\documentclass[11pt]{article}
\usepackage{hyphenat}
\usepackage[kerning, tracking, spacing]{microtype}
\usepackage[margin=1in]{geometry}
\usepackage{amsmath, amsthm,amssymb}
\usepackage{mathrsfs}
\usepackage{enumitem}
\usepackage{graphicx}
\usepackage{mathdots}
\usepackage{hyperref,url}
\usepackage{tikz, tikz-cd}
\usetikzlibrary{matrix}
\usetikzlibrary{shapes}
\usetikzlibrary{arrows,arrows.meta,decorations.markings, tikzmark}
\usepackage{mathtools}
\usepackage{ stmaryrd }
\usepackage[normalem]{ulem}
\usepackage[utf8]{inputenc}

\usepackage{xcolor}
\hypersetup{
    colorlinks,
    linkcolor={red!50!black},
    citecolor={blue!50!black},
    urlcolor={blue!80!black}
}

\pgfarrowsdeclare{bad to}{bad to}
{
  \pgfarrowsleftextend{-2\pgflinewidth}
  \pgfarrowsrightextend{\pgflinewidth}
}
{
  \pgfsetlinewidth{0.8\pgflinewidth}
  \pgfsetdash{}{0pt}
  \pgfsetroundcap
  \pgfsetroundjoin
  \pgfpathmoveto{\pgfpoint{-3\pgflinewidth}{4\pgflinewidth}}
  \pgfpathcurveto
  {\pgfpoint{-2.75\pgflinewidth}{2.5\pgflinewidth}}
  {\pgfpoint{0pt}{0.25\pgflinewidth}}
  {\pgfpoint{0.75\pgflinewidth}{0pt}}
  \pgfpathcurveto
  {\pgfpoint{0pt}{-0.25\pgflinewidth}}
  {\pgfpoint{-2.75\pgflinewidth}{-2.5\pgflinewidth}}
  {\pgfpoint{-3\pgflinewidth}{-4\pgflinewidth}}
  \pgfusepathqstroke
}

\DeclareMathAlphabet{\mymathbb}{U}{BOONDOX-ds}{m}{n}

\usepackage{quiver}

\def\Mbar{\overline{\mathcal{M}}}

\def\Mc{\mathcal{M}}

\def\CH{\mathsf{CH}}
\def\CHop{\mathsf{CH}_{\mathsf{op}}}

\def\R{\mathsf{R}}
\def\CO{\mathcal{O}}
\def\A{\mathcal{A}}

\def\P{\mathcal{P}}
\def\X{\mathcal{X}}
\def\id{\mathrm{id}}

\def\EE{\mathbb{E}}

\def\PP{\mathbb{P}}
\def\QQ{\mathbb{Q}}

\def\ZZ{\mathbb{Z}}
\def\LL{\mathbb{L}}

\def\GL{\mathrm{GL}}

\def\RPC{\mathrm{RPC}}

\def\PL{\mathrm{PL}}
\def\Aut{\mathrm{Aut}}

\def\hom{\mathcal{H}om}
\DeclareMathOperator{\Hom}{Hom}
\DeclareMathOperator{\lcm}{lcm}

\def\Gm{\mathbb{G}_m}
\def\Glog{\mathbb{G}_{m,\log}}
\def\sR{\mathsf{sR}}

\def\codim{\mathrm{codim}}
\def\vir{\mathrm{vir}}

\def\ch{\mathrm{ch}}

\def\dx{\frac{d}{dx}}

\def\im{\mathrm{Im}\,}

\def\Jbar{\overline{J}}

\def\Td{\mathrm{td}}
\def\fm{\mathfrak{F}}
\def\dbcoh{\mathrm{D}^b_{\mathrm{coh}}}

\def\Pbar{\overline{\mathcal{P}}}

\def\uniDR{\mathsf{uniDR}}

\def\cR{\mathcal{R}}

\def\gp{\mathrm{gp}}

\def\Abar{\overline{\mathcal{A}}}
\def\Xbar{\overline{\mathcal{X}}}

\def\Coh{\mathrm{Coh}}
\def\FM{\mathsf{FM}}

\newcommand{\RR}{\mathbb{R}}
\DeclareMathOperator{\Span}{Span}

\newcommand{\trop}{\mathrm{trop}}
\renewcommand{\log}{\mathrm{log}}

\newcommand{\CHM}{\mathrm{CHM}}
\newcommand{\p}{\mathfrak{p}}
\newcommand{\barA}{\overline{A}}

\newcommand{\sfpp}{\mathsf{PP}}

\newcommand{\sfS}{\mathsf{S}}
\newcommand{\logA}{{A^{\log}}}
\newcommand{\tropA}{{A^{\trop}}}
\newcommand{\fB}{\mathfrak{B}}
\newcommand{\ext}{\mathcal{E}xt}

\newcommand{\sbar}{\overline{s}}
\newcommand{\sh}{\sigma}
\newcommand{\shb}{\overline{\sh}}
\newcommand{\corr}{\mathrm{Corr}}
\newcommand{\w}{\mathfrak{w}}
\newcommand{\St}{\widetilde{\mathsf{S}}}
\newcommand{\sPP}{\mathsf{sPP}}
\newcommand{\spp}{\sPP}

\newcommand{\K}{\mathcal{K}}
\newcommand{\supp}{\mathrm{Supp}\,}
\newcommand{\Pic}{\operatorname{Pic}}
\newcommand{\diff}{\operatorname{dif}}

\theoremstyle{definition}
\newtheorem{definition}{Definition}[section]
\newtheorem{theorem}[definition]{Theorem}
\newtheorem{example}[definition]{Example}
\newtheorem{proposition}[definition]{Proposition}
\newtheorem{corollary}[definition]{Corollary}
\newtheorem{lemma}[definition]{Lemma}

\newtheorem{remark}[definition]{Remark}

\newtheorem{assumption}[definition]{Assumption}

\newtheorem{conjecture*}{Conjecture}
\newtheorem{question*}[conjecture*]{Question}
\newtheorem{theoremintro}{Theorem}

\title{Weight decomposition for toroidal abelian fibrations}

\author{Younghan Bae, Jeremy Feusi, Aitor Iribar L\'opez, Sam Molcho}
\date{}
\begin{document}

\maketitle

\begin{abstract}
     We study the action of the rational multiplication by $N$ map on the Chow groups of degenerations of principally polarized abelian varieties of torus rank at most one, as well as its interaction with the Fourier transform. As applications, we compute the class of the unit section, prove a generalized weight decomposition of the relative Chow motive of the universal family of such degenerations, and completely determine its tautological ring.
   \end{abstract}

\tableofcontents

\section{Introduction}

Let $\pi : \X_g \to \A_g$ be the universal abelian scheme over the moduli stack of principally polarized abelian varieties of dimension $g$ over a field $\Bbbk$. There are two distinguished features of algebraic cycles on $\X_g$:
\begin{enumerate}
\item[(i)] The rational Chow ring $\CH^*(\X_g)$ admits a multiplicative splitting whose direct summands are eigenspaces of the endomorphism $N^*$ induced by the multiplication by $N$ map (\cite{Beauville86,DM91}).

\item[(ii)] The tautological ring $\R^*(\X_g)\subseteq \CH^*(\X_g)$ generated by the Chern classes $\lambda_i = c_i(\EE)$ of the Hodge bundle and the Theta divisor has a complete description (\cite{vdG99}).
\end{enumerate}
Similar results hold on the $s$-fold self-fiber products $\X_g^s \to \A_g$. Some of the key ingredients in proving these statements are the relative group structure and the Fourier--Mukai transform.

The purpose of this paper is to extend the above results to the simplest partial toroidal compactification of $\X_g$, which was introduced by Mumford in \cite{mumford_kodaira}. Let $\A_g\subset \A_g'$ be the canonical partial compactification parametrizing degenerations of torus rank at most one, and let $\pi : \X_g' \to \A_g'$ be the universal family. This family is no longer a relative group scheme, but it is a birational model of a log abelian scheme in the sense of Kajiwara--Kato--Nakayama \cite{KKN1, KKN2}. We prove generalizations of both (i) and (ii) using the Fourier transform and logarithmic geometry. Along the way, we develop the theory of weights for these toroidal abelian fibrations.

\subsection{The class of the unit section}

Our first result is a closed formula for the fundamental class of the unit
section $e: \A_g' \to \X_g'$ in terms of tautological classes. As in the
case of abelian schemes, this result is the key input for our study of
algebraic cycles on $\X_g'$.

We introduce natural classes on $\X_g'$ by extending tautological classes from the interior and utilizing the toroidal structure of the boundary:
\begin{itemize}
    \item Let $\theta$ be the unique extension of the Theta divisor on $\X_g$ to $\X_g'$ which is trivialized along the unit section.
    \item The boundary $\partial \mathcal X_g'\subset \X_g'$ is a divisor with self-intersection. Let $i: Y\to \X_g'$ denote the normalization of a certain \'etale double cover of the boundary $\partial \X_g' \subset \X_g'$, and let $N_i$ be the normal bundle of $i$.
    \item Let $C=\mathcal X_{g-1}\times_{\mathcal A_{g-1}}\mathcal X_{g-1}$. There exists a degree-two morphism \(j: C \to \X_g'\), whose image is the singular locus of $\partial \X_g'$. Let $N_j$ be the normal bundle to $j$.
\end{itemize}
Consider the following formula
\begin{align}
\sfS_g^c
=
\Bigg[
\exp(\theta) \cup
&\Bigg(
1
+
\frac{1}{2}\,
i_* \Bigg(\sum_{k\ge 1}
\frac{(-1)^k B_{2k}}{2^k k!}
\,c_1(N_i)^{k-1}\Bigg)\label{eq:S} \\[4pt]
&+
\frac{1}{2}\,
j_* \Bigg(
\sum_{k_1,k_2\ge 1}
\frac{(-1)^{k_1+k_2}}{2^{k_1+k_2}}
\sum_{m=0}^{2k_2}
\binom{2k_2}{m}
B_{2k_1+m}
\frac{\alpha_1(N_j)^{k_1-1}\cup\alpha_2(N_j)^{k_2-1}}{k_1!k_2!}
\Bigg)
\Bigg)
\Bigg]_{\codim = c}\,,\nonumber
\end{align}
Here $\alpha_1(N_j)$ and $\alpha_2(N_j)$ are the Chern roots of $N_j$, $B_k$ is the $k$-th Bernoulli number, and $[-]_{\codim =c}$ denotes the codimension $c$ part of the class.

\begin{theoremintro}\label{thm:unit}
	The fundamental class $[e]$ of the unit section $e:\A_g'\to \X_g'$ is given by the formula
    \[
    [e]= \sfS^g_g \in\CH^g(\X_g')\,.
    \]
\end{theoremintro}

Theorem \ref{thm:unit} generalizes the formula on the semi-abelian locus $\X_g^\circ \subset \X_g'$ from \cite{BMP}; the formula itself is closely related to \cite{BHPSS}.
In \cite{GZ}, Grushevsky--Zakharov proposed a formula $\mathsf{GZ}_g$ for $[e]$; see Remark \ref{rmk:gz} for a discussion of an oversight in the proof of this formula in \cite{GZ}. After combinatorial manipulation, we show that $\sfS_g^g = \mathsf{GZ}_g$, thereby proving their proposed formula.

\subsection{Motivic decomposition and multiplicativity}

Our second result gives a motivic decomposition of $\pi : \X_g' \to \A_g'$ in the category of relative Chow motives constructed by Corti--Hanamura \cite{CortiHanamura}.

The starting point is a generalization of the ``multiplication by $N$'' map on abelian schemes. For a positive integer $N$, consider the locally closed locus
\[
\tau_N := \{(x,y) \in \X_g\times_{\A_g}\X_g : Nx-y=0\} \subset \X_g'\times_{\A_g'} \X_g'\,.
\]
Let $[\overline{\tau_N}] \in \CH^g(\X_g' \times_{\A_g'}\X_g')$ be the fundamental class of its Zariski closure, and $[\overline{\tau_N}]^t$ be its transpose. As a relative correspondence, the class $[\overline{\tau_N}]^t$ defines a (rational) ``multiplication by $N$'' map
\begin{equation}\label{eq:intro_n}
    [N]^* : \CH^*(\X_g') \to \CH^*(\X_g')\,.
\end{equation}

\begin{theoremintro}\label{thm:dec}
    \begin{enumerate}[label = (\alph*)]
        \item The class $[\overline{\tau_N}]^t$ is polynomial in $N$ of degree $2g$. That is, we can write
        \[
        [\overline{\tau_N}]^t = \w_0 + \cdots + N^{2g}\w_{2g}, \quad \w_k \in \CH_*(\X_g'\times_{\A_g'}\X_g')\,.
        \]
        \item For $0\leq k \leq 2g$, let $h_k(\X_g') := (\X_g', \w_k,0)$. Then there exists a decomposition of motives
        \[
        h(\X_g') = h_0(\X_g') \oplus \cdots \oplus h_{2g}(\X_g') \in \CHM(\A_g')\,.
        \]
		Moreover, when the base field is algebraically closed of characteristic zero, this decomposition gives a motivic lift of the decomposition of $R\pi_* \QQ$ into shifted semisimple perverse sheaves.
        \item Let $W_k \CH^*(\X_g'):= \bigoplus_{w\leq k} \CH^*(h_w(\X_g'))$ be the filtration by weights for $[N]^*$. The filtration $W_\bullet\CH^*(\X_g')$ is multiplicative.
    \end{enumerate}
\end{theoremintro}

Therefore, $\CH^*(\X_g')$ additively decomposes into eigenspaces for the operator \eqref{eq:intro_n}, generalizing results by Beauville \cite{Beauville86} and Deninger--Murre \cite{DM91}. Point~(b) proves the Corti--Hanamura conjecture \cite{CortiHanamura} for $\X_g' \to \A_g'$.

\subsection{The tautological ring of \texorpdfstring{$\X_g'$}{Xg'}}\label{sec:intro_taut}

Our third result computes the tautological ring $\R^*(\X_g') \subseteq \CH^*(\X_g')$. Let $\spp^{\ast}(\X_g')$ denote the ring of strict piecewise polynomials on the cone over the dual intersection complex of the normal crossings pair $(\X_g',\partial \X_g')$. There is a natural map
\begin{equation}\label{eq:xi}
    \Xi : \spp^{\ast}(\X_g')[\lambda_1,\ldots,\lambda_g,\theta] \to \CH^*(\X_g'),
\end{equation}
sending $\lambda_i$ to the $i$-th Chern class of the Hodge bundle $\EE$ and $\theta$ to the Theta divisor. We let $\R^*(\X_g')$ be the image of $\Xi$. An example of a piecewise polynomial is the fundamental class of the boundary $\partial \mathcal X_g'$, which we denote by $D \in \CH^1(\X_g')$. See \S \ref{sec:taut_def} for details.

The elements (relations) in the kernel of $\Xi$ come in three flavors:
\begin{itemize}
    \item[(a)] Relations coming from the base:
    \begin{itemize}
    \item[$\bullet$] $\lambda_{g-1}\cup\gamma$ (if $g > 1$) and $\lambda_g\cup\gamma$, for all
    $\gamma \in \spp^{>0}(\X_g')$.

    \item[$\bullet$] The homogeneous components of Mumford's relation
    \[
    c(\mathbb E \oplus \mathbb E^\vee)-1
    =
    (1+\lambda_1+\cdots+\lambda_g)\cup
    (1-\lambda_1+\cdots+(-1)^g\lambda_g)-1.
    \]
    \item[$\bullet$] The class
    \begin{equation}\label{eq:Pixtonclass}
        \lambda_g - \frac{B_{2g}}{2^g g!}D^g.
    \end{equation}
    \end{itemize}
\end{itemize}
These were proved by van der Geer \cite{vdG99}, Esnault--Viehweg \cite{EV04}, and Ekedahl--van der Geer \cite{EvdG05}, respectively.

\begin{itemize}
    \item[(b)] Relations coming from the boundary: The boundary of the toroidal compactifications has a recursive structure and is a union of toroidal compactifications of torus bundles over $\mathcal X_h^r$ for different $h$ and $r$. In our case, the normalization of the boundary $Y$ is a $\mathbb P^1$-bundle over the twofold product of the universal family in dimension $g-1$:
    $$
    \pi_{\mathbb P} : Y \to \X_{g-1}^2.
    $$
    The relations in the tautological ring of $\X_{g-1}^2$ were determined in \cite{GZ} and will be reviewed in \S\ref{sec:proofA}. Pulling back these relations along $\pi_{\mathbb P}$ and pushing forward along $i$ yields relations in $\R^*(\mathcal{X}_g')$.
    \item[(c)] Weight vanishing relations: Consider the tautological expression:
\end{itemize}
    \begin{equation}\label{eq:St}
    \widetilde{\sfS}_g^c := \left[
    \exp(\theta)\cup
    \left(
    1-
    \frac{1}{2}j_*
    \sum_{k_1,k_2\ge 1} \frac{(-1)^{k_1+k_2}(2k_1)!(2k_2)!}{2^{k_1+k_2}(2k_1+2k_2)!}
    \frac{\alpha_1(N_j)^{k_1-1}\cup\alpha_{2}(N_j)^{k_2-1}}{k_1!k_2!}
    \right)
    \right]_{\codim = c}.
    \end{equation}
    \begin{itemize}
        \item[] We show that the following holds:
    \end{itemize}

\begin{theoremintro}\label{thm:rel}
    If $c > g$, then $\widetilde{\sfS}_g^c = 0$ in $\CH^c(\X_g')$.
\end{theoremintro}

The formula \eqref{eq:St} is the coefficient of $N^{2c}$ in the polynomial $[N]^*(\theta^c/c!)$. Theorem \ref{thm:rel} then follows from Theorem \ref{thm:dec}~(b), as it implies that every class of weight greater than $2g$ vanishes.

Theorem \ref{thm:rel} is a generalization of the relation $\theta^{g+1}=0$ on $\mathcal X_g$ and is the analogue of the \emph{universal DR} relations \cite[Theorem 0.8]{BHPSS} for compactified Jacobians.

There are no further relations between tautological classes. More specifically, let
$$
\mathfrak I_g\subseteq \sPP(\X_g')[\lambda_1, \ldots , \lambda_g, \theta]
$$
be the ideal generated by the relations in (a), (b), (c) above.

\begin{theoremintro}\label{thm:taut}
    The kernel of the map $\Xi$ in \eqref{eq:xi} is $\mathfrak{I}_g$. Hence the tautological ring has the presentation
    \[
    \R^*(\X_g') \cong \spp^{\ast}(\X_g')[\lambda_1,\cdots, \lambda_g,\theta]/\mathfrak{I}_g\,.
    \]
\end{theoremintro}

Finally, we prove that the Fourier transform preserves the tautological ring:
By \cite[Theorem 1.5]{BMP}, there exists an extension $\Pbar$ of the
Poincar\'e line bundle which induces an autoequivalence
\begin{equation}\label{eq:FM_intro}
    \FM : \dbcoh(\X_g') \xrightarrow{\cong} \dbcoh(\X_g').
\end{equation}
Consider the Chow-theoretic realization of $\FM$,
\begin{equation}\label{eq:F_intro}
    \fm : \CH^*(\X_g') \xrightarrow{\cong} \CH^*(\X_g'),
\end{equation}
with respect to an appropriate Todd class normalization, and let $\fm^{-1}$
denote its inverse.

\begin{theoremintro}\label{thm:FourierIso}
    The maps $\fm$ and $\fm^{-1}$ preserve the tautological subring
    $\R^*(\X_g')\subset \CH^*(\X_g')$.
\end{theoremintro}

In fact, we provide an algorithm for computing the Fourier transform of tautological classes and
use this to prove closure of the tautological ring under pushforward maps from the $s$-fold relative
fiber product of $\X_g'$ (Theorem \ref{thm:small} and Remark \ref{rmk:pushforward}).

\subsection{Weights of tautological classes}

We consider a small resolution of the $s$-fold relative fiber product
\[
\X_g'^{(s)} \to
\X_g' \times_{\A_g'} \cdots \times_{\A_g'} \X_g'\,,
\]
generalizing \cite[Section 3]{BMP}. A direct generalization of \eqref{eq:intro_n} defines
the relative multiplication by $N$ map
\[
[1 \times \cdots \times N]^* : \CH^*(\X_g'^{(s)}) \to \CH^*(\X_g'^{(s)})\,.
\]

The main input for studying algebraic cycles on $\X_g'$ and proving Theorems \ref{thm:unit}, \ref{thm:dec}, \ref{thm:rel} and \ref{thm:FourierIso} is a method for controlling the action of the multiplication by $N$ maps on tautological classes on these small resolutions of relative self-products of $\X_g'$.

We introduce some divisorial classes on $\X_g'^{(s)}$: Let $p_i : \X_g'^{(s)} \to \X_g'$ be the projection onto the $i$-th component, and let
$p_{i,j} : \X_g'^{(s)} \to \X_g'^{(2)}$ be the projection onto the $(i,j)$-th
components. Set
\[
\theta_i := p_i^*\theta,\quad
\ell_{i,j} := p_{i,j}^*\ell
\in \CH^1(\X_g'^{(s)})\,,
\]
where $\theta$ is the Theta divisor and $\ell$ is the first Chern class of the
extended Poincar\'e line bundle.

Over $\mathcal A_g$, we have
$$
[1\times \ldots \times N]^*(\theta_s^m
    \cup \ell_{1,s}^{k_1}
    \cup \cdots
    \cup \ell_{s-1,s}^{k_{s-1}}) = N^{2m+\sum_{i=1}^{s-1}k_i} ( \theta_s^m
    \cup \ell_{1,s}^{k_1}
    \cup \cdots
    \cup \ell_{s-1,s}^{k_{s-1}})\,.
$$
After passing to the partial compactification, these classes are no longer of pure weight, but we can locate them in the filtration by weights $W_\bullet\CH^*((\X_g')^{(s)})$:

\begin{theoremintro}\label{thm:weight_main}
    Let $m \geq 0$, $k_1,\ldots,k_{s-1} \geq 0$, and
    $\varphi \in \spp^{d}(\X_g'^{(s)})$. For
    \[
    \alpha
    =
    \theta_s^m
    \cup \ell_{1,s}^{k_1}
    \cup \cdots
    \cup \ell_{s-1,s}^{k_{s-1}}
    \cup \varphi,
    \]
    the class $[1 \times \cdots \times N]^*(\alpha)$ is polynomial in $N$ of
    degree at most $2m+\sum_{i=1}^{s-1} k_i+d$.
\end{theoremintro}

Since Theorem \ref{thm:weight_main} is used in the proof of Theorem \ref{thm:dec}, the weight of each monomial appearing in Theorem \ref{thm:weight_main} must be computed independently without relying on the multiplicativity of weights.
Our proof uses logarithmic resolutions of the multiplication by $N$ map and Brion's formula for piecewise polynomial classes \cite{Brion}.
A more precise statement is proved in \S\ref{sec:wt}; see Theorem \ref{thm:wt}. A related combinatorial framework using $b$-divisors has recently been developed by Botero--Burgos Gil--Holmes--de Jong in \cite{BGHdJ}.

Note that by Theorem \ref{thm:dec}~(b), the calculations of Theorem \ref{thm:weight_main} compute the perverse degrees of tautological classes.

\subsection{Idea of the proofs}

\subsubsection{Sketch of the proof of Theorem \ref{thm:unit}}

Our first observation is that the Fourier transform of the line bundle\footnote{Since $\CO(\theta)$ is only a $\QQ$-line bundle, the argument below must be slightly modified; see \S\ref{sec:FMn}.} $\CO(\theta)$ has a formula analogous to the one for abelian schemes. Consider the autoequivalence
\eqref{eq:FM_intro}. Using the explicit description of the kernel $\Pbar$ as in \cite{BMP}, we have
\begin{equation*}
    \FM (\CO(\theta))
    \cong
    \CO(-\theta) \otimes \pi^*\Mc
\end{equation*}
for some line bundle $\Mc$ on $\A_g'$. Let $\fm$ be the Chow-theoretic Fourier transform \eqref{eq:F_intro}. Applying the Baum--Fulton--MacPherson isomorphism, the above equality in the derived category implies that
\begin{equation}\label{eq:intro2}
    \fm\bigl(\exp(-\theta)\cup \Td^\vee(\cR_\pi)^{-1}\bigr)
    =
	(-1)^g\exp(\theta-D/8),
\end{equation}
where $\cR_\pi$ is the residue sheaf of $\pi :\X_g'\to \A_g'$ and $D$ is the class of the boundary divisor in $\A_g'$.

As observed in \cite{BMP}, the class of the unit section is the Fourier inverse
of the fundamental class:
\[
[e] = \fm^{-1}(1).
\]
Up to Todd class contribution and signs, the leading contribution to the right-hand side is $\fm(1)$. The main difficulty in computing $\fm(1)$ from
\eqref{eq:intro2} is that it is hard to isolate this contribution unless $\CH^*(\X_g')$ carries a meaningful {\em second grading}, in addition to the grading by codimension.

We overcome this difficulty by studying the rational multiplication by $N$ map $[N]^{\ast}$ defined in \eqref{eq:intro_n}. Theorem \ref{thm:weight_main} controls the action of $[N]^{\ast}$ on tautological classes. This enables
us to compute the contribution of $\fm(1)$ in \eqref{eq:intro2} modulo tautological classes supported on the codimension-two stratum of $\X_g'$.
After a careful analysis of the intersection theory of this stratum, we obtain Theorem \ref{thm:unit}.

\subsubsection{Sketch of the proof of Theorem \ref{thm:dec}}

The polynomiality of $[\overline{\tau_N}]$ in $N$ follows from the fact that the class of the unit section $[e]$ lies in the tautological ring; see Theorem
\ref{thm:unit}. Consider the rational pullback along
\[
\phi_N : \X_g' \times_{\A_g'} \X_g' \dashrightarrow \X_g',
\quad
(x,y) \mapsto Nx-y,
\]
defined in the same way as $[N]^*$. Then $[\overline{\tau_N}]$ is obtained as the rational pullback $[\phi_N]^*([e])$. Hence the polynomiality follows from the weight bounds for tautological classes proved in \S\ref{sec:wt}.

The motivic decomposition of $\X_g' \to \A_g'$ follows from the multiplicativity of the classes $[\overline{\tau_N}]$:
\[
[\overline{\tau_N}] \circ [\overline{\tau_M}]
=
[\overline{\tau_{NM}}] \in \CH_*(\X_g' \times_{\A_g'}\X_g')\,,
\]
where $\circ$ denotes the composition of relative correspondences.
Since $[\overline{\tau_1}]$ is the class of the relative diagonal, this equality shows that the coefficients $\{\w_i\}_{i=0}^{2g}$ of $[\overline{\tau_N}]^{t}$ form mutually orthogonal projectors.

To prove the multiplicativity of the weight filtration, we establish a ``$P=W$'' phenomenon\footnote{Here, the $W$-side is {\em not} the weight filtration on the Betti moduli in the context of the non-abelian Hodge correspondence.}
\[
P_k\CH^*(\X_g') = W_k\CH^*(\X_g'),
\]
where the Chow-theoretic perverse filtration is defined using the images of the codimension at most $k$ components of $\fm$. The inclusion $P \subseteq W$ follows from a degree bound for $[1\times N]^*(\fm)$ in $N$ (Theorem \ref{thm:weight_main}).
The inclusion $W \subseteq P$ follows by studying $[N]_*$, the pushforward along (a resolution of) the rational multiplication by $N$ map on the Chow group, using the results of \S\ref{sec:wt}. The $P$-filtration is multiplicative by Maulik--Shen--Yin \cite{PerverseFourierMSY25}, and hence the multiplicativity of the $W$-filtration follows.

\subsubsection{Sketch of the proof of Theorem \ref{thm:taut}}

First, the relation \eqref{eq:Pixtonclass} on $\A_g'$ follows by pulling back the equality in Theorem \ref{thm:unit} along the unit section $e$. The complete description of $\R^*(\A_g')$ then follows from the existence of a smooth toroidal compactification of $\A_g'$, van der Geer's result in \cite{vdG99} and a vanishing result from \cite{CMOP}.

For $\X_g'$, the inclusion $\mathfrak I_g \subset \ker \Xi$ follows from Theorem \ref{thm:rel}. To prove the reverse inclusion, we analyze the restrictions of tautological classes on the strata of $\X_g'$ and the proper pushforward $\pi_* : \R^*(\X_g') \to \R^*(\A_g')$ using carefully chosen elements spanning $\R^*(\X_g')$.

\subsection{Future directions}

We discuss several open directions concerning the structure of algebraic cycles on general toroidal compactifications of $\X_g$ and $\A_g$. Although several key arguments in this paper rely on the relatively simple geometry of $\X_g'\to \A_g'$, we expect that some of the methods extend to deeper toroidal compactifications, at the cost of a heavier combinatorial burden.

Let $ \pi : \overline{\mathcal X}_g \to \overline{\mathcal A}_g$ be a toroidal partial compactification of the universal family, with both $\overline{\mathcal X}_g$ and $\overline{\mathcal A}_g$ smooth and $\pi$ flat. Define the small tautological rings
$$
\sR^*(\overline{\mathcal X}_g) = \spp^*(\overline{\mathcal X}_g)[\theta]\quad \text{and}\quad \sR^*(\overline{\mathcal A}_g) = \sPP^*(\overline{\mathcal A}_g),
$$
where $\theta$ is a suitable extension of the Theta divisor.

We view the following four questions as foundational challenges in understanding the intersection theory of toroidal compactifications of $\A_g$ and $\X_g$.

\begin{conjecture*}\label{conj1}
    The pushforward map $\pi_*$ preserves small tautological rings:
    $$
    \pi_* : \sR^*(\overline{\mathcal X}_g) \to \sR^*(\overline{\mathcal A}_g).
    $$
\end{conjecture*}

\begin{conjecture*}\label{conj2}
    Let $\overline{e}$ be the Zariski closure of the unit section. Then $[\overline{e}] \in \sR^g(\overline{\mathcal X}_g)$.
\end{conjecture*}

\begin{conjecture*}\label{conj3}
    There exists an exhaustive filtration on $\CH^*(\Xbar_g)$ which extends the weight filtration.
\end{conjecture*}

\begin{question*}\label{conj4}
    Can the rings $\sR^*(\overline{\mathcal X}_g)$ and $\sR^*(\overline{\mathcal A}_g)$ be described by \emph{natural} generators and relations?
\end{question*}

In addition to our main results, further evidence for the above conjectures comes from the study of fine compactified Jacobians $\Jbar_{g,n}$ over the moduli space of stable curves $\Mbar_{g,n}$ \cite{Caporaso94,KP19}. For these toroidal abelian fibrations, Conjecture \ref{conj2} was proved in \cite{BHPSS} (see also \cite[Proposition 8.3]{BMP}). Conjecture \ref{conj1} will be studied in \cite{BM}, and Conjecture \ref{conj3} in \cite{BP}.

\subsection{Structure of the paper}

\S\ref{sec:tr1} is devoted to preliminaries and to the relation between $\X_g'$ and the universal log abelian scheme.
In \S\ref{sec:Fourier}, we show that the extension of the Poincar\'e line bundle constructed in \cite{BMP} endows $\X_g'$ with the structure of a self-dualizable abelian fibration, and we compute the Fourier inverse of $\exp(\theta)$.
In \S\ref{sec:wt}, we develop a theory for computing the weights of tautological classes using log abelian schemes and prove Theorem \ref{thm:weight_main}. This is the technical heart of the paper.
We then prove Theorem \ref{thm:unit} in
\S\ref{sec:unit}, Theorem \ref{thm:dec} in \S\ref{sec:weight}, and Theorems \ref{thm:rel} and \ref{thm:taut} in \S\ref{sec:taut}.

\subsection*{Conventions}

\begin{itemize}
	\item We work over a base field $\Bbbk$.
	\item All Chow groups are taken with $\QQ$-coefficients.
	\item All monoids and log structures will always be fine and saturated (f.s.) in the sense of \cite{Kato}.
	\item For a Deligne--Mumford stack $X$, let $\dbcoh(X)$ denote the bounded derived category of coherent sheaves on $X$. Let $K_0(X)$ (resp. $K^0(X)$) denote the Grothendieck group of coherent sheaves (resp. locally free sheaves).
	\item If $f$ is a polynomial in an indeterminate $u$, we write $[u ^{m}]f$ for the coefficient of $u ^{m}$ in $f$.
	\item We write $\Pbar^{\vee}:=\hom(\Pbar,\CO_{\barA^{2}})$ for the dual of the Poincar\'e line bundle and $\Pbar^{-1}:=\hom(\Pbar,\CO_{\barA^2}) \otimes \pi_2^*\omega_\pi [g]$ for the kernel of the inverse Fourier transform. Similarly, for $\mathcal{P}$, $\mathcal{P}^{\circ}$, etc.
\end{itemize}

\subsection*{Acknowledgments}
We would like to thank Gerard van der Geer, Samuel Grushevsky, David Holmes, Davesh Maulik, Rahul Pandharipande, Alex Perry, Johannes Schmitt, Pim Spelier, Junliang Shen and Qizheng Yin for helpful discussions. Y.B. is especially grateful to Aaron Pixton for their collaboration, which inspired many of the ideas developed in this paper.

Large Language Models were used in the preparation of this manuscript. They helped with image creation, handling the combinatorics and finding typos. The authors have reviewed all of these contributions and maintain full responsibility for the contents of this paper.

Y.B. was supported by the June E Huh Visiting Fellowship and the AMS-Simons travel grant. A.I.L. and J.F. were supported by SNF-200020-219369. We thank Sapienza University for supporting research visits of A.I.L. and Y.B. during which parts of this project were completed.

\section{Torus rank at most one toroidal abelian fibrations}\label{sec:tr1}

\subsection{The partial toroidal compactification of $\X_g$}\label{sec:mum}

For $g\ge 1$, let $\A_g/\Bbbk$ be the moduli stack of principally polarized abelian varieties of dimension $g$. It is a smooth, geometrically connected Deligne--Mumford stack of dimension $g(g+1)/2$. Let $\X_g \to \A_g$ denote the universal family, and $\X_g^\vee$ the dual abelian scheme of $\X_g$. We use the following notation:

\begin{itemize}
    \item We write $C$ for the fiber product of the universal family and its dual in dimension $g-1$:
    $$
    C := \X_{g-1} \times_{\A_{g-1}} \X_{g-1}^\vee
    $$
    which we think of as the moduli space of $(X, \lambda, x, L)$, where $\lambda: X \to X^\vee$ is a principal polarization, $x \in X$ and $L \in X^\vee$.
    \item We write $\P_C$ for the Poincar\'e line bundle on $C$ trivialized along the unit sections. Let $Y :=\PP_C(\P_C \oplus\CO_C)$ be the projective closure of the total space of $\P_C$ and $s_0, s_\infty : C \to Y$ be the zero and the infinity sections.
    \item We write $\sh : C \to C$ for the shift operator
    \begin{equation}\label{eq:sh}
        \sh (X, \lambda, x, L)=(X, \lambda, x + \lambda^{-1}(L), L)\,,
    \end{equation}
    and $\sbar_{\infty}$ for the shifted infinity section:
    $$
    \sbar_{\infty} = s_{\infty}\circ \sh\,.
    $$
    The inverse of the shift operator is denoted by $\shb$.
    \item We write $D:=\X_{g-1}^\vee$ for the dual abelian scheme.
\end{itemize}

The universal family $\X_g \to \A_g$ admits a canonical partial compactification constructed as follows (see also \cite{mumford_kodaira}): Let $\Abar_g$ be a
smooth toroidal compactification of $\A_g$. It has a natural morphism $q : \Abar_g \to \Abar_g^{\mathrm{Sat}}$ to the Satake compactification.
The inverse image $\A_g' := q^{-1}(\A_g \sqcup \A_{g-1})$ is independent of the choice of $\Abar_g$; we call it the Mumford partial compactification.

By \cite{FC}, the universal abelian variety $\X_g \to \A_g$ extends to a semi-abelian scheme
$$
\X_g^{\circ} \to  \A_g'\, ,
$$
and $\X_g^{\circ}$ admits a minimal toroidal compactification $\X_g'$, constructed in \cite{Namikawa,FC}; see also \cite{Alexeev,EdGS}.
It is an open subset of the Delaunay--Voronoi compactification $\Xbar_g^{\operatorname{D--V}} \to \Abar_g^{\operatorname{vor}}$ constructed in \cite{namikawa76}; see \cite{EdGS}.

It will be convenient to introduce the following notation:
\begin{itemize}
    \item Let $B:= \A_g'$, and let $G:= \X_g^\circ \to B$ be the universal semi-abelian group scheme.
	\item Let $\barA:=\X_g'$ be the minimal toroidal compactification of $G$ and let $\pi :\barA \to B$ be the extension of $\X_g\to \A_g$.
	\item Let $A^{\circ}:=\X_g$ and $B^{\circ}:=\mathcal{A}_g$.
    \item If $X\to B$ is one of the spaces used below (e.g., $\barA$, $G$, etc.), write
    $X^s:=X\times_B\cdots\times_BX$ for its $s$-fold relative fiber
    product.
\end{itemize}

The natural action $G\curvearrowright G$ extends to an action $\mu : G\times_B
\barA \to \barA$, and the tuple $(\barA, B, G)$ is an Ng\^o fibration: it is a
weak $\delta$-regular abelian fibration; see \cite[\S 7.1]{Ngo}.
The boundary strata of $\pi :\barA \to B$ are studied in
\cite{mumford_kodaira,EvdG05,GZ}:

\begin{theorem}\label{thm:bdy}
    There are surjective finite maps of degree $2$:
    $$
    \iota : D \to \partial B := B \setminus\A_g\, , \qquad i : Y \to \partial \barA := \barA \setminus \X_g
    $$
    that are compatible with the structure map $Y \to D$:
    \begin{equation}\label{eqn:geom_setup}
    \begin{tikzcd}
	{Y} & {\partial \barA \subset \barA} \\
	{D} & {\partial B \subset B.}
	\arrow["i", from=1-1, to=1-2]
	\arrow["\pi_Y", from=1-1, to=2-1]
	\arrow["\pi", from=1-2, to=2-2]
	\arrow["\iota", from=2-1, to=2-2]
    \end{tikzcd}
    \end{equation}
    Moreover, the map $i$ is the normalization of an \'etale double cover of the boundary. There is a surjective finite \'etale map from $C$ to the singular locus of the boundary $\partial \barA$,
    $$
    j: C \to (\partial \barA)^{\mathrm{sing}}\subset \barA
    $$
    such that we can identify $\partial \barA$, up to an \'etale double cover, with the quotient of $Y$ obtained by gluing the zero section $s_0$ and the shifted infinity section $\sbar_\infty$, and
    \begin{equation}\label{eq: j = is0=isinf}
        j = i \circ s_0=i\circ \sbar_{\infty}\, .
    \end{equation}
\end{theorem}

\begin{figure}
	\begin{center}
	\begin{minipage}[c]{0.52\textwidth}
		\centering
		\begin{tikzpicture}[scale=0.82,transform shape,every node/.style={font=\small}]
			\draw[thick] (0.25,-2.35) -- (7.15,-2.35);
			\fill (5.23,-2.35) circle (1.6pt);
			\node[below] at (0.65,-2.42) {$B$};
			\node[below] at (5.23,-2.42) {$D$};

			\draw[thick,black!35]
				(0.63,-0.40) -- (1.88,0.02) -- (2.08,3.17) -- (0.83,2.75) -- cycle;
			\draw[thick,black!55]
				(1.93,-0.40) -- (3.18,0.02) -- (3.38,3.17) -- (2.13,2.75) -- cycle;
			\draw[thick,black!75]
				(3.23,-0.40) -- (4.48,0.02) -- (4.68,3.17) -- (3.43,2.75) -- cycle;
			\node at (0.52,3.45) {$\barA$};

			\coordinate (Fbottom) at (4.83,-0.40);
			\coordinate (Fnode) at (5.03,1.475);
			\coordinate (Ftop) at (5.23,3.35);
			\coordinate (Floop) at (4.05,1.475);
			\coordinate (Bbottom) at (6.08,0.02);
			\coordinate (Bnode) at (6.28,1.895);
			\coordinate (Btop) at (6.48,3.77);
			\coordinate (Bloop) at (5.30,1.895);
			\coordinate (Flooptangent) at (4.286643,1.869648);
			\coordinate (Blooptangent) at (5.536643,2.289648);
			\coordinate (Fedge) at (5.139244,2.156122);

			\draw[thick,red] (Fnode) -- (Bnode);

			\draw[thick,densely dashed,black!45]
				(Bbottom) .. controls (6.09,0.52) and (6.43,1.445) .. (Bnode);
			\draw[thick,densely dashed,black!45]
				(Bnode) .. controls (6.43,2.345) and (6.39,3.17) .. (Btop);
			\draw[thick,densely dashed,black!45]
				(Bnode) .. controls (6.13,2.345) and (5.30,2.545) ..
				(Bloop) .. controls (5.30,1.245) and (6.13,1.445) .. (Bnode);

			\draw[thick] (Fbottom) -- (Bbottom);
			\draw[thick] (Ftop) -- (Btop);
			\draw[thick] (Flooptangent) -- (Fedge);
			\draw[thick,densely dashed,black!45] (Fedge) -- (Blooptangent);

			\draw[thick]
				(Fbottom) .. controls (4.84,0.10) and (5.18,1.025) .. (Fnode);
			\draw[thick]
				(Fnode) .. controls (5.18,1.925) and (5.14,2.75) .. (Ftop);
			\draw[thick]
				(Fnode) .. controls (4.88,1.925) and (4.05,2.125) ..
				(Floop);
			\draw[thick]
				(Floop) .. controls (4.05,0.825) and (4.88,1.025) .. (Fnode);

			\node[red] at (5.35,1.38) {$C$};
			\node at (5.75,3.05) {$D$};

			\draw[->,thick]
				(3.35,-0.48) -- node[right] {$\pi$} (3.35,-2.08);

			\draw[thick]
				(8.25,1.55) -- (9.50,1.97) -- (9.70,5.27) -- (8.45,4.85) -- cycle;
			\draw[thick] (8.30,2.38) -- (9.55,2.80);
			\draw[thick] (8.40,4.03) -- (9.65,4.45);
			\node[left] at (8.30,2.38) {$s_0$};
			\node[left] at (8.40,4.03) {$s_\infty$};
			\node at (9.00,5.58) {$Y$};

			\draw[->,thick]
				(8.58,2.56) -- node[right] {$\sh$} (9.31,4.25);
			\draw[->,thick]
				(8.05,3.15) .. controls (7.55,2.982) and (7.22,1.788) ..
				node[pos=0.52,below] {$\pi_{\mathbb P}$} (6.72,1.62);
		\end{tikzpicture}
	\end{minipage}%
	\hspace{0.01\textwidth}%
	\begin{minipage}[c]{0.45\textwidth}
		\centering
		\resizebox{\linewidth}{!}{%
			\begin{tikzcd}[ampersand replacement=\&]
				\barA \arrow[dd, "\pi"'] \& Y \arrow[l, "i"'] \arrow[d, "\pi_{\mathbb P}"] \& \mathbb P_C(\mathcal O_C \oplus \mathcal P_C) \arrow[l, Rightarrow, no head] \\
				{} \& C \arrow[d, "\pi_C"] \arrow[lu, "j", bend left=30] \arrow[u, "\sbar_{\infty}", bend left=60] \arrow[u, "s_{0}", bend left=18] \& \mathcal X_{g-1}\times_{\mathcal A_{g-1}}\mathcal X_{g-1}^\vee \arrow[l, Rightarrow, no head] \\
				B \& D \arrow[l, "\iota"'] \& \mathcal X_{g-1}^\vee \arrow[l, Rightarrow, no head]
			\end{tikzcd}%
		}
	\end{minipage}
	\end{center}
	\caption{\label{fig:thm_bdy}A visualization of the maps and spaces appearing in Theorem \ref{thm:bdy}. The left-hand side depicts the geometry of the universal family $\barA$ over the base $B$.}
\end{figure}
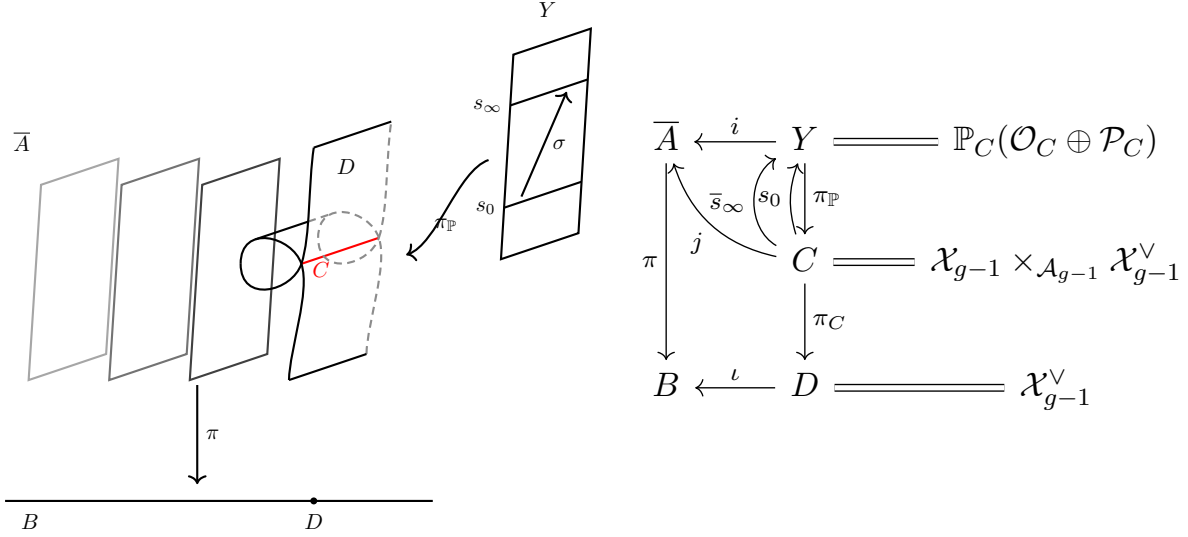

\subsection{The geometry of self-fiber products} \label{sec:An}

Henceforth, we equip $B$ and $\barA$ with the divisorial log structure associated to the
toroidal boundaries $\partial B\subseteq B$, $\partial \barA\subseteq \barA$.
Since $\partial B$ is smooth by Theorem \ref{thm:bdy}, the space $B$ is log
smooth and there is a strict smooth morphism
$$
B\to \mathcal{A}^{1}:=[\mathbb{A}^{1}/\mathbb{G}_m]
$$
from $B$ to its Artin fan.
See \cite{SkeletonsAndFAbramo2015,AModuliStackCavali2017,LecturesOnLogOgus2018} for reference material on logarithmic geometry and Artin fans.\footnote{In this paper we do not distinguish between ``an Artin fan'' of a log scheme $X$ in the sense of \cite[Definition 43]{LogarithmicTauPandha2024} and ``the Artin fan'' of $X$ in the sense of \cite{Birational_inva_Abramo_2013}. All Artin fans considered in this paper will agree with the canonical Artin fans of Abramovich--Wise.}

In this section we compute the Artin fan of
$\barA$. This will allow us to then construct a desingularization $\barA^{(s)} \to\barA^s$ for $s \geq 2$ and define a subring of $\CH^{\ast}(\barA^{(s)})$ of ``piecewise polynomial classes'' in \S \ref{sec:taut_def}.

\begin{definition}
    Let $\Sigma^1$ be the cone complex obtained by identifying the two rays of $\RR_{\geq0}^2$, and let $\mathsf{t}:\Sigma^1\to\mathbb R_{\geq0}$ be the map of cone complexes that adds the two coordinates.
\end{definition}

\begin{proposition}\label{pro:Atrop}
	There is a 2-commutative diagram
	\[
	\begin{tikzcd}
		\barA \arrow[d] \arrow[r] & \mathcal A_{\Sigma^1} \arrow[d, "\mathsf{t}"] \\
		B \arrow[r] & \mathcal{A}^1,
	\end{tikzcd}
	\]
	where the horizontal morphisms are strict, smooth and surjective with connected fibers and the vertical morphisms are
	flat and log smooth, and the left-hand morphism is proper. In particular, $\barA$ is log smooth with Artin fan $\mathcal A_{\Sigma^1}$.
\end{proposition}

\begin{proof}[Proof of Proposition \ref{pro:Atrop}]
	The vertical morphisms are flat by miracle flatness, using that no positive-dimensional cone in $\Sigma^1$ maps to the
	zero-cone in $\mathbb{R}_{\ge 0}$ under the projection onto the $\mathsf{t}$-coordinate. Properness of the left-hand vertical morphism
	is clear over $B^{\circ}$ and follows from the description of its fibers in Theorem \ref{thm:bdy}.
	The morphism $B \to \A^1$ is strict by construction and smooth since $B$ is log regular. It is surjective since $\A^{1}$ has two geometric points which are, respectively, the images of the dense open subset of $B$ and the boundary divisor
	$\partial B$. The fibers are the dense open and $\partial B$, respectively, and hence are connected.
	It remains to show that $\mathcal{A}_{\Sigma^{1}}$ is the Artin fan of $\barA$ and prove the commutativity of the diagram.
	By Theorem \ref{thm:bdy}, the singular locus of $\partial\barA$
	is smooth, so the Artin fan of $\barA$ is two-dimensional. To determine
	it, it suffices to consider an analytic neighborhood
	of the singular locus inside $\barA$. Moreover, away from the singular locus of $\partial\barA$, the
	simple normal crossings pair $(\barA,\partial\barA)$ is clearly log smooth, and the map $\barA\to B$ is
	log smooth on this open subset.

    By \cite[Theorem 1.13 (3)]{FC},
	an analytic neighborhood of the singular locus is isomorphic
	over an infinitesimal neighborhood of $[\mathcal{X}_{g-1}^{\vee}/(\mathbb{Z}/2\mathbb{Z})]$ to the
	space $\overline{X}$ constructed as follows: Let $X$ be the toric variety
	bundle over $\mathcal{X}_{g-1}\times_{\mathcal{A}_{g-1}} \mathcal{X}_{g-1}^{\vee}$ with
	torus bundle associated to the vector bundle
	$\mathcal{O}(-2\theta_2^C)\oplus \mathcal{P}$ (where $\theta_2^C$
	is the universal Theta divisor on the $\mathcal{X}_{g-1}^{\vee}$
	factor and $\mathcal{P}$ is the universal Poincar\'e bundle) and fan
	in $\mathbb{R}_{\mathsf{x},\mathsf{t}}^{2}$ given by some mixed DV-cell
	with respect to $\mathsf{t}$, viewed as a quadratic form on $\mathbb{R}^{2}$.
	Note that all mixed DV-cells are translates of the cone
	$\sigma:=\{(\mathsf{x},\mathsf{t})\in \mathbb{R}_{\ge 0}^{2}|\mathsf{x}\le \mathsf{t}\}$, so we may
	assume that the cone corresponding to $X$ is $\sigma$. Hence
	$X$ is the relative spectrum of the sheaf of algebras:
	$$\bigoplus_{\substack{(a,b)\in \mathbb{Z}_{\ge 0}\times\mathbb{Z}\\a+b\ge 0}}\mathcal{O}(-2a\theta_2^C)\otimes \mathcal{P}^{b}.$$
	We define $\overline{X}$ to be the quotient of $X$ by the following action
	of $\mathbb{Z}/2\mathbb{Z}=\GL(\mathbb{Z})$: On the space
	$\mathcal{X}_{g-1}\times_{\mathcal{A}_{g-1}}\mathcal{X}_{g-1}^{\vee}$, the nontrivial
	element $-1$ of $\mathbb{Z}/2\mathbb{Z}$ acts on both factors via multiplication by $-1$.
	Since
	$$[-1\times -1]^{\ast}(\mathcal{O}(-2a\theta_2^C)\otimes \mathcal{P}^{b})\cong \mathcal{O}(-2a\theta_2^C)\otimes \mathcal{P}^{b}$$
	canonically, this action lifts to an action on $X$, acting trivially on the toric variety bundle. Then
	$\overline{X}:=X/(\mathbb{Z}/2\mathbb{Z})$.

	An infinitesimal neighborhood $\overline{Y}$ of $\partial B=[\mathcal{X}_{g-1}^{\vee}/(\mathbb{Z}/2\mathbb{Z})]$ in $B$ is given by the relative spectrum of
	the sheaf of algebras
	$$\bigoplus_{a\in \mathbb{Z}_{\ge 0}}\mathcal{O}(-2a\theta_2^C)$$
	over $\mathcal{X}_{g-1}^{\vee}$ modulo the action of $\mathbb{Z}/2\mathbb{Z}$ on $\mathcal{X}_{g-1}^{\vee}$ by multiplication by
	$-1$. The map $\overline{X}\to \overline{Y}$ is given by the canonical map
	$$\bigoplus_{a\in \mathbb{Z}_{\ge 0}}\mathcal{O}(-2a\theta_2^C)\to \bigoplus_{\substack{(a,b)\in \mathbb{Z}_{\ge 0}\times\mathbb{Z}\\a+b\ge 0}}\mathcal{O}(-2a\theta_2^C)\otimes \mathcal{P}^{b}.$$
	Hence, locally on the source, the map $\overline{X}\to \overline{Y}$ is given by
	$$[(\mathcal{X}_{g-1}\times_{\mathcal{A}_{g-1}}\mathcal{X}_{g-1}^{\vee})/(\mathbb{Z}/2\mathbb{Z})]\times\mathbb{A}^{2}\to [\mathcal{X}_{g-1}^{\vee}/(\mathbb{Z}/2\mathbb{Z})]\times\mathbb{A}^{1}$$
	equipped with the strict log structures pulled back from $\mathbb{A}^{2}$ and $\mathbb{A}^{1}$,
	where the map to $\mathcal{X}_{g-1}^{\vee}$ is the projection onto the second factor and the map to $\mathbb{A}^{1}$ is the projection
	onto $\mathbb{A}^{2}$ followed by the map $\mathbb{A}^{2}\to \mathbb{A}^{1}$ sending $(x,y)$ to $xy$. Since the latter map is log smooth,
	it follows that $\barA\to B$ is log smooth. Moreover, it follows that
	the Artin fan of $\barA$ is a quotient of the Artin fan of $\mathbb{A}^{2}$ and since the latter space is log smooth (i.e.,
	the map to its Artin fan is smooth), so is $\barA$. The map of fans corresponding to $(x,y)\mapsto xy$ is the map
	$\mathbb{R}_{\ge 0}^{2}\to \mathbb{R}_{\ge 0},\ (\mathsf{a},\mathsf{b})\mapsto \mathsf{a}+\mathsf{b}$. Letting $\mathsf{t}:=\mathsf{a}+\mathsf{b}$ and $\mathsf{u}:=\mathsf{a}$, the cone $\mathbb{R}_{\ge 0}^{2}$ is
	identified with $\{(\mathsf{u},\mathsf{t})|0\le \mathsf{u}\le \mathsf{t}\}$ and the map of cones is identified with projecting onto the $\mathsf{t}$-coordinate. Note that since
	there is only one irreducible boundary component in $\barA$, the two rays in $\mathbb{R}_{\ge 0}^{2}$ are identified in the Artin fan.
	Since in the formal neighborhood $\overline{X}$, the toric variety bundle is associated to a rank-two vector bundle which is a global
	direct sum, the two normal directions are not interchanged by the monodromy on the singular locus, so no further quotient is needed.
	It is clear from the construction that the map $\barA\to \mathcal A_{\Sigma^1}$ is surjective with connected fibers.
\end{proof}

It will be more convenient to think of $\Sigma^1$ as a family of tropical tori over $\mathbb R_{\geq 0}$. For each $m \in \mathbb Z$, subdivide $\RR \times \RR_{\ge 0}$ along the hyperplanes $m\mathsf{t} = \mathsf{x}$. The group $\mathbb Z$ acts on $\RR\times \RR_{\ge 0}$ by $(\mathsf{x}, \mathsf{t}) \mapsto (\mathsf{x} + m\mathsf{t}, \mathsf{t})$ for $m \in \ZZ$, and this action preserves the subdivision. $\Sigma^1$ can be identified with the quotient of this subdivision by the action of $\ZZ$, see Figure \ref{fig:Sigma1}.

\begin{figure}
\begin{center}
	\[
	\begin{tikzpicture}[every node/.style={font=\small}]
		\coordinate (L0) at (0,-0.7);
		\coordinate (L1) at (1.4,-0.7);
		\coordinate (L2) at (1.4,0.45);
		\draw[thick] (L0) -- (L1);
		\draw[thick] (L0) -- (L2);
		\draw[<->,thick]
			(1.62,-0.7)
			to[out=18,in=-18,looseness=1.5] (1.62,0.45);
		\node at (0.8,-1.25) {$\mathbb R_{\geq 0}^{2}$};

		\draw[arrows={Hooks[harpoon,swap]-Latex},thick] (2.65,0) -- (3.70,0);

		\coordinate (U0) at (4.25,0);
		\def\linkx{6.75}
		\def\step{0.75}
		\def\actionx{6.95}
		\def\actiongap{0.10}
		\def\ellipsisgap{0.48}
		\def\ellipsisstep{0.13}
		\foreach \m in {-3,-2,-1,0,1,2,3} {
			\draw[thick] (U0) -- (\linkx,{\m*\step});
		}

		\foreach \m in {-4,-3,-2,-1,0,1,2} {
			\draw[->,thick]
				(\actionx,{(\m+0.5)*\step+\actiongap})
				to[out=18,in=-18]
				(\actionx,{(\m+1.5)*\step-\actiongap});
		}
		\foreach \d in {0,1,2} {
			\fill (\linkx,{3*\step+\ellipsisgap+\d*\ellipsisstep})
				circle (0.65pt);
			\fill (\linkx,{-3*\step-\ellipsisgap-\d*\ellipsisstep})
				circle (0.65pt);
		}
		\node at (5.50,-3.30) {$\mathcal U^{1}\Sigma^{1}$};

		\draw[->>,thick] (7.75,0) -- (9.05,0);

		\coordinate (Q0) at (9.65,-0.7);
		\def\basex{12.05}
		\def\baserx{0.34}
		\def\basecy{-0.0210842}
		\def\basery{0.6789158}
		\coordinate (Qbottom) at (\basex,-0.7);
		\coordinate (Qtangent) at (11.955562,0.631117);
		\draw[thick] (Q0) -- (Qbottom);
		\draw[thick,densely dashed] (Q0) -- (Qtangent);
		\draw[thick] (\basex,\basecy)
			ellipse[x radius=\baserx,y radius=\basery];
		\node at (10.95,-1.25) {$\Sigma^{1}$};
	\end{tikzpicture}
	\]
\end{center}
\caption{\label{fig:Sigma1}The cone $\mathbb R_{\geq 0}^{2}$ (left), the cone complex $\Sigma^{1}$ (right) and its ``universal cover'' $\mathcal{U}^{1}\Sigma^{1}$ (center).}
\end{figure}
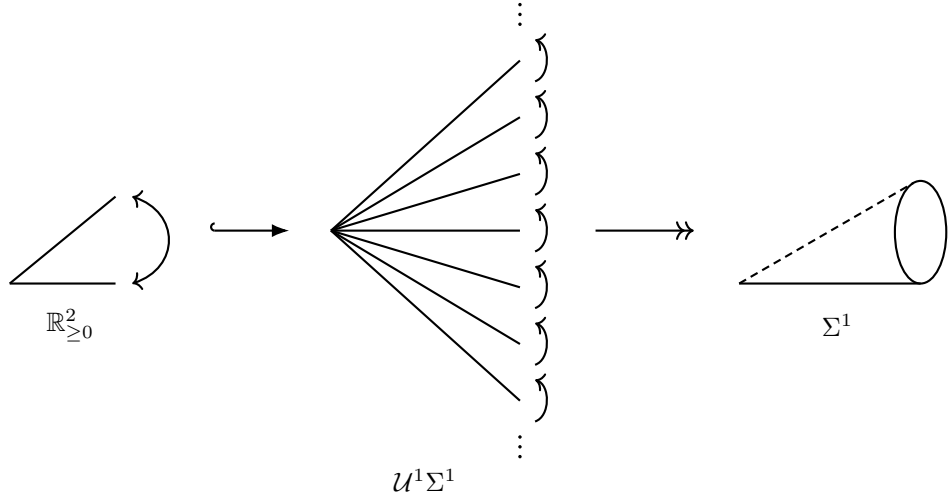

From Proposition \ref{pro:Atrop}, it follows that the map $\barA\to B$ is proper and log
smooth, so the $s$-fold self-fiber product $\barA^s \to B$ is log smooth with
Artin fan associated to the cone complex
$$
\Sigma^s := \Sigma^1 \times_{\mathbb R_{\geq 0}} \cdots \times _{\mathbb R_{\geq 0}} \Sigma^1 \stackrel{\mathsf{t}}{\longrightarrow} \RR_{\ge 0}\, .
$$
Its universal cover, which we denote by $\mathcal U^s\Sigma^s$, is the subdivision of $\RR^s \times \RR_{\ge 0}$ along the hyperplanes $\{\mathsf{x}_i\in \mathsf{t}\ZZ\}_{i=1,\ldots,s}$. The cones in $\Sigma^s$ are not simplicial for $s\geq2$, and consequently $\barA^s$ is not smooth.

Let $\mathfrak S_s$ be the symmetric group on $s$ elements. For $\tau \in \mathfrak S_s$, consider the simplex
\begin{equation}\label{eq:sigma_cones}
\sigma_{\tau} = \{(\mathsf{x}_1, \ldots , \mathsf{x}_s, \mathsf{t}) \mid 0\le \mathsf{x}_{\tau(1)}\le \cdots \le \mathsf{x}_{\tau(s)}\le \mathsf{t}\} \subset \RR^s \times \RR_{\ge 0}\, .
\end{equation}
Let $\ZZ^s$ act on $\RR^s \times \RR_{\ge 0}$ by
$$
(\mathsf{x}_1,\dots,\mathsf{x}_s,\mathsf{t})\longmapsto (\mathsf{x}_1 + m_1\mathsf{t},\dots,\mathsf{x}_s+m_s\mathsf{t},\mathsf{t}),\qquad \underline{m} = (m_1, \ldots, m_s)\in \ZZ^s,
$$
and let $\sigma_{\underline{m}, \tau}$ be the $\underline{m}$-translate of $\sigma_{\tau}$. The cones $\sigma_{\underline{m}, \tau}$ are equivalently obtained by subdividing $\RR^s\times \RR_{\ge 0}$ along the hyperplanes $\{\mathsf{x}_i \in \mathsf{t}\mathbb Z\}_{i=1, \ldots , s}$ and $\{\mathsf{x}_i - \mathsf{x}_j \in \mathsf{t}\mathbb Z\}_{1 \leq i<j \leq s}$.

Gluing all the cones $\sigma_{\underline{m}, \tau}$ along their common faces forms a $\ZZ^s$-periodic cone complex $\mathcal U^s\Sigma^{(s)}$.

\begin{definition}
    Let $G$ be a group acting on a combinatorial cone stack $\Sigma$ with faithful monodromy. We write $\Sigma/G$ for the categorical quotient in cone
    stacks with faithful monodromy; equivalently, it is obtained from $[\Sigma/G]$ by quotienting by the subgroupoid of morphisms acting trivially on the rays.
\end{definition}

\begin{definition}
	For $s\geq 1$, we define the cone complex $\Sigma^{(s)}$ as the quotient $\mathcal U^s\Sigma^{(s)}/\ZZ^s$ (see Figure \ref{fig:Sigma_s}). It is a subdivision of $\Sigma^s$. The desingularization map $f:\barA^{(s)}\to\barA^s$ is obtained from the f.s. fiber diagram
    $$
    \begin{tikzcd}
        \barA^{(s)} \ar[r, "f"] \ar[d]  & \barA^s \ar[d] \\
        \mathcal A_{\Sigma^{(s)}} \ar[r] & \mathcal{A}_{\Sigma^s}.
    \end{tikzcd}
    $$
\end{definition}

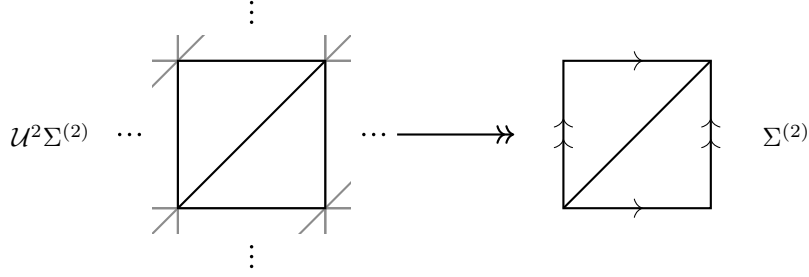
\begin{figure}
\begin{center}
	\[
	\begin{tikzpicture}[every node/.style={font=\small}]
		\def\s{1.95}
		\def\p{0.33}
		\def\dx{5.10}
		\def\continuationgap{0.18}
		\def\continuationstep{0.13}
		\colorlet{extensiongray}{black!45}

		\newcommand{\plaintile}{%
			\draw[thick] (0,0) rectangle (\s,\s);
			\draw[thick] (0,0) -- (\s,\s);
		}
		\newcommand{\periodicplain}[2]{%
			\begin{scope}[shift={(#1,#2)}]
				\begin{scope}
					\clip (-\p,-\p) rectangle (\s+\p,\s+\p);
					\foreach \xx/\yy in {-\s/-\s,0/-\s,\s/-\s,-\s/0,\s/0,-\s/\s,0/\s,\s/\s} {
						\begin{scope}[draw=extensiongray,shift={(\xx,\yy)}]
							\plaintile
						\end{scope}
					}
					\plaintile
				\end{scope}
			\end{scope}
		}
		\newcommand{\edgeidentifications}{%
			\path[decorate,decoration={markings,
				mark=at position 0.54 with {\arrow{>[scale=1.4]}}}]
				(0,0) -- (\s,0);
			\path[decorate,decoration={markings,
				mark=at position 0.54 with {\arrow{>[scale=1.4]}}}]
				(0,\s) -- (\s,\s);
			\path[decorate,decoration={markings,
				mark=at position 0.47 with {\arrow{>[scale=1.4]}},
				mark=at position 0.61 with {\arrow{>[scale=1.4]}}}]
				(0,0) -- (0,\s);
			\path[decorate,decoration={markings,
				mark=at position 0.47 with {\arrow{>[scale=1.4]}},
				mark=at position 0.61 with {\arrow{>[scale=1.4]}}}]
				(\s,0) -- (\s,\s);
		}

		\periodicplain{0}{0}
		\begin{scope}[shift={(\dx,0)}]
			\plaintile
			\edgeidentifications
		\end{scope}

		\draw[->>,thick]
			(\s+0.95,0.5*\s) -- (\dx-0.65,0.5*\s);

		\foreach \d in {0,1,2} {
			\fill (0.5*\s,{\s+\p+\continuationgap+\d*\continuationstep})
				circle (0.65pt);
			\fill (0.5*\s,{-\p-\continuationgap-\d*\continuationstep})
				circle (0.65pt);
			\fill ({-\p-\continuationgap-\d*\continuationstep},0.5*\s)
				circle (0.65pt);
			\fill ({\s+\p+\continuationgap+\d*\continuationstep},0.5*\s)
				circle (0.65pt);
		}

		\node[anchor=east]
			at ({-\p-\continuationgap-2*\continuationstep-0.25},0.5*\s)
			{$\mathcal U^2\Sigma^{(2)}$};
		\node[anchor=west] at (\dx+\s+0.55,0.5*\s)
			{$\Sigma^{(2)}$};
	\end{tikzpicture}
	\]
\end{center}
\caption{\label{fig:Sigma_s}The quotient map $\mathcal U^2\Sigma^{(2)}\to\Sigma^{(2)}$. To simplify the presentation, we restrict to the link at $\mathsf{t}=1$.}
\end{figure}

For $s=1$, it is clear that $\Sigma^{(1)}=\Sigma^1$.
Since the cones in $\Sigma^{(s)}$ are simplicial and unimodular, the following statement follows immediately from Proposition \ref{pro:Atrop} and the definition.

\begin{corollary}\label{cor:As}
    $\barA^{(s)}$ is a smooth Deligne--Mumford stack which is a log modification of the relative $s$-fold fiber product of $\barA \to B$. Moreover, its Artin fan is $\mathcal A_{\Sigma^{(s)}}$.
\end{corollary}

We end this section with the following lemma on maps between the spaces $\barA^{(s)}$:

\begin{lemma}\label{lem:maps_regular}
	Let $s\ge r\ge 0$.
\begin{enumerate}[label = (\alph*)]
	\item There are (unique) extensions $p_{i_1,i_2,\dots,i_r}:\barA^{(s)}\to \barA^{(r)}$ of the projection operators $G^{s}\to G^{r},\ (x_k)_k\mapsto (x_{i_1},\dots,x_{i_r})$
	onto the factors $i_1,\dots,i_r$.
	\item There is a (unique) extension $\Delta:\barA\to \barA^{(2)}$ of the diagonal embedding $G\to G^{2},\ x\mapsto (x,x)$.
	\item There is a (unique) extension $\diff:\barA^{(2)}\to \barA$ of the difference map $G^{2}\to G,\ (x,y)\mapsto x-y$.
\end{enumerate}
\end{lemma}

\begin{proof}
	For point (a), observe that by symmetry it suffices to consider the case $(i_1,\dots,i_r)=(1,\dots,r)$. It follows from \eqref{eq:sigma_cones} that the morphism $\mathcal U^s\Sigma^{(s)}\to\mathcal U^s\Sigma^s$ factors as
	$$
	\mathcal U^s\Sigma^{(s)}\to \mathcal U^r\Sigma^{(r)} \times_{\RR_{\ge 0}} \mathcal U^{s-r}\Sigma^{s-r} \to \mathcal U^s\Sigma^s.
	$$
	Descending this factorization along the action of $\ZZ^s$, we get a factorization
	$$
	\Sigma^{(s)}\to \Sigma^{(r)} \times_{\RR_{\ge 0}} \Sigma^{s-r} \to \Sigma^s.
	$$
	By base-change, we therefore get a factorization:
	$$\barA^{(s)}\to \barA^{(r)}\times_B \barA^{s-r}\to \barA^{s}\, ,$$
	and (a) follows. Point (b) is proved similarly. For point (c), one can argue using Mumford models as in \cite{vdgK} (where the authors study the addition map) that the map $\diff$ is toroidal and induces the map
	$$
    \Sigma^{(2)} \to \Sigma^{(1)}\, , \quad (\mathsf{x}_1, \mathsf{x_2}) \mapsto \mathsf{x_1}-\mathsf{x}_2\,.
    $$
    This function maps cones into cones, so it is regular. Alternatively, one can apply the theory of log abelian varieties,
    using ideas similar to those in the proof of Proposition \ref{prop:FM-theta}.
\end{proof}

\subsection{Extending the Poincar\'e line bundle}\label{subs:extP}

Since the fibers of $G \to B$ are irreducible and $G^2$ is smooth, there exists a unique line bundle $\mathcal P \in \Pic(G^2)$ that agrees with the usual Poincar\'e line bundle $\mathcal P^{\circ} \in \Pic(A^{\circ }\times_{B^\circ}A^{\circ })$ and such that $(1 \times e)^* \mathcal P = (e\times 1)^*\mathcal P = \mathcal O$. Moreover, since $\barA^{2} \setminus G^{2}$ has codimension $2$ and the resolution $f: \barA^{(2)} \to \barA^2$ is small, we have $\operatorname{codim}\big(\barA^{(2)} \setminus G^2\big)=2$. Since $\barA^{(2)}$ is smooth, it follows that $\mathcal{P}$ extends uniquely to a line bundle on $\barA^{(2)}$.

\begin{definition}
	The \emph{extended Poincar\'e line bundle} is the unique $\mathcal P \in \Pic(\barA^{(2)})$ that extends the line bundle on $G^2$ constructed above. The \emph{extended Theta divisor} is the $\mathbb Q$-line bundle $\mathcal O(\Theta)$ on $\barA$ defined by the formula $\mathcal O(\Theta)^{\otimes 2} = \Delta^*(\mathcal P)$, where $\Delta : \barA \to \barA^{(2)}$ is the diagonal map (see Lemma \ref{lem:maps_regular}). The first Chern classes of $\mathcal O(\Theta)$ and $\P$ are denoted by
    $$
    \theta := c_1(\mathcal O(\Theta)) \in \CH^1(\barA)\, , \quad \ell := c_1(\mathcal P) \in \CH^1(\barA^{(2)})\, .
    $$
\end{definition}

Note that, by definition, $\Theta$ agrees with the usual symmetric Theta divisor on $A^\circ$ and $e^* \mathcal O(\Theta)$ is trivial. In particular, since $\Pic(\barA)_{\QQ} = \mathbb Q\langle\lambda_1, \Theta, D\rangle$, the $\mathbb{Q}$-line bundle $\mathcal O(\Theta)$ agrees with the one constructed by Alexeev and Nakamura in \cite{AleNak}. We recall some of its properties:
\begin{lemma}\label{lem:thetample}
    The extended Theta divisor is $\pi$-ample and has no higher cohomology:
    $$
    R^{>0}\pi_* \mathcal O(n\Theta) =0\, ,\quad n \in 2\mathbb Z_{>0}\, .
    $$
\end{lemma}

\begin{proof}
    Since $\pi$ is flat, the vanishing of the higher derived pushforwards can be checked on each fiber. By \cite[Theorem 1.7.8]{Positive1}, since $\pi$ is proper, relative ampleness is also a fiberwise condition.
    By the construction of $\barA$ in \cite{FC}, each fiber can be realized as the special fiber of a Mumford construction \cite{Mumfordanalytic, FC} over a complete DVR. The proofs of \cite[Theorems 4.4 and 4.7]{AleNak} then show that the restriction of $\mathcal O(\Theta)$ to each fiber is ample and has no higher cohomology.
\end{proof}

By Lemma \ref{lem:maps_regular}, the difference map
$$
\diff : \barA^{(2)} \to \barA
$$
is regular, and Mumford's formula\footnote{We follow the convention that $\mathcal P^{\circ} = \mu^*\mathcal O(\Theta) \otimes p_1^*\mathcal O(-\Theta) \otimes p_2^*\mathcal O(-\Theta)$. This convention differs by a sign from the usual convention on Jacobians; see the discussion in \cite[(1.6.)]{beauvilleSL2}. In particular, the signs differ from \cite{BMP}. Our convention is the same as the one used by Kajiwara, Kato and Nakayama; see \cite[Proposition 2.4.]{KKN5}.}
\begin{equation}\label{eq:Mumform}
    \mathcal P^{\vee} \cong \diff^* \mathcal O(\Theta) \otimes p_1^*\mathcal O(-\Theta) \otimes p_2^*\mathcal O(-\Theta)
\end{equation}
holds since both sides agree on $A^\circ \times_{B^\circ}A^\circ$ and their pullbacks along the unit sections agree.

\subsection{The log Poincar\'e line bundle}\label{sec:Plog}

We relate the extended Poincar\'e line bundle to the theory of
$\Glog$-biextensions on log abelian varieties developed by Kajiwara,
Kato and Nakayama \cite{KKN2, KKN4, KKN5, KKN7}.

For this, we start by recalling some basic properties of log abelian varieties.
If $B$ is a log scheme, a log abelian variety is a sheaf of abelian groups
$A^{\log}$ on $B$ such that there is a short exact sequence
$$
0 \longrightarrow G \longrightarrow A^{\log} \longrightarrow A^{\trop} \longrightarrow 0
$$
where $G\to B$ is a semi-abelian scheme and $A^{\trop}$ is a tropical abelian variety\footnote{Tropical abelian varieties are sheaves on the category of rational polyhedral cones that are analogues of abelian varieties. A study of tropical abelian varieties will appear in \cite{troAV}. In \cite{KKN2}, $A^{\trop}$ is denoted by $\hom (X, \mathbb G_{m,\trop})^{(Y)}/\overline{Y}$. The latter is the geometric realization of the tropical abelian variety, and we call it a tropical abelian variety by abuse of notation}.
The sheaf $A^{\trop}$ is of combinatorial nature and will only appear in the proofs of Lemma \ref{lem:Artin_fan_A_N}
and Lemma \ref{lem:glob_secs_ptrop}.
A principal polarization on $A^{\log}$ is a symmetric $\Glog$-biextension (and, in particular, a logarithmic line bundle)
$$
\P^{\log} \to A^{\log}\times_B A^{\log}
$$
such that the induced map $A^{\log} \to \ext (A^{\log}, \Glog)$ is injective and restricts to a polarization of log 1-motives on each point.
By \cite[2.5.4.1 and 3.7.5]{grothendieck1973groupes} we have an exact sequence:
\begin{align*}
	\operatorname{Bil}(G,G;\mathbb{G}_{m,\trop})\cong \Hom(G,&\hom(G,\mathbb{G}_{m,\trop}))\to \operatorname{Biext}(G,G;\mathbb{G}_m)\to\\
	&\to \operatorname{Biext}(G,G;\mathbb{G}_{m,\log})\to \operatorname{Biext}(G,G;\mathbb{G}_{m,\trop}).
\end{align*}
By \cite[Proposition 1.6 (1)]{KKN5} the right-hand group vanishes
and by \cite[Lemma 6.1.1]{KKN2} the left-hand group vanishes.
Hence the restriction of $\P^{\log}$ to $G^2$ is induced by a unique $\Gm$-biextension of $G$.

By \cite[Theorem 1.8]{KKN7} there exists a universal principally polarized log abelian scheme
\[
\pi^{\log}: \logA\to B
\]
whose associated semi-abelian scheme agrees with $G \subset \barA$.
By \cite[Theorem 9.1]{KKN5}\footnote{The model $\barA$ is the model associated to the star $\{0,1,-1\}\subseteq \mathbb{Z}$ (see \cite[4.6]{KKN5}).}, $\barA^{(s)}$ is a projective model of $(\logA)^s$, and there is an f.s. fiber diagram
\begin{equation}
\begin{tikzcd}
	\barA^{(s)} & {\A_{\Sigma^{(s)}}} \\
	{(\logA)}^{s} & {(A^{\trop})}^{s}
	\arrow[from=1-1, to=1-2]
	\arrow[from=1-1, to=2-1,"\beta"]
	\arrow[from=1-2, to=2-2]
	\arrow[from=2-1, to=2-2]
\end{tikzcd}
\end{equation}
where the vertical arrows are log blowups and the horizontal arrows are smooth, strict and surjective, with connected fibers.

On any log scheme $X$, we have an exact sequence
$$
 H^1(X, \Gm) \longrightarrow H^1(X, \Glog) \longrightarrow H^1(X, \mathbb G_{m,\trop}).
$$
The middle group is the group of logarithmic line bundles. The following statement appears in \cite{J}. For completeness, we include a proof.

\begin{proposition}\label{pro:logline}
    Let $X$ be a log scheme over an algebraically closed field with smooth underlying scheme and log structure induced by an effective Cartier divisor $\partial X$.
    Then $H^1(X, \mathbb G_{m,\trop})=0$. In particular, any logarithmic line bundle on $X$ is represented by some line bundle.
\end{proposition}

\begin{proof}
	The assumptions imply that $M_X^\gp \cong j_*\CO_U^\times$, where $j: U:=X\setminus\partial X \to X$, so $\overline{M}_X^\gp$ is the sheaf of Weil divisors supported on $\partial X$.
	This sheaf is flasque because the closure of any codimension-$1$ subset of an open subset of $X$ supported on $\partial X$ is again a codimension-$1$ subset supported on $\partial X$.
	It follows that $H^1(X,\overline{M}_X^\gp)=0$.
\end{proof}

It follows that the restriction of $\P^{\log} \to (\logA)^2$ to $\barA^{(2)}$ is represented by a line bundle. Since $\operatorname{codim}\big(\barA^{(2)} \setminus G^2\big)=2$, this line bundle is unique if we require it to restrict to the $\Gm$-biextension on $G$ from above. Moreover, since $\Gm$-biextensions are trivialized at the origin, this line bundle representing $\P^{\log}$ must agree with the extended Poincar\'e line bundle $\P$ constructed in \S\ref{subs:extP}.

\begin{remark}\label{rem:Ptrop}
	In Lemma \ref{lem:glob_secs_ptrop} we will see that the set of lifts of $\mathcal{P}^{\log}$ to a line bundle on $\barA^{(2)}$
	is in bijection with the set of strict piecewise linear functions $p$ on the universal cover $\mathcal U^2\Sigma^{(2)}$ of
	$\Sigma^{(2)}$ such that
	$$
    p(\mathsf{x}_1,\mathsf{x}_2,\mathsf{t})=\mathsf{x}_1v_2+\mathsf{x}_2v_1+\mathsf{t}v_1v_2+p(\mathsf{x}_1+v_1\mathsf{t},\mathsf{x}_2+v_2\mathsf{t},\mathsf{t})
    $$
    for all $(v_1, v_2) \in \ZZ^2$. A choice of such a function is given by
    $$
    \mathsf{p}^{\trop} (\mathsf{x}_1, \mathsf{x}_2, \mathsf{t}) = -\mathsf{t} \left(\frac{\mathsf{x}_1 \mathsf{x}_2}{\mathsf{t}^2} - \left\lbrace\frac{\mathsf{x}_1}{\mathsf{t}}\right\rbrace\left\lbrace\frac{\mathsf{x}_2}{\mathsf{t}}\right\rbrace + \min \left(\left\lbrace\frac{\mathsf{x}_1}{\mathsf{t}}\right\rbrace, \left\lbrace\frac{\mathsf{x}_2}{\mathsf{t}}\right\rbrace\right)\right)
    $$
    Equivalently, $\mathsf{p}^{\trop}$ is characterized as the unique strict piecewise linear function on $\mathcal U^2\Sigma^{(2)}$ that agrees with $-\frac{\mathsf{x}_1\mathsf{x}_2}{\mathsf{t}}$ on the rays given by integral lattice points.
\end{remark}

\subsection{Multiplication by $N$ maps}
\label{sec:mbyNmaps}

In this section, we define the action of the multiplication by $N$ map
on $\CH^*(\barA)$ and, more generally, on $\CH^*(\barA^{(s)})$.

For $N \in \ZZ_{>0}$, there exists a {\em rational} multiplication by $N$ map
\[
N : \barA \dashrightarrow \barA\,, x \mapsto Nx\,,
\]
which is regular on the semi-abelian part $G \subset \barA$. We resolve the indeterminacy using the theory of log abelian varieties, see \S\ref{sec:Plog} and the references therein.

\begin{definition}\label{def:n}
    Let $N_i =(1, \ldots ,N, \ldots , 1)\in \ZZ^s_{>0}$ with $N$ placed at the $i$-th position. For any diagram
    $$
    \begin{tikzcd}
	& {\barA_{N_i}^{(s)}} && {\barA^{(s)}} \\
	{\barA^{(s)}} & {(\logA)^s} && {(\logA)^{s}}
	\arrow["{N_i}", from=1-2, to=1-4]
	\arrow["b"', from=1-2, to=2-1]
	\arrow[from=1-2, to=2-2]
	\arrow["\beta", from=1-4, to=2-4]
	\arrow["\beta", from=2-1, to=2-2]
	\arrow["{1\times \ldots\times N\times \ldots  \times 1}", from=2-2, to=2-4]
    \end{tikzcd}
    $$
    such that the right square is commutative, and $b$ is a log modification, we define
    $$
    [N_i]^* := b_* \circ N_i^* : \CH^*(\barA^{(s)}) \longrightarrow \CH^*(\barA^{(s)})\, .
    $$
	When $i$ is clear from context, we write $\barA_{N}^{(s)}:=\barA_{N_i}^{(s)}$ and $[N]^{\ast}:=[N_i]^{\ast}$.
\end{definition}

Since $b$ is a log modification, $\barA_{N_i}^{(s)}$ is irreducible and reduced. Therefore, the map $[N_i]^*$ is independent of the choice of diagram in Definition \ref{def:n}.

A convenient choice of $\barA_{N_i}^{(s)}$ is obtained as the fiber product in the following f.s. fiber diagram
\begin{equation}\label{eq:n}
    \begin{tikzcd}
        \barA^{(s)}_{N_i} \ar[r] \ar[rrr, bend left=20, "N_i"] \ar[d] & \barA^{(s)'}_{N_i} \ar[rr] \ar[d] && \barA^{(s)} \ar[d,"\beta"] \\
        \barA^{(s)} \ar[r, "\beta"] & (\logA)^{s} \ar[rr,"1\times \cdots \times N \times \cdots \times 1"] && (\logA)^{s}\,.
    \end{tikzcd}
\end{equation}

\begin{definition}
    We say that a class $\alpha \in \CH^*(\barA)$ has weight at most $k$ if $[N]^*\alpha$ is a polynomial in $N$ of degree at most $k$.
    We say that $\alpha$ has finite weight if $[N]^*\alpha$ is polynomial in $N$. That is, $\alpha$ has finite weight if and only if
	there exist classes $[\alpha]^{(0)},\dots,[\alpha]^{(m)}\in\CH^{\ast}(\barA)$ such that
    $$
    [N]^*\alpha = \sum_{k=0}^m N^k [\alpha]^{(k)}
    $$
	holds for all $N$. The class $[\alpha]^{(k)}$ is called the \emph{weight $k$ part} of $\alpha$. The associated filtration is called
	the \emph{weight filtration}:
    $$
    W_k \CH^*(\barA) = \{\alpha \in \CH^*(\barA) : \alpha\text{ has weight } \leq k\}\,.
    $$
\end{definition}

More generally, we say that $\alpha\in \CH^*(\barA^{(s)})$ has weight $ \leq k$ in the $i$-th direction if $[N_i]^*\alpha$ is polynomial in $N$ of degree at most $k$. We write
    $$
    W_{i,k}\CH^*(\barA^{(s)})
    $$
for the associated filtration, and $[\alpha]^{i, (k)}$ for the weight parts in the $i$-th direction. We will omit $i$ when it is clear from the context.

\begin{lemma}\label{lem:Npush}
    Let $1 \leq i\leq s \leq r$ and let $\pi : \barA^{(r)} \to \barA^{(s)}$ be the projection onto the first $s$ factors. For $N_i = (1, \ldots , N, \ldots , 1) \in \ZZ^s_{>0}$, the following holds:
    \begin{enumerate}[label = (\alph*)]
        \item $\pi_{*}\circ [N_i\times 1^{r-s}]^* = [N_i]^*\circ \pi_*$.
        \item $[N_i\times 1^{r-s}]^* \circ \pi^* = \pi^* \circ [N_i]^*$.
    \end{enumerate}
\end{lemma}

\begin{proof}
    We prove (a). Consider the following diagram
    \[
    \begin{tikzcd}
        \barA^{(r)}_{N_i \times 1^{r-s}} \ar[r,"h"] \ar[d,"b^r"] & \barA^{(r)''}_{N_i \times 1^{r-s}} \ar[d,"\pi'"] \ar[r]  &\barA^{(r)'}_{N_i \times 1^{r-s}} \ar[r] \ar[d] & \barA^{(r)} \ar[d,"\pi"] \\
        \barA^{(r)} \ar[dr,"\pi"] &\barA^{(s)}_{N_i} \ar[r] \ar[d,"b^s"] & \barA^{(s)'}_{N_i} \ar[r] \ar[d] & \barA^{(s)} \ar[d] \\
        &\barA^{(s)} \ar[r] & (\logA)^s \ar[r,"N_i"] & (\logA)^s\,.
    \end{tikzcd}
    \]
    Here the four squares are f.s. fiber products, and $h$ is a proper birational morphism making the left trapezoid commute. Since $\pi : \barA^{(r)} \to \barA^{(s)}$ is semi-stable, the upper two squares are cartesian. In particular, $\barA^{(r)}_{N_i\times 1^{r-s}}$ is a model for $[N_i \times 1^{r-s}]^*$ in Definition \ref{def:n}.

    In the above diagram, denote by $N_i :\barA_{N_i}^{(s)} \to \barA^{(s)}$ and $(N_i \times 1^{r-s})' : \barA^{(r)''}_{N_i \times 1^{r-s}} \to \barA^{(r)}$ the corresponding compositions, and set $N_i \times 1^{r-s}= (N_i \times 1^{r-s})' \circ h$. Then, for $\alpha \in \CH^*(\barA^{(r)})$, we have
    \begin{align*}
        [N_i]^*\pi_*(\alpha)
        & = b^s_* N_i^*\pi_*(\alpha) \\
        & = b^s_*\pi'_*h_*((N_i\times 1^{r-s})' \circ h)^*(\alpha)\\
        & = \pi_*b^r_*(N_i \times 1^{r-s})^*(\alpha)\\
        & = \pi_*[N_i\times 1^{r-s}]^*(\alpha)\,,
    \end{align*}
    where the second equality follows from the fact that $h$ is proper and birational, using \cite[Lemma C.8]{BS1}, and the third follows from the commutativity of the left trapezoid. This proves the claim.

    The proof of part (b) is similar to (a).
\end{proof}

\begin{remark}
    Pullback along the multiplication by $N$ map induces an action on both the
    small and big logarithmic Chow groups of the universal log abelian
	scheme $\X_g^{\log} \to \A_g^{\log}$. However, contrary to our situation (see Proposition \ref{cor:pure}),
	the eigenspaces of this action do not span the logarithmic Chow group. See also \cite{BGHdJ}.
\end{remark}

\subsection{The tautological ring and tautological maps}\label{sec:taut_def}

There are three types of tautological classes on $\barA^{(s)}$: divisorial classes, $\lambda$-classes, and piecewise polynomial classes.

First, we introduce the divisorial classes. For $1 \le i \le s$, let $p_i :\barA^{(s)} \to \barA$ be the projection onto the $i$-th component, and let $p_{i,j} :\barA^{(s)} \to \barA^{(2)}$ be the unique extension of the projection onto the $(i,j)$-th component from Lemma \ref{lem:maps_regular}. Write
\[
\theta_i := p_i^*\theta, \quad \ell_{i,j} := p_{i,j}^*(\ell) \in \CH^1(\barA^{(s)})\,.
\]

We define the Hodge bundle to be the conormal bundle to the unit section, $\EE := N_e^\vee$, and set $\lambda_i := c_i(\EE)$. We use the same notation for the pullbacks of these classes along $\barA^{(s)} \to B$.

\begin{remark}\label{rem:Hodge_is_Omlog}
    By \cite[Chapter IV. Theorem 1.1.]{FC}, the sheaf of logarithmic differentials of $\pi$ is the pullback of the Hodge bundle:
    $$
    \Omega_{\barA/B}^{\log} = \pi^* \mathbb E\, ,\quad \Omega^{\log}_{\barA^{(s)}/B} =(\pi^s)^* \mathbb E^{\oplus s}\,.
    $$
    Alternatively, using the group structure on $\logA$, it is easy to show (see for example the argument in \cite[Tag 047I]{stacks-project}) that $\Omega_{\logA/B}^{\log} = (\pi^{\log})^*\mathbb E$, and, since $\barA\to \logA$ is a log modification, it preserves the logarithmic sheaf of differentials.
\end{remark}

Piecewise polynomial classes appear as a result of the logarithmic structure. We briefly recall the notation: Let $X$ be a log smooth Deligne--Mumford stack with a strict morphism
\[
X \to \mathcal A_\Sigma
\]
to its Artin fan. Let $\Sigma$ be the corresponding cone stack, and let $\spp^{\ast}(\Sigma)$ denote the ring of strict piecewise polynomials on $\Sigma$.

By \cite[Theorem B]{A_case_study_of_Molcho_2021} (or \cite[Theorem 14]{MPS} when $X$ is smooth), there is a ring isomorphism
\[
\CHop^*(\mathcal A_\Sigma) \cong \spp^{\ast}(\Sigma).
\]
Composing the inverse of this isomorphism with the pullback to $X$ gives a ring morphism
\begin{equation}\label{eq:Phi}
    \Phi_X : \spp^{\ast}(\Sigma) \longrightarrow \CHop^*(X).
\end{equation}
When $X$ is clear from context, we omit the subscript and write $\Phi=\Phi_X$.

\begin{definition}
    \begin{enumerate}[label = (\alph*)]
        \item The {\em small tautological ring} $\sR^*(\barA^{(s)}) \subseteq \CH^*(\barA^{(s)})$ is the subring generated by the classes $\theta_i$, $\ell_{i,j}$, and the piecewise polynomial classes.

        \item The {\em tautological ring} $\R^*(\barA^{(s)}) \subseteq \CH^*(\barA^{(s)})$ is the subring generated by $\sR^*(\barA^{(s)})$ and the classes $\lambda_i$.
    \end{enumerate}
\end{definition}

We identify $\spp^{\ast}(\Sigma^{(s)})$ with the $\ZZ^{s}$-periodic strict piecewise polynomials on $\mathcal U^{s}\Sigma^{(s)}$.

\begin{example}
	The tautological ring of $B$ has a simple description. Let $\iota : D = \X_{g-1}^\vee \to \partial B \subset B$ be the double cover
	of the boundary and write
	\[
		D := \frac{1}{2}\iota_*(1) \in \CH^1(B)\,.
	\]
	Then $\sR^*(\barA^{(0)}) = \sR^*(B)$ is generated by $D$, while $\R^*(B)$ is generated by $D$ and the classes $\lambda_i$.
\end{example}

\begin{example}\label{expl:pp_s_1}
    We unpack the definition of piecewise polynomial classes on $\barA$. Piecewise polynomials on $\Sigma^{1}$ can be identified with polynomials
    $p \in \mathbb{Q}[\mathsf{x},\mathsf{y}]$ satisfying $p(0,\mathsf{t})=p(\mathsf{t},0)$. This is a finitely generated $\mathbb{Q}$-algebra with generators
    \[
        \mathsf{t} = \mathsf{x} + \mathsf{y}\, ,\quad
        \mathsf{x} \mathsf{y}\, ,\quad
        \mathsf{x}\mathsf{y}(\mathsf{x} - \mathsf{y})\,.
    \]
    With this notation, the piecewise polynomial classes pulled back from $B$ correspond precisely to the polynomials in $\mathsf{t}$.
	Recall the situation from Figure \ref{fig:thm_bdy} (see also Theorem \ref{thm:bdy}). By definition, piecewise polynomial classes on
	$\barA$ are pushforwards along $i$ and $j$ of polynomials in the Chern roots of the normal bundles of these maps. To compute these
	Chern roots, we introduce the following notation:

	Let $\theta^C_1, \theta^C_2, \ell^C \in \CH^1(C)$ be the symmetric Theta class
	pulled back from $\mathcal X_{g-1}$, the symmetric Theta class pulled back from $\mathcal X_{g-1}^{\vee}$
	and the first Chern class of the Poincar\'e line bundle $\mathcal P_C$, respectively. These
	classes can also be seen in $\CH^1(Y)$ by pullback. Let $h = c_1(\mathcal O_Y(1)) \in \CH^1(Y)$
	be the hyperplane class. We denote by $\lambda_1^C, \ldots, \lambda_{g-1}^C\in \CH^{\ast}(C)$ the $\lambda$ classes pulled back
	from $\mathcal A_{g-1}$. We define:
	\begin{enumerate}[label = (\alph*)]
		\item The small tautological ring of $C$ (resp. $Y$) is the subring $\sR^*(C)$ (resp. $\sR^*(Y)$) generated by $\theta_1^C, \theta_2^C, \ell^C$ (resp. $\theta_1^C, \theta_2^C, \ell^C, h$).

		\item The tautological ring of $C$ (resp. $Y$) is the subring $\R^*(C)$ (resp. $\R^*(Y)$) generated by the corresponding small tautological ring and the $\lambda^C_i$ classes.
	\end{enumerate}
	Note that $\lambda_i^{C}$ is the restriction of $\lambda_i\in \CH^{\ast}(\barA)$ to $C$ by \cite{vdG99}.
	Finally, write
	\[
		U := 2h-\ell^C = [s_0] + [\sbar_\infty].
	\]
	Then
	$$
	[s_0]=\frac{U+\ell^C}{2}\,,\qquad[\sbar_{\infty}]=\frac{U-\ell^C}{2}\,.
	$$
	A direct computation shows that the Chern roots of the normal bundles of $i$ and $j$ are given by
	$$
	c_1(N_i) = -2\theta_2^C -U\,,\quad \alpha_1(N_j) = \ell^C\, ,\quad \alpha_2(N_j)=-2\theta_2^C-\ell^C\, .
	$$
	These formulas, together with
	\begin{equation}\label{eq:normalbunex}
	1 = \Phi(1)\, ,\quad \frac{1}{2}i_*(c_1(N_i)^a) = \Phi (\mathsf{x}^{a+1} + \mathsf{y}^{a+1})\, ,\text{ and}\quad \frac12j_*(\alpha_1(N_j)^{b_1}\cup\alpha_2(N_j)^{b_2}) = \Phi(\mathsf{x}^{b_1+1}\mathsf{y}^{b_2+1})\, ,
	\end{equation}
	allow one to compute explicitly the image of any piecewise polynomial on $\Sigma^{1}$ under $\Phi:\spp^{\ast}(\Sigma^{1})\to \CH^{\ast}(\barA)$
	as a ``decorated strata class'', i.e., a class of the form
	\begin{equation}\label{eq:dec_strata}
		\alpha + \frac{1}{2}i_*(\beta) + \frac{1}{2}j_*(\gamma)\, ,
	\end{equation}
	with $\alpha \in \mathbb Q[\theta, \lambda_i]$, $\beta\in \R^*(Y)$ and $\gamma \in \R^*(C)$.
	In Corollary \ref{cor:i_* in sR} we will show that any decorated strata class lies in $\R^{\ast}(\barA)$.
\end{example}

\begin{remark}
	By the projective bundle formula, we get
	$$
	\sR^*(Y) \cong \frac{\sR^*(C)[U]}{\langle U^2-(\ell^C)^2\rangle}\,.
	$$
\end{remark}

\section{The Fourier transform}\label{sec:Fourier}

\subsection{Semi-abelian Ng\^o fibration}

Abelian fibrations appear in several contexts; see \cite{Ngo,AF,PerverseFourierMSY25}. In this paper, we will
use the following notation:

\begin{definition}\label{def:sa}
    A {\em semi-abelian Ng\^o fibration} over a smooth algebraic stack $B$ is a tuple
    \[
    (\pi : \barA \to B, G\to B, \mu : G\times_B \barA \to \barA),
    \]
	where
    \begin{enumerate}[label=(\alph*)]
        \item $\pi :\barA \to B$ is a surjective, flat, Gorenstein, locally projective morphism whose geometric fibers are reduced, connected and pure of dimension $g$.
        \item $G \to B$ is a $\delta$-regular semi-abelian group scheme of dimension $g$. Moreover, over the locus $B^\circ \subset B$ where the morphism is proper, $G|_{B^\circ} \to B^\circ$ is principally polarized.
        \item For every $b\in B$, the action $G_b \curvearrowright \barA_b$ has affine stabilizers.
        \item The locus of $b \in B$ where $\barA_b$ is geometrically reducible has codimension at least two.
		\item There is a $G$-equivariant open immersion $G\hookrightarrow \barA$. In particular, there is a section $B\to \barA$ induced by the unit section $B\to G$.
    \end{enumerate}
\end{definition}

For any geometric point $b \in B$, the semi-abelian variety $G_b$ admits a unique Chevalley decomposition
\begin{equation}\label{eq:Chev}
	0 \to T_b \to G_b \stackrel{q}{\to} A_b \to 0\,,
\end{equation}
where $A_b$ is an abelian variety and $T_b$ is an algebraic torus. Let $\delta : B \to \ZZ_{\ge 0}$, $b \mapsto \dim T_b$. The family $G\to B$
is said to be $\delta$-regular if
\[
\codim \{ b \in B : \delta(b) \ge k\} \geq k
\]
for all $k \ge 0$.

Point (d) implies that the relative dualizing sheaf $\omega_\pi$ is isomorphic to the pullback of a line bundle from $B$ (\cite[Remark 2.2]{AF}).

We use the language of biextensions from \cite{Mum_ext,Hodge3}. For a semi-abelian variety $G_{\Bbbk}$ over a field $\Bbbk$, $\mathcal P \in \operatorname{Biext}(G_{\Bbbk}, G_{\Bbbk};\mathbb G_{m})$ induces a homomorphism from $G_{\Bbbk}$ to its Picard variety:
$$
\varphi_{\mathcal P} : G_{\Bbbk} \to \operatorname{Ext}(G_{\Bbbk}, \mathbb G_{m}) \subset \Pic(G_{\Bbbk})\,, \quad x \mapsto \mathcal P|_x\,.
$$

If we identify $G_{\Bbbk}$ with the $1$-motive
    $M_{G_{\Bbbk}} = [0 \to G_{\Bbbk}]$ as in \cite{Hodge3}, then its dual is given by $M_{G_\Bbbk}^\vee = [X \to A_{\Bbbk}^\vee]$, where $X$ is the character lattice of $T_{\Bbbk}$. In particular:
    $$
    \operatorname{Biext}(M_{G_\Bbbk}, M_{G_\Bbbk}; \mathbb G_m) = \Hom(M_{G_{\Bbbk}}, M_{G_\Bbbk}^\vee) = \Hom(G_{\Bbbk}, A_{\Bbbk}^\vee) = \Hom(A_{\Bbbk}, A_{\Bbbk}^\vee) = \operatorname{Biext}(A_{\Bbbk}, A_{\Bbbk}; \mathbb G_m)\, ,
    $$
    since there are no nonconstant morphisms from $T_{\Bbbk}$ to $A_{\Bbbk}$. It follows that every biextension $\mathcal P$ of $G_\Bbbk\times G_\Bbbk$ by $\mathbb{G}_{m}$ is of the form $q^* \mathcal Q$ for some biextension $\mathcal Q$ of $A_\Bbbk\times A_\Bbbk$ by $\mathbb{G}_{m}$.

\begin{definition}
    Let $G\to B$ be a semi-abelian group scheme. A $\mathbb G_m$-biextension $\mathcal P \to G^2$ is
    \emph{nondegenerate} if, for every point $b \in B$ with Chevalley
    decomposition \eqref{eq:Chev}, the unique biextension
    $\mathcal Q_b \in \operatorname{Biext}(A_b,A_b;\mathbb G_m)$ satisfying
    $q^*\mathcal Q_b = \mathcal P_b$ induces an isogeny
    $\varphi_{\mathcal Q_b} : A_b \to A_b^\vee$.
\end{definition}

We also make the following assumption, which is related to extending the Fourier--Mukai equivalence to the boundary.

\begin{assumption}\label{assum1}
    Let $(\barA,B,G)$ be a semi-abelian Ng\^o fibration. We assume the following:
    \begin{enumerate}[label=(\alph*)]
        \item There exists a symmetric nondegenerate biextension
        $\mathcal P \to G^2$ whose associated line bundle extends the
        normalized Poincar\'e line bundle over $B^\circ$. Moreover,
        $\mathcal P$ extends\footnote{Since $\mathcal P$ is symmetric, it is
        enough to require an extension to
        $ G \times_B \barA$.} to a symmetric line bundle
        \[
        \mathcal P^\circ \in \Pic(G \times_B \barA \cup \barA \times_B G)
        \]
        that is compatible with the group action.

        \item The line bundle $\mathcal P^\circ$ extends to a maximal
        Cohen--Macaulay sheaf
        \[
        \Pbar \in \operatorname{Coh}(\barA^2)
        \]
        which is flat along the projections
        $\pi_1,\pi_2 : \barA^2 \to \barA$.

    \end{enumerate}
\end{assumption}

\begin{remark}
    \begin{enumerate}[label = (\alph*)]
        \item Compatibility with the group action
        $\mu : G \times_B \barA \to \barA$ means that
        \begin{equation}\label{eq:compPoin}
            (\operatorname{id}_G \times \mu)^*\mathcal P^{\circ}
            \cong
            \pi_{1,2}^*\mathcal P^\circ \otimes \pi_{1,3}^*\mathcal P^\circ .
        \end{equation}
        For any point $a \in G$ mapping to $b \in B$, we obtain a line
		bundle
        \[
        \mathcal P_a^\circ
        :=
        \mathcal P^\circ|_{\{a\}\times \barA_b}
        \in \Pic(\barA_b)
        \]
        such that $\mathcal P_a^\circ|_{G_b} = \varphi_{\mathcal P}(a)$.

        \item Let $\phi : \barA \times_B G \cup G \times_B \barA \hookrightarrow \barA^2$ be the open embedding. Since the complement of the image of $\phi$ has codimension at least $2$, the extension property of Cohen--Macaulay sheaves implies that $\Pbar = \phi_*\mathcal P^\circ$.
    \end{enumerate}
\end{remark}

We now state the main result of this section, whose proof will be given in the
next section.
\begin{theorem}\label{thm:daf}
    Let $(\barA, B, G)$ be a semi-abelian Ng\^o fibration satisfying Assumption \ref{assum1}.
    \begin{enumerate}[label = (\alph*)]
        \item For $\Pbar \in \Coh(\barA^2)$ and $\Pbar^{-1}:= \hom(\Pbar,\CO_{\barA^2}) \otimes \pi_2^*\omega_\pi [g] $, we have
        \[
        \Pbar \circ \Pbar^{-1} \cong \CO_{\Delta_{\barA/B}}, \quad \Pbar^{-1} \circ \Pbar \cong \CO_{\Delta_{\barA/B}}\in \dbcoh(\barA^2)\,.
        \]
        \item There exists an object $\K \in \dbcoh(\barA^3)$ supported in codimension $\ge g$ which satisfies
        \[
        \Pbar \circ \K \cong \CO_{\Delta^{\mathrm{sm}}_{\barA/B}} \circ (\Pbar \boxtimes \Pbar) \in \dbcoh(\barA^3)\,.
        \]
        Here $\Delta^{\mathrm{sm}}_{\barA/B}$ denotes the small diagonal in $\barA^{3}$.
    \end{enumerate}
\end{theorem}

Point (a) is a generalization of Arinkin's result \cite{Arinkin2}, and point (b) first appears in \cite{PerverseFourierMSY25} (see also \cite{Daf}).
In contrast to loc. cit., we do not assume that the morphism $\pi : \barA \to B$ has full support.

We end this section with the following result:

\begin{theorem}\label{thm:deq}
    The tuple $(\mathcal{X}_g',\mathcal{A}_g',G)$, with
    $\Pbar := \mathrm{R}f_*\P$, is a semi-abelian Ng\^o fibration satisfying Assumption
    \ref{assum1}. In particular, $\Pbar$ induces an autoequivalence
    \[
    \FM_{\Pbar} : \dbcoh(\barA) \xrightarrow{\cong} \dbcoh(\barA),
    \]
    which is linear over $\dbcoh(B)$.
\end{theorem}

\begin{proof}
    Since $\pi : \barA \to B$ is a semi-stable morphism with integral fibers,
    the conditions in Definition \ref{def:sa} are easy to check. Assumption
    \ref{assum1} follows from \cite[Section 9.5]{BMP}. In particular,
    \[
    \mathrm{R}^{>0} f_*\P = 0, \quad \mathrm{R}^{>0} f_*\P^\vee = 0\,,
    \]
    and $f_*\P$ is the maximal Cohen--Macaulay extension of $\P^\circ$.
    The last statement follows from Theorem~\ref{thm:daf}.
\end{proof}

\begin{remark}\label{rmk:DbFinv}
    Since $\mathrm{R}f_* \mathcal P = \Pbar$, we also have that
    $$
    \FM_{\Pbar}(E) = \FM_{\mathcal P}(E) = \mathrm{R}p_{2,*} (p_1^* E \otimes \P)\,.
    $$
    By Theorem \ref{thm:daf}, its inverse is given by
    $$
    \FM_{\Pbar}^{-1}(E) = \FM_{\P^{\vee} \otimes \det(\mathbb E)[g]}(E) = \mathrm{R}p_{2,*}(p_1^* E \otimes \mathcal P^{\vee})\otimes \det(\mathbb E)[g]
    $$
\end{remark}

\begin{remark}\label{rmk:full}
    Since $(\mathcal{X}_g',\mathcal{A}_g',G)$ is a weak abelian fibration, with $G/B$ polarizable by \cite{AD}, and since all fibers are integral, every irreducible summand in the decomposition theorem \cite{BBDG} for $\pi : \barA \to B$ has full support by \cite[Th\'eor\`eme 7.2.1]{Ngo}.
\end{remark}

\subsection{Proof of Theorem \ref{thm:daf}}

In this section, we prove Theorem \ref{thm:daf}. The following result controls the dimensions of compositions of integral kernels, replacing the Abel--Jacobi arguments used for compactified Jacobians in \cite{Arinkin1,Arinkin2}.

\begin{lemma}\label{lem:semi-abelian_bound}
    Let $G_\Bbbk$ be a semi-abelian variety over a field $\Bbbk$ with Chevalley decomposition \eqref{eq:Chev}. If $\mathcal Q \in \operatorname{Biext}(A_\Bbbk,A_\Bbbk;\mathbb G_m)$ induces an isogeny $\varphi_{\mathcal Q} : A_\Bbbk \to A_\Bbbk^\vee$ and $\mathcal P = q^*\mathcal Q$, then the fibers of $\varphi_{\mathcal P}$ are countable unions of subvarieties of codimension $\dim A_{\Bbbk}$.
\end{lemma}

\begin{proof}
    Consider the diagram
    \[
    \begin{tikzcd}
        G_{\Bbbk} \ar[d,"\varphi_{\mathcal{P}}"] \ar[r,"q"] & A_\Bbbk \ar[d, "\varphi_{\mathcal{Q}}"]\\
        \Pic (G_\Bbbk) & \Pic(A_\Bbbk) \ar[l, "q^*"']
    \end{tikzcd}
    \]
    The square is commutative because $\mathcal P = q^*\mathcal Q$. Since
    $\varphi_{\mathcal Q}$ is an isogeny, its fibers are finite. Since $q$ is a torus bundle, each fiber of $q^*$ is a countable set of points. Hence, for any $e \in \Pic(G_\Bbbk)$, the set-theoretic inverse
    image $(q^* \circ \varphi_{\mathcal Q})^{-1}(e)$ is a countable discrete set. Since $q$ is flat, it follows that
    \[
    \varphi_{\mathcal P}^{-1}(e)
    =
    (q^* \circ \varphi_{\mathcal Q} \circ q)^{-1}(e)
    \]
    is a countable union of subvarieties of codimension $\dim A_\Bbbk$.
\end{proof}

Denote by $\barA^{s}$ the $s$-fold relative fiber product.
Consider the projections
\[
p : \barA^{s} \to B, \quad \pi_i :\barA^s \to \barA,\quad \pi_{i,j} : \barA^s \to \barA^2\,, \cdots\,.
\]
Let $\Pbar^{\otimes -1} = \Pbar^\vee$. For a vector $\varepsilon = (\varepsilon_1,\ldots, \varepsilon_s)$, $\varepsilon_i \in \{\pm 1\}$, denote
\[
\K_s^{\varepsilon} := \mathrm{R}\pi_{1,\dots,s*}(\pi_{1,s+1}^*\Pbar^{\otimes \varepsilon_1} \otimes \cdots \otimes \pi_{s,s+1}^*\Pbar^{\otimes \varepsilon_s}) \in \dbcoh(\barA^s)\,.
\]
By the flatness of $\Pbar$ and $\Pbar^{\vee}$, the tensor product is quasi-isomorphic to the derived tensor product.

\begin{proposition}[\cite{Arinkin2}]\label{pro:cdim}
    For $s \ge 1$, $\codim_{\barA^s} (\supp (\K^{\varepsilon}_s)) \ge g$.
\end{proposition}

\begin{proof}
    This condition is smooth-local on $B$, so we may assume that $B$ is a smooth variety. Since $G \to B$ is $\delta$-regular, every irreducible subvariety $W \subset B$ satisfies $\codim(W) \geq \delta(W)$. Therefore, it suffices to show that, for every geometric point $b \in B$,
    \[
    \codim\bigl(\supp \K_s^{\varepsilon} \cap \barA_b^s\bigr) \geq g - \delta(b).
    \]

    For $x \in \barA$, let $\Pbar_x$ denote the restriction of $\Pbar$ to
    $\barA \times \{x\}$. By cohomology and base change,
    $(x_1,\ldots,x_s) \in \supp \K^{\varepsilon}_s$ if and only if
    \[
    \mathbb H^i\bigl(
        \barA_b,
        \Pbar^{\otimes\varepsilon_1}_{x_1} \otimes^{\LL} \cdots \otimes^{\LL} \Pbar^{\otimes \varepsilon_s}_{x_s}
    \bigr) \neq 0
    \]
    for some $i$. By the same argument as in
    \cite[Proposition 7.2]{Arinkin2}, we therefore have
    \begin{equation}\label{eq:cdim1}
        \bigotimes_{i=1}^s \, (\Pbar_{x_i}|_{G_b})^{\otimes \varepsilon_i} \cong \CO_{G_b}
    \end{equation}
	whenever $(x_1,\dots,x_s)\in \supp \K^{\varepsilon}_s$.

    Consider the action morphism $\mu \times \id : G_b \times \barA_b^s \to \barA_b^s$ on the first component. For $x_1,\ldots,x_s \in \barA_b$, set
    \[
    Z :=
    (\mu \times \id)^{-1}\supp(\K_s^\varepsilon)
    \cap
    \bigl(G_b \times \{(x_1,\ldots,x_s)\}\bigr).
    \]
    By \eqref{eq:cdim1} and \eqref{eq:compPoin}, if
    $(y,x_1,\ldots,x_s) \in Z$, then
    \[
    \varphi_{\mathcal P}(y)
    =
    \P_y
    \cong
    \bigotimes_{i=1}^s (\P_{x_i})^{\otimes (-\varepsilon_1\varepsilon_i)}.
    \]
    By Lemma \ref{lem:semi-abelian_bound}, the fiber of
    $\varphi_{\mathcal P}$ has codimension $g-\delta(b)$. Hence
    \[
    \codim \, Z \geq g-\delta(b).
    \]
    This proves the desired fiberwise codimension bound, and hence the claim.
\end{proof}

\begin{proposition}[\cite{Arinkin1}]\label{pro:partial}
    Let $p_1 : G \times_B \barA \to G$ be the projection onto the first component. Denote by $e : B \to G$ the unit section and let $\EE \to B$ be the Hodge bundle. Then
    \[
    \mathrm{R}p_{1,*} \P^\circ \cong e_* (\det \EE^{\vee}) [-g] \in \dbcoh(G)\,.
    \]
\end{proposition}

\begin{proof}
    The proof is identical to \cite{Arinkin1}. The argument involving the irreducibility of the fibers can be replaced by \cite[Lemma 6.5]{MRV1}.
\end{proof}

\begin{proof}[Proof of Theorem \ref{thm:daf}]
    We first prove part (a). Consider the composition
    \[
    \Psi
    :=
    \mathrm{R}\pi_{1,3,*}\bigl(\pi_{1,2}^*\Pbar^{-1} \otimes \pi_{2,3}^*\Pbar\bigr)
    \in \dbcoh(\barA^2),
    \]
    where, by the flatness of $\Pbar$, the ordinary tensor product is
    quasi-isomorphic to the derived tensor product. It is enough to show that
    \begin{equation}\label{eq:psi}
        \Psi
        \cong
        \CO_{\Delta_{\barA/B}} \in \dbcoh(\barA^2).
    \end{equation}

    By the same argument as in Proposition \ref{pro:cdim} and
    \cite[Proposition 7.6]{Arinkin2}, for any $b \in B$, we have
    \[
    \codim(\supp \Psi_b) \geq g - \delta(b),
    \]
    and the inequality is strict if $\delta(b)>0$. On the other hand, pulling
    back the equality from Proposition \ref{pro:partial} along the difference
    map $G^2 \to G$ gives \eqref{eq:psi} over $G^2$. The rest
    of the argument is identical to \cite[Proof of Proposition 7.1]{Arinkin2}.

    For part (b), by definition of the convolution kernel $\K$, we have
    \[
    \K
    \cong
    \Pbar^{-1}
    \circ
    \CO_{\Delta^{\mathrm{sm}}_{\barA/B}}
    \circ
	(\Pbar \boxtimes \Pbar) \cong \mathcal K_3^{(1,1,-1)}\otimes \omega_{\pi}[g].
    \]
    Therefore, the desired codimension estimate follows from Proposition
    \ref{pro:cdim}.
\end{proof}

From now on, we will restrict ourselves to the case $(\barA,B,G):=(\mathcal{X}_g',\mathcal{A}_g',G)$.

\subsection{Fourier image of the principal polarization}\label{sec:FMn}

We study the interaction of the Fourier transform with the principal polarization. While $\CO(2\Theta)$ is a line bundle,
$\CO(\Theta)$ is only a $\QQ$-line bundle, so one must treat cohomological calculations for $\CO(\Theta)$ with additional care.

For an even positive integer $n$, consider the Fourier transform
\[
\FM^{-1}_n(E) := \mathrm{R}p_{2,*}\bigl(p_1^*E\otimes \P^{\otimes -n}\bigr)\otimes \det(\mathbb E)[g] : \dbcoh(\barA) \to \dbcoh(\barA)
\]
induced by the $n$-th power of the extended Poincar\'e line bundle. By Lemma \ref{lem:thetample}, the derived pushforward
\begin{equation}\label{eq:defMn}
\Mc_n := \mathrm{R}\pi_*\CO(n \Theta) \cong \pi_*\CO(n\Theta)
\end{equation}
is a vector bundle of rank $n^g$.

\begin{proposition}\label{prop:FM-theta}
    Let $p_2 : \barA^{(2)} \to \barA$ be the projection onto the second factor. For any even positive integer $n$, we have $
    \mathrm{R}p_{2,*}\diff^*\CO(n\Theta) \cong \pi^*\mathcal M_n$.
\end{proposition}
\begin{proof}
    Consider the diagram
    \begin{equation}\label{eq:diagg}
    \begin{tikzcd}
    	{\barA^{(2)}} & {\barA^{2}} & \barA \\
    	& \barA & B
    	\arrow["d", from=1-1, to=1-2]
    	\arrow["{p_2}"', from=1-1, to=2-2]
	    \arrow["{\pi_1}", from=1-2, to=1-3]
	    \arrow["{\pi_2}"', from=1-2, to=2-2]
    	\arrow[from=1-3, to=2-3]
    	\arrow[from=2-2, to=2-3,"\pi"]
    \end{tikzcd}
    \end{equation}
    where $d(a_1,a_2) = (\diff(a_1, a_2),a_2)$. Then, $\mathrm{L}d^*\pi_1^*\mathcal O(n\Theta) = \diff^*\mathcal O(n\Theta)$, so
    $$
    \mathrm{R}p_{2,*}(\diff^*\mathcal O(n\Theta)) = \mathrm{R}\pi_{2,*}(\mathrm{R}d_*\mathrm{L}d^*(\pi_1^*\mathcal O(n\Theta))) = \mathrm{R}\pi_{2,*}((\mathrm{R}d_*\mathrm{L}d^*\mathcal O)\otimes \pi_1^*\mathcal O(n\Theta))
    $$
    We claim that the map $d$ is a log blowup. This can be seen by considering the following diagram:
    $$
    \begin{tikzcd}
	{\barA^{(2)}} && \\
	& {\barA^{2, \Sigma'}} & {\barA^{2}} \\
	& {(A^{\log})^2} & {(A^{\log})^2}
	\arrow["b"', from=1-1, to=2-2]
	\arrow["d", curve={height=-6pt}, from=1-1, to=2-3]
	\arrow["{d'}", from=2-2, to=2-3]
	\arrow[from=2-2, to=3-2]
	\arrow[from=2-3, to=3-3]
	\arrow["{d^{\log}}"', from=3-2, to=3-3]
\end{tikzcd}
    $$
	where the square is an f.s. fiber square.
	Here $\Sigma'$ is the subdivision of $({A^{\trop}})^2$ corresponding to the periodic subdivision of $\mathbb R^2\times{\mathbb R}_{\geq 0}$ along the hyperplanes $\mathsf{x}_1-\mathsf{x}_2 \in \mathsf{t}\mathbb Z$, $\mathsf{x}_2 \in \mathsf{t}\mathbb Z$. Note that $\Sigma^{(2)}$ is a refinement of $\Sigma'$, so $b$ is a log blowup. Since $d^{\log}$ is an isomorphism, with inverse $(a_1, a_2)\mapsto (a_1+a_2, a_2)$, so is $d'$. Therefore, $d$ is a composition of a log blowup and an isomorphism. In particular, $\mathrm{R}d_*\mathrm{L}d^* \mathcal O = \mathcal O$.

    Since the square in \eqref{eq:diagg} is Cartesian and the maps are flat, we can apply flat base-change to obtain
    $$
    \mathrm{R}p_{2,*}(\diff^*\mathcal O(n\Theta)) = \mathrm{R}\pi_{2,*}( \pi_1^*\mathcal O(n\Theta)) = \pi^* \mathcal M_n.
    $$
\end{proof}

\begin{theorem}\label{thm:FM-1Db}
For every even positive integer $n$, $\FM^{-1}_n(\mathcal O(n\Theta)) \cong \mathcal O(-n\Theta)\otimes \pi^*\mathcal M_n\otimes \det(\mathbb E)[g]$ in $\dbcoh(\barA)$.
\end{theorem}
\begin{proof}
    By Mumford's formula \eqref{eq:Mumform}, we have
    \begin{align*}
        \FM^{-1}_n\big(\CO(n\Theta)\big)
    &= \mathrm{R}p_{2,*}\!\Big(p_1^*\CO(n\Theta)
    \otimes \diff^*\CO(n\Theta)
    \otimes p_1^*\CO(-n\Theta)
    \otimes p_2^*\CO(-n\Theta)
    \Big)\otimes \det(\mathbb E)[g] \\
    &\cong
    \CO(-n\Theta) \otimes \mathrm{R}p_{2,*}\diff^*\CO(n\Theta)\otimes \det(\mathbb E)[g] \\
    &\cong
    \CO(-n\Theta) \otimes \pi^*\Mc_n\otimes \det(\mathbb E)[g]\,,
    \end{align*}
    where the last equality follows from Proposition \ref{prop:FM-theta}.
\end{proof}

\subsection{Chow-theoretic realization}\label{sec:chow}

We briefly recall the relative correspondences constructed by Corti--Hanamura \cite{CortiHanamura}. Fix a smooth Deligne--Mumford stack $B$. Let $f_1 : X_1 \to B, \,f_2 : X_2 \to B$ be proper representable morphisms from smooth Deligne--Mumford stacks $X_1$ and $X_2$, and assume that $X_2$ is equidimensional. We define
\[
\corr_B^k(X_1,X_2) := \CH_{\dim X_2-k}(X_1\times_B X_2).
\]

In what follows, many relative correspondences, such as $\fm$, are defined as proper pushforwards of algebraic cycles from a smooth model $X_{12} \to X_1 \times_B X_2$. In this case, for $Z \in \CH^{\dim X_1-\dim B+k}(X_{12})$, we also denote by
\[
Z \in \CH^{\dim X_1-\dim B+k}(X_{12}) \to \corr_B^{k}(X_1,X_2)
\]
its proper pushforward to $X_1 \times_B X_2$. A detailed study will appear in \cite{BP}.

We also use the same notation for a relative correspondence and for the induced map on Chow groups.

We consider the Chow-theoretic realization of the extended Poincar\'e line bundle with the Todd class normalization used in \cite{PerverseFourierMSY25}. For any separated Deligne--Mumford stack $M$, consider the Baum--Fulton--MacPherson map (\cite[Chapter 18]{Fulton})
\begin{equation}\label{eq:bfm}
    \tau : K_0(M) \to \CH_*(M)\,.
\end{equation}

Since $\barA^2$ is lci, it has a virtual tangent bundle $T_{\barA^2} \in K^0(\barA^2)$. Write
\[
\fm := \Td(-T_{\barA^2})\cap \tau(\Pbar), \quad \fm^{-1} := \Td(-T_B) \cap \tau(\Pbar^{-1})\,,
\]
and define $\fm_k$ and $\fm_k^{-1}$ to be the codimension $k$ parts of the respective correspondences,
so that
\[
\fm = \sum_{k=0}^\infty \fm_k\,, \quad \fm^{-1} = \sum_{k=0}^{\infty} \fm_k^{-1}.
\]
Applying the map $\tau$ from \eqref{eq:bfm} to Theorem \ref{thm:daf}, we get
\begin{equation}\label{eq:fourier_is_iso}
\fm \circ \fm^{-1} = \id, \quad \fm^{-1} \circ \fm = \id\,.
\end{equation}
In particular, the class of the relative diagonal has the form
\begin{equation}\label{eq:diag}
    \Delta_{\barA/B} = \sum_{k=0}^{2g} \fm_k \circ \fm_{2g-k}^{-1}\,.
\end{equation}

For a semi-stable morphism $f: X\to B$ between log schemes, the residue sheaf $\cR_f$ is defined as the quotient sheaf $\Omega^1_{f,\log}/\Omega_f^1$ (\cite[\S 5]{BMP}).

Let $\pi :\barA\to B$ and $\pi^{(2)} : \barA^{(2)} \to B$. By \cite[Proposition 4.4]{BMP}, the Todd class of $\cR_\pi$ has the following form

\begin{equation}\label{eq:todd}
    \Td^\vee(\cR_\pi)^{-1} =1 + \frac{1}{2}j_*\Big(\sum_{m=1}^\infty \frac{(-1)^{m-1}|B_{2m}|}{(2m)!} \frac{(\ell^C)^{2m-1}-(\ell^C+2\theta_2^C)^{2m-1}}{-2\theta_2^C}\Big)\,,
\end{equation}
where $B_k$ is the $k$-th Bernoulli number
\begin{equation*}
    \frac{x}{e^x-1} = \sum_{k=0}^\infty B_k \frac{x^k}{k!}\,.
\end{equation*}

Combining \eqref{eq:todd} and the following lemma gives explicit formulas for the Chow-theoretic correspondences $\fm$ and $\fm^{-1}$.

\begin{lemma}\label{lem:Fchow}
    Let $\fm, \fm^{-1}\in \CH_*(\barA^2)$ be defined as above and write $f:\barA^{(2)}\to \barA^{2}$ for the blowdown
	morphism. Then:
    \begin{enumerate}[label = (\alph*)]
        \item $\fm = f_*(\ch(\P) \cup \Td^\vee(\cR_{\pi^{(2)}}- p_1^*\cR_\pi - p_2^*\cR_\pi))$.
        \item $\fm^{-1} =(-1)^g f_*(\ch(\P^{\vee}) \cup \Td^\vee(\cR_{\pi^{(2)}})^{-1})$.
    \end{enumerate}
\end{lemma}

\begin{proof}
    This follows from the same argument as in \cite[\S 4]{BMP}.
\end{proof}

\begin{corollary}\label{cor:F(exp_theta)}
    \begin{equation}\label{eq:Ftheta}
        \fm(\exp(-\theta) \cup \Td^\vee(\cR_\pi)^{-1}) = (-1)^g\exp(\theta- D/8)\in \CH^{\ast}(\barA)\,.
    \end{equation}
\end{corollary}

\begin{proof}
    We apply the Baum--Fulton--MacPherson map \eqref{eq:bfm} to Theorem~\ref{thm:FM-1Db}. For even positive integers $n$, the left-hand side is
    \begin{align*}
    \tau\big(\FM^{-1}_n(\CO(n\Theta))\big)
    &= \tau\big(\mathrm{R}p_{2,*}(p_1^*\CO(n\Theta) \otimes \P^{\otimes -n})\otimes \det(\mathbb E)[g]\big) \\
    &= (-1)^g\ch(\det\mathbb E) \cup p_{2,*}\!\left(\tau(p_1^*\CO(n\Theta) \otimes \P^{-\otimes n})\right) \\
    &= (-1)^g\ch (\det \mathbb E)\cup p_{2,*}\!\left(\ch(\P^{\otimes- n})\cup p_1^*\ch(\CO(n\Theta))\cup \Td(T_{\barA^{(2)}})\right)\,.
    \end{align*}
    In $K^0(\barA^{(2)})$, we have by Remark \ref{rem:Hodge_is_Omlog}
    $$
    T_{\barA^{(2)}} = T_B + T_{\barA^{(2)}/B}^{\vir}\, ,\quad (T^{\vir}_{\barA^{(2)}/B})^\vee = \Omega_{\pi^{(2)}}^{\log} - \cR_{\pi^{(2)}} = \mathbb E^{\oplus 2} - \cR_{\pi^{(2)}}\, .
    $$
	By the main theorem of \cite{EV04}, Mumford's relation
    \begin{equation}\label{eq:mumford}
        c(\EE \oplus \EE^\vee)-1=0
    \end{equation}
    extends to $B$; see \S\ref{sec:intro_taut}.
    It is not hard to see that \eqref{eq:mumford} implies the relation $\Td(\EE^\vee)\cup \,e^{\lambda_1/2} = 1$, so
    \begin{equation}\label{eq:simplifytau}
    \tau\big(\FM_n^{-1}(\mathcal O(n\Theta))\big) = (-1)^g\Td(T_B) \cup p_{2,*}\!\left(\ch(\P^{\otimes- n})\cup p_1^*\ch(\CO(n\Theta))\cup \Td^\vee(\cR_{\pi^{(2)}})^{-1}\right)\,.
    \end{equation}
    Similarly, the right-hand side of Theorem \ref{thm:FM-1Db} is
    \begin{align*}
    \tau\big(\CO(-n\Theta) \otimes \pi^*\Mc_n\otimes \det (\mathbb E)[g]\big)
    &=
    (-1)^g\ch(\det(\mathbb E))\cup\ch\big(\CO(-n\Theta) \otimes \pi^*\Mc_n\big)\cup\Td(T_{\barA})\\
    &=(-1)^g\Td(T_B)\cup \ch(\mathcal O(-n\Theta))\cup \Td^\vee(\cR_{\pi})^{-1}\cup \ch(\pi^*\mathcal M_n)\cup e^{\lambda_1/2}\,.
    \end{align*}
    We divide both sides by $\Td(T_B)$. Applying the map $\tau$ from \eqref{eq:bfm} and Mumford's relation \eqref{eq:mumford} to \eqref{eq:defMn}, we see that
    $$
    \ch(\mathcal M_n)\cup e^{\lambda_1/2} = \pi_*(\ch(\mathcal O(n\Theta))\cup\Td^\vee(\cR_{\pi})^{-1})
    $$
    is a polynomial function in $n \in 2\ZZ_{>0}$.
	Therefore, both sides of Theorem \ref{thm:FM-1Db} become polynomial functions in $n \in 2\ZZ_{>0}$ after applying $\tau$, and hence the identity continues to hold for $n=1$. We will see in Corollary \ref{cor:pushThet} that
    \begin{equation}\label{eq:evalMn1}
    \ch(\mathcal M_n) \cup e^{\lambda_1/2} \Big|_{n=1} = \exp(D/8)\, ,
    \end{equation}
    so the right-hand side of Theorem \ref{thm:FM-1Db} becomes
    $$
    \tau\big(\CO(-n\Theta) \otimes \pi^*\Mc_n\otimes \det (\mathbb E)[g]\big)\Big|_{n=1} = (-1)^g \exp(-\theta + D/8) \cup \Td^\vee (\cR_{\pi})^{-1}\,.
    $$
    By \eqref{eq:simplifytau} and Lemma \ref{lem:Fchow}~(b), the left-hand side becomes
    $$
    \tau(\FM_n^{-1}(\mathcal O(n\Theta))) \Big|_{n=1} = \mathfrak F^{-1}(\exp(\theta))\,.
    $$
    Therefore,
    $$
    \mathfrak F^{-1}(\exp(\theta)) = (-1)^g\exp(-\theta+D/8)\cup\Td^\vee(\cR_{\pi})^{-1}\, .
    $$
    Applying $\mathfrak F$ to both sides completes the proof.
\end{proof}

\section{Weight of tautological classes}\label{sec:wt}

\subsection{Sketch of the proof}\label{sec:wt_intro}

The goal of this section is to prove the following theorem:

\begin{theorem}\label{thm:wt}
    Let $\gamma \in \R^{c}(\barA^{(s)})$.
    For $1\le i \le s$, denote by $[\gamma]^{i,(w)}$ the weight $w$ part of $\gamma$ with respect to $[N_i]^{\ast}$.
	\begin{enumerate}[label = (\alph*)]
		\item $[N_i]^* \gamma$ is polynomial in $N$ and for any $w \ge 0$, $[\gamma]^{i,(w)}$ is a tautological class.
		\item If $w>2c$, then $[\gamma]^{i,(w)} = 0$.
		\item More precisely, if $\varphi \in \spp^{d}(\Sigma^{(s)})$ and $\alpha$ is any class pulled back from $\barA^{(s-1)}$ (where we forget the $i$-th factor), then
			$$\left[\theta_i^{m}\cup\ell_{1,i}^{k_1}\cup\dots \cup\widehat{\ell_{i,i}^{k_i}}\cup\dots\cup\ell_{s,i}^{k_{s}}\cup\Phi(\varphi)\cup\alpha\right]^{i,(w)}=0$$
			for $w>2m+\sum_{j\neq i}k_j+d$.
	\end{enumerate}
\end{theorem}

We omit the label $i$ when it is clear from context. In \S \ref{sec:Nlow} we obtain a similar result for the ``complementary'' operator $[N_i]_*$, see Theorem \ref{thm:wt_lower}.

For the remainder of this section, we outline the strategy of the proof of Theorem \ref{thm:wt}.
The key tool is logarithmic geometry, which reduces the action of $[N_i]^{\ast}$ on $\CH^*(\barA^{(s)})$ to a combinatorial calculation on tropical abelian schemes.

Let us begin by briefly sketching the reduction to combinatorics. For simplicity, we ignore the classes $\theta_{i}$ for the moment. By the projection formula and symmetry, it is clear that it suffices to consider the multiplication by $N$ map in a given fixed direction; e.g., we can consider the case $i=s$. From Lemma \ref{lem:Npush} (b) it follows that $[N_s]^{\ast}(\ell_{i,j} \cup \gamma) = [1\times\dots\times N]^{\ast}(\ell_{i,j}\cup \gamma)=\ell_{i,j}\cup [1\times\dots\times N]^{\ast}(\gamma)$ if $1 \leq i,j <s$, so we may assume that $\gamma$ is a polynomial in the classes $\ell_{j,s}$ and piecewise polynomial classes.
Recall the definition of the operator $[1\times\dots\times N]^{\ast}$ from \S\ref{sec:mbyNmaps}: the rational map
\[
1\times \cdots \times N: \barA^{(s)} \dashrightarrow \barA^{(s)}
\]
is resolved by a correspondence
\[
\begin{tikzcd}
    & \barA_N^{(s)} \ar[dl,"b", swap] \ar[dr,"N_s"] & \\
    \barA^{(s)} & & \barA^{(s)}
\end{tikzcd}
\]
where $b$ is a logarithmic blowup and
\[
[1 \times \cdots \times N]^* := b_*N_s^* : \CH^*(\barA^{(s)}) \to \CH^*(\barA^{(s)}).
\]
We will see in Lemma \ref{lem:Artin_fan_A_N} that this correspondence tropicalizes to an analogous correspondence of fans,
\[
\begin{tikzcd}
    & \Sigma_N^{(s)} \ar[dl,"b", swap] \ar[dr,"N_s"] & \\
    \Sigma^{(s)} & & \Sigma^{(s)}
\end{tikzcd}
\]
The operators $(1 \times \cdots \times N)^*$, $b_*$, and $b^*$ lift naturally from $\CH^{\ast}(\barA^{(s)})$ to $\spp^{\ast}(\Sigma^{(s)})$, compatibly with the operator $\Phi$ defined in \S\ref{sec:taut_def}. Hence we obtain a commutative diagram:
$$
\begin{tikzcd}
	{\sPP^*(\Sigma^{(s)})} & {\sPP^*(\Sigma^{(s)}_N)} & {\sPP^*(\Sigma^{(s)})} \\
	{\CH^*(\barA^{(s)})} & {\CH^*(\barA^{(s)}_N)} & {\CH^*(\barA^{(s)})}
	\arrow["{N_s^*}", from=1-1, to=1-2]
	\arrow["{[N_s]^*}", curve={height=-24pt}, from=1-1, to=1-3]
	\arrow["\Phi", from=1-1, to=2-1]
	\arrow["{b_*}", from=1-2, to=1-3]
	\arrow["\Phi", from=1-2, to=2-2]
	\arrow["\Phi", from=1-3, to=2-3]
	\arrow["{N_s^*}", from=2-1, to=2-2]
	\arrow["{[N_s]^*}"', curve={height=24pt}, from=2-1, to=2-3]
	\arrow["{b_*}", from=2-2, to=2-3]
\end{tikzcd}
$$
Having performed this lift, the special case of Theorem \ref{thm:wt} for piecewise polynomials is reduced to showing the combinatorial statement
\begin{equation}\label{eq:PPisNpoly}
[1 \times \cdots \times N]^*(\gamma) \in \spp^{\ast}(\Sigma^{(s)})[N]^{\le \codim(\gamma)}.
\end{equation}
That is, the operator takes every piecewise polynomial to a polynomial in $N$
of degree at most $\codim(\gamma)$, whose coefficients are piecewise
polynomials. The main complexity in the proof of this fact is due to the
presence of the operator $b_*$. It is computed via \emph{Brion's formula}
(see \cite{Brion} and \cite[Proposition~76]{LogarithmicTauPandha2024}),
which produces a piecewise polynomial on
$\Sigma^{(s)}$ from one on $\Sigma^{(s)}_N$ by averaging the contributions of
certain rational functions on each maximal cone in $\Sigma^{(s)}_N$. This is the
first point at which combinatorial complexity enters: the cones in
$\Sigma^{(s)}_N$ are not uniform (and there are many of them: on the order of $N\cdot s!$), so showing that the average is polynomial in
$N$ rather than a rational function requires a delicate cancellation.

Extending this reasoning to the general case including the classes $\ell_{i,s}, \theta_s$ requires a
further argument, as these classes are not piecewise polynomial. We explain this argument for the classes $\ell_{i,s}$, the case of classes involving $\theta_s$ being analogous. On the abelian locus $(A^{\circ})^s \subset \barA^{(s)}$ the Poincar\'e bundle satisfies
\[
(1 \times \cdots \times N)^*\mathcal{P} = \mathcal{P}^{\otimes N}
\]
and one might naively expect that $[1 \times \cdots \times N]^*\ell_{i,s}^k = N^k\ell_{i,s}^k$. This expectation is false, and it is the correction to this expectation that is piecewise polynomial:

\begin{theorem}[Corollary \ref{cor:pp_corr_poincare_precise}]\label{thm:pp_corr_poincare}
    For $1\le i \le s-1$, there exists a piecewise linear function $c_{\mathcal P,i,N}$ on $\Sigma_N^{(s)}$ such that
    \[
    (1\times\dots\times N)^{\ast}\ell_{i,s}=N \cdot b^{\ast} (\ell_{i,s}) + \Phi(c_{\mathcal P,i,N}).
    \]
\end{theorem}

This introduces the second complication: $c_{\P,i,N}$ is a piecewise linear function
on $\Sigma_{N}^{(s)}$ which is not pulled back from $\Sigma^{(s)}$, so the
polynomiality result in \eqref{eq:PPisNpoly} is not sufficient. The polynomiality result
must be extended at least for products of piecewise polynomials on
$\Sigma^{(s)}$ with powers of $c_{\P,i,N}$. However, it is not straightforward to
give a clean statement of the required result, as polynomiality is not a well-defined notion for
arbitrary piecewise polynomials on $\Sigma_N^{(s)}$. This is due to the fact that the spaces
$\Sigma_N^{(s)}$ do not map to one another as $N$ varies.
Thus the result depends on the precise expression for
$c_{\P,i,N}$. Of course, this class is not arbitrary: the expectation
\[
(1 \times \cdots \times N)^*\mathcal{P}^{\log} =(\mathcal{P}^{\log})^{\otimes N}
\]
is true for the logarithmic Poincar\'e bundle by \cite[Theorem 4.4.1]{Molcho_Wise_2022}, and therefore for its induced
$\mathbb{G}_m^{\textup{trop}}\cong \PL$-torsor. This torsor is not trivial on
$\Sigma^{(s)}$, but it becomes trivial after pulling back to the universal
cover $\mathcal{U}^{s}\Sigma^{(s)}$. The practical consequence is that the
correction term $c_{\mathcal P,i,N}$ in Theorem \ref{thm:pp_corr_poincare} has
the following explicit description.

\begin{lemma}\label{lem:cPi}
    There exists a piecewise linear function $c_{\mathcal P,i}$ on $\mathcal U^{s}\Sigma^{(s)}$ (see Remark \ref{rem:Ptrop}) which
    satisfies
    \[
    c_{\mathcal P,i,N} = (1\times\dots\times N)^{\ast}(c_{\mathcal P,i})-N\cdot b^{\ast}(c_{\mathcal P,i}).
    \]
\end{lemma}

See Corollary \ref{cor:pp_corr_poincare_precise} for a more precise statement. The corresponding statement for $\theta_s$ is Corollary \ref{cor:pp_corr_theta_precise}, which also appears in \cite{BGHdJ}.

To make use of this structure, we move the computation to the universal cover. This does not resolve the issue automatically, as the full ring $\spp(\mathcal{U}^s\Sigma^{(s)})$ is too large: not all of its elements are polynomial in $N$ under $[1 \times \cdots \times N]^*$. However, by the lemma and Theorem \ref{thm:pp_corr_poincare}, it suffices to prove that $b_*(1\times \cdots \times N)^*(f)$ is polynomial in $N$
for $f$ contained in some subring of $\spp^{\ast}(\mathcal U^{s}\Sigma^{(s)})$
containing the functions $c_{\mathcal P,i}$ and the pullback of any piecewise
polynomial on $\Sigma^{(s)}$. This introduces the third complication, which is
finding an adequate candidate ring: small enough that its elements satisfy the above polynomiality, big enough to contain $c_{\mathcal P,i}$ and $\sPP(\Sigma^{(s)})$, and containing a convenient basis in which the weight calculations can actually be carried out.

Morally, the ring we choose is the subring generated by piecewise polynomials
and pullbacks of sections of nontrivial $\PL$-torsors. More precisely, we
introduce the following definitions:

\begin{definition}
	Let $\mathcal{U}\Sigma^{(s)}$ be the cone complex obtained by restricting $\mathcal{U}^{s}\Sigma^{(s)}$ to the region $0\le \mathsf{x}_i\le \mathsf{t}$ for $i\in \{1,\dots,s-1\}$ and let $\tilde{\Sigma}^{(s)}$ be the cone complex obtained by restricting $\mathcal{U}^{s}\Sigma^{(s)}$ to the region $0\le \mathsf{x}_i\le \mathsf{t}$ for $i\in \{1,\dots,s\}$
	(see Definition \ref{def:cone_complexes} and Figure \ref{fig:Sigma_N}).
	\begin{enumerate}[label=(\alph*)]
		\item We write $\spp^{\ast}(\tilde{\Sigma}^{(s)})[m]^{\mathsf{gl}}$ for the ring of polynomials $f(m)\in \spp^{\ast}(\tilde{\Sigma}^{(s)})[m]$
			such that
			\[
				f(q-1)|_{\{\mathsf{x}_s=\mathsf{t}\}}=f(q)|_{\{\mathsf{x}_s=0\}}\,,\quad q\in \mathbb{Z}.
			\]
		\item Moreover, $\Psi:\spp^{\ast}(\tilde{\Sigma}^{(s)})[m]^{\mathsf{gl}}\to \spp^{\ast}(\mathcal U\Sigma^{(s)})$ is defined by
			$$\Psi(f)(\mathsf{x}_1,\dots,\mathsf{x}_s,\mathsf{t})=f(m)(\mathsf{x}_1,\dots,\mathsf{x}_{s-1},\mathsf{x}_s-m\mathsf{t},\mathsf{t}),$$
			where $m$ is the unique integer satisfying $m\mathsf{t}\le \mathsf{x}_s<(m+1)\mathsf{t}$.
		\item For a maximal cone $\sigma \in \tilde{\Sigma}^{(s)}$ and $f\in \spp^{\ast}(\tilde{\Sigma}^{(s)})[m]^{\mathsf{gl}}$, we define $f|_{\sigma}$ to be the polynomial in $\mathsf{x}_1,\dots,\mathsf{x}_s,\mathsf{t},m$ obtained by restricting each coefficient of $f$ (viewed as a polynomial in $m$) to $\sigma$.
	\end{enumerate}
\end{definition}

The ring of piecewise polynomials $\spp^{\ast}(\Sigma^{(s)})$ is canonically identified with the subring of $\ZZ$-periodic piecewise polynomial functions on $\mathcal U\Sigma^{(s)}$ whose restrictions to ${\mathsf{x}_i=0}$ and ${\mathsf{x}_i=\mathsf{t}}$ agree under translation for each $i=1,\dots,s-1$.
Equivalently, it is the image under $\Psi$ of the subring of $\spp^{\ast}(\tilde{\Sigma}^{(s)})[m]^{\mathsf{gl}}$ consisting of elements that are constant in $m$ and satisfy the same boundary conditions. However, not every element of the ring $\spp^{\ast}(\tilde{\Sigma}^{(s)})[m]^{\mathsf{gl}}$ maps to a periodic piecewise polynomial under $\Psi$, and in fact the piecewise linear functions $c_{\mathcal P,i}$ from Lemma \ref{lem:cPi} are contained in it. Thus, to prove the polynomiality and degree estimates,
it suffices to restrict to $\spp^{\ast}(\tilde{\Sigma}^{(s)})[m]^{\mathsf{gl}}$. For this, we require the following statement:

\begin{lemma}[Lemma \ref{lem:combinat_lift} and Corollary \ref{cor:push_pull_trop_geom}]
	\label{lem:exists_explicit_lift}
	There exists an explicit map $\mathfrak{N}_{N,m}^{\ast}:\spp^{\ast}(\tilde{\Sigma}^{(s)})[m]^{\mathsf{gl}}\to \spp^{\ast}(\tilde{\Sigma}^{(s)})[m]^{\mathsf{gl}}$
	lifting the action of $[1\times\dots\times N]^{\ast}$ on $\spp^{\ast}(\Sigma^{(s)})$.
\end{lemma}

\begin{definition}
	For $f\in \spp^{\ast}(\tilde{\Sigma}^{(s)})[m]^{\mathsf{gl}}$, we define $\deg_{m,\mathsf{x}_s}(f)$
	to be the maximal total degree of $f|_{\sigma}$ in $(m,\mathsf{x}_s)$
	as $\sigma$ ranges over the maximal cones in $\tilde{\Sigma}^{(s)}$.
\end{definition}

The following is the main combinatorial result on the weights of piecewise polynomials.

\begin{theorem}
	\label{thm:polynomiality}
	Let $f\in \spp^{\ast}(\tilde{\Sigma}^{(s)})[m]^{\mathsf{gl}}$ and set $d:=\deg_{m,\mathsf{x}_s}(f)$. For every maximal cone
	$\sigma$ of $\tilde{\Sigma}^{(s)}$, the restriction
	$$\left.\mathfrak N_{N,m}^\ast(f)\right|_\sigma$$
	is a polynomial in $N$ of degree at most $d$.
\end{theorem}

Theorem \ref{thm:polynomiality} will be proved in \S \ref{sec:trop}.
The proof requires fairly involved combinatorial arguments and proceeds via explicit calculation based on the expression for $\mathfrak{N}_{N,m}^{\ast}$ from Lemma~\ref{lem:exists_explicit_lift}. Summarizing the above discussion, the main source of the combinatorial complexity is the application of Brion's formula to non-periodic functions on the universal covers of the irregular subdivisions $\Sigma_N^{(s)}$.
The reader comfortable with assuming Theorem \ref{thm:polynomiality} may wish to continue with \S\ref{sec:prf_wt_thm}.
The reader who wishes to continue is encouraged to assume $s=2$ on a first reading.

\subsection{The multiplication by \texorpdfstring{$N$}{N} map on tropical abelian schemes}\label{sec:trop}

The goal of this section is to prove Theorem \ref{thm:polynomiality}.
The comparison with the geometric operator $[1\times\dots \times N]^{\ast}$ on $\CH^{\ast}(\barA^{(s)})$
will be given in \S\ref{sec:prf_wt_thm}. We start with the construction of the map
$\mathfrak N_{N,m}^{\ast}$. For the convenience of the reader, the following definition collects the cone complexes
that will appear most frequently throughout \S\ref{sec:wt}. See also Figure \ref{fig:Sigma_N} for a visualization of these
cone complexes.

\begin{definition}\label{def:cone_complexes}
	Fix $s\ge 1$.
	\begin{enumerate}[label = (\alph*)]
		\item Let $\mathcal{U}\Sigma^{(s)}$ be the restriction of $\mathcal{U}^{s}\Sigma^{(s)}$ to the region $0\le \mathsf{x}_i\le \mathsf{t}$ for $i\in \{1,\dots,s-1\}$.
		\item For $N\ge 1$, let $(1\times\dots\times N)$ be the map from
			$\mathcal U\Sigma^{(s)}$ to the cone over $[0,1]^{s-1}\times
			\mathbb{R}$ obtained by sending
			$(\mathsf{x}_1,\dots,\mathsf{x}_s,\mathsf{t})$ to
			$(\mathsf{x}_1,\dots,\mathsf{x}_{s-1},N\mathsf{x}_s,\mathsf{t})$,
			where $\mathsf{t}$ is the coordinate in the cone direction.
		\item Let $\mathcal U\Sigma_N^{(s)}$ be the refinement of $\mathcal U\Sigma^{(s)}$ obtained by pulling back $\mathcal U\Sigma^{(s)}$
			along $(1\times\dots\times N)$.
		\item Let $\mathcal{U}^{s}\Sigma_N^{(s)}$ be the $\mathbb{Z}^{s-1}$-periodic continuation of
			$\mathcal{U}\Sigma_N^{(s)}$ in the directions $\mathsf{x}_i$ for $i\in \{1,\dots,s-1\}$. Equivalently, $\mathcal{U}^{s}\Sigma_N^{(s)}$
			is the refinement of $\mathcal{U}^{s}\Sigma^{(s)}$ obtained by pulling back $\mathcal{U}^{s}\Sigma^{(s)}$ along
			$(1\times\dots\times N)$.
		\item Let $\tilde{\Sigma}^{(s)}$ (resp. $\tilde{\Sigma}_N^{(s)}$) denote the intersection of $\mathcal{U}\Sigma^{(s)}$ (resp. $\mathcal{U}\Sigma_N^{(s)}$) with the region $0\le \mathsf{x}_s\le \mathsf{t}$.
		\item Let $\Sigma_N^{(s)}:=\mathcal{U}^{s}\Sigma_N^{(s)}/\mathbb{Z}^{s}$. Equivalently,
			it is the cone complex obtained from $\mathcal U\Sigma_N^{(s)}/\mathbb{Z}$ by identifying the
			face $\{\mathsf{x}_i=0\}$ with the face $\{\mathsf{x}_i=\mathsf{t}\}$ for $i\in \{1,\dots,s-1\}$. Here, $a\in \mathbb{Z}$ acts on the $\mathsf{x}_s$-direction by $\mathsf{x}_s\mapsto \mathsf{x}_s+a\mathsf{t}$.
	\end{enumerate}
	By abuse of notation, we denote by $(1\times\dots\times N)$ the map
	$\mathcal U\Sigma_N^{(s)}\to \mathcal U\Sigma^{(s)}$ (resp. $\mathcal{U}^{s}\Sigma_N^{(s)}\to \mathcal{U}^{s}\Sigma^{(s)}$, etc.)
	induced by $(1\times\dots\times N)$. We also write $p_N:\mathcal U\Sigma_N^{(s)}\to \mathcal U\Sigma^{(s)}$
	(resp. $p_N:\mathcal{U}^{s}\Sigma_N^{(s)}\to \mathcal{U}^{s}\Sigma^{(s)}$, etc.) for the subdivision map.
\end{definition}

By Lemma \ref{lem:Artin_fan_A_N}, $\Sigma_N^{(s)}$ is the cone complex corresponding to the Artin fan of $\barA_N^{(s)}$ from \eqref{eq:n}.
The reason for introducing $\mathcal{U}\Sigma^{(s)}$ is that both $p_N$ and $(1\times\dots\times N)$ preserve this subcomplex
of $\mathcal{U}^{s}\Sigma^{(s)}$ and restricting the computation to this region simplifies the notation by removing the $s-1$
``unused'' directions.

\begin{definition}
	We define an operator $\mathfrak N_N^{\ast}:\spp^{\ast}(\Sigma^{(s)})\to \spp^{\ast}(\Sigma^{(s)})$ by:
	$$\mathfrak{N}_N^{\ast}:=p_{N,\ast}(1\times\dots \times N)^{\ast}.$$
\end{definition}

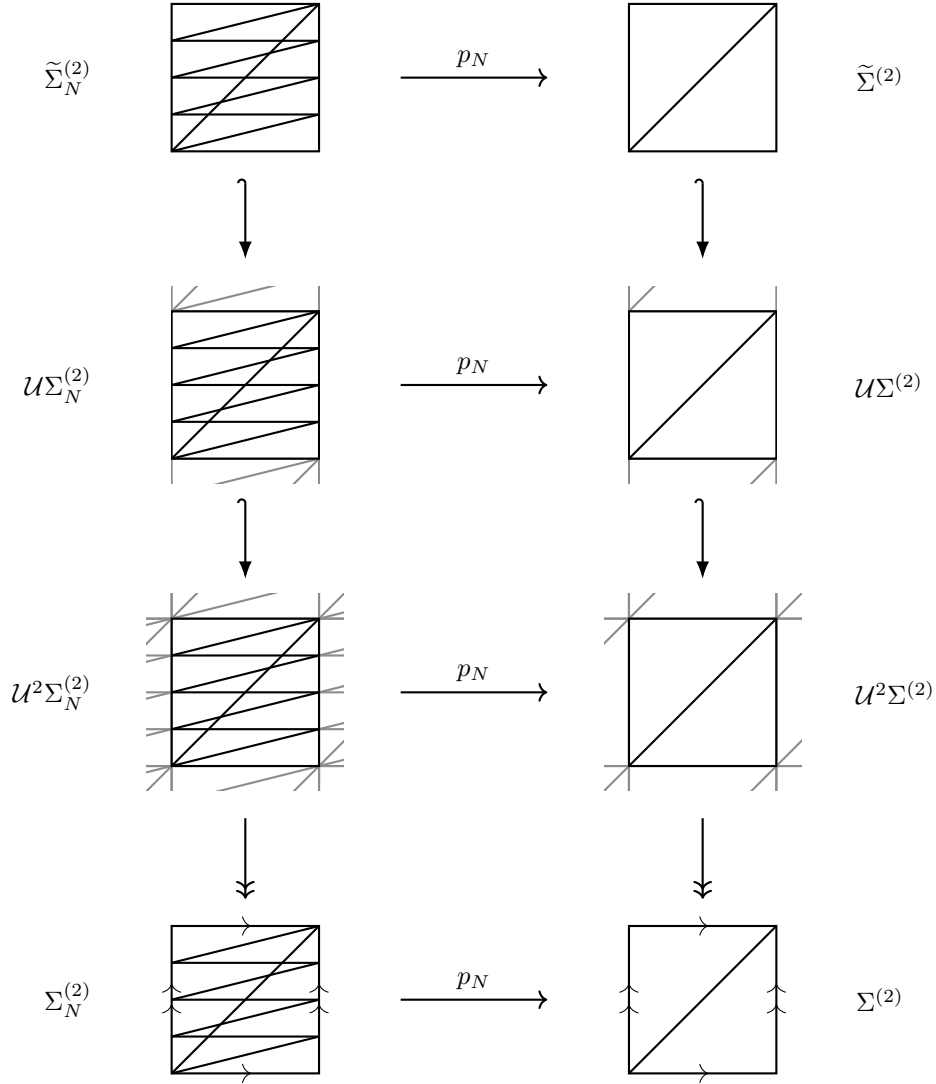
\begin{figure}[!t]
    \centering
	\[
	\begin{tikzpicture}[every node/.style={font=\small}]
		\def\s{1.95}
		\def\p{0.33}
		\colorlet{extensiongray}{black!45}
		\def\dx{6.05}
		\def\ytilde{12.198}
        \def\yU{8.132}
        \def\yUU{4.066}
        \def\yquot{0}
		\def\varrowhalf{0.533}
		\pgfmathsetmacro{\yarrowone}{0.5*(\ytilde+\yU+\s+\p)}
		\pgfmathsetmacro{\yarrowtwo}{0.5*(\yU-\p+\yUU+\s+\p)}
		\pgfmathsetmacro{\yarrowthree}{0.5*(\yUU-\p+\yquot+\s)}

		\newcommand{\plaintile}{%
			\draw[thick] (0,0) rectangle (\s,\s);
			\draw[thick] (0,0) -- (\s,\s);
		}
		\newcommand{\refinedtile}{%
			\draw[thick] (0,0) rectangle (\s,\s);
			\foreach \k in {1,2,3} {
				\draw[thick] (0,{\k*\s/4}) -- (\s,{\k*\s/4});
			}
			\foreach \k in {0,1,2,3} {
				\draw[thick] (0,{\k*\s/4}) -- (\s,{(\k+1)*\s/4});
			}
			\draw[thick] (0,0) -- (\s,\s);
		}
		\newcommand{\edgeidentifications}{%
			\path[decorate,decoration={markings,
				mark=at position 0.54 with {\arrow{>[scale=1.4]}}}]
				(0,0) -- (\s,0);
			\path[decorate,decoration={markings,
				mark=at position 0.54 with {\arrow{>[scale=1.4]}}}]
				(0,\s) -- (\s,\s);
			\path[decorate,decoration={markings,
				mark=at position 0.47 with {\arrow{>[scale=1.4]}},
				mark=at position 0.61 with {\arrow{>[scale=1.4]}}}]
				(0,0) -- (0,\s);
			\path[decorate,decoration={markings,
				mark=at position 0.47 with {\arrow{>[scale=1.4]}},
				mark=at position 0.61 with {\arrow{>[scale=1.4]}}}]
				(\s,0) -- (\s,\s);
		}

		\newcommand{\verticalplain}[2]{%
			\begin{scope}[shift={(#1,#2)}]
				\begin{scope}
					\clip (0,-\p) rectangle (\s,\s+\p);
					\begin{scope}[draw=extensiongray,shift={(0,-\s)}]\plaintile\end{scope}
					\begin{scope}[draw=extensiongray,shift={(0,\s)}]\plaintile\end{scope}
					\plaintile
				\end{scope}
			\end{scope}
		}
		\newcommand{\verticalrefined}[2]{%
			\begin{scope}[shift={(#1,#2)}]
				\begin{scope}
					\clip (0,-\p) rectangle (\s,\s+\p);
					\begin{scope}[draw=extensiongray,shift={(0,-\s)}]\refinedtile\end{scope}
					\begin{scope}[draw=extensiongray,shift={(0,\s)}]\refinedtile\end{scope}
					\refinedtile
				\end{scope}
			\end{scope}
		}

		\newcommand{\periodicplain}[2]{%
			\begin{scope}[shift={(#1,#2)}]
				\begin{scope}
					\clip (-\p,-\p) rectangle (\s+\p,\s+\p);
					\foreach \xx/\yy in {-\s/-\s,0/-\s,\s/-\s,-\s/0,\s/0,-\s/\s,0/\s,\s/\s} {
						\begin{scope}[draw=extensiongray,shift={(\xx,\yy)}]\plaintile\end{scope}
					}
					\plaintile
				\end{scope}
			\end{scope}
		}
		\newcommand{\periodicrefined}[2]{%
			\begin{scope}[shift={(#1,#2)}]
				\begin{scope}
					\clip (-\p,-\p) rectangle (\s+\p,\s+\p);
					\foreach \xx/\yy in {-\s/-\s,0/-\s,\s/-\s,-\s/0,\s/0,-\s/\s,0/\s,\s/\s} {
						\begin{scope}[draw=extensiongray,shift={(\xx,\yy)}]\refinedtile\end{scope}
					}
					\refinedtile
				\end{scope}
			\end{scope}
		}

		\begin{scope}[shift={(0,\ytilde)}]\refinedtile\end{scope}
		\begin{scope}[shift={(\dx,\ytilde)}]\plaintile\end{scope}
		\node[anchor=east] at (-0.92,\ytilde+0.5*\s) {$\widetilde{\Sigma}_N^{(2)}$};
		\node[anchor=west] at (\dx+\s+0.92,\ytilde+0.5*\s) {$\widetilde{\Sigma}^{(2)}$};

		\verticalrefined{0}{\yU}
		\verticalplain{\dx}{\yU}
		\node[anchor=east] at (-0.92,\yU+0.5*\s) {$\mathcal U\Sigma_N^{(2)}$};
		\node[anchor=west] at (\dx+\s+0.92,\yU+0.5*\s) {$\mathcal U\Sigma^{(2)}$};

		\periodicrefined{0}{\yUU}
		\periodicplain{\dx}{\yUU}
		\node[anchor=east] at (-0.92,\yUU+0.5*\s) {$\mathcal U^2\Sigma_N^{(2)}$};
		\node[anchor=west] at (\dx+\s+0.92,\yUU+0.5*\s) {$\mathcal U^2\Sigma^{(2)}$};

		\begin{scope}[shift={(0,\yquot)}]
			\refinedtile
			\edgeidentifications
		\end{scope}
		\begin{scope}[shift={(\dx,\yquot)}]
			\plaintile
			\edgeidentifications
		\end{scope}
		\node[anchor=east] at (-0.92,\yquot+0.5*\s) {$\Sigma_N^{(2)}$};
		\node[anchor=west] at (\dx+\s+0.92,\yquot+0.5*\s) {$\Sigma^{(2)}$};

		\foreach \yy in {\ytilde,\yU,\yUU,\yquot} {
			\draw[->,thick] (\s+\p+0.75,\yy+0.5*\s)
				-- node[above] {$p_N$} (\dx-\p-0.75,\yy+0.5*\s);
		}

		\foreach \x in {0.5*\s,\dx+0.5*\s} {
			\draw[arrows={Hooks[harpoon]-Latex},thick]
				(\x,{\yarrowone+\varrowhalf}) -- (\x,{\yarrowone-\varrowhalf});
			\draw[arrows={Hooks[harpoon]-Latex},thick]
				(\x,{\yarrowtwo+\varrowhalf}) -- (\x,{\yarrowtwo-\varrowhalf});
			\draw[->>,thick]
				(\x,{\yarrowthree+\varrowhalf}) -- (\x,{\yarrowthree-\varrowhalf});
		}
	\end{tikzpicture}
	\]
\caption{\label{fig:Sigma_N}An illustration of the cone complexes defined in Definition \ref{def:cone_complexes} in the case $s=2$ and $N=4$. To simplify the presentation, we restrict to the link at $\mathsf{t}=1$.}
\end{figure}

By Corollary \ref{cor:push_pull_trop_geom}, the restriction of the operator $\mathfrak{N}_N^{\ast}$ to
$\spp^{\ast}(\Sigma^{(s)})$ defines a lift of the operator $[1\times\dots\times N]^{\ast}$
on $\CH^{\ast}(\barA^{(s)})$. The reason for introducing the ring $\spp^{\ast}(\tilde{\Sigma}^{(s)})[m]^{\mathsf{gl}}$ is that the restriction of $\mathfrak{N}_N^{\ast}$
to the image of this ring under $\Psi$ is polynomial in $N$ in the sense of Theorem \ref{thm:polynomiality}. To prove
polynomiality, we must give a more concrete description of $\mathfrak{N}_N^{\ast}$.

The main
technical difficulty is the computation of $p_{N,\ast}$ via Brion's formula. Namely, we must determine
the maximal cones in $\mathcal U\Sigma^{(s)}$ and $\mathcal U\Sigma_N^{(s)}$ as well as their support functions. In fact, it
will be computationally simpler to replace $\mathcal U\Sigma_N^{(s)}$ with a further subdivision $\mathcal U\overline{\Sigma}_N^{(s)}$.
By the projection formula, $p_{N,\ast}$ is invariant under this replacement.

\subsubsection{Maximal cones of $\mathcal U\Sigma_N^{(s)}$ and the subdivision $\mathcal U\overline{\Sigma}_N^{(s)}$}

To simplify notation, we introduce the functions
$$\mathsf{z}_i^{(N)}(\mathsf{u}):=N\mathsf{u}-(i-1)\mathsf{t}.$$
We also write $\mathsf{z}_i^{(N)}:=\mathsf{z}_i^{(N)}(\mathsf{x}_s)$.

\begin{lemma}
	\label{lem:explicit-N-refinement}
	The subdivision $\mathcal U\Sigma_N^{(s)}\to \mathcal U\Sigma^{(s)}$
	is an iterated subdivision along hyperplanes. Explicitly, it is obtained by subdividing
	along the hyperplanes:
	$$\mathsf{z}_{i}^{(N)}=\mathsf{x}_j,\quad i\in \mathbb{Z},\quad 1\le j\le s-1$$
	and
	$$\mathsf{z}_{i}^{(N)}=0,\quad i\in \mathbb{Z}.$$
\end{lemma}

\begin{proof}
	Observe that $\mathcal U\Sigma^{(s)}$ is an iterated subdivision of the
	cone over $[0,1]^{s-1}\times \mathbb{R}$ along the hyperplanes
	\begin{enumerate}[label = (\alph*)]
		\item $\mathsf{x}_i=\mathsf{x}_j$ for $i,j\in \{1,\dots,s-1\}$,
		\item $\mathsf{x}_i=\mathsf{t}$ for $i\in \{1,\dots,s-1\}$,
		\item $\mathsf{x}_j=\mathsf{x}_s+q\mathsf{t}$ for $q\in \mathbb{Z}$ and $j\in \{1,\dots,s-1\}$ and
		\item $\mathsf{x}_s+q\mathsf{t}=0$ for $q\in \mathbb{Z}$.
	\end{enumerate}
	The first two classes of hyperplanes pull back to the same
	hyperplanes along $(1\times\dots\times N)$ and hence are already
	contained in the subdivision $\mathcal U\Sigma^{(s)}$. The third class
	of hyperplanes pulls back to the hyperplanes:
	$$\mathsf{x}_j=N\mathsf{x}_s+q\mathsf{t}\Leftrightarrow \mathsf{x}_j=\mathsf{z}_{1-q}^{(N)}\,.$$
	Finally, the fourth class of hyperplanes pulls back to
	$$N\mathsf{x}_s+q\mathsf{t}=0\Leftrightarrow \mathsf{z}_{1-q}^{(N)}=0\,.$$
\end{proof}

It follows from the definition of $\mathcal U\Sigma^{(s)}$
that the maximal cones of $\mathcal U\Sigma^{(s)}$ are
indexed by tuples $(\rho,q,r)$ with $\rho\in \mathfrak S_{s-1}$, $r\in \{0,\dots,s-1\}$ and $q\in \mathbb{Z}$.
The cone corresponding to such a tuple is:
$$\sigma_{\rho,q,r}:=\{(\mathsf{x},\mathsf{t})\in \mathbb{R}^{s}\times \mathbb{R}_{\ge 0}|0\le \mathsf{x}_{\rho(1)}\le\dots \le \mathsf{x}_{\rho(r)}\le \mathsf{x}_s-q\mathsf{t}\le \mathsf{x}_{\rho(r+1)}\le \dots\le \mathsf{x}_{\rho(s-1)}\le \mathsf{t}\}.$$

\begin{corollary}
	\label{cor:max_cones_sigma_n}
	The maximal cones in $\mathcal U\Sigma_N^{(s)}$ are indexed by a choice of tuple $(\rho,q,i,a,b)$, where $\rho\in \mathfrak S_{s-1}$,
	$i,q\in \mathbb{Z}$, $Nq+1\le i\le N(q+1)$ and $a,b\in\{1,\dots,s\}$. The maximal cone corresponding to the
	tuple $(\rho,q,i,a,b)$ is:
	\begin{align*}
		\sigma_{\rho,q,i,a,b}&:=\left\{(\mathsf{x},\mathsf{t})\in \RR^s\times\RR_{\ge 0}\middle|0\le \mathsf{x}_{\rho(1)}\le\dots\le \mathsf{x}_{\rho(s-1)}\le \mathsf{t},\right.\\
							 &\left.\mathsf{x}_{\rho(a-1)}\le \mathsf{x}_s-q\mathsf{t}\le \mathsf{x}_{\rho(a)},\quad \mathsf{x}_{\rho(b-1)}\le \mathsf{z}_i^{(N)}\le \mathsf{x}_{\rho(b)}\right\},
	\end{align*}
	with the convention $\mathsf{x}_{\rho(0)}=0$ and $\mathsf{x}_{\rho(s)}=\mathsf{t}$, subject
	to the additional conditions:
	\begin{itemize}
		\item $a\le b$ if $i=Nq+1$ and
		\item $b\le a$ if $i=N(q+1)$.
	\end{itemize}
\end{corollary}

\begin{proof}
	From Lemma \ref{lem:explicit-N-refinement}, we know that the subdivision $\mathcal U\Sigma_N^{(s)}$ is obtained by
	subdividing $\mathcal U\Sigma^{(s)}$ along the hyperplanes
	$$\mathsf{z}_{i}^{(N)}=\mathsf{x}_j,\quad i\in \mathbb{Z},\quad 1\le j\le s-1$$
	and
	$$\mathsf{z}_{i}^{(N)}=0,\quad i\in \mathbb{Z}.$$
	Note that $\mathsf{z}_{i}^{(N)}=0\Leftrightarrow \mathsf{z}_{i-1}^{(N)}=\mathsf{t}$.
	Hence we see that on any maximal cone in the subdivision,
	the values $(0,\mathsf{x}_1,\dots,\mathsf{x}_{s-1},\mathsf{t},\mathsf{z}_i^{(N)})$ and the values $(0,\mathsf{x}_1,\dots,\mathsf{x}_{s-1},\mathsf{t},\mathsf{x}_s-q\mathsf{t})$ have a fixed order.
	In the latter case, we have the additional requirement that $0$ is minimal and $\mathsf{t}$ is maximal. Note that $0\le \mathsf{x}_s-q\mathsf{t}\le \mathsf{t}$
	is equivalent to $(Nq-i+1)\mathsf{t}\le \mathsf{z}_i^{(N)}\le (N(q+1)-i+1)\mathsf{t}$. Hence the location of $\mathsf{z}_{i}^{(N)}$ in the order
	is fixed on $\sigma_{\rho,q,r}$ if either $(Nq-i+1)\mathsf{t}\ge \mathsf{t}\Leftrightarrow i\le Nq$ or $(N(q+1)-i+1)\mathsf{t}\le 0\Leftrightarrow i\ge N(q+1)+1$.
	Hence we only obtain nontrivial walls in $\sigma_{\rho,q,r}$
	for $Nq+1\le i\le N(q+1)$. Since $\sigma_{\rho,q,i,a,b}$
	is obtained by choosing an ordering of $(0,\mathsf{x}_1,\dots,\mathsf{x}_{s-1},\mathsf{t},\mathsf{z}_i^{(N)})$ compatible with the fixed order
	of $(0,\mathsf{x}_1,\dots,\mathsf{x}_{s-1},\mathsf{t})$ on $\sigma_{\rho,q,r}$, it
	remains to show that each of the cones $\sigma_{\rho,q,i,a,b}$
	has maximal dimension.

	It is enough to do this on the link $\mathsf{t}=1$. Put $\mathsf{y}=\mathsf{x}_s-q$ and $h=i-Nq$. Then $1\le h\le N$ and
	\[
		\mathsf{z}_i^{(N)}=N\mathsf{y}-(h-1).
	\]
	The strip condition $0<\mathsf{z}_i^{(N)}<1$ is equivalent to
	\[
		\frac{h-1}{N}<\mathsf{y}<\frac{h}{N}.
	\]
	If $1<h<N$, the point where $\mathsf{y}=\mathsf{z}_i^{(N)}$ lies in this open interval, so $\mathsf{y}$ and $\mathsf{z}_i^{(N)}$ can be chosen in either
	order. If $h=1$, then $\mathsf{y}\le \mathsf{z}_i^{(N)}$ on the strip, and if $h=N$, then $\mathsf{z}_i^{(N)}\le \mathsf{y}$ on the strip. Thus we
	see that $a\le b$ for $i=Nq+1$ and $b\le a$ for $i=N(q+1)$ are necessary conditions.

	Conversely, assume these conditions hold. If $1<h<N$, choose $\mathsf{y}$ in the open strip so that $\mathsf{y}<\mathsf{z}_i^{(N)}$ when
	$a\le b$ and $\mathsf{z}_i^{(N)}<\mathsf{y}$ when $b\le a$. If $h=1$ or $h=N$, choose any interior point of the strip. We can then choose
	\[
		0<\mathsf{x}_{\rho(1)}<\cdots <\mathsf{x}_{\rho(s-1)}<1
	\]
	so that exactly $a-1$ of the $\mathsf{x}_{\rho(j)}$ lie below $\mathsf{y}$ and exactly $b-1$ lie below $\mathsf{z}_i^{(N)}$. This gives a point
	at which all inequalities defining $\sigma_{\rho,q,i,a,b}$ are strict. Hence the cone has nonempty interior in the
	corresponding maximal cone of the refinement, and therefore has maximal dimension.
\end{proof}

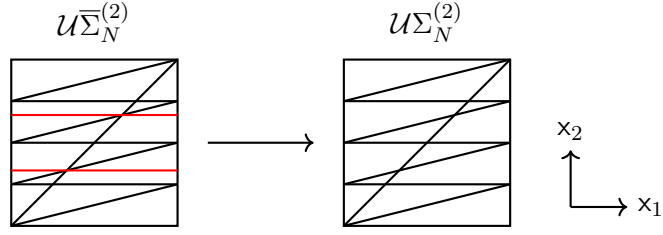
\begin{figure}
\begin{center}
	\[
	\begin{tikzpicture}
		\def\s{2.2}
		\def\dx{4.4}
		\def\y{0}

		\draw[thick] (0,\y) rectangle ++(\s,\s);

		\foreach \k in {1,2,3} {
			\draw[thick] (0,\y+\k*\s/4) -- ++(\s,0);
		}
		\foreach \k in {0,1,2,3} {
			\draw[thick] (0,\y+\k*\s/4) -- ++(\s,\s/4);
		}
		\draw[thick] (0,\y) -- (\s,\y+\s);

		\foreach \k in {1,2} {
			\draw[red, thick] (0,\y+\k*\s/3) -- ++(\s,0);
		}

		\node at (0.5*\s,\y+\s+0.5) {$\mathcal U\overline{\Sigma}_N^{(2)}$};

		\draw[thick] (\dx,\y) rectangle ++(\s,\s);
		\foreach \k in {1,2,3} {
			\draw[thick] (\dx,\y+\k*\s/4) -- ++(\s,0);
		}
		\foreach \k in {0,1,2,3} {
			\draw[thick] (\dx,\y+\k*\s/4) -- ++(\s,\s/4);
		}
		\draw[thick] (\dx,\y) -- (\dx+\s,\y+\s);

		\node at (\dx+0.5*\s,\y+\s+0.5) {$\mathcal U\Sigma_N^{(2)}$};

		\draw[->, thick]
			(\s+0.4,\y+0.5*\s) -- (\dx-0.4,\y+0.5*\s);

		\draw[->, thick] (\dx+\s+0.8,\y+0.25) -- ++(0.75,0)
			node[right] {$\mathsf{x}_1$};
		\draw[->, thick] (\dx+\s+0.8,\y+0.25) -- ++(0,0.75)
			node[above] {$\mathsf{x}_2$};
	\end{tikzpicture}
	\]
\end{center}
\caption{\label{fig:Sigma_N_bar}This figure depicts the (infinite) cone complexes $\mathcal U\Sigma_N^{(2)}$ and $\mathcal U\overline{\Sigma}_N^{(2)}$ in the case $N=4$. To simplify the illustration, we only depict their intersection with the region $0\le \mathsf{x}_1,\mathsf{x}_2\le \mathsf{t}$ and restrict to the links at $\mathsf{t}=1$.}
\end{figure}

From this result it follows in particular that, in general, the
order between $\mathsf{z}_i^{(N)}$ and $\mathsf{x}_s-q\mathsf{t}$ is not fixed on $\sigma_{\rho,q,i,a,b}$.
This complicates the computation of the support function of this cone.
For this reason, we introduce the following subdivision:
\begin{definition}
Let $\mathcal U\overline{\Sigma}_N^{(s)}$ be
the subdivision of $\mathcal U\Sigma_N^{(s)}$ obtained by subdividing along the hyperplanes
$$\mathsf{z}_i^{(N)}=\mathsf{x}_s$$
for all $i\in \mathbb{Z}$.
\end{definition}

See Figure \ref{fig:Sigma_N_bar} for an illustration of $\mathcal U\overline{\Sigma}_N^{(s)}$. From Corollary \ref{cor:max_cones_sigma_n}
we have:

\begin{corollary}
	\label{cor:max_cones_sigma_n_bar}
	The maximal cones of $\mathcal U\overline{\Sigma}_N^{(s)}$ are indexed by tuples $(\pi,q,i)$, where $\pi\in \mathfrak S_{s+1}$,
	$q\in \mathbb{Z}$ and $Nq+1\le i\le N(q+1)$ satisfying the constraints:
	\begin{enumerate}[label = (\alph*)]
		\item If $N=1$, then $\pi^{-1}(s+1)=\pi^{-1}(s)+1$. Otherwise, the following two conditions hold:
		\item if $i=Nq+1$, then $\pi^{-1}(s+1)>\pi^{-1}(s)$ and
		\item if $i=N(q+1)$, then $\pi^{-1}(s+1)<\pi^{-1}(s)$.
	\end{enumerate}
	The maximal cone corresponding to $(\pi,q,i)$ is constructed as follows:
	Let
    \[
    (b_1,\dots,b_{s+1})=(\mathsf{x}_1,\dots,\mathsf{x}_{s-1},\mathsf{x}_s-q\mathsf{t},\mathsf{z}_i^{(N)}).
    \]
    Then we define
	$$\sigma_{\pi,q,i}:=\{(\mathsf{x},\mathsf{t})\in \RR^s\times\RR_{\ge 0}|0\le b_{\pi(1)}\le\dots\le b_{\pi(s+1)}\le \mathsf{t}\}.$$
\end{corollary}

\subsubsection{Support functions of cones in $\mathcal U\overline{\Sigma}_N^{(s)}$}
\label{sec:supp_functions}

In this section, we compute the support functions of the cones $\sigma_{\pi,q,i}$ as in Corollary \ref{cor:max_cones_sigma_n_bar}.
We first determine the support functions for a more general class of cones (Corollary \ref{cor:ineq_cone}) and then specialize the result to $\sigma_{\pi,q,i}$ (Corollary \ref{cor:supp_function_sigma}).

The following elementary result on support functions is well known:

\begin{lemma}
	\label{lem:delta-pullback-linear}
	Let $V$ be a finite-dimensional real vector space equipped with a full-dimensional lattice.
	Let $A:V\to V$ be an invertible linear map mapping the lattice to a sublattice of finite index, and let $\sigma$ be a full-dimensional cone in $V$. Then
	\[
		\delta_{A^{-1}\sigma}(\mathsf{x})=\frac{\delta_\sigma(A\mathsf{x})}{|\det A|}.
	\]
\end{lemma}

\begin{lemma}[Cone over a product]
	\label{lem:cone-over-product}
	Let $P\subseteq \RR^p$ and $Q\subseteq \RR^q$ be rational polyhedra. Let $\sigma_P,\sigma_Q,\sigma_{P\times Q}$ denote the cones over $P,Q,P\times Q$, each regarded inside its own linear span. Let $N_P:=\Span(\sigma_P)\cap \mathbb{Z}^{p+1}$ and define $N_Q$ similarly. Let $\pi_P$ (resp. $\pi_Q$) denote the projection from $N_P$ (resp. $N_Q$) to $\mathbb{Z}$ induced by the projection onto the cone direction. Then
	\[
		\delta_{\sigma_{P\times Q}}(\mathsf{x},\mathsf{y},\mathsf{t})=\gcd([\mathbb{Z}:\pi_P(N_P)],[\mathbb{Z}:\pi_Q(N_Q)])\frac{\delta_{\sigma_P}(\mathsf{x},\mathsf{t})\,\delta_{\sigma_Q}(\mathsf{y},\mathsf{t})}{\mathsf{t}}.
	\]
\end{lemma}

\begin{proof}
	First assume that $P$ and $Q$ are full-dimensional in $\mathbb{R}^{p}$ and $\mathbb{R}^{q}$ respectively, so that $[\mathbb{Z}:\pi_P(N_P)]=[\mathbb{Z}:\pi_Q(N_Q)]=1$.

	Let
	\[
		f_P(u):=-\inf_{p\in P}\langle u,p\rangle,\qquad
		f_Q(v):=-\inf_{q\in Q}\langle v,q\rangle.
	\]
	Then
	\[
		\sigma_P^\vee=\{(u,\lambda)\mid \lambda\ge f_P(u)\},\qquad
		\sigma_Q^\vee=\{(v,\lambda)\mid \lambda\ge f_Q(v)\},
	\]
	and
	\[
		\sigma_{P\times Q}^\vee=\{(u,v,\lambda)\mid \lambda\ge f_P(u)+f_Q(v)\}.
	\]
	Therefore
	\begin{align*}
		\frac{1}{\delta_{\sigma_{P\times Q}}(\mathsf{x},\mathsf{y},\mathsf{t})}
	&=\int_{\sigma_{P\times Q}^\vee} e^{-\langle u,\mathsf{x}\rangle-\langle v,\mathsf{y}\rangle-\lambda \mathsf{t}}\,du\,dv\,d\lambda\\
	&=\int_{\RR^p\times \RR^q}\int_{f_P(u)+f_Q(v)}^\infty
	e^{-\langle u,\mathsf{x}\rangle-\langle v,\mathsf{y}\rangle-\lambda \mathsf{t}}\,d\lambda\,du\,dv\\
&=\frac1{\mathsf{t}}\left(\int_{\RR^p} e^{-\langle u,\mathsf{x}\rangle-\mathsf{t}f_P(u)}\,du\right)
\left(\int_{\RR^q} e^{-\langle v,\mathsf{y}\rangle-\mathsf{t}f_Q(v)}\,dv\right)\\
&=\frac{\mathsf{t}}{\delta_{\sigma_P}(\mathsf{x},\mathsf{t})\,\delta_{\sigma_Q}(\mathsf{y},\mathsf{t})}.
	\end{align*}
	Taking reciprocals gives the result.

	We now return to the general case when $P$ and $Q$ are not necessarily full-dimensional.
	In this case, the image of the lattice $\Span(\sigma_{P\times Q})\cap \mathbb{Z}^{p+q+1}$ in $\mathbb{Z}$ under projection onto
	the cone direction has index $\ell:=\lcm([\mathbb{Z}:\pi_P(N_P)],[\mathbb{Z}:\pi_Q(N_Q)])$. Let $\tilde{\sigma}_P$
	(resp. $\tilde{\sigma}_Q$) be the cone $\sigma_P$ (resp. $\sigma_Q$)
	equipped with the integral structure which is the preimage of $\ell \mathbb{Z}$ in the integral structure of the respective cones
	under the projection onto the cone direction. Then, replacing $\mathbb{R}^{p}$ (resp. $\mathbb{R}^{q}$) with the span of $\sigma_P$ (resp. $\sigma_Q$),
	the same argument as above applies and we get:
	$$\delta_{\sigma_{P\times Q}}(\mathsf{x},\mathsf{y},\mathsf{t})=\frac{\delta_{\tilde{\sigma}_P}(\mathsf{x},\mathsf{t})\,\delta_{\tilde{\sigma}_Q}(\mathsf{y},\mathsf{t})}{\frac{\mathsf{t}}{\ell}}.$$
	From the definition of the support function as an integral, it is easy to see that changing the integral structure changes the support
	function by the index of the lattice inclusion, so we get:
	$$\delta_{\sigma_{P\times Q}}(\mathsf{x},\mathsf{y},\mathsf{t})=\frac{[\mathbb{Z}:\pi_P(N_P)][\mathbb{Z}:\pi_Q(N_Q)]}{\ell}\frac{\delta_{\sigma_P}(\mathsf{x},\mathsf{t})\,\delta_{\sigma_Q}(\mathsf{y},\mathsf{t})}{\mathsf{t}}.$$
\end{proof}

\begin{lemma}
	\label{lem:balanced-cone}
	Let $p,q\ge 1$ and let
	\[
		a_1,\dots,a_p,b_1,\dots,b_q\in \ZZ_{\ge 0}
	\]
	with
	\[
		\gcd(a_1,\dots,a_p,b_1,\dots,b_q)=1
	\]
	and such that neither $(a_1,\dots,a_p)$ nor $(b_1,\dots,b_q)$ is identically zero.
	Consider the cone
	\[
		\sigma:=\{(\mathsf{x},\mathsf{y})\in \RR_{\ge 0}^p\times \RR_{\ge 0}^q \mid a_1\mathsf{x}_1+\cdots+a_p\mathsf{x}_p=b_1\mathsf{y}_1+\cdots+b_q\mathsf{y}_q\}.
	\]
	Then
	\[
		\delta_\sigma=\frac{\prod_{i=1}^p \mathsf{x}_i\cdot \prod_{j=1}^q \mathsf{y}_j}{a_1\mathsf{x}_1+\cdots+a_p\mathsf{x}_p}.
	\]
\end{lemma}

\begin{proof}
	First observe that if $a_i=0$ for some $i$, then $\sigma$ is the product of $\mathbb{R}_{\ge 0}$ (corresponding to the coordinate $\mathsf{x}_i$)
	with the cone $\sigma'$ obtained by the same procedure as in the statement of the lemma, but with $a_i$ removed. Hence
	$\delta_{\sigma}=\mathsf{x}_i\cdot \delta_{\sigma'}$. Since the formula for $\delta_{\sigma}$ satisfies the same multiplicativity property
	when $a_i=0$, we may reduce to the case when $a_i>0$ for all $i$. By symmetry, we may also assume that $b_i>0$ for all $i$.
	Let
	\[
		P_a:=\{\mathsf{x}\in \RR_{\ge 0}^p\mid a_1\mathsf{x}_1+\cdots+a_p\mathsf{x}_p=1\},
		\qquad
		P_b:=\{\mathsf{y}\in \RR_{\ge 0}^q\mid b_1\mathsf{y}_1+\cdots+b_q\mathsf{y}_q=1\}.
	\]
	Let $\sigma_a$ and $\sigma_b$ be the cones over $P_a$ and $P_b$, respectively. Thus
	\[
		\sigma_a=\{(\mathsf{x},\mathsf{t})\in \RR_{\ge 0}^p\times \RR_{\ge 0}\mid a_1\mathsf{x}_1+\cdots+a_p\mathsf{x}_p=\mathsf{t}\},
	\]
	\[
		\sigma_b=\{(\mathsf{y},\mathsf{t})\in \RR_{\ge 0}^q\times \RR_{\ge 0}\mid b_1\mathsf{y}_1+\cdots+b_q\mathsf{y}_q=\mathsf{t}\}.
	\]

	Consider the diagonal map
	\[
		\phi_a:\RR^{p+1}\to \RR^{p+1},
		\qquad
		\phi_a(\mathsf{x}_1,\dots,\mathsf{x}_p,\mathsf{t}):=(a_1\mathsf{x}_1,\dots,a_p\mathsf{x}_p,\mathsf{t}).
	\]
	Then $\sigma_a=\phi_a^{-1}(\sigma_{\Delta^{p-1}})$, where $\sigma_{\Delta^{p-1}}$ is the cone over the standard simplex
	\[
		\{\mathsf{u}\in \RR_{\ge 0}^p\mid \mathsf{u}_1+\cdots+\mathsf{u}_p=1\}.
	\]
	Since
	\[
		\delta_{\sigma_{\Delta^{p-1}}}(\mathsf{u},\mathsf{t})=\mathsf{u}_1\cdots \mathsf{u}_p
	\]
	and $|\det(\phi_a)|=a_1\cdots a_p$, Lemma~\ref{lem:delta-pullback-linear} gives
	\[
		\delta_{\sigma_a}(\mathsf{x},\mathsf{t})=\mathsf{x}_1\cdots \mathsf{x}_p.
	\]
	The same argument for $P_b$ gives
	\[
		\delta_{\sigma_b}(\mathsf{y},\mathsf{t})=\mathsf{y}_1\cdots \mathsf{y}_q.
	\]

	Now let $\sigma_{P_a\times P_b}$ be the cone over $P_a\times P_b$. We have already shown that
	\[
		\delta_{\sigma_a}(\mathsf{x},\mathsf{t})=\mathsf{x}_1\cdots \mathsf{x}_p,
		\qquad
		\delta_{\sigma_b}(\mathsf{y},\mathsf{t})=\mathsf{y}_1\cdots \mathsf{y}_q.
	\]
	The vertices of $P_a$ (resp. $P_b$) are $\frac{e_i}{a_i}$ (resp. $\frac{e_i}{b_i}$), so
	one sees that the two indices appearing in the statement of Lemma~\ref{lem:cone-over-product}
	are given by $\gcd(a_1,\dots,a_p)$ and $\gcd(b_1,\dots,b_q)$. By assumption, we have
	$\gcd(a_1,\dots,a_p,b_1,\dots,b_q)=1$,
	so Lemma~\ref{lem:cone-over-product} gives
	\[
		\delta_{\sigma_{P_a\times P_b}}(\mathsf{x},\mathsf{y},\mathsf{t})
		=
		\frac{\mathsf{x}_1\cdots \mathsf{x}_p\,\mathsf{y}_1\cdots \mathsf{y}_q}{\mathsf{t}}.
	\]

	Finally, $\sigma$ is naturally identified with $\sigma_{P_a\times P_b}$ via
	\[
		\sigma\longrightarrow \sigma_{P_a\times P_b},
		\qquad
		(\mathsf{x},\mathsf{y})\longmapsto (\mathsf{x},\mathsf{y},a_1\mathsf{x}_1+\cdots+a_p\mathsf{x}_p).
	\]
	This identification is lattice-preserving, so the support function pulls back unchanged. Under it, the coordinate $\mathsf{t}$ becomes
	\[
		\mathsf{t}=a_1\mathsf{x}_1+\cdots+a_p\mathsf{x}_p=b_1\mathsf{y}_1+\cdots+b_q\mathsf{y}_q,
	\]
	so the displayed formula pulls back to
	\[
		\delta_\sigma=\frac{\prod_{i=1}^p \mathsf{x}_i\cdot \prod_{j=1}^q \mathsf{y}_j}{a_1\mathsf{x}_1+\cdots+a_p\mathsf{x}_p}.
	\]
\end{proof}

\begin{corollary}
	\label{cor:ineq_cone}
	Let $\mathsf{y}_1,\dots,\mathsf{y}_n$ be the coordinate functions on $\mathbb{R}^{n}$ and let
	$\sigma\subseteq \mathbb{R}^{n}$ be the cone defined by the inequalities $0\le \mathsf{y}_1\le \dots\le \mathsf{y}_n$
	and the equality $\mathsf{y}_i=a\mathsf{y}_j-b\mathsf{y}_n$ for $a,b\in \mathbb{Z}_{\ge 0}$, $a>0$, $b<a$ and $i,j\in \{1,\dots,n-1\}$
	distinct. We additionally require the following conditions in edge cases:
	\begin{itemize}
		\item If $b=0$, then we require that $i>j$.
		\item If $b=a-1$, then we require $i<j$.
	\end{itemize}
	Then
	$$\delta_{\sigma}=\frac{\prod_{r=1}^n (\mathsf{y}_r-\mathsf{y}_{r-1})}{b(\mathsf{y}_n-\mathsf{y}_j)}$$
	if $i<j$ and
	$$\delta_{\sigma}=\frac{\prod_{r=1}^n (\mathsf{y}_r-\mathsf{y}_{r-1})}{(a-b-1)\mathsf{y}_j}$$
	if $i>j$, with the convention that $\mathsf{y}_0=0$.
\end{corollary}

\begin{proof}
	Define $\mathsf{z}_r:=\mathsf{y}_{r}-\mathsf{y}_{r-1}$ for $r\in \{1,\dots,n\}$. Then the inequalities
	$0\le \mathsf{y}_1\le\dots\le \mathsf{y}_n$ are equivalent to $\mathsf{z}_r\ge 0$ for all $r$. The equality $\mathsf{y}_i=a\mathsf{y}_j-b\mathsf{y}_n$ becomes
	$$\sum_{r=1}^{i}\mathsf{z}_r=\sum_{r=1}^{j}a\mathsf{z}_r-\sum_{r=1}^{n}b\mathsf{z}_r.$$
	Assume first that $i<j$. Then we can rewrite the equality as:
	$$\sum_{r=j+1}^{n}b\mathsf{z}_r=\sum_{r=1}^{i}(a-b-1)\mathsf{z}_r+\sum_{r=i+1}^{j}(a-b)\mathsf{z}_r.$$
	Hence, by Lemma \ref{lem:balanced-cone}, we obtain
	$$\delta_{\sigma}=\frac{\prod_{r=1}^n \mathsf{z}_r}{\sum_{r=j+1}^{n}b\mathsf{z}_r}=\frac{\prod_{r=1}^n (\mathsf{y}_r-\mathsf{y}_{r-1})}{b(\mathsf{y}_n-\mathsf{y}_j)}.$$
	Now assume that $i>j$. Then we instead rewrite the equality as:
	$$\sum_{r=i+1}^{n}b\mathsf{z}_r+\sum_{r=j+1}^{i}(b+1)\mathsf{z}_r=\sum_{r=1}^{j}(a-b-1)\mathsf{z}_r.$$
	Hence, by Lemma \ref{lem:balanced-cone}, we obtain
	$$\delta_{\sigma}=\frac{\prod_{r=1}^n \mathsf{z}_r}{\sum_{r=1}^{j}(a-b-1)\mathsf{z}_r}=\frac{\prod_{r=1}^n (\mathsf{y}_r-\mathsf{y}_{r-1})}{(a-b-1)\mathsf{y}_j}.$$
\end{proof}

Combining Corollary \ref{cor:max_cones_sigma_n_bar} and
Corollary \ref{cor:ineq_cone}, we get:

\begin{corollary}
	\label{cor:supp_function_sigma}
	Let $\pi\in \mathfrak S_{s+1}$,
	$q\in \mathbb{Z}$ and $Nq+1\le i\le N(q+1)$ satisfying the constraints:
	\begin{enumerate}[label = (\alph*)]
		\item If $N=1$, then $\pi^{-1}(s+1)=\pi^{-1}(s)+1$. Otherwise, the following two conditions hold:
		\item if $i=Nq+1$, then $\pi^{-1}(s+1)>\pi^{-1}(s)$ and
		\item if $i=N(q+1)$, then $\pi^{-1}(s+1)<\pi^{-1}(s)$.
	\end{enumerate}
	Use notation as in Corollary \ref{cor:max_cones_sigma_n_bar}.
	Then
	$$\delta_{\sigma_{\pi,q,i}}=\frac{(\mathsf{t}-b_{\pi(s+1)})\prod_{r=1}^{s+1} (b_{\pi(r)}-b_{\pi(r-1)})}{(i-Nq-1)((q+1)\mathsf{t}-\mathsf{x}_s)}$$
	if $\pi^{-1}(s+1)<\pi^{-1}(s)$ and
	$$\delta_{\sigma_{\pi,q,i}}=\frac{(\mathsf{t}-b_{\pi(s+1)})\prod_{r=1}^{s+1} (b_{\pi(r)}-b_{\pi(r-1)})}{((q+1)N-i)(\mathsf{x}_s-q\mathsf{t})}$$
	if $\pi^{-1}(s+1)>\pi^{-1}(s)$, with the convention that
	$b_{\pi(0)}=0$ and when $N=1$, the (vanishing) factors $((q+1)N-i)(\mathsf{x}_s-q\mathsf{t})$
	and $b_{s+1}-b_s$ are omitted.
\end{corollary}

\subsubsection{Proof of Theorem \ref{thm:polynomiality}}

For $\pi\in \mathfrak S_d$ represented by $(j_1,\dots,j_d)$
and $r\in \{0,\dots,d\}$ let $\pi_r\in \mathfrak S_{d+1}$ be the permutation given by
$(j_1,\dots,j_r,d+1,j_{r+1},\dots,j_d)$ with the convention that $d+1$ is inserted
at the beginning when $r=0$ and at the end when $r=d$.
If $r\in \{1,\dots,d\}$, we write $\pi_{-r}\in \mathfrak S_{d-1}$ for the permutation obtained by
removing $r$ from $(j_1,\dots,j_d)$ and replacing the remaining entries $j_i$ with $j_i-1$
whenever $j_i>r$.

\begin{lemma}
	\label{lem:combinat_lift}
	For $f\in \spp^{\ast}(\tilde{\Sigma}^{(s)})[m]^{\mathsf{gl}}$, define
	$\mathfrak{N}_{N,m}^\ast(f)$ by the following formula: for $\pi\in \mathfrak S_{s}$,
	\[
		\left.\mathfrak{N}_{N,m}^\ast(f)\right|_{\sigma_{\pi}}
			:=\sum_{i=1}^{N}\sum_{a=0}^{s}
			\frac{\delta_{\sigma_{\pi}}}{\delta_{\sigma_{\pi_a,0,i}}}\,
			f|_{\sigma_{(\pi_a)_{-s}}}
			\bigl(Nm+i-1\bigr)
			\bigl(\mathsf{x}_1,\dots,\mathsf{x}_{s-1},\mathsf{z}_i^{(N)},\mathsf{t}\bigr)
	\]
	with the convention that when $i=1$ we additionally require
	$a\ge \pi^{-1}(s)$ and when $i=N$ we additionally require $a<\pi^{-1}(s)$ (if $N=1$, we require $a=\pi^{-1}(s)$).
	Then $\mathfrak{N}_{N,m}^\ast(f)\in \spp^{\ast}(\tilde{\Sigma}^{(s)})[m]^{\mathsf{gl}}$ and
	$$\Psi(\mathfrak{N}_{N,m}^\ast(f))=\mathfrak{N}_N^\ast(\Psi(f)).$$
\end{lemma}

\begin{proof}
	For each $q \in \ZZ$ and each maximal cone $\sigma_{\rho,q,r}$ of $\mathcal U\Sigma^{(s)}$, it remains to prove that
    \[
    \left.\Psi(\mathfrak N_{N,m}^*(f))\right|_{\sigma_{\rho,q,r}}
    =
    \left.\mathfrak N_N^*(\Psi(f))\right|_{\sigma_{\rho,q,r}}.
    \]

	For the right-hand side, we compute using Brion's formula, where in the cases $i=1$ and $i=N$ we implicitly omit all terms that do not satisfy the conditions in the statement:
	\begin{align*}
		\mathfrak{N}_N^\ast(\Psi(f))|_{\sigma_{\rho,q,r}}&=\sum_{i=1}^{N}\sum_{a=0}^{s}\frac{\delta_{\sigma_{\rho,q,r}}}{\delta_{\sigma_{(\rho_r)_a,q,Nq+i}}}((1\times\cdots\times N)^{\ast}\Psi(f))|_{\sigma_{(\rho_r)_a,q,Nq+i}}\\
													&=\sum_{i=1}^{N}\sum_{a=0}^{s}\frac{\delta_{\sigma_{\rho,q,r}}}{\delta_{\sigma_{(\rho_r)_a,q,Nq+i}}}\Psi(f)|_{\sigma_{((\rho_r)_a)_{-s},Nq+i-1}}(\mathsf{x}_s\to N\mathsf{x}_s)\\
													&=\sum_{i=1}^{N}\sum_{a=0}^{s}\frac{\delta_{\sigma_{\rho,q,r}}}{\delta_{\sigma_{(\rho_r)_a,q,Nq+i}}}f|_{\sigma_{((\rho_r)_a)_{-s}}}(Nq+i-1)(\mathsf{x}_1,\dots,\mathsf{x}_{s-1},N\mathsf{x}_s-(Nq+i-1)\mathsf{t},\mathsf{t}).
	\end{align*}
	Similarly, for the left-hand side, we get:
	\begin{align*}
		\Psi(\mathfrak{N}_{N,m}^\ast(f))|_{\sigma_{\rho,q,r}}&=\mathfrak{N}_{N,m}^\ast(f)|_{\sigma_{\rho_r}}(q)(\mathsf{x}_1,\dots,\mathsf{x}_{s-1},\mathsf{x}_s-q\mathsf{t},\mathsf{t})\\
														   &=\sum_{i=1}^{N}\sum_{a=0}^{s} \frac{\delta_{\sigma_{\rho_r}}(\mathsf{x}_1,\dots,\mathsf{x}_{s-1},\mathsf{x}_s-q\mathsf{t},\mathsf{t})}{\delta_{\sigma_{(\rho_r)_a,0,i}}(\mathsf{x}_1,\dots,\mathsf{x}_{s-1},\mathsf{x}_s-q\mathsf{t},\mathsf{t})}\, \\
														   &\quad \cdot f|_{\sigma_{((\rho_r)_a)_{-s}}} \bigl(Nq+i-1\bigr) \bigl(\mathsf{x}_1,\dots,\mathsf{x}_{s-1},N(\mathsf{x}_s-q\mathsf{t})-(i-1)\mathsf{t},\mathsf{t}\bigr)\\
														   &=\sum_{i=1}^{N}\sum_{a=0}^{s} \frac{\delta_{\sigma_{\rho,q,r}}}{\delta_{\sigma_{(\rho_r)_a,q,Nq+i}}}f|_{\sigma_{((\rho_r)_a)_{-s}}} \bigl(Nq+i-1\bigr) \bigl(\mathsf{x}_1,\dots,\mathsf{x}_{s-1},N(\mathsf{x}_s-q\mathsf{t})-(i-1)\mathsf{t},\mathsf{t}\bigr).
	\end{align*}
	From this it is clear that the two expressions agree.
	One easily checks that $\mathfrak{N}_{N,m}^\ast(f)$ lies in $\spp^{\ast}(\tilde{\Sigma}^{(s)})[m]^{\mathsf{gl}}$,
	finishing the proof.
\end{proof}

For simplicity, we introduce the following notation:
Fix $\rho\in \mathfrak S_{s-1}$. For
$f\in \spp^{\ast}(\tilde{\Sigma}^{(s)})[m]^{\mathsf{gl}}$ and $1\le r\le s-1$, set
\[
	H_{\rho,r}^f(m):=
	\frac{f|_{\sigma_{\rho_{r-1}}}(m)-f|_{\sigma_{\rho_r}}(m)}{\mathsf{x}_{\rho(r)}-\mathsf{x}_s}
\]
This quotient is polynomial in $m,\mathsf{x}_1,\dots,\mathsf{x}_s,\mathsf{t}$ by the continuity of $f$.
In preparation for the proof of Theorem \ref{thm:polynomiality}, the following
lemma uses Corollary \ref{cor:supp_function_sigma} to give a more explicit
expression for $\mathfrak N_{N,m}^{\ast}$.

\begin{lemma}
	\label{lem:contribution-formula-m}
	Let $f\in \spp^{\ast}(\tilde{\Sigma}^{(s)})[m]^{\mathsf{gl}}$, fix $\rho\in \mathfrak S_{s-1}$ and $0\le r\le s-1$.
	For $1\le i\le N$, define
	\begin{align}
		S_{i,r}={}&
		\frac{(i-1)(\mathsf{t}-\mathsf{x}_s)}{\mathsf{z}_i^{(N)}}f|_{\sigma_{\rho_0}}(m\to Nm+i-1,\mathsf{x}_s\to \mathsf{z}_i^{(N)})
		+\sum_{h=1}^{r}(i-1)(\mathsf{t}-\mathsf{x}_s)H_{\rho,h}^f(m\to Nm+i-1,\mathsf{x}_s\to \mathsf{z}_i^{(N)})\notag\\
				  &+f|_{\sigma_{\rho_r}}(m\to Nm+i-1,\mathsf{x}_s\to \mathsf{z}_i^{(N)})
				  +\sum_{h=r+1}^{s-1}(N-i)\mathsf{x}_sH_{\rho,h}^f(m\to Nm+i-1,\mathsf{x}_s\to \mathsf{z}_i^{(N)})\notag\\
				  &+\frac{(N-i)\mathsf{x}_s}{\mathsf{t}-\mathsf{z}_i^{(N)}}f|_{\sigma_{\rho_{s-1}}}(m\to Nm+i-1,\mathsf{x}_s\to \mathsf{z}_i^{(N)}).
				  \label{eq:contribution-formula-m}
	\end{align}
	Then
	\[
		\left.\mathfrak N_{N,m}^\ast(f)\right|_{\sigma_{\rho_r}}=\sum_{i=1}^N S_{i,r}.
	\]
\end{lemma}

\begin{proof}
	Let $\pi\in \mathfrak S_s$. We start by computing the factor $\frac{\delta_{\sigma_{\pi}}}{\delta_{\sigma_{\pi_a,0,i}}}$
	from Lemma \ref{lem:combinat_lift}. Observe that $\sigma_{\pi}$ is simplicial, defined by the conditions
	$\mathsf{x}_{\pi(j+1)}-\mathsf{x}_{\pi(j)}\ge 0$ for all $j\in \{0,\dots,s\}$, with the usual convention that $\mathsf{x}_{\pi(s+1)}=\mathsf{t}$
	and $\mathsf{x}_{\pi(0)}=0$. Hence
	$$\delta_{\sigma_{\pi}}=\prod_{j=0}^{s}(\mathsf{x}_{\pi(j+1)}-\mathsf{x}_{\pi(j)}).$$
	By Corollary \ref{cor:supp_function_sigma}, we therefore get:
	$$\frac{\delta_{\sigma_{\pi}}}{\delta_{\sigma_{\pi_a,0,i}}}=\frac{(i-1)(\mathsf{t}-\mathsf{x}_s)(\mathsf{x}_{\pi(a+1)}-\mathsf{x}_{\pi(a)})}{(\mathsf{z}_i^{(N)}-\mathsf{x}_{\pi(a)})(\mathsf{x}_{\pi(a+1)}-\mathsf{z}_i^{(N)})}$$
	when $a<\pi^{-1}(s)$ and
	$$\frac{\delta_{\sigma_{\pi}}}{\delta_{\sigma_{\pi_a,0,i}}}=\frac{(N-i)\mathsf{x}_s(\mathsf{x}_{\pi(a+1)}-\mathsf{x}_{\pi(a)})}{(\mathsf{z}_i^{(N)}-\mathsf{x}_{\pi(a)})(\mathsf{x}_{\pi(a+1)}-\mathsf{z}_i^{(N)})}$$
	when $a\ge \pi^{-1}(s)$. Note that in the cases excluded in Corollary \ref{cor:max_cones_sigma_n_bar}, these
	expressions vanish identically, so no special treatment of the edge cases is necessary.

	Using this, we can rewrite $\left.\mathfrak{N}_{N,m}^{\ast}(f)\right|_{\sigma_{\rho_r}}$ as:
	\begin{align*}
		\sum_{i=1}^{N}&\left(\sum_{a=0}^{r}
			\frac{(i-1)(\mathsf{t}-\mathsf{x}_s)(\mathsf{x}_{\rho_r(a+1)}-\mathsf{x}_{\rho_r(a)})}{(\mathsf{z}_i^{(N)}-\mathsf{x}_{\rho_r(a)})(\mathsf{x}_{\rho_r(a+1)}-\mathsf{z}_i^{(N)})}\,
			f|_{\sigma_{\rho_a}}
			\bigl(Nm+i-1\bigr)
			\bigl(\mathsf{x}_1,\dots,\mathsf{x}_{s-1},\mathsf{z}_i^{(N)},\mathsf{t}\bigr)\right.\\
		&\left.+\sum_{a=r+1}^{s}
			  \frac{(N-i)\mathsf{x}_s(\mathsf{x}_{\rho_r(a+1)}-\mathsf{x}_{\rho_r(a)})}{(\mathsf{z}_i^{(N)}-\mathsf{x}_{\rho_r(a)})(\mathsf{x}_{\rho_r(a+1)}-\mathsf{z}_i^{(N)})}\,
			  f|_{\sigma_{\rho_{a-1}}}
			  \bigl(Nm+i-1\bigr)
		  \bigl(\mathsf{x}_1,\dots,\mathsf{x}_{s-1},\mathsf{z}_i^{(N)},\mathsf{t}\bigr)\right).
	\end{align*}
	We simplify this expression by applying the partial fraction decomposition
	\[
		\frac{c-b}{(c-\mathsf{z})(\mathsf{z}-b)}=\frac{1}{c-\mathsf{z}}-\frac{1}{b-\mathsf{z}},
	\]
	to each term in the two inner sums. For $1\le a\le r$, we can reorder the sum to group the two terms with the same denominator, which together give
	\begin{align*}
	&(i-1)(\mathsf{t}-\mathsf{x}_s)\frac{f|_{\sigma_{\rho_{a-1}}}(m\to Nm+i-1,\mathsf{x}_s\to \mathsf{z}_i^{(N)})-f|_{\sigma_{\rho_a}}(m\to Nm+i-1,\mathsf{x}_s\to \mathsf{z}_i^{(N)})}{\mathsf{x}_{\rho_r(a)}-\mathsf{z}_i^{(N)}}\\
	&=(i-1)(\mathsf{t}-\mathsf{x}_s)H_{\rho,a}^f(m\to Nm+i-1,\mathsf{x}_s\to \mathsf{z}_i^{(N)}).
	\end{align*}
	Similarly, for $r+2\le a\le s$ we get
	\begin{align*}
&(N-i)\mathsf{x}_s\frac{f|_{\sigma_{\rho_{a-2}}}(m\to Nm+i-1,\mathsf{x}_s\to \mathsf{z}_i^{(N)})-f|_{\sigma_{\rho_{a-1}}}(m\to Nm+i-1,\mathsf{x}_s\to \mathsf{z}_i^{(N)})}{\mathsf{x}_{\rho_r(a)}-\mathsf{z}_i^{(N)}}\\
&=(N-i)\mathsf{x}_sH_{\rho,a-1}^f(m\to Nm+i-1,\mathsf{x}_s\to \mathsf{z}_i^{(N)}).
	\end{align*}
	Using this, and keeping track of the boundary terms, we get:
	\begin{align*}
		\left.\mathfrak{N}_{N,m}^{\ast}(f)\right|_{\sigma_{\rho_r}}=\sum_{i=1}^{N}&\left(\sum_{a=1}^{r}(i-1)(\mathsf{t}-\mathsf{x}_s)H_{\rho,a}^f(m\to Nm+i-1,\mathsf{x}_s\to \mathsf{z}_i^{(N)})\right.\\
		&+\sum_{a=r+2}^{s}
			  (N-i)\mathsf{x}_sH_{\rho,a-1}^f(m\to Nm+i-1,\mathsf{x}_s\to \mathsf{z}_i^{(N)})\\
		&+\frac{(i-1)(\mathsf{t}-\mathsf{x}_s)}{\mathsf{z}_i^{(N)}}f|_{\sigma_{\rho_0}}
			\bigl(Nm+i-1\bigr)
			\bigl(\mathsf{x}_1,\dots,\mathsf{x}_{s-1},\mathsf{z}_i^{(N)},\mathsf{t}\bigr)\\
		&+\left(\frac{(i-1)(\mathsf{t}-\mathsf{x}_s)}{\mathsf{x}_s-\mathsf{z}_i^{(N)}}-\frac{(N-i)\mathsf{x}_s}{\mathsf{x}_{s}-\mathsf{z}_i^{(N)}}\right)f|_{\sigma_{\rho_r}}
			\bigl(Nm+i-1\bigr)
			\bigl(\mathsf{x}_1,\dots,\mathsf{x}_{s-1},\mathsf{z}_i^{(N)},\mathsf{t}\bigr)\\
		&\left.+ \frac{(N-i)\mathsf{x}_s}{\mathsf{t}-\mathsf{z}_i^{(N)}}f|_{\sigma_{\rho_{s-1}}}
			  \bigl(Nm+i-1\bigr)
		  \bigl(\mathsf{x}_1,\dots,\mathsf{x}_{s-1},\mathsf{z}_i^{(N)},\mathsf{t}\bigr)\right)
	\end{align*}
	After simplifying and reindexing the sums, this precisely matches
	the formula in the statement.
\end{proof}

\begin{proof}[Proof of Theorem \ref{thm:polynomiality}]
	Fix $\rho\in \mathfrak S_{s-1}$ and $0\le r\le s-1$. For simplicity, set
	\[
		c(m):=f|_{\sigma_{\rho_0}}(m)(\mathsf{x}_s\to 0).
	\]
	The continuity condition in the definition of $\spp^{\ast}(\tilde{\Sigma}^{(s)})[m]^{\mathsf{gl}}$ gives
	\[
		f|_{\sigma_{\rho_{s-1}}}(m)(\mathsf{x}_s\to \mathsf{t})=c(m+1).
	\]
	Thus we may write
	\[
		f|_{\sigma_{\rho_0}}(m)=c(m)+\mathsf{x}_sR_0(m),\qquad
		f|_{\sigma_{\rho_{s-1}}}(m)=c(m+1)+(\mathsf{x}_s-\mathsf{t})R_s(m)
	\]
	with $R_0$ and $R_s$ in $\mathbb{Q}[m,\mathsf{x}_1,\dots,\mathsf{x}_s,\mathsf{t}]$. Substituting this into $S_{i,r}$, we rewrite
    {\allowdisplaybreaks
	\begin{align*}
		S_{i,r}={}&\left(1+\frac{(i-1)(\mathsf{t}-\mathsf{x}_s)}{\mathsf{z}_i^{(N)}}\right)c(Nm+i-1)
		+\frac{(N-i)\mathsf{x}_s}{\mathsf{t}-\mathsf{z}_i^{(N)}}c(Nm+i)\\
				  &+(i-1)(\mathsf{t}-\mathsf{x}_s)R_0(m\to Nm+i-1,\mathsf{x}_s\to \mathsf{z}_i^{(N)})
				  +\sum_{h=1}^{r}(i-1)(\mathsf{t}-\mathsf{x}_s)H_{\rho,h}^f(m\to Nm+i-1,\mathsf{x}_s\to \mathsf{z}_i^{(N)})\notag\\
				  &+f|_{\sigma_{\rho_r}}(m\to Nm+i-1,\mathsf{x}_s\to \mathsf{z}_i^{(N)})-c(Nm+i-1)
				  +\sum_{h=r+1}^{s-1}(N-i)\mathsf{x}_sH_{\rho,h}^f(m\to Nm+i-1,\mathsf{x}_s\to \mathsf{z}_i^{(N)})\notag\\
				  &-(N-i)\mathsf{x}_sR_s(m\to Nm+i-1,\mathsf{x}_s\to \mathsf{z}_i^{(N)}).
	\end{align*}}
	The coefficient of $c(Nm+i-1)$ simplifies as
	\[
		1+\frac{(i-1)(\mathsf{t}-\mathsf{x}_s)}{\mathsf{z}_i^{(N)}}=\frac{(N-i+1)\mathsf{x}_s}{\mathsf{z}_i^{(N)}}.
	\]
	Using $\mathsf{t}-\mathsf{z}_i^{(N)}=-\mathsf{z}_{i+1}^{(N)}$, the sum of the terms involving
	$c$ is
	\[
		\sum_{i=1}^N\left(\frac{(N-i+1)\mathsf{x}_s}{\mathsf{z}_i^{(N)}}c(Nm+i-1)
		-\frac{(N-i)\mathsf{x}_s}{\mathsf{z}_{i+1}^{(N)}}c(Nm+i)\right),
	\]
	which telescopes to
	\[
		c(Nm).
	\]
	All remaining terms are polynomial in $(i,N)$. Hence polynomiality in $N$ follows.

	It remains to prove the degree bound. If $d=0$, all $H_{\rho,h}^{f}$ and the functions $R_0,R_s$ vanish, $c$ is independent
	of $m$ and
	\[
		\left.\mathfrak N_{N,m}^\ast(f)\right|_{\sigma_{\rho_r}}=c(Nm),
	\]
	which has degree at most $0$ in $N$.

	Assume $d\ge 1$. Note that the degree of $c$ in $m$ is at most $d$ and hence the contribution $c(Nm)$ has
	degree at most $d$ in $N$. The remaining terms in the above expression for $S_{i,r}$ are clearly polynomial in $i$ and $N$ of total
	degree at most $d$. Hence it suffices to show that the degree $d$ part of the remainder of $S_{i,r}$ vanishes.

	For $0\le h\le s-1$, let $G_h(m,\mathsf{x}_s)$ be the homogeneous degree-$d$ part in $m$ and $\mathsf{x}_s$ of $f|_{\sigma_{\rho_h}}(m)$, and let $C(m)=G_0(m,0)$. The homogeneous degree-$(d-1)$ parts in $m$ and $\mathsf{x}_s$ of $R_0$, $R_s$, and $H_{\rho,h}^f$ are respectively
	\[
		\frac{G_0-C}{\mathsf{x}_s},\qquad \frac{G_{s-1}-C}{\mathsf{x}_s},\qquad \frac{G_h-G_{h-1}}{\mathsf{x}_s}.
	\]
	Since the substitutions $m\to Nm+i-1$ and $\mathsf{x}_s\to \mathsf{z}_i^{(N)}$ are of total degree 1 in $i$ and $N$, it follows that the
	possible degree-$d$ contribution in $i$ and $N$ to $S_{i,r}$ (omitting the $c$ terms as explained above) is
	the degree $d$ part of
	\[
		\frac{(i-1)(\mathsf{t}-\mathsf{x}_s)}{\mathsf{z}_i^{(N)}}(G_0-C)
		+\frac{(i-1)(\mathsf{t}-\mathsf{x}_s)}{\mathsf{z}_i^{(N)}}\sum_{h=1}^r(G_h-G_{h-1})
		+(G_r-C)
	\]
	\[
		+\frac{(N-i)\mathsf{x}_s}{\mathsf{z}_i^{(N)}}\sum_{h=r+1}^{s-1}(G_h-G_{h-1})
		-\frac{(N-i)\mathsf{x}_s}{\mathsf{z}_i^{(N)}}(G_{s-1}-C).
	\]
	The sums telescope, so this equals
	\[
		(G_r-C)\left(1+\frac{(i-1)(\mathsf{t}-\mathsf{x}_s)}{\mathsf{z}_i^{(N)}}-\frac{(N-i)\mathsf{x}_s}{\mathsf{z}_i^{(N)}}\right)
		=\frac{\mathsf{x}_s(G_r-C)}{\mathsf{z}_i^{(N)}}.
	\]
	This has total degree $d-1$ in $(m,\mathsf{x}_s,i,N)$ as required. Substituting $m$ and $\mathsf{x}_s$ as above and summing over $i=1,\dots,N$ gives degree at most $d$ in $N$.
\end{proof}

\begin{corollary}
	\label{cor:glob_poly}
	Let $f\in \spp^{\ast}(\tilde{\Sigma}^{(s)})[m]^{\mathsf{gl}}$. There is a unique element
	$f_N\in \spp^{\ast}(\tilde{\Sigma}^{(s)})[m]^{\mathsf{gl}}[N]$ of degree at most $\deg_{m,\mathsf{x}_s}(f)$ such that, for every
	$N\in \mathbb{Z}_{\ge 1}$, we have $f_N(N)=\mathfrak{N}_{N,m}^\ast(f)$.
\end{corollary}

\begin{proof}
	Let $K$ be the maximum of the degrees in $N$ of the restrictions of
	$\mathfrak{N}_{N,m}^\ast(f)$ to each maximal cone in $\tilde{\Sigma}^{(s)}$.
	Let $f_N$ be the polynomial defined by Lagrange interpolation of the
	evaluations $\mathfrak{N}_{1,m}^\ast(f),\dots,\mathfrak{N}_{K+1,m}^\ast(f)$.
	The restriction of $f_N$ to each maximal cone in $\tilde{\Sigma}^{(s)}$
	agrees with $\mathfrak{N}_{N,m}^\ast(f)$ for every $N$ and hence the two
	functions agree on all of $\tilde{\Sigma}^{(s)}$. Since $f_N$ is uniquely
	determined by its evaluations at all $N\in \mathbb{Z}_{\ge 1}$, the
	uniqueness of $f_N$ is clear. Moreover, since the coefficient of $N^{j}$ in
	$f_N$ is a linear combination of the elements of
	$\mathfrak{N}_{1,m}^\ast(f),\dots,\mathfrak{N}_{K+1,m}^\ast(f)\in
	\spp^{\ast}(\tilde{\Sigma}^{(s)})[m]^{\mathsf{gl}}$ for all $j\ge 0$, it follows that
	each such coefficient lies in $\spp^{\ast}(\tilde{\Sigma}^{(s)})[m]^{\mathsf{gl}}$.
\end{proof}

\subsection{Proof of the main theorem and corollaries}\label{sec:prf_wt_thm}

\subsubsection{Pullbacks and the log Poincar\'e line bundle}
\label{sec:pullbacks_poincare}

Over the interior, the fact that $\mathcal{P}^{\circ}$ is a biextension of ${A^{\circ}}^{2}$ implies that
$$(1\times N)^{\ast}(\mathcal{P}^{\circ})={\mathcal{P}^{\circ}}^{\otimes N}.$$
The same argument applies after replacing $\mathcal{P}^{\circ}$ with $\mathcal{P}^{\log}$.
However, the analogous equation for $\mathcal{P}$ is false, as
boundary corrections appear. The goal of this section is to use the description of $\mathcal{P}$
as a line bundle representing $\mathcal{P}^{\log}$ from \S\ref{sec:Plog} to determine these boundary corrections
(see Corollary \ref{cor:pp_corr_poincare_precise}). A related computation appears in \cite{BGHdJ}.

Let $\mathcal{P}^{\trop}:=\mathcal{P}^{\log}\times^{\mathbb{G}_{m,\log}}\mathbb{G}_{m,\trop}$ (equivalently,
$\mathcal{P}^{\trop}$ is the image of $\mathcal{P}^{\log}$ in $H^{1}(\barA^{(2)},\mathbb{G}_{m,\trop})$ under the map
$H^1(X, \Glog) \longrightarrow H^1(X, \mathbb G_{m,\trop})$ from \S\ref{sec:Plog}).
Write $\mathsf{PL}$ for the sheaf on $\RPC$ which sends $\sigma\in \RPC$ to its character 
lattice $M_{\sigma}$. Since $M_{\sigma}$ is the group of linear functions on $\sigma$, the set of morphisms
from a cone complex $\Sigma$ to $\mathsf{PL}$ is the set of strict piecewise linear functions
on $\Sigma$. Observe that $\mathbb{G}_{m,\trop}=\mathcal{A}_{\mathsf{PL}}$.
Let $\mathcal{P}^{\sfpp}$ denote the $\mathsf{PL}$-torsor on $\Sigma^{(2)}$ obtained as the quotient
$$\mathcal{P}^{\sfpp}:=(\mathcal U^2\Sigma^{(2)}\times\mathsf{PL})/\mathbb{Z}^{2},$$
where $(v_1,v_2)\in \mathbb{Z}^{2}$ sends $((\mathsf{x}_1,\mathsf{x}_2,\mathsf{t}),p)$ to
$$((\mathsf{x}_1+v_1\mathsf{t},\mathsf{x}_2+v_2\mathsf{t},\mathsf{t}),(\mathsf{y}_1,\mathsf{y}_2,s)\mapsto v_2\mathsf{y}_1+v_1\mathsf{y}_2+v_1v_2s+p(s,\mathsf{y}_1+v_1s,\mathsf{y}_2+v_2s)).$$

\begin{lemma}\label{lem:glob_secs_ptrop}
	\begin{enumerate}[label = (\alph*)]
		\item $\mathcal{P}^{\trop}$ is canonically isomorphic to the pullback of the geometric realization $\mathcal{A}_{\mathcal{P}^{\sfpp}}$ of $\mathcal{P}^{\sfpp}$
			along $\barA^{(2)}\to \mathcal{A}_{\Sigma^{(2)}}$.
		\item The pullback morphism induces an isomorphism from the group of global sections of $\mathcal{P}^{\sfpp}$ to the
			group of global sections of $\mathcal{P}^{\trop}$.
	\end{enumerate}
\end{lemma}

For the proof, we require the following definition which will also be used in the proof of Lemma \ref{lem:Artin_fan_A_N}:
Let $\mathcal{Q}$ be the sheaf on the category of rational polyhedral cones which sends a cone $\sigma$ with integral structure $N_{\sigma}$ to
the set of morphisms $f:N_{\sigma}\to \mathbb{Z}^{2}$ satisfying
$f_{\mathbb{R}}(\sigma)\subseteq \mathbb{R}_{\ge 0}\times \mathbb{R}$ and
$f_{\mathbb{R}}(\sigma)\cap \{0\}\times \mathbb{R}=\{0\}$. There is a natural action of $v\in \mathbb{Z}$ on $\mathcal{Q}$
which sends $f\in \Hom(\sigma,\mathcal{Q})$ to
$$w\mapsto f(w)+(0,f_1(w)v),\ \forall w\in N_{\sigma}.$$
Here, for $i\in \{1,2\}$, we write $f_i$ for the composition of $f$
with the canonical projection from $\mathbb{Z}^{2}$ to the $i$-th factor.
We define $\tropA:=\mathcal{Q}/\mathbb{Z}$.
Pointwise addition on the second factor of $\mathbb{Z}^2$ makes
$\tropA$ a group object over the cone $\mathcal{A}^{1}=(\mathbb{R}_{\ge 0},\mathbb{Z})$ via projection onto the first coordinate.
We write $(\tropA)^{s}:=\tropA\times_{\mathcal{A}^{1}}\dots\times_{\mathcal{A}^{1}}\tropA$. Observe that a morphism $\sigma\to (\tropA)^{s}$
similarly has a description as an equivalence class of morphisms $N_{\sigma}\to \mathbb{Z}^{s+1}$ satisfying conditions as above.

\begin{proof}[Proof of Lemma \ref{lem:glob_secs_ptrop}]
	For point (a), by \cite[Proposition 1.6 and Lemma 3.10]{KKN5}, $\mathcal{P}^{\trop}$
	is canonically isomorphic to the pullback along $(\logA)^{2}\to \mathcal{A}_{(\tropA)^{2}}$
	of $\mathcal{A}_{(\mathcal{Q}\times_{\mathcal{A}^{1}}\mathcal{Q}\times \PL)/\mathbb{Z}^{2}}$,
	where $(v_1,v_2)\in \mathbb{Z}^{2}$ acts on the $\sigma$-points of $\mathcal{Q}\times_{\mathcal{A}^{1}}\mathcal{Q}\times \PL$
	by $(f,g,m)\mapsto (f_{v_1},g_{v_2},-f_2^{\vee}(v_2)-g_2^{\vee}(v_1)-f_1^{\vee}(v_1v_2)+m)$.
	Here, $f_2^{\vee}:\mathbb{Z}\to M_{\sigma}$ (resp. $f_1^{\vee}:\mathbb{Z}\to M_{\sigma}$) denotes the dual of $f_2$
	(resp. $f_1$) (similarly for $g$). Observe that since the fiber product of $\mathcal{Q}$
	with itself is taken over $\mathcal{A}^{1}$, we have
	$f_1=g_1$, from which it is easy to show that the above formula
	determines a group action. Note that there is a canonical $\mathbb{Z}^{2}$-equivariant map
	$\mathcal U^2\Sigma^{(2)}\to \mathcal{Q}\times_{\mathcal{A}^{1}}\mathcal{Q}$
	induced by the inclusion $\mathcal U^2\Sigma^{(2)}\subseteq \mathbb{R}_{\ge 0}\times \mathbb{R}^{2}$
	(which gives two canonical maps $N_{\sigma}\to \mathbb{Z}^{2}$ for all $\sigma\in \mathcal U^2\Sigma^{(2)}$ by projecting onto the factor $\mathbb{R}_{\ge 0}$ and either the first or second factor of $\mathbb{R}^{2}$).
	Pulling back the above canonical isomorphism along $\Sigma^{(2)}=\mathcal U^2\Sigma^{(2)}/\mathbb{Z}^{2}\to (\tropA)^{2}$
	gives the canonical isomorphism of $\mathcal{P}^{\trop}$ with the pullback of $\mathcal{A}_{\mathcal{P}^{\sfpp}}$.

	We continue with point (b): Since $\mathcal{P}^{\sfpp}\times_{(\tropA)^{2}}\Sigma^{(2)}$
	is a trivial $\mathbb{G}_{m,\trop}$-torsor by Lemma \ref{lem:P_trop_triv} below, this is equivalent
	to the claim that the space of sections of $\mathbb{G}_{m,\trop}$
	over $\mathcal{A}_{\Sigma^{(2)}}$ is equal to the space of sections over $\barA^{(2)}$.
	For this, observe that $H^{0}(\mathbb{G}_{m,\trop})$ is canonically isomorphic to the group of
	Cartier divisors supported on $\partial\barA^{(2)}$. Since $\partial\barA^{(2)}$ is an irreducible Cartier divisor
	(corresponding to the unique ray in $\Sigma^{(2)}$), it follows that $H^{0}(\mathbb{G}_{m,\trop})\cong \mathbb{Z}$,
	canonically generated by the pullback of $\mathsf{t}\in \spp^{\ast}(\Sigma^{(2)})$. By the final statement in Lemma \ref{lem:P_trop_triv},
	this is precisely equal to $H^{0}(\Sigma^{(2)},\PL)$.
\end{proof}

From the definition of $\mathcal{P}^{\sfpp}$ it follows that the global sections of $\mathcal{P}^{\sfpp}$
are canonically in bijection with the set of strict piecewise linear functions $p$ on $\mathcal U^2\Sigma^{(2)}$ satisfying
$$p(\mathsf{x}_1,\mathsf{x}_2,\mathsf{t})=\mathsf{x}_1v_2+\mathsf{x}_2v_1+\mathsf{t}v_1v_2+p(\mathsf{x}_1+v_1\mathsf{t},\mathsf{x}_2+v_2\mathsf{t},\mathsf{t}).$$
Note that any such function is uniquely determined by its restriction to $\tilde{\Sigma}^{(2)}$ and in particular
by its restriction to $\mathcal U^2\Sigma^{(2)}$.

\begin{definition}
    Let $c_{\mathcal P}\in\spp^{\ast}(\tilde{\Sigma}^{(2)})[m]^{\mathsf{gl}}$ be the strict piecewise linear function which is
    $-(m+1)\mathsf{x}_1$ on the $0\le \mathsf{x}_1\le \mathsf{x}_{2}\le \mathsf{t}$ and which is $-\mathsf{x}_2-m\mathsf{x}_1$ on $0\le \mathsf{x}_2\le \mathsf{x}_1\le \mathsf{t}$.
\end{definition}

\begin{lemma}\label{lem:P_trop_triv}
	The $\PL$-torsor $\mathcal{P}^{\sfpp}$ is trivial. Under the restriction of the above canonical identification
	to $\mathcal U\Sigma^{(2)}$, the global sections of $\mathcal{P}^{\sfpp}$ are identified with the set:
	$$\{\Psi(c_{\mathcal P})+k\mathsf{t}|k\in \mathbb{Z}\}.$$
	In particular, $\PL(\Sigma^{(2)})=\{k\mathsf{t}|k\in \mathbb{Z}\}$.
\end{lemma}

\begin{proof}
	One easily checks that $\Psi(c_{\mathcal P})$ extends (uniquely) to a piecewise linear function on $\mathcal U^2\Sigma^{(2)}$
	satisfying the above equivariance property. Since $k\mathsf{t}$ is globally linear and independent of $(\mathsf{x}_1,\mathsf{x}_2)$, it follows that
	each $\Psi(c_{\mathcal P})+k\mathsf{t}$ gives a global section of $\mathcal{P}^{\sfpp}$. If $p\in \spp^{\ast}(\mathcal U^2\Sigma^{(2)})$
	is a second function satisfying the above equivariance property, the difference $\Psi(c_{\mathcal P})-p$
	is $\mathbb{Z}^{2}$-invariant. Subtracting some multiple of $\mathsf{t}$ from $p$, we may assume that $\Psi(c_{\mathcal P})-p$ vanishes
	at $(0,0,\mathsf{t})\in \mathbb{R}^{2}\times\mathbb{R}_{\ge 0}$. By $\mathbb{Z}^{2}$-invariance, it hence vanishes on all rays
	$(v_1\mathsf{t},v_2\mathsf{t},\mathsf{t})$ for $(v_1,v_2)\in \mathbb{Z}^{2}$. Since the rays of each maximal cone in $\mathcal U^2\Sigma^{(2)}$
	are of this form, it follows that the restriction of $\Psi(c_{\mathcal{P}})-p$ to each maximal cone vanishes at every ray. Since this restriction
	is linear, it vanishes, so $p=\Psi(c_{\mathcal P})$. The final statement follows from the fact that $\Sigma^{(2)}=\mathcal U^2\Sigma^{(2)}/\mathbb{Z}^{2}$.
\end{proof}

Recall the definition of $\Phi$ from \eqref{eq:Phi}.

\begin{corollary}\label{cor:pp_corr_poincare_precise}
	The piecewise linear function $c_{\mathcal P,N}:=(1\times N)^{\ast}\Psi(c_{\mathcal P})-Np_N^{\ast}\Psi(c_{\mathcal P})$
	on $\mathcal{U}\Sigma_N^{(2)}$ descends to a piecewise linear function on $\Sigma_N^{(2)}$ and
	$$(1\times N)^{\ast}\ell=Nb^{\ast}\ell+\Phi(c_{\mathcal P,N}) \in \CHop^1(\barA_N^{(2)}).$$
\end{corollary}

\begin{proof}
	The fact that $c_{\mathcal P,N}$ descends to $\Sigma_N^{(2)}$ can either be checked explicitly
	by showing that $c_{\mathcal P,N}$ is $\mathbb{Z}$-periodic with respect to the action $\mathsf{x}_2\mapsto \mathsf{x}_2+q\mathsf{t}$
	and that $c_{\mathcal P,N}$ restricts to zero at $\mathsf{x}_1\in \{0,\mathsf{t}\}$, or it follows from the fact that
	$c_{\mathcal P,N}$ is canonically a global section of the pullback of the trivial $\PL$-torsor
	on $(\tropA)^2$, since $(1\times N)^{\ast}\mathcal{P}^{\sfpp}\cong {\mathcal{P}^{\sfpp}}^{\otimes N}$. Recall from \S\ref{subs:extP}
	that $\mathcal{P}$ is a line bundle representing $\mathcal{P}^{\log}$.
	From \S\ref{sec:Plog} and Lemma \ref{lem:glob_secs_ptrop}, we can write
	the $\mathbb{G}_{m}$-torsor $\mathcal{P}$ as the preimage of $\Psi(c_{\mathcal P})+\lambda \mathsf{t}$ in $\mathcal{P}^{\log}$
	for some $\lambda\in \mathbb{Z}$. Hence $(1\times N)^{\ast}\mathcal{P}$ is the preimage of $(1\times N)^{\ast}(\Psi(c_{\mathcal P})+\lambda \mathsf{t})$
	in $(1\times N)^{\ast}\mathcal{P}^{\log}$ and $b^{\ast}(\mathcal{P})^{\otimes N}$ is the preimage of $N(p_N^{\ast}(\Psi(c_{\mathcal P})+\lambda \mathsf{t}))$
	in $b^{\ast}{\mathcal{P}^{\log}}^{\otimes N}$. Hence
	$$
    (1\times N)^{\ast}\mathcal{P}\otimes b^{\ast}(\mathcal{P})^{\otimes -N}
    $$
	is the preimage of $c_{\mathcal P,N}+(1-N)\lambda \mathsf{t}$ in the trivial $\mathbb{G}_{m,\log}$-torsor. Taking Chern classes, we get
	$$
    (1\times N)^{\ast}\ell-Nb^{\ast}\ell=\Phi(c_{\mathcal P,N}+(1-N)\lambda \mathsf{t})\,.
    $$
	We pull back this equality along the unit section $e^2:B \to \barA^{(2)}$. Since $c_{\mathcal P,N}$ vanishes on the rays, it pulls back to zero, and, by our conventions on $\mathcal P$, the left-hand side also pulls back to zero. We get
    $$
    0 = (e^2)^* \Phi((1-N)\lambda \mathsf{t}) = (1-N)\lambda D\,. $$
    By \cite[Corollary 1.6]{mumford_kodaira}, the class $D$ is a nonzero element of the Picard group of $B$, so we
	conclude that $\lambda=0$.
\end{proof}

We end this section by using the above result to show that the Theta divisor is of weight $2$ up to a piecewise linear correction. Let
$$c_{\theta}=-\frac{1}{2}((2m+1)\mathsf{x}_1+m^{2}\mathsf{t})\in \spp^{\ast}(\tilde{\Sigma}^{(1)})[m]^{\mathsf{gl}}.$$

\begin{corollary}\label{cor:pp_corr_theta_precise}
	The piecewise linear function $c_{\theta,N}=N^{\ast}\Psi(c_{\theta})-N^{2}p_N^{\ast}\Psi(c_{\theta})$
	on $\mathcal{U}\Sigma_N^{(1)}$ descends to a piecewise linear function on $\Sigma_N^{(1)}$ and
	$$N^{\ast}\theta=N^{2}b^{\ast}\theta+\Phi(c_{\theta,N}) \in \CHop^1(\barA_N^{(1)}).$$
\end{corollary}

\begin{figure}
	\begin{center}
		\[
			\begin{tikzpicture}[scale=0.82, transform shape, every node/.style={font=\scriptsize}]
				\tikzset{cone value/.style={align=center, inner sep=0.5pt}}
				\def\s{2.3}
				\def\lx{1.35}
				\def\dx{13.2}

				\foreach \m in {-1,0,1} {
					\draw[thick] (\lx+\m*\s,\m*\s) -- ++(\s,\s);
				}
				\foreach \m in {-1,0,1} {
					\fill (\lx+\m*\s,\m*\s) circle (1.2pt);
				}
				\fill (\lx,0) circle (2pt);
				\node[fill=white, inner sep=1pt] at (\lx-0.14*\s,0) {$0$};
				\node at (\lx+0.5*\s,2.50*\s) {$\Sigma^{(1)}$};
				\node[rotate=-45, font=\bfseries\large] at (\lx+2.18*\s,2.18*\s) {$\vdots$};
				\node[rotate=-45, font=\bfseries\large] at (\lx-1.18*\s,-1.18*\s) {$\vdots$};
				\node[above left, inner sep=1pt] at (\lx-0.50*\s,-0.50*\s) {$\mathsf{x}$};
				\node[above left, inner sep=1pt] at (\lx+0.50*\s,0.50*\s) {$-\mathsf{x}$};
				\node[above left, inner sep=1pt] at (\lx+1.50*\s,1.50*\s) {$-3\mathsf{x}+2\mathsf{t}$};

				\draw[->, thick] (\lx+1.25*\s,0.5*\s) -- node[above] {$\Delta$} (\dx-1.95*\s,0.5*\s);

				\foreach \i in {-1,0,1} {
					\foreach \j in {-1,0,1} {
						\draw[thick] (\dx+\i*\s,\j*\s) rectangle ++(\s,\s);
						\draw[thick] (\dx+\i*\s,\j*\s) -- ++(\s,\s);
					}
				}
				\draw[red, very thick] (\dx-\s,-\s) -- (\dx+2*\s,2*\s);
				\fill (\dx,0) circle (2pt);
				\node[inner sep=1pt] at (\dx-0.34*\s,-0.10*\s) {$(0,0)$};
				\node at (\dx+0.5*\s,2.50*\s) {$\Sigma^{(2)}$};
				\node[font=\bfseries\large] at (\dx+0.5*\s,2.28*\s) {$\vdots$};
				\node[font=\bfseries\large] at (\dx+0.5*\s,-1.28*\s) {$\vdots$};
				\node[font=\bfseries\large] at (\dx-1.28*\s,0.5*\s) {$\cdots$};
				\node[font=\bfseries\large] at (\dx+2.28*\s,0.5*\s) {$\cdots$};

				\node[cone value] at (\dx-0.67*\s,-0.22*\s) {$\mathsf{x}_2$};
				\node[cone value] at (\dx-0.30*\s,-0.75*\s) {$\mathsf{x}_1$};
				\node[cone value] at (\dx+0.33*\s,-0.22*\s) {$0$};
				\node[cone value] at (\dx+0.70*\s,-0.75*\s) {$\mathsf{x}_1-\mathsf{x}_2$\\$-\mathsf{t}$};
				\node[cone value] at (\dx+1.33*\s,-0.22*\s) {$-\mathsf{x}_2$};
				\node[cone value] at (\dx+1.70*\s,-0.75*\s) {$\mathsf{x}_1-2\mathsf{x}_2$\\$-2\mathsf{t}$};

				\node[cone value] at (\dx-0.67*\s,0.78*\s) {$-\mathsf{x}_1+\mathsf{x}_2$\\$-\mathsf{t}$};
				\node[cone value] at (\dx-0.30*\s,0.25*\s) {$0$};
				\node[cone value] at (\dx+0.33*\s,0.78*\s) {$-\mathsf{x}_1$};
				\node[cone value] at (\dx+0.70*\s,0.25*\s) {$-\mathsf{x}_2$};
				\node[cone value] at (\dx+1.33*\s,0.78*\s) {$-\mathsf{x}_1-\mathsf{x}_2$\\$+\mathsf{t}$};
				\node[cone value] at (\dx+1.70*\s,0.25*\s) {$-2\mathsf{x}_2$};

				\node[cone value] at (\dx-0.67*\s,1.78*\s) {$-2\mathsf{x}_1+\mathsf{x}_2$\\$-2\mathsf{t}$};
				\node[cone value] at (\dx-0.30*\s,1.25*\s) {$-\mathsf{x}_1$};
				\node[cone value] at (\dx+0.33*\s,1.78*\s) {$-2\mathsf{x}_1$};
				\node[cone value] at (\dx+0.70*\s,1.25*\s) {$-\mathsf{x}_1-\mathsf{x}_2$\\$+\mathsf{t}$};
				\node[cone value] at (\dx+1.33*\s,1.78*\s) {$-2\mathsf{x}_1-\mathsf{x}_2$\\$+2\mathsf{t}$};
				\node[cone value] at (\dx+1.70*\s,1.25*\s) {$-\mathsf{x}_1-2\mathsf{x}_2$\\$+2\mathsf{t}$};
			\end{tikzpicture}
		\]
	\end{center}
	\caption{\label{fig:Delta}The diagonal map $\Delta:\Sigma^{(1)}\to \Sigma^{(2)}$ on the links at $\mathsf{t}=1$. The labels denote the unique extension of $\Psi(c_{\mathcal P})$ as in the proof of Lemma \ref{lem:glob_secs_ptrop} (right) and its restriction to the diagonal (left).}
\end{figure}
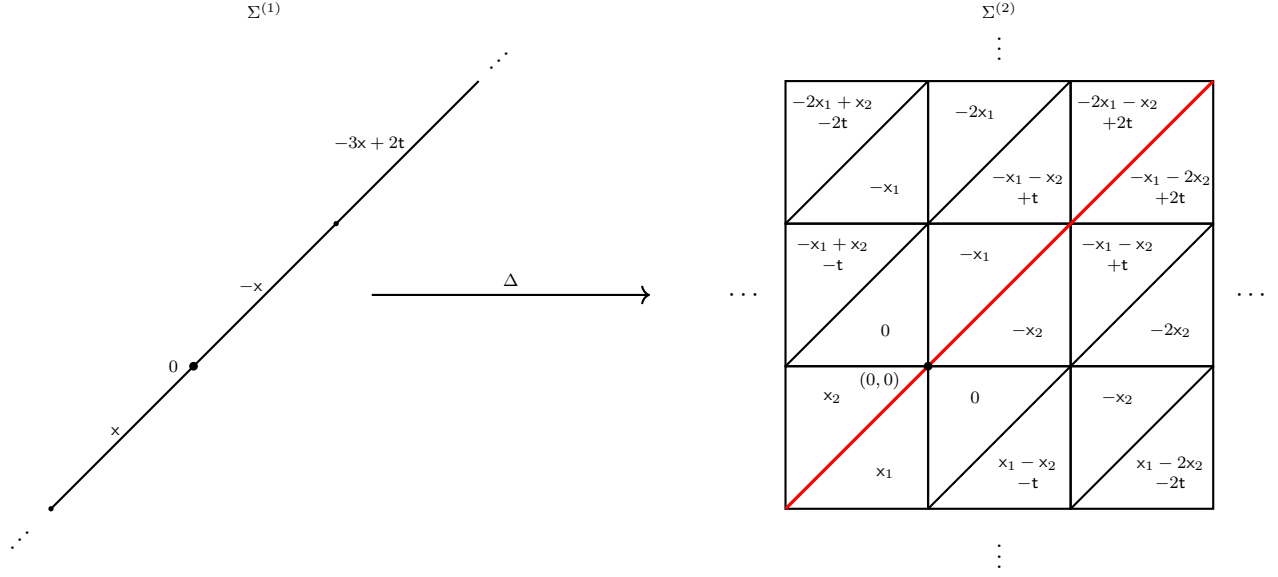

\begin{proof}
	Let $\Delta:\barA^{(1)}\to \barA^{(2)}$ be the map induced by the diagonal morphism
	$\logA\to (\logA)^{2}$ (see Figure \ref{fig:Delta}). Since $\Delta$ maps the unit section in $\barA^{(1)}$ to the unit section
	in $\barA^{(2)}$, the pullback $\Delta^{\ast}\ell$ is trivial along the unit section. Moreover,
	over the interior $B\setminus \partial B$, the class $\frac{1}{2}\Delta^{\ast}\ell$ restricts to the Theta divisor
	on $\mathcal{X}_g$. Hence $\theta=\frac{1}{2}\Delta^{\ast}\ell$. Let $\varrho:\Sigma^{(2)}\to \Sigma^{(2)}$
	be the automorphism induced by $(\mathsf{x}_1,\mathsf{x}_2,\mathsf{t})\mapsto (\mathsf{x}_2,\mathsf{x}_1,\mathsf{t})$ and write $\Sigma_{N,1}^{(2)}$ for the subdivision
	of $\Sigma^{(2)}$ obtained by pulling back $\Sigma_N^{(2)}$ along $\varrho$. We set
	$$\Sigma_{N,N}^{(2)}:=\Sigma_N^{(2)}\times_{\Sigma^{(2)}}\Sigma_{N,1}^{(2)}.$$
	Consider the following commutative diagram:
	$$\begin{tikzcd}
		\Sigma_{N,N}^{(2)}\arrow[r,"(N\times 1)"]\arrow[d,"q_N"]&\Sigma_{N}^{(2)}\arrow[d,"p_N"]\arrow[r,"(1\times N)"]&\Sigma^{(2)}\\
		\Sigma_{N,1}^{(2)}\arrow[r,"(N\times 1)"]&\Sigma^{(2)}
	\end{tikzcd}$$
	where the vertical morphisms are subdivisions and by abuse of notation $(N\times 1)$ denotes any morphism
	of the form $\varrho\circ (1\times N)\circ \varrho$ on some subdivision. This diagram induces a diagram of models of $(\logA)^{2}$:
	$$\begin{tikzcd}
		\barA_{N,N}^{(2)}\arrow[r,"(N\times 1)"]\arrow[d,"b'"]&\barA_N^{(2)}\arrow[d,"b"]\arrow[r,"(1\times N)"]&\barA^{(2)}\\
		\barA_{N,1}^{(2)}\arrow[r,"(N\times 1)"]&\barA^{(2)},
	\end{tikzcd}$$
	where $\barA_{N,N}^{(2)}$ and $\barA_{N,1}^{(2)}$ are defined in the obvious way.
	We compute using Corollary \ref{cor:pp_corr_poincare_precise}:
	\begin{align*}
		(N\times 1)^{\ast}(1\times N)^{\ast}\ell&=(N\times 1)^{\ast}(Nb^{\ast}\ell+\Phi(c_{\mathcal P,N}))\\
												&={b'}^{\ast}(N\times 1)^{\ast}N\ell+(N\times 1)^{\ast}\Phi(c_{\mathcal P,N})\\
												&={b'}^{\ast}\varrho^{\ast}\circ (1\times N)^{\ast}\circ \varrho^{\ast}N\ell+(N\times 1)^{\ast}\Phi(c_{\mathcal P,N})\\
												&={b'}^{\ast}\varrho^{\ast}(N^{2}b^{\ast}\ell+N\Phi(c_{\mathcal P,N}))+(N\times 1)^{\ast}\Phi(c_{\mathcal P,N})\\
												&=N^{2}{b'}^{\ast}\varrho^{\ast}b^{\ast}\ell+{b'}^{\ast}\varrho^{\ast}b^{\ast}\Phi(Nc_{\mathcal P,N})+(N\times 1)^{\ast}\Phi(c_{\mathcal P,N}).
	\end{align*}
	Now consider the commutative diagrams:
	$$\begin{tikzcd}[column sep=5em]
		\barA\arrow[d,"\Delta"]&\arrow[l,swap,"b"]\barA_{N}^{(1)}\arrow[d,"\Delta"]\arrow[r,"N"]&\barA\arrow[d,"\Delta"]\\
		\barA^{(2)}&\arrow[l,swap,"b\circ \varrho\circ b'"]\barA_{N,N}^{(2)}\arrow[r,"(1\times N)\circ (N\times 1)"]&\barA^{(2)}.
	\end{tikzcd}$$
	Using this, we get:
	\begin{align*}
		N^{\ast}2\theta&=N^{\ast}\Delta^{\ast}\ell\\
					  &=\Delta^{\ast}(N\times 1)^{\ast}(1\times N)^{\ast}\ell\\
					  &=\Delta^{\ast}\left(N^{2}{b'}^{\ast}\varrho^{\ast}b^{\ast}\ell+{b'}^{\ast}\varrho^{\ast}b^{\ast}\Phi(Nc_{\mathcal P,N})+(N\times 1)^{\ast}\Phi(c_{\mathcal P,N})\right)\\
					  &=2N^{2}b^*\theta+\Delta^{\ast}\Phi(Nc_{\mathcal P,N})+\Delta^{\ast}(N\times 1)^{\ast}\Phi(c_{\mathcal P,N}).
	\end{align*}
	It remains to show:
	$$\Phi(c_{\theta,N})=\frac{1}{2}\Delta^{\ast}\Phi(Nc_{\mathcal P,N})+\frac{1}{2}\Delta^{\ast}(N\times 1)^{\ast}\Phi(c_{\mathcal P,N})=\frac{1}{2}\Delta^{\ast}\Phi(Nc_{\mathcal P,N}+(N\times 1)^{\ast}c_{\mathcal P,N})$$
	(note that from this expression, it is clear that $c_{\theta,N}$ descends to a piecewise linear function on $\Sigma_N^{(1)}$).
	Consider the commutative diagram:
	$$\begin{tikzcd}
		\mathcal U\Sigma^{(1)}\arrow[d,"\Delta"]\arrow[r]&\tropA\arrow[d,"\Delta"]\\
		\mathcal U^2\Sigma^{(2)}\arrow[r]&(\tropA)^{2}
	\end{tikzcd}$$
	where the left-hand vertical map is induced by the diagonal map $\mathbb{R}\to \mathbb{R}^{2}$.
	From this, it follows that
	$$\frac{1}{2}\Delta^{\ast}\Phi(Nc_{\mathcal P,N}+(N\times 1)^{\ast}c_{\mathcal P,N})=\frac{1}{2}\Phi(\Delta^{\ast}(Nc_{\mathcal P,N}+(N\times 1)^{\ast}c_{\mathcal P,N})),$$
	where by abuse of notation, we identify $c_{\mathcal P,N}$ and $(N\times 1)^{\ast}c_{\mathcal P,N}$ with their unique extensions to $\mathcal U^2\Sigma^{(2)}$
	as in the proof of Lemma \ref{lem:glob_secs_ptrop}. To compute these extensions, we first write:
	$$Nc_{\mathcal P,N}+(N\times 1)^{\ast}c_{\mathcal P,N}=N(1\times N)^{\ast}\Psi(c_{\mathcal P})-N^{2}p_N^{\ast}\Psi(c_{\mathcal P})+(N\times 1)^{\ast}(1\times N)^{\ast}\Psi(c_{\mathcal P})-N(N\times 1)^{\ast}p_N^{\ast}\Psi(c_{\mathcal P}).$$
	We recall that the extension of $\Psi(c_{\mathcal P})$ must satisfy the functional equation:
	$$p(\mathsf{x}_1+v_1\mathsf{t},\mathsf{x}_2+v_2\mathsf{t},\mathsf{t})=-\mathsf{x}_1v_2-\mathsf{x}_2v_1-\mathsf{t}v_1v_2+p(\mathsf{x}_1,\mathsf{x}_2,\mathsf{t})$$
	for all $v_1,v_2\in \mathbb{Z}$. Using this, one easily computes that $\Delta^{\ast}\Psi(c_{\mathcal P})=2\Psi(c_{\theta})$. Moreover, observe
	that $\Psi(c_{\mathcal P})(\mathsf{x}_1,\mathsf{x}_2,\mathsf{t})=\Psi(c_{\mathcal P})(\mathsf{x}_2,\mathsf{x}_1,\mathsf{t})$ on $\tilde{\Sigma}^{(2)}$. The functional equation above preserves this symmetry, so it
	follows that the unique extension of $\Psi(c_{\mathcal P})$ to $\mathcal U^2\Sigma^{(2)}$ satisfies $\Psi(c_{\mathcal P})(\mathsf{x}_1,\mathsf{x}_2,\mathsf{t})=\Psi(c_{\mathcal P})(\mathsf{x}_2,\mathsf{x}_1,\mathsf{t})$. From this, it follows that
	$$\Delta^{\ast}(N(1\times N)^{\ast}\Psi(c_{\mathcal P}))=\Delta^{\ast}(N(N\times 1)^{\ast}p_N^{\ast}\Psi(c_{\mathcal P})),$$
	so, after applying $\Delta^{\ast}$, these two terms cancel. Hence:
	\begin{align*}
		\Delta^{\ast}(Nc_{\mathcal P,N}+(N\times 1)^{\ast}c_{\mathcal P,N})&=\Delta^{\ast}(-N^{2}p_N^{\ast}\Psi(c_{\mathcal P})+(N\times 1)^{\ast}(1\times N)^{\ast}\Psi(c_{\mathcal P}))\\
														   &=-2N^{2}p_N^{\ast}\Psi(c_{\theta})+\Delta^{\ast}(N\times 1)^{\ast}(1\times N)^{\ast}\Psi(c_{\mathcal P})\\
														   &=-2N^{2}p_N^{\ast}\Psi(c_{\theta})+N^{\ast}\Delta^{\ast}\Psi(c_{\mathcal P})\\
														   &=-2N^{2}p_N^{\ast}\Psi(c_{\theta})+2N^{\ast}\Psi(c_{\theta})\\
														   &=2c_{\theta,N}.
	\end{align*}
	This concludes the proof.
\end{proof}

\subsubsection{Comparison between the combinatorial and algebraic settings}

We return to the setting of \S\ref{sec:tr1} and study the action of $[1\times\dots\times N]^{\ast}$
on tautological classes of $\barA^{(s)}$.
Recall the space
$\barA_N^{(s)}:=\barA_{(1,\dots,1,N)}^{(s)}$ from equation \ref{eq:n}.
The following result
establishes the link between the tropical and algebraic settings:

\begin{lemma}\label{lem:Artin_fan_A_N}
	There is a commutative diagram:
	$$\begin{tikzcd}
		\barA_N^{(s)}\arrow[r,"{(1,\dots,N)}"]\arrow[d]&\barA^{(s)}\arrow[d]\\
		\mathcal{A}_{\Sigma_N^{(s)}}\arrow[r,"(1\times\dots\times N)"]&\mathcal{A}_{\Sigma^{(s)}}
	\end{tikzcd}$$
	where the vertical morphisms are strict, smooth, and surjective and
	the top horizontal morphism is proper. Moreover, the diagram:
	$$\begin{tikzcd}
		\barA_N^{(s)}\arrow[d]\arrow[r,"b"]&\barA^{(s)}\arrow[d]\\
		\mathcal{A}_{\Sigma_N^{(s)}}\arrow[r]&\mathcal{A}_{\Sigma^{(s)}}
	\end{tikzcd}$$
	is cartesian, where the bottom horizontal morphism is induced by $p_N$.
\end{lemma}

\begin{proof}
	We start with the first commutative diagram.

	The top horizontal morphism is proper since both $\barA_N^{(s)}$ and $\barA^{(s)}$ are proper over $B$. The right-hand
	vertical map is strict and smooth by Proposition \ref{pro:Atrop} and the definition of $\barA^{(s)}$. Recall the sheaf
	$\tropA$ from \S\ref{sec:pullbacks_poincare}.
	By \cite[9.3]{KKN7}, there is a strict smooth morphism $\logA\to \mathcal{A}_{\tropA}$ of group objects and there is a cartesian diagram:
	$$\begin{tikzcd}
		\barA^{(s)}\arrow[d]\arrow[r]&(\logA)^{s}\arrow[d]\\
		\mathcal{A}_{\Sigma^{(s)}}\arrow[r]&\mathcal{A}_{(\tropA)^{s}},
	\end{tikzcd}$$
	where the bottom horizontal morphism is induced by the canonical morphism $\Sigma^{(s)}\to \mathbb{R}^{s+1}$ by viewing
	the cone over $[0,1]^{s-1}\times \mathbb{R}$ as a subset of $\mathbb{R}^{s+1}$ with the cone direction in the first coordinate.
	The map $\barA^{(s)}\to (\logA)^{s}$ is a log modification corresponding to the subdivision
	$\Sigma^{(s)}\to (\tropA)^{s}$. Since log modifications are stable under base-change, it
	follows that $\barA_N^{(s)}$ is the log modification of $\barA^{(s)}$ obtained by pulling back the subdivision
	$f:\Sigma^{(s)}\to (\tropA)^{s}$ along the composition $\Sigma^{(s)}\xrightarrow{f} (\tropA)^{s}\xrightarrow{1\times\dots\times N} (\tropA)^{s}$.
	It follows from the definitions that this subdivision is $\Sigma_N^{(s)}$. From this, we get the commutative
	diagram in the statement. The left-hand vertical map is smooth since $\logA$ is log smooth and $\barA_N^{(s)}\to (\logA)^{s}$
	is log \'etale. For surjectivity, note that, since all involved spaces are log smooth, the images of the vertical maps contain
	the dense open point. Since $(\logA)^{s}$ and $\barA^{(s)}$ are proper over $B$, it follows that the vertical maps are surjective.

	For the second diagram, the above shows that $\barA_N^{(s)}\to (\logA)^{s}$ is a log modification and
	$$\barA_N^{(s)}\cong (\logA)^{s}\times_{\mathcal{A}_{(\tropA)^{s}}}\mathcal{A}_{\Sigma_N^{(s)}}.$$
	So we obtain a diagram:
	$$\begin{tikzcd}
		\barA_N^{(s)}\arrow[d]\arrow[r]&\mathcal{A}_{\Sigma_N^{(s)}}\arrow[d]\\
		\barA^{(s)}\arrow[d]\arrow[r]&\mathcal{A}_{\Sigma^{(s)}}\arrow[d]\\
		(\logA)^{s}\arrow[r]&\mathcal{A}_{(\tropA)^{s}},
	\end{tikzcd}$$
	where the outer rectangle and lower square are cartesian. The result follows from the transitivity
	of fiber products.
\end{proof}

Recall the definition of $\Phi$ from \eqref{eq:Phi}.

\begin{corollary}\label{cor:push_pull_trop_geom}
\begin{enumerate}[label = (\alph*)]
	\item For all $f\in \spp^{\ast}(\Sigma^{(s)})$, $(1\times\dots\times N)^{\ast}(\Phi(f))=\Phi((1\times\dots\times N)^{\ast}(f))$.
	\item For all $f\in \spp^{\ast}(\Sigma_N^{(s)})$, $b_{\ast}(\Phi(f))=\Phi(p_{N,\ast}(f))$.
\end{enumerate}
\end{corollary}

\begin{proof}
    The first statement follows from Lemma \ref{lem:Artin_fan_A_N} and the functoriality of the identification $\CH^{\ast}(\mathcal{A}_{\Sigma}) \cong \spp^{\ast}(\Sigma)$ from \S\ref{sec:taut_def}.
    For the second statement, consider the cartesian diagram
    $$\begin{tikzcd} \barA_N^{(s)}\arrow[d]\arrow[r,"b"]&\barA^{(s)}\arrow[d]\\ \mathcal{A}_{\Sigma_N^{(s)}}\arrow[r]&\mathcal{A}_{\Sigma^{(s)}} \end{tikzcd}$$
    from the lemma.
    Since the vertical morphisms are strict, this is both a fiber diagram of fs log stacks and a fiber diagram of algebraic stacks with log structures.
    As the vertical morphisms are smooth and the horizontal morphisms are proper, the second statement follows from flat base change.
\end{proof}

\subsubsection{Proof of Theorem \ref{thm:wt}}

\begin{proof}[Proof of Theorem \ref{thm:wt}]
	We prove the result only for $\gamma$ in the subring generated by the $\ell_{i,j}$, $\lambda_i$ and piecewise polynomial classes.
	The case including the $\theta_i$ is similar, additionally using Corollary \ref{cor:pp_corr_theta_precise}.
	We start with some elementary reductions: As explained in \S\ref{sec:wt_intro}, it suffices to prove the statement of the theorem for $i=s$,
	i.e., we are considering weight with respect to the multiplication by $N$ map on the $s$-th component. By Lemma \ref{lem:Npush} we have
	$$[\alpha\cup \gamma]^{(w)}=\alpha\cup [\gamma]^{(w)}$$
	for any $\alpha$ pulled back along $p_{1, \ldots , s-1}:\barA^{(s)}\to \barA^{(s-1)}$ (forgetting the $s$-th factor, see Lemma \ref{lem:maps_regular}).
	Hence any such $\alpha$ can be ignored. In particular, this applies to the classes $\ell_{i,j}$ for $i,j\ne s$, and
	we may assume that $\gamma=\prod_{i=1}^{s-1}\ell_{i,s}^{e_i}\cup\Phi(f)$ for some $e_i\in \mathbb{N}$
	and $f\in \spp^{\ast}(\Sigma^{(s)})$. Define
	$$\gamma(a_1,\dots,a_{s-1}):=\Phi(f)\cup\prod_{i=1}^{s-1}e^{a_i\ell_{i,s}}$$
	where the $a_i$ are formal indeterminates. The class $\gamma$ is recovered
	from $\gamma(a_1,\dots,a_{s-1})$ by extracting the coefficient of $a_1^{e_1}\cdot\dots\cdot a_{s-1}^{e_{s-1}}$
	(up to the constant global factor $\prod_{i=1}^{s-1}e_i!$, which we henceforth ignore).
	By the remark
	following Lemma \ref{lem:exists_explicit_lift}, we can
	write $f=\Psi(g)$ for some
	$g\in \spp^{\ast}(\tilde{\Sigma}^{(s)})[m]^{\mathsf{gl}}$ constant
	in $m$. It is clear that $\deg_{m,\mathsf{x}_s}(g)=\deg_{\mathsf{x}_s}(g)=\deg_{\mathsf{x}_s}(f)$.
	For $i\in \{1,\dots,s-1\}$, let $c_{\mathcal P,i,N}$ be the pullback
	of $c_{\mathcal P,N}$ (see Corollary \ref{cor:pp_corr_poincare_precise})
	along the map $\pi_i:\Sigma_N^{(s)}\to \Sigma_N^{(2)}$ induced by
	the projection to the $(i,s)$ factors.
	We now compute:
	\begin{align*}
		(1\times \dots\times N)^{\ast}(\gamma(a_1,\dots,a_{s-1}))&=\prod_{i=1}^{s-1}(1\times \dots\times N)^{\ast}(e^{a_i\ell_{i,s}})\cup(1\times \dots\times N)^{\ast}(\Phi(\Psi(g)))\\
		&=\prod_{i=1}^{s-1}\exp(Na_ib^{\ast}\ell_{i,s}+a_i\Phi(c_{\mathcal P,i,N}))\cup(1\times \dots\times N)^{\ast}(\Phi(\Psi(g)))\\
		&=\prod_{i=1}^{s-1}\exp(Na_ib^{\ast}\ell_{i,s}+a_i\Phi(c_{\mathcal P,i,N}))\cup(\Phi((1\times \dots\times N)^{\ast}\Psi(g))),
	\end{align*}
	where on the second line we applied Corollary \ref{cor:pp_corr_poincare_precise}
	and in the third line we used Corollary \ref{cor:push_pull_trop_geom}. Pushing forward along $b$, we get:
	{\allowdisplaybreaks\begin{align}
		&[1\times \dots\times N]^{\ast}(\gamma(a_1,\dots,a_{s-1}))=b_{\ast}\left(\prod_{i=1}^{s-1}\exp(Na_ib^{\ast}\ell_{i,s}+a_i\Phi(c_{\mathcal P,i,N}))\cup(\Phi((1\times \dots\times N)^{\ast}\Psi(g)))\right)\notag\\
		&=\prod_{i=1}^{s-1}\exp(Na_i\ell_{i,s})\cup b_{\ast}\left(\prod_{i=1}^{s-1}\exp(a_i\Phi(c_{\mathcal P,i,N}))\cup(\Phi((1\times \dots\times N)^{\ast}\Psi(g)))\right)\notag\\
		&=\prod_{i=1}^{s-1}\exp(Na_i\ell_{i,s})\cup\Phi\left(p_{N,\ast}\left(\prod_{i=1}^{s-1}\exp(a_ic_{\mathcal P,i,N})(1\times \dots\times N)^{\ast}(\Psi(g))\right)\right)\notag\\
		&=\prod_{i=1}^{s-1}\exp(Na_i\ell_{i,s})\cup\Phi \left(p_{N,\ast}\left(\prod_{i=1}^{s-1}\exp(a_i\pi_i^{\ast}((1\times N)^{\ast}\Psi(c_{\mathcal P})-Np_N^{\ast}\Psi(c_{\mathcal P})))(1\times \dots\times N)^{\ast}(\Psi(g))\right)\right)\notag\\
		&=\prod_{i=1}^{s-1}\exp(Na_i\ell_{i,s})\cup\Phi\left(\prod_{i=1}^{s-1}\exp(a_i\pi_i^{\ast}(-N\Psi(c_{\mathcal P}))) p_{N,\ast}(1\times\dots\times N)^{\ast}\left(\prod_{i=1}^{s-1}\exp(a_i\pi_i^{\ast}(\Psi(c_{\mathcal P})))\Psi(g)\right)\right)\notag\\
		&=\prod_{i=1}^{s-1}\exp(Na_i\ell_{i,s})\cup\Phi\left(\prod_{i=1}^{s-1}\exp(-Na_i\pi_i^{\ast}(\Psi(c_{\mathcal P}))) \Psi\left(\mathfrak{N}_{N,m}^{\ast}\left(\prod_{i=1}^{s-1}\exp(a_i\pi_i^{\ast}(c_{\mathcal P}))g\right)\right)\right)\label{eq:expansion_gamma}.
	\end{align}}
	where on the third line we applied Corollary \ref{cor:push_pull_trop_geom}.
	By abuse of notation, we also wrote $\pi_i^{\ast}$ for the map $\spp^{\ast}(\tilde{\Sigma}^{(2)})[m]^{\mathsf{gl}}\to \spp^{\ast}(\tilde{\Sigma}^{(s)})[m]^{\mathsf{gl}}$
	induced by the pullback along the projection onto the $(i,s)$ factors. Observe that $\deg_{m,\mathsf{x}_2}(c_{\mathcal P})=1$, so by
	Corollary \ref{cor:glob_poly}, the piecewise polynomial
	$$\mathfrak{N}_{N,m}^{\ast}\left(\prod_{i=1}^{s-1}\pi_i^{\ast}c_{\mathcal P}^{c_i}\cdot g \right)$$
	is polynomial in $N$ of degree at most
	$$\deg_{m,\mathsf{x}_s}(g)+\sum_{i=1}^{s-1}c_i.$$
	It follows that the contribution of the coefficient of $\prod_{i=1}^{s-1}a_i^{c_i}$ in the expansion of the factor $\mathfrak N_{N,m}^{\ast}(\ast)$
	in the expression \ref{eq:expansion_gamma} is polynomial in $N$ of degree at most $\deg_{m,\mathsf{x}_s}(g)+\sum_{i=1}^{s-1}c_i$.
	Taking into account the remaining factors, it follows that the coefficient of $\prod_{i=1}^{s-1}a_i^{e_i}$ in the expansion of the
	full expression \ref{eq:expansion_gamma} is polynomial in $N$ of degree at most
	$$\deg_{m,\mathsf{x}_s}(g)+\sum_{i=1}^{s-1}e_i\le \deg(f)+\sum_{i=1}^{s-1}e_i=c.$$
	This proves parts (b) and (c) of Theorem \ref{thm:wt}.
	Each pure degree $d$ part of the expression \ref{eq:expansion_gamma} is obtained by extracting the pure degree
	$d'$-parts of the coefficient of $\prod_{i=1}^{s-1}a_i^{c_i}$ in the factor
	$$\mathfrak{N}_{N,m}^{\ast}\left(\prod_{i=1}^{s-1}\pi_i^{\ast}c_{\mathcal P}^{c_i}\cdot g \right)$$
	in $N$, expanding the remaining factors in $N$ and grouping the resulting terms by their degrees. Each pure weight part
	is hence a linear combination of tautological classes, which proves
	part (a) of Theorem \ref{thm:wt}.
\end{proof}

\subsection{Weights under \texorpdfstring{$N_{\ast}$}{N*}}\label{sec:Nlow}
In this section, we study the action of the ``complementary'' map $[N_i]_{\ast}$ on $\CH^{\ast}(\barA^{(s)})$:

\begin{definition}
	For $\alpha\in \CH^{\ast}(\barA^{(s)})$ and $1\le i\le s$, let
	$$
    [N_i]_{\ast}(\alpha):=N^{-2g}N_{i,\ast}b^{\ast}(\alpha)\,.
    $$
    We say that $\alpha\in \CH^{\ast}(\barA^{(s)})$ has weight at most $k$ with respect to the $i$-th component for $i\in\{1,\dots,s\}$ if the classes
    $$
    [N_i]_{\ast}(\alpha)
    $$
    are polynomial in $N^{-1}$ of degree $k$. In this case, we denote the coefficient of
    $N^{-k}$ by $[\alpha]_{i,(k)}$ and will omit $i$ whenever it is clear from context.
    The class $[\alpha]_{(k)}$ is called the \emph{weight $k$ part} of $\alpha$. The associated filtration is denoted by
    $$
    \widehat{W}_k\CH^*(\barA^{(s)}).
    $$
\end{definition}
As with $[N_i]^{\ast}$, the map $[N_i]_{\ast}$ has a combinatorial counterpart (also denoted by $[N_i]_{\ast}$; see Proposition \ref{prop:push_pull_trop_geom_lower}).

The goal of this section is to prove the following analogue of Theorem \ref{thm:wt}:

\begin{theorem}\label{thm:wt_lower}
    Let $\gamma \in \R^{c}(\barA^{(s)})$.
	\begin{enumerate}[label = (\alph*)]
		\item $[N_i]_* \gamma$ is polynomial in $N^{-1}$ and for any $w \ge 0$, $[\gamma]_{i,(w)}$ is a tautological class.
		\item If $w>2c$, then $[\gamma]_{i,(w)} = 0$.
		\item More precisely, if $\varphi \in \spp^{d}(\Sigma^{(s)})$ and $\alpha$ is any class pulled back from $\barA^{(s-1)}$ (where we forget the $i$-th factor), then
			$$\left[\theta_i^{m}\cup\ell_{1,i}^{k_1}\cup\dots \cup\widehat{\ell_{i,i}^{k_i}}\cup\dots\cup\ell_{s,i}^{k_{s}}\cup\Phi(\varphi)\cup\alpha\right]_{i,(w)}=0$$
			for $w>2m+\sum_{j\neq i}k_j+d$.
	\end{enumerate}
\end{theorem}

The main additional difficulty in the proof of this theorem is that the tropical map $(1\times\dots\times N):\Sigma_N^{(s)}\to \Sigma^{(s)}$
is not separated, which complicates the construction of a lift of the geometric operator $(1\times\dots\times N)_{\ast}$.
In Proposition \ref{prop:push_pull_trop_geom_lower}, we show that there is nevertheless a Brion-like combinatorial formula for this pushforward
operator on toroidal compactifications of semi-abelian fibrations.

The proof of Theorem \ref{thm:wt_lower} will be given in \S\ref{sec:prf_wt_lower}. As in the proof of Theorem \ref{thm:wt},
we start with the tropical part of Theorem \ref{thm:wt_lower} in \S\ref{sec:wt_lower_trop} and then compare
with the geometric setup in \S\ref{sec:wt_lower_pushforwards} and \S\ref{sec:prf_wt_lower}.

\subsubsection{Weights with respect to $[1\times\dots\times N]_\ast$ on tropical abelian schemes}
\label{sec:wt_lower_trop}

Let $(\mathcal U\Sigma_N^{(s)})'$ be the (infinite) cone complex obtained from
$\mathcal U\overline{\Sigma}_N^{(s)}$ by forgetting the location of $\mathsf{x}_s-q\mathsf{t}$ in the
given order of $(0,\mathsf{x}_1,\dots,\mathsf{x}_{s-1},\mathsf{x}_{s}-q\mathsf{t},\mathsf{z}_i^{(N)})$ on
the maximal cone $\sigma_{\pi,q,i}$. More precisely, we let
$(\mathcal U\Sigma_N^{(s)})'$ be the subdivision of the cone over
$[0,1]^{s-1}\times \mathbb{R}$ whose maximal cones are of the
form $\tau_{\rho,i,r}$, where $\rho\in \mathfrak S_{s-1}$, $r\in \{0,\dots,s-1\}$
and $i\in \mathbb{Z}$. The cone $\tau_{\rho,i,r}$ is defined as:
$$\tau_{\rho,i,r}:=\{(\mathsf{x},\mathsf{t})\in \mathbb{R}^{s}\times\mathbb{R}_{\ge 0}|0\le \mathsf{x}_{\rho(1)}\le\dots \le \mathsf{x}_{\rho(r)}\le \mathsf{z}_i^{(N)}\le \mathsf{x}_{\rho(r+1)}\le \dots\le \mathsf{x}_{\rho(s-1)}\le \mathsf{t}\}$$
with the usual conventions $\mathsf{x}_{\rho(0)}=0$ and $\mathsf{x}_{\rho(s)}=\mathsf{t}$.
In particular, $\tau_{\rho,i,r}$ is simplicial, defined by the
hyperplanes
$$\mathsf{x}_{\rho(j+1)}-\mathsf{x}_{\rho(j)}\ge 0,\ \forall j\in \{0,\dots,\hat{r},\dots,s-1\}$$
and
$$\mathsf{z}_i^{(N)}-\mathsf{x}_{\rho(r)}\ge 0,\quad \mathsf{x}_{\rho(r+1)}-\mathsf{z}_i^{(N)}\ge 0.$$
Observe that the sublattice of $\langle \mathsf{x}_1,\dots,\mathsf{x}_s,\mathsf{t}\rangle$
spanned by the defining equations is
$\langle \mathsf{x}_1,\dots,\mathsf{x}_{s-1},N\mathsf{x}_s,\mathsf{t}\rangle$. In particular, it has index
$N$, so the support function of $\tau_{\rho,i,r}$ is
$$\delta_{\tau_{\rho,i,r}}=\frac{1}{N}(\mathsf{z}_i^{(N)}-\mathsf{x}_{\rho(r)})(\mathsf{x}_{\rho(r+1)}-\mathsf{z}_i^{(N)})\prod_{j\in \{0,\dots,\hat{r},\dots,s-1\}}^{}(\mathsf{x}_{\rho(j+1)}-\mathsf{x}_{\rho(j)}).$$
We write $q_N:\mathcal U\Sigma_N^{(s)}\to (\mathcal U\Sigma_N^{(s)})'$ for the subdivision
map (see also Figure \ref{fig:Sigma_prime}). Moreover, we define $(\Sigma_N^{(s)})'$ to be the cone complex
obtained from $(\mathcal{U}\Sigma_N^{(s)})'/\mathbb{Z}$ by identifying the faces $\{\mathsf{x}_i=0\}$ and $\{\mathsf{x}_i=\mathsf{t}\}$
for $i\in \{1,\dots,s-1\}$ (see also Definition \ref{def:cone_complexes}).

\begin{figure}
\begin{center}
	\[
	\begin{tikzpicture}
		\def\s{2.0}
		\def\dx{3.8}
		\def\yTop{3.6}
		\def\yBot{0}


		\draw[thick] (0,\yTop) rectangle ++(\s,\s);
		\foreach \k in {1,2,3} {
			\draw[thick] (0,\yTop+\k*\s/4) -- ++(\s,0);
		}
		\foreach \k in {0,1,2,3} {
			\draw[thick] (0,\yTop+\k*\s/4) -- ++(\s,\s/4);
		}
		\draw[thick] (0,\yTop) -- (\s,\yTop+\s);
		\node at (0.5*\s,\yTop+\s+0.5) {$\mathcal U\Sigma_N^{(2)}$};

		\draw[thick] (\dx,\yTop) rectangle ++(\s,\s);
		\foreach \k in {1,2,3} {
			\draw[thick] (\dx,\yTop+\k*\s/4) -- ++(\s,0);
		}
		\foreach \k in {0,1,2,3} {
			\draw[thick] (\dx,\yTop+\k*\s/4) -- ++(\s,\s/4);
		}
		\node at (\dx+0.5*\s,\yTop+\s+0.5) {$(\mathcal U\Sigma_N^{(2)})'$};

		\draw[thick] (2*\dx,\yTop) rectangle ++(\s,\s);
		\draw[thick] (2*\dx,\yTop) -- ++(\s,\s);
		\node at (2*\dx+0.5*\s,\yTop+\s+0.5) {$\mathcal U\Sigma^{(2)}$};

		\draw[->, thick]
			(\s+0.35,\yTop+0.5*\s) -- node[above] {$q_N$} (\dx-0.35,\yTop+0.5*\s);
		\draw[->, thick]
			(\dx+\s+0.35,\yTop+0.5*\s) -- node[above] {$c_N$} (2*\dx-0.35,\yTop+0.5*\s);


		\draw[thick] (0,\yBot) rectangle ++(\s,\s);
		\draw[thick] (0,\yBot) -- ++(\s,\s);
		\node at (0.5*\s,\yBot-0.7) {$\mathcal U\Sigma^{(2)}$};

		\draw[->, thick]
			(0.5*\s,\yTop-0.35) -- (0.5*\s,\yBot+\s+0.35);
		\node[left] at (0.5*\s,\yBot+\s+0.8) {$p_N$};
	\end{tikzpicture}
	\]
\end{center}
\caption{\label{fig:Sigma_prime}This figure depicts the (infinite) cone complexes $\mathcal U\Sigma_N^{(2)}$, $(\mathcal U\Sigma_N^{(2)})'$ and $\mathcal U\Sigma^{(2)}$, as well as the maps between them, in the case $N=4$. To simplify the illustration, we only depict their intersection with the region $0\le \mathsf{x}_1,\mathsf{x}_2\le \mathsf{t}$ and restrict to the links at $\mathsf{t}=1$. The morphisms $q_N$ and $p_N$ are subdivision morphisms, while $c_N$ corresponds to scaling the $\mathsf{x}_2$-coordinate by $N$.}
\end{figure}
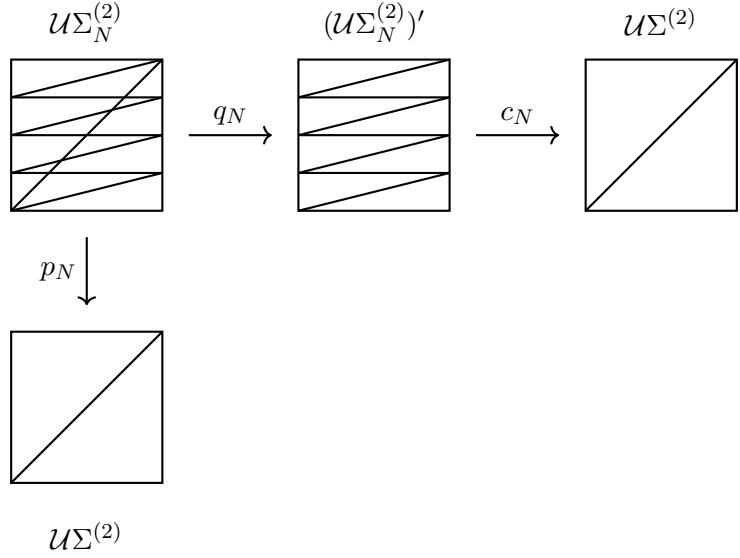

\begin{lemma}\label{lem:combinat_expr_lower}
	Let $f\in \spp^{\ast}(\tilde{\Sigma}^{(s)})[m]^{\mathsf{gl}}$ and let $\tau_{\rho,i,r}$
	be a maximal cone of $(\mathcal U\Sigma_N^{(s)})'$. Write $q:=\lfloor \frac{i-1}{N}\rfloor$. That is,
	$q$ is the unique integer satisfying $Nq+1\le i\le N(q+1)$. Then
	\begin{align*}
		q_{N,\ast}p_N^{\ast}(\Psi(f))|_{\tau_{\rho,i,r}}
		={}&
		\frac{N(q+1)-i}{N}f|_{\sigma_{\rho_0}}(q)(\mathsf{x}_1,\dots,\mathsf{x}_{s-1},\mathsf{x}_s-q\mathsf{t},\mathsf{t})\\
		   &+\frac{(N(q+1)-i)(\mathsf{x}_s-q\mathsf{t})}{N}\sum_{a=1}^{r}H_{\rho,a}^f(q)(\mathsf{x}_s\to \mathsf{x}_s-q\mathsf{t})\\
		&+\frac{1}{N}f|_{\sigma_{\rho_r}}(q)(\mathsf{x}_1,\dots,\mathsf{x}_{s-1},\mathsf{x}_s-q\mathsf{t},\mathsf{t})\\
		&+\frac{(i-1-Nq)(\mathsf{t}-\mathsf{x}_s+q\mathsf{t})}{N}\sum_{a=r+1}^{s-1}H_{\rho,a}^f(q)(\mathsf{x}_s\to \mathsf{x}_s-q\mathsf{t})\\
		&+\frac{i-1-Nq}{N}f|_{\sigma_{\rho_{s-1}}}(q)(\mathsf{x}_1,\dots,\mathsf{x}_{s-1},\mathsf{x}_s-q\mathsf{t},\mathsf{t}).
	\end{align*}
\end{lemma}

\begin{proof}
	If $N=1$, the formula is clear, so assume $N>1$.
	Let
	$$(c_0,\dots,c_{s+1})=(0,\mathsf{x}_{\rho(1)},\dots,\mathsf{x}_{\rho(r)},\mathsf{z}_i^{(N)},\mathsf{x}_{\rho(r+1)},\dots,\mathsf{x}_{\rho(s-1)},\mathsf{t}).$$
	By the above computation and Corollary \ref{cor:supp_function_sigma}, we have for every $b\in \{0,\dots,s-1\}$:
	$$\frac{\delta_{\tau_{\rho,i,r}}}{\delta_{\sigma_{(\rho_{b})_r,q,i}}}=\frac{(i-Nq-1)((q+1)\mathsf{t}-\mathsf{x}_s)(c_{b+2}-c_{b+1})}{N(c_{b+2}-\mathsf{x}_s+q\mathsf{t})(\mathsf{x}_s-q\mathsf{t}-c_{b+1})}$$
	if $r\le b$ and
	$$\frac{\delta_{\tau_{\rho,i,r}}}{\delta_{\sigma_{(\rho_{b})_{r+1},q,i}}}=\frac{((q+1)N-i)(\mathsf{x}_s-q\mathsf{t})(c_{b+1}-c_{b})}{N(c_{b+1}-\mathsf{x}_s+q\mathsf{t})(\mathsf{x}_s-q\mathsf{t}-c_{b})}$$
	if $r\ge b$.

	By Brion's formula, we have
    {\allowdisplaybreaks
	\begin{align*}
		q_{N,\ast}p_N^{\ast}(\Psi(f))|_{\tau_{\rho,i,r}}&=\sum_{b=0}^{r}\frac{\delta_{\tau_{\rho,i,r}}}{\delta_{\sigma_{(\rho_{b})_{r+1},q,i}}}p_N^{\ast}(\Psi(f))|_{\sigma_{(\rho_{b})_{r+1},q,i}}\\
														&+\sum_{b=r}^{s-1}\frac{\delta_{\tau_{\rho,i,r}}}{\delta_{\sigma_{(\rho_{b})_{r},q,i}}}p_N^{\ast}(\Psi(f))|_{\sigma_{(\rho_{b})_{r},q,i}}\\
														&=\sum_{b=0}^{r}\frac{((q+1)N-i)(\mathsf{x}_s-q\mathsf{t})(c_{b+1}-c_{b})}{N(c_{b+1}-\mathsf{x}_s+q\mathsf{t})(\mathsf{x}_s-q\mathsf{t}-c_{b})}\Psi(f)|_{\sigma_{\rho_{b},q}}\\
														&+\sum_{b=r}^{s-1}\frac{(i-Nq-1)((q+1)\mathsf{t}-\mathsf{x}_s)(c_{b+2}-c_{b+1})}{N(c_{b+2}-\mathsf{x}_s+q\mathsf{t})(\mathsf{x}_s-q\mathsf{t}-c_{b+1})}\Psi(f)|_{\sigma_{\rho_{b},q}}\\
														&=\sum_{b=0}^{r}\frac{((q+1)N-i)(\mathsf{x}_s-q\mathsf{t})(c_{b+1}-c_{b})}{N(c_{b+1}-\mathsf{x}_s+q\mathsf{t})(\mathsf{x}_s-q\mathsf{t}-c_{b})}f|_{\sigma_{\rho_b}}(q)(\mathsf{x}_1,\dots,\mathsf{x}_{s-1},\mathsf{x}_s-q\mathsf{t},\mathsf{t})\\
														&+\sum_{b=r}^{s-1}\frac{(i-Nq-1)((q+1)\mathsf{t}-\mathsf{x}_s)(c_{b+2}-c_{b+1})}{N(c_{b+2}-\mathsf{x}_s+q\mathsf{t})(\mathsf{x}_s-q\mathsf{t}-c_{b+1})}f|_{\sigma_{\rho_{b}}}(q)(\mathsf{x}_1,\dots,\mathsf{x}_{s-1},\mathsf{x}_s-q\mathsf{t},\mathsf{t}).
	\end{align*}}

	Note that after applying the partial fraction decomposition:
	$$\frac{c_1-c_2}{(c_1-\mathsf{x}_s+q\mathsf{t})(\mathsf{x}_s-q\mathsf{t}-c_2)}=\frac{1}{c_1-\mathsf{x}_s+q\mathsf{t}}-\frac{1}{c_2-\mathsf{x}_s+q\mathsf{t}},$$
	we can rearrange both sums to group terms with equal denominator.
	Keeping track of boundary terms gives:
    {\allowdisplaybreaks
	\begin{align*}
		\sum_{b=0}^{r}&\frac{((q+1)N-i)(\mathsf{x}_s-q\mathsf{t})(c_{b+1}-c_{b})}{N(c_{b+1}-\mathsf{x}_s+q\mathsf{t})(\mathsf{x}_s-q\mathsf{t}-c_{b})}f|_{\sigma_{\rho_b}}(q)(\mathsf{x}_1,\dots,\mathsf{x}_{s-1},\mathsf{x}_s-q\mathsf{t},\mathsf{t})\\
														&+\sum_{b=r}^{s-1}\frac{(i-Nq-1)((q+1)\mathsf{t}-\mathsf{x}_s)(c_{b+2}-c_{b+1})}{N(c_{b+2}-\mathsf{x}_s+q\mathsf{t})(\mathsf{x}_s-q\mathsf{t}-c_{b+1})}f|_{\sigma_{\rho_{b}}}(q)(\mathsf{x}_1,\dots,\mathsf{x}_{s-1},\mathsf{x}_s-q\mathsf{t},\mathsf{t})\\
														&=\frac{((q+1)N-i)(\mathsf{x}_s-q\mathsf{t})}{N}\sum_{b=1}^{r}H_{\rho,b}^{f}(q)(\mathsf{x}_s\to \mathsf{x}_s-q\mathsf{t})\\
														&+\frac{(i-Nq-1)((q+1)\mathsf{t}-\mathsf{x}_s)}{N}\sum_{b=r+1}^{s-1}H_{\rho,b}^{f}(q)(\mathsf{x}_s\to \mathsf{x}_s-q\mathsf{t})\\
														&-\frac{((q+1)N-i)(\mathsf{x}_s-q\mathsf{t})}{N(c_0-\mathsf{x}_s+q\mathsf{t})}f|_{\sigma_{\rho_0}}(q)(\mathsf{x}_1,\dots,\mathsf{x}_{s-1},\mathsf{x}_s-q\mathsf{t},\mathsf{t})\\
														&-\frac{(i-Nq-1)((q+1)\mathsf{t}-\mathsf{x}_s)}{N(c_{r+1}-\mathsf{x}_s+q\mathsf{t})}f|_{\sigma_{\rho_{r}}}(q)(\mathsf{x}_1,\dots,\mathsf{x}_{s-1},\mathsf{x}_s-q\mathsf{t},\mathsf{t})\\
														&+\frac{((q+1)N-i)(\mathsf{x}_s-q\mathsf{t})}{N(c_{r+1}-\mathsf{x}_s+q\mathsf{t})}f|_{\sigma_{\rho_r}}(q)(\mathsf{x}_1,\dots,\mathsf{x}_{s-1},\mathsf{x}_s-q\mathsf{t},\mathsf{t})\\
														&+\frac{(i-Nq-1)((q+1)\mathsf{t}-\mathsf{x}_s)}{N(c_{s+1}-\mathsf{x}_s+q\mathsf{t})}f|_{\sigma_{\rho_{s-1}}}(q)(\mathsf{x}_1,\dots,\mathsf{x}_{s-1},\mathsf{x}_s-q\mathsf{t},\mathsf{t})\\
														&=\frac{((q+1)N-i)(\mathsf{x}_s-q\mathsf{t})}{N}\sum_{b=1}^{r}H_{\rho,b}^{f}(q)(\mathsf{x}_s\to \mathsf{x}_s-q\mathsf{t})\\
														&+\frac{(i-Nq-1)((q+1)\mathsf{t}-\mathsf{x}_s)}{N}\sum_{b=r+1}^{s-1}H_{\rho,b}^{f}(q)(\mathsf{x}_s\to \mathsf{x}_s-q\mathsf{t})\\
														&+\frac{((q+1)N-i)}{N}f|_{\sigma_{\rho_0}}(q)(\mathsf{x}_1,\dots,\mathsf{x}_{s-1},\mathsf{x}_s-q\mathsf{t},\mathsf{t})\\
														&+\frac{1}{N}f|_{\sigma_{\rho_r}}(q)(\mathsf{x}_1,\dots,\mathsf{x}_{s-1},\mathsf{x}_s-q\mathsf{t},\mathsf{t})\\
														&+\frac{(i-Nq-1)}{N}f|_{\sigma_{\rho_{s-1}}}(q)(\mathsf{x}_1,\dots,\mathsf{x}_{s-1},\mathsf{x}_s-q\mathsf{t},\mathsf{t}).
		\end{align*}}
		This is exactly the expression in the statement, up to reordering.
\end{proof}

Let $c_N:(\mathcal U\Sigma_N^{(s)})'\to \mathcal U\Sigma^{(s)}$ be the morphism
induced by $\mathsf{x}_s\mapsto N\mathsf{x}_s$. Note that $c_N$ sends $\tau_{\rho,i,r}$
isomorphically to $\sigma_{\rho,i-1,r}$. Hence by Brion's formula,
we have
$$c_{N,\ast}(f)=f\left(\mathsf{x}_s\to \frac{\mathsf{x}_s}{N}\right)$$
for any $f\in \spp^{\ast}((\mathcal U\Sigma_N^{(s)})')$. For $j\in \mathbb{Z}$,
let $T_j:\mathcal U\Sigma^{(s)}\to \mathcal U\Sigma^{(s)}$ be the translation by $j$
operator, obtained by sending $\mathsf{x}_s$ to $\mathsf{x}_s+j\mathsf{t}$.

\begin{definition}
	For $N\ge 1$, we define an operator
	$\mathfrak{N}_{N,\ast}:\spp^{\ast}(\mathcal U\Sigma^{(s)})\left[\frac{j}{N}\right]\to \spp^{\ast}(\mathcal U\Sigma^{(s)})$ by
	$$\mathfrak{N}_{N,\ast}(f):=\frac{1}{N}\sum_{j=0}^{N-1}T_j ^{\ast}(c_{N,\ast}q_{N,\ast}p_N^{\ast}(f)).$$
\end{definition}

\begin{theorem}\label{thm:polynomiality_lower}
	Let $f\in \spp^{\ast}(\tilde{\Sigma}^{(s)})[m]^{\mathsf{gl}}\left[\frac{j}{N}\right]$. Then
	the restriction of $\mathfrak{N}_{N,\ast}(\Psi(f))$ to the maximal
	cone $\sigma_{\rho,q,r}$ of $\mathcal U\Sigma^{(s)}$ is a polynomial in $N^{-1}$ of degree at most
	$\deg_{m,\frac{j}{N},\mathsf{x}_s}(f)$ for every $N\in \mathbb{Z}_{\ge 1}$ satisfying $N\ge |q|$.
\end{theorem}

\begin{proof}
	For $k\ge 0$ and $f\in \spp^{\ast}(\tilde{\Sigma}^{(s)})[m]^{\mathsf{gl}}$ define:
	$$\mathfrak{N}_{N,k,\ast}(\Psi(f)):=\frac{1}{N}\sum_{j=0}^{N-1}\left(\frac{j}{N}\right)^{k}T_j ^{\ast}(c_{N,\ast}q_{N,\ast}p_N^{\ast}(\Psi(f))).$$
	By linearity it suffices to show that $\mathfrak{N}_{N,k,\ast}(\Psi(f))$ is polynomial in $N^{-1}$ of degree at most $d+k$, where
	$d:=\deg_{\mathsf{x}_s,m}(f)$. By Lemma \ref{lem:combinat_expr_lower}, we can write:
    {\allowdisplaybreaks
	\begin{align*}
		&\mathfrak{N}_{N,k,\ast}(\Psi(f))|_{\sigma_{\rho,q,r}}=\frac{1}{N}\sum_{j=0}^{N-1}\left(\frac{j}{N}\right)^{k}q_{N,\ast}p_N^{\ast}(\Psi(f))|_{\tau_{\rho,q+j+1,r}}\left(\mathsf{x}_s\to \frac{\mathsf{x}_s+j\mathsf{t}}{N}\right)\\
		&\quad=\frac{1}{N}\sum_{j=0}^{N-1}\left(\frac{j}{N}\right)^{k}\left(\frac{N\left(\left\lfloor \frac{q+j}{N}\right\rfloor+1\right)-q-j-1}{N}f|_{\sigma_{\rho_0}}\left(\left\lfloor \frac{q+j}{N}\right\rfloor\right)\left(\mathsf{x}_1,\dots,\mathsf{x}_{s-1},\frac{\mathsf{x}_s+j\mathsf{t}}{N}-\left\lfloor \frac{q+j}{N}\right\rfloor \mathsf{t},\mathsf{t}\right)\right.\\
		&\quad +\frac{\left(N\left(\left\lfloor \frac{q+j}{N}\right\rfloor+1\right)-q-j-1\right)\left(\frac{\mathsf{x}_s+j\mathsf{t}}{N}-\left\lfloor \frac{q+j}{N}\right\rfloor \mathsf{t}\right)}{N}\sum_{a=1}^{r}H_{\rho,a}^f\left(\left\lfloor \frac{q+j}{N}\right\rfloor\right)\left(\mathsf{x}_s\to \frac{\mathsf{x}_s+j\mathsf{t}}{N}-\left\lfloor \frac{q+j}{N}\right\rfloor \mathsf{t}\right)\\
		&\quad+\frac{1}{N}f|_{\sigma_{\rho_r}}\left(\left\lfloor \frac{q+j}{N}\right\rfloor\right)\left(\mathsf{x}_1,\dots,\mathsf{x}_{s-1},\frac{\mathsf{x}_s+j\mathsf{t}}{N}-\left\lfloor \frac{q+j}{N}\right\rfloor \mathsf{t},\mathsf{t}\right)\\
		&\quad+\frac{\left(q+j-N\left\lfloor \frac{q+j}{N}\right\rfloor\right)\left(\mathsf{t}-\frac{\mathsf{x}_s+j\mathsf{t}}{N}+\left\lfloor \frac{q+j}{N}\right\rfloor \mathsf{t}\right)}{N}\sum_{a=r+1}^{s-1}H_{\rho,a}^f\left(\left\lfloor \frac{q+j}{N}\right\rfloor\right)\left(\mathsf{x}_s\to \frac{\mathsf{x}_s+j\mathsf{t}}{N}-\left\lfloor \frac{q+j}{N}\right\rfloor \mathsf{t}\right)\\
		&\quad\left.+\frac{q+j-N\left\lfloor \frac{q+j}{N}\right\rfloor}{N}f|_{\sigma_{\rho_{s-1}}}\left(\left\lfloor \frac{q+j}{N}\right\rfloor\right)\left(\mathsf{x}_1,\dots,\mathsf{x}_{s-1},\frac{\mathsf{x}_s+j\mathsf{t}}{N}-\left\lfloor \frac{q+j}{N}\right\rfloor \mathsf{t},\mathsf{t}\right)\right)\\
	\end{align*}}
	Hence, up to constants, the expression in the outer sum is a polynomial in $\frac{j}{N}$, $\left\lfloor \frac{q+j}{N}\right\rfloor$ and $N^{-1}$ of total degree
	at most $d+k+1$. We show:

	\emph{Claim:} Let $a,b,c\ge 0$ and $q\in \mathbb{Z}$. Then
	$$\frac{1}{N}\sum_{j=0}^{N-1}\left(\frac{j}{N}\right)^{a}\left\lfloor \frac{q+j}{N}\right\rfloor^{b}N^{-c}$$
	is polynomial in $N^{-1}$ of degree at most $a+b+c$ for all
	$N\ge |q|$.
	\emph{Proof of Claim:} If $b=0$, we have by Faulhaber's formula
	that $\sum_{j=0}^{N-1}j^{a}$ is polynomial of degree at most $a+1$,
	from which the result follows. Hence we may assume that $b>0$.

	Since $|q|\le N$, we have
	$$\left\lfloor \frac{q+j}{N}\right\rfloor=\mathbf{1}_{j\ge N-q}-\mathbf{1}_{j< -q}.$$
	First assume that $q<0$. Then we can rewrite the above expression as:
	$$\sum_{j=0}^{-q-1}\left(\frac{j}{N}\right)^{a}(-1)^{b}N^{-c-1}$$
	which is clearly a polynomial in $N^{-1}$ of degree $a+c+1\le a+b+c$.
	Conversely, if $q\ge 0$, then the expression becomes:
	$$\sum_{j=N-q}^{N-1}\left(\frac{j}{N}\right)^{a}N^{-c-1}.$$
	By Faulhaber's formula, this is a polynomial in $N^{-1}$ of
	degree at most $a+c+1\le a+b+c$. This finishes the proof of the
	claim.

	From the claim, it follows that the expression above is polynomial
	of degree at most $d+k+1$. It remains to show that the degree
	$d+k+1$ contribution vanishes. For this, we write
	$$\mathsf{x}=\frac{j}{N},\quad \mathsf{y}=\left\lfloor \frac{q+j}{N}\right\rfloor,\quad \mathsf{z}=N^{-1}.$$
	Then we can rewrite the above expression as:
	{\allowdisplaybreaks
    \begin{align*}
		\frac{1}{N}\sum_{j=0}^{N-1}&\mathsf{x}^{k}\biggl((\mathsf{y}+1-q\mathsf{z}-\mathsf{x}-\mathsf{z})f|_{\sigma_{\rho_0}}\left(\mathsf{y}\right)\left(\mathsf{x}_1,\dots,\mathsf{x}_{s-1},\mathsf{z}\mathsf{x}_s+\mathsf{x}\mathsf{t}-\mathsf{y} \mathsf{t},\mathsf{t}\right)\\
			&+\left(\mathsf{y}+1-\mathsf{z}q-\mathsf{x}-\mathsf{z}\right)\left(\mathsf{z}\mathsf{x}_s+\mathsf{x}\mathsf{t}-\mathsf{y} \mathsf{t}\right)\sum_{a=1}^{r}H_{\rho,a}^f\left(\mathsf{y}\right)\left(\mathsf{x}_s\to \mathsf{z}\mathsf{x}_s+\mathsf{x}\mathsf{t}-\mathsf{y} \mathsf{t}\right)\\
		    &+\mathsf{z}f|_{\sigma_{\rho_r}}\left(\mathsf{y}\right)\left(\mathsf{x}_1,\dots,\mathsf{x}_{s-1},\mathsf{z}\mathsf{x}_s+\mathsf{x}\mathsf{t}-\mathsf{y} \mathsf{t},\mathsf{t}\right)\\
			&+\left(\mathsf{z}q+\mathsf{x}-\mathsf{y}\right)\left(\mathsf{t}-\mathsf{z}\mathsf{x}_s-\mathsf{x}\mathsf{t}+\mathsf{y} \mathsf{t}\right)\sum_{a=r+1}^{s-1}H_{\rho,a}^f\left(\mathsf{y}\right)\left(\mathsf{x}_s\to \mathsf{z}\mathsf{x}_s+\mathsf{x}\mathsf{t}-\mathsf{y} \mathsf{t}\right)\\
			&+(\mathsf{z}q+\mathsf{x}-\mathsf{y})f|_{\sigma_{\rho_{s-1}}}\left(\mathsf{y}\right)\left(\mathsf{x}_1,\dots,\mathsf{x}_{s-1},\mathsf{z}\mathsf{x}_s+\mathsf{x}\mathsf{t}-\mathsf{y} \mathsf{t},\mathsf{t}\right)\biggr)\\
	\end{align*}}
	Write $G_a$ for the homogeneous degree $d$ part of $f|_{{\sigma}_{\rho_a}}$
	in $\mathsf{x}_s$ and $m$. Since in the above expression we are replacing
	$m$ and $\mathsf{x}_s$ with homogeneous degree 1 terms in $\mathsf{x}$, $\mathsf{y}$ and
	$\mathsf{z}$ in the arguments of $H_{\rho,a}^{f}$, we may replace it
	with its degree $d-1$ homogeneous part in $\mathsf{x}_s$ and $m$ which
	is given by $\frac{G_{a-1}-G_a}{-\mathsf{x}_s}$. Hence the two sums telescope and, replacing $\mathsf{x}_s$ with $\mathsf{z}\mathsf{x}_s+\mathsf{x}\mathsf{t}-\mathsf{y}\mathsf{t}$, the
	degree $d+k+1$ part simplifies as the degree $d+k+1$ part of:
    {\allowdisplaybreaks
	\begin{align*}
		\frac{1}{N}\sum_{j=0}^{N-1}&\mathsf{x}^{k}\biggl((\mathsf{y}+1-q\mathsf{z}-\mathsf{x}-\mathsf{z})G_0\left(\mathsf{y}\right)\left(\mathsf{x}_1,\dots,\mathsf{x}_{s-1},\mathsf{z}\mathsf{x}_s+\mathsf{x}\mathsf{t}-\mathsf{y} \mathsf{t},\mathsf{t}\right)\\
				&+\left(\mathsf{y}+1-\mathsf{z}q-\mathsf{x}-\mathsf{z}\right)(G_{r}(\mathsf{y})(\mathsf{x}_s\to \mathsf{z}\mathsf{x}_s+\mathsf{x}\mathsf{t}-\mathsf{y}\mathsf{t})-G_0(\mathsf{y})(\mathsf{x}_s\to \mathsf{z}\mathsf{x}_s+\mathsf{x}\mathsf{t}-\mathsf{y}\mathsf{t}))\\
				&+\mathsf{z}G_r\left(\mathsf{y}\right)\left(\mathsf{x}_1,\dots,\mathsf{x}_{s-1},\mathsf{z}\mathsf{x}_s+\mathsf{x}\mathsf{t}-\mathsf{y} \mathsf{t},\mathsf{t}\right)\\
				&+\left(\mathsf{z}q+\mathsf{x}-\mathsf{y}\right)(G_r(\mathsf{y})(\mathsf{x}_s\to \mathsf{z}\mathsf{x}_s+\mathsf{x}\mathsf{t}-\mathsf{y}\mathsf{t})-G_{s-1}(\mathsf{y})(\mathsf{x}_s\to \mathsf{z}\mathsf{x}_s+\mathsf{x}\mathsf{t}-\mathsf{y}\mathsf{t}))\\
				&+(\mathsf{z}q+\mathsf{x}-\mathsf{y})G_{s-1}\left(\mathsf{y}\right)\left(\mathsf{x}_1,\dots,\mathsf{x}_{s-1},\mathsf{z}\mathsf{x}_s+\mathsf{x}\mathsf{t}-\mathsf{y} \mathsf{t},\mathsf{t}\right)\biggr)\\
				&=\frac{1}{N}\sum_{j=0}^{N-1}\mathsf{x}^{k}G_r\left(\mathsf{y}\right)\left(\mathsf{x}_1,\dots,\mathsf{x}_{s-1},\mathsf{z}\mathsf{x}_s+\mathsf{x}\mathsf{t}-\mathsf{y} \mathsf{t},\mathsf{t}\right).
		\end{align*}}
		The term in the sum has total degree at most $d+k$ in $\mathsf{x}$, $\mathsf{y}$ and $\mathsf{z}$ so, by the
		claim, the sum is again polynomial in $N^{-1}$
		of degree at most $d+k$. Hence the degree $d+k+1$ part vanishes, as claimed.
\end{proof}

\subsubsection{Comparison between the combinatorial and algebraic settings}

\label{sec:wt_lower_pushforwards}

In this section, we compare the tropical machinery from
\S\ref{sec:wt_lower_trop} with the geometric setting in
order to prove the following result, which will be the main technical
tool in the proof of Theorem \ref{thm:wt_lower}:

\begin{proposition}\label{prop:push_pull_trop_geom_lower}
\begin{enumerate}[label = (\alph*)]
	\item For all $f\in \spp^{\ast}(\Sigma^{(s)})$ we have $b^{\ast}(\Phi(f))=\Phi(p_N^{\ast}(f))$.
	\item For all $f\in \spp^{\ast}(\Sigma_N^{(s)})$ we have
		$$(1\times\dots\times N)_{\ast}(\Phi(f))=N^{2g}\Phi\left(\frac{1}{N}\sum_{j=0}^{N-1}T_j ^{\ast}c_{N,\ast}q_{N,\ast}(f)\right).$$
\end{enumerate}
\end{proposition}

Contrary to the proof of Corollary \ref{cor:push_pull_trop_geom},
the proof of point (2) in Proposition \ref{prop:push_pull_trop_geom_lower} is
significantly more technical. The reason for this is that the
morphism $(1\times\dots\times N):\mathcal{A}_{\Sigma_{N}^{(s)}}\to \mathcal{A}_{\Sigma^{(s)}}$
is not separated, so the usual flat base-change argument does not
apply. We solve this by factoring the map
$(1\times\dots\times N)$ into several simpler maps as shown
in Figure \ref{fig:Sigma_prime}. These are treated separately
in the proof of Proposition \ref{prop:push_pull_trop_geom_lower}.
The main idea is to use the fact that the morphism
$\logA\xrightarrow{N} \logA$ is a group morphism, so after symmetrizing
by its kernel, every class is pulled back from the source and we
can apply the projection formula. This is made significantly more
complicated by the fact that the $N$-torsion of $\logA$ is not
strict over the base, so one must perform a root stack operation
before being able to execute this plan.
We require the following technical lemma.

\begin{lemma}\label{lem:surj_mult_N}
	Let $A\to S$ be a family of log abelian varieties and $A_{\RPC}$ a tropical abelian variety over the cone stack $\Sigma_S$ corresponding
	to the (an) Artin fan of $S$ such that there
	exists a short exact sequence:
	$$
    0\to G_A\to A\to \mathcal{A}_{A_{\RPC}}\to 0
    $$
	as in \cite[4.1.2]{KKN2}, where $G_A$ denotes the maximal semi-abelian
	subsheaf of $A$. Then there is a short exact sequence:
	$$
    0\to G_A[N]\to A[N]\to \mathcal{A}_{A_{\RPC}}[N]\to 0
    $$
	of $N$-torsion subsheaves.
\end{lemma}

\begin{proof}
	Since the $N$-torsion subgroup is the kernel of the multiplication
	by $N$ operator, the snake lemma gives a long exact sequence:
	$$0\to G_A[N]\to A[N]\to \mathcal{A}_{A_{\RPC}}[N]\to G_A/(NG_A).$$
	Since semi-abelian schemes are divisible, the right-hand sheaf
	is trivial.
\end{proof}

\begin{proof}[Proof of Proposition \ref{prop:push_pull_trop_geom_lower}]
	Point (1) follows from the definition of $\Phi$ and the commutative diagram:
	$$\begin{tikzcd}
		\barA_N^{(s)}\arrow[d]\arrow[r,"b"]&\barA^{(s)}\arrow[d]\\
		\mathcal{A}_{\Sigma_N^{(s)}}\arrow[r]&\mathcal{A}_{\Sigma^{(s)}}
	\end{tikzcd}$$
	from Lemma \ref{lem:Artin_fan_A_N}. For point (2), consider the map $\phi_N:\mathcal{A}^{1}\to \mathcal{A}^{1}=(\mathbb{R}_{\ge 0},\mathbb{Z})$
	given by sending $\mathsf{x}$ to $N\mathsf{x}$. For any stack $\Sigma$ on $\RPC/\mathcal{A}^{1}$ (e.g., $\Sigma_N^{(s)}$, $\Sigma^{(s)}$, etc.
	via the projection onto the $\mathsf{t}$-coordinate) we write $\Sigma^{N}$ (e.g., $\Sigma_N^{(s),N}$, $\Sigma^{(s),N}$, $(\Sigma_N^{(s),N})'$)
	for the pullback of $\Sigma$ along $\phi_N$. Note that this simply replaces the integral structure $\mathbb{Z}^{s+1}$ with the
	integral structure $\mathbb{Z}^{s}\times N\mathbb{Z}$, where the $N\mathbb{Z}$ factor is in the $\mathsf{t}$-direction.
	Moreover, the action of $a\in \mathbb{Z}$ on $\mathcal U\Sigma^{(s)}$, etc. pulls back to the action of $a\in \mathbb{Z}$ on $\mathcal U\Sigma^{(s),N}$, etc.
	given by $(\mathsf{x}_1,\dots,\mathsf{x}_{s-1},\mathsf{x}_s,\mathsf{t})\mapsto (\mathsf{x}_1,\dots,\mathsf{x}_{s-1},\mathsf{x}_s+Na\mathsf{t},\mathsf{t})$. The corresponding quotient cone
	complex, obtained from $\mathcal U\Sigma^{(s),N}/\mathbb{Z}$ by identifying the sides $\{\mathsf{x}_i=0\}$ and $\{\mathsf{x}_i=\mathsf{t}\}$ for all $i\in \{1,\dots,s-1\}$, is $\Sigma^{(s),N}$.
	We define
	$$(\barA_N^{(s)})':=\barA_N^{(s)}\times_{\mathcal{A}_{\Sigma_N^{(s)}}}\mathcal{A}_{(\Sigma_N^{(s)})'}.$$
	Write $B^{N}:=B\times_{\phi_N}\mathcal{A}^{1}$ and $\barA_N^{(s),N}:=\barA_N^{(s)}\times_{B}B^{N}$ (similarly for $(\barA_N^{(s)})'$, etc.).
	Note that $B^{N}\to B$ is a
	root stack along the boundary divisor, so in particular, it is birational. Let $B(N)\to B^{N}$ denote the moduli space of full level $N$ structures of $(\logA)^{N}$ and let $\logA(N)\to (\logA)^{N}$ denote the moduli space of
	pairs consisting of a section of $(\logA)^{N}$ and an isomorphism of the $N$-torsion subsheaf of the corresponding family of log abelian varieties
	with $(\mathbb{Z}/N\mathbb{Z})^{2g}$. By \cite[Proposition 18.1]{KKN4},
	the morphism $\logA(N)\to (\logA)^{N}$ is finite flat of degree $|\Aut((\mathbb{Z}/N\mathbb{Z})^{2g})|$.
	Note that the $N$-torsion global $\mathcal{A}^{1}$-sections of $(\tropA)^{N}$ are precisely the maps
	from $\mathcal{A}^{1}=(\mathbb{R}_{\ge 0},\mathbb{Z})$
	to $\mathbb{R}_{\ge 0}\times \mathbb{R}$ sending $1$ to $(1,k)$ for $k\in \mathbb{Z}$ modulo the equivalence relation generated by $(1,k)\sim (1,k+N)$. Indeed, the projection map
	$(\tropA)^{N}\to \tropA$ corresponds to the map $\mathbb{R}_{\ge 0}\times \mathbb{R}\xrightarrow{N\times 1}\mathbb{R}_{\ge 0}\times \mathbb{R}$
	so from the definition of $\tropA$ it follows that a map $\mathcal{A}^{1}$ to $(\tropA)^{N}$ is the unit section whenever it
	sends $1$ to $(1,Nk)$ for $k\in \mathbb{Z}$. Let $(\logA)^{s}(N):=(\logA)^{N}\times_{B^{N}}\dots\times_{B^{N}} (\logA)^{N}\times_{B^{N}} \logA(N)$
	and $\barA^{(s)}(N):=(\logA)^{s}(N)\times_{\mathcal{A}_{(\tropA)^{s,N}}}\mathcal{A}_{(\Sigma_N^{(s),N})'}$. Note that the subdivision
	$(\Sigma_N^{(s)})'$ is invariant under the map $(\mathsf{x}_1,\dots,\mathsf{x}_s,\mathsf{t})\mapsto (\mathsf{x}_1,\dots,\mathsf{x}_s+\frac{1}{N}\mathsf{t},\mathsf{t})$, so the pullback
	$(\Sigma_N^{(s),N})'$ is invariant under the action of $\mathbb{Z}$ by $(\mathsf{x}_1,\dots,\mathsf{x}_s,\mathsf{t})\mapsto (\mathsf{x}_1,\dots,\mathsf{x}_s+\mathsf{t},\mathsf{t})$.
	Let $\xi:B(N)\to (\logA)^{s}(N)$ be an $N$-torsion section in the $s$-th coordinate (and the unit section in the first $s-1$ coordinates).
	This gives a corresponding section $\mathcal{A}^{1}\to (\tropA)^{s,N}$ which is $N$-torsion in the last coordinate. By the above,
	it corresponds to sending $1\in \mathbb{R}_{\ge 0}$ to $(0,\dots,0,k,1)\in \mathbb{R}^{s}\times\mathbb{R}_{\ge 0}$ for some
	$k\in \mathbb{Z}$. In particular, $(\Sigma_N^{(s),N})'$ is invariant under translation by $N$-torsion points, so the
	action of $(\mathbb{Z}/N\mathbb{Z})^{2g}$ on $(\logA)^{s}(N)$ given by multiplication by $N$-torsion points in the last coordinate lifts to an action
	on $\barA^{(s)}(N)$ such that the map $\barA^{(s)}(N)\to \mathcal{A}_{(\Sigma_N^{(s),N})'}$ is $(\mathbb{Z}/N\mathbb{Z})^{2g}$-equivariant.
	Here, the action of $(\mathbb{Z}/N\mathbb{Z})^{2g}$ on the target is given by the composition
	$$(\mathbb{Z}/N\mathbb{Z})^{2g}\xrightarrow{\sim}\logA[N]\to \mathcal{A}_{\tropA}[N]$$
	and the action by translation of $\mathcal{A}_{\tropA}[N]=\mathcal{A}_{\tropA[N]}$ on $\mathcal{A}_{(\Sigma_N^{(s),N})'}$.
	Now consider the map $(\logA)^{s}(N)\to (\logA)^{s,N}$ given by multiplication by $N$ in the last coordinate. Since the $N$-torsion
	in the last coordinate is killed by this map, it is $(\mathbb{Z}/N\mathbb{Z})^{2g}$-invariant. Hence the same is true for the map
	$$(1\times\dots\times N)_N:\barA^{(s)}(N)\to \barA^{(s),N}.$$
	By Lemma \ref{lem:surj_mult_N},
	the map $(\mathbb{Z}/N\mathbb{Z})^{2g}\to \tropA[N]\cong \mathbb{Z}/N\mathbb{Z}$ is surjective, so it follows
	that for any $f\in \spp^{\ast}((\Sigma_N^{(s),N})')$ and any $j\in \mathbb{Z}$:
	$$(1\times\dots\times N)_{N,\ast}(\Phi(f))=(1\times\dots\times N)_{N,\ast}(\Phi(T_j^{\ast}(f))),$$
	where by abuse of notation we also write $T_j:(\mathcal{U}\Sigma_N^{(s),N})'\to(\mathcal{U}\Sigma_N^{(s),N})'$ for the translation operator
	$$(\mathsf{x}_1,\dots,\mathsf{x}_s,\mathsf{t})\mapsto (\mathsf{x}_1,\dots,\mathsf{x}_{s-1},\mathsf{x}_s+j\mathsf{t},\mathsf{t}).$$
	Now consider the following commutative diagram:
	$$\begin{tikzcd}
		&&\barA^{(s)}(N)\arrow[d,"\rho"]\arrow[rrd,sloped,"(1\times\dots\times N)_N"]&&\\
		\barA_N^{(s),N}\arrow[d]\arrow[rdd,pos=0.25,"\pi_1"]\arrow[rr,"r_{N,N}"]&&(\barA_N^{(s),N})'\arrow[d]\arrow[rdd,pos=0.25,"\pi_2"]\arrow[rr,"(1\times\dots \times N)'_N"]&&\barA^{(s),N}\arrow[d]\arrow[rdd,pos=0.25,"\pi_3"]&\\
		\mathcal{A}_{\Sigma_N^{(s),N}}\arrow[rr,"q_N"]\arrow[dd,pos=0.25,"\phi_1"]&&\mathcal{A}_{(\Sigma_N^{(s),N})'}\arrow[dd,pos=0.25,"\phi_2"]\arrow[rr,"c_N"]&&\mathcal{A}_{\Sigma^{(s),N}}\arrow[dd,pos=0.25,"\phi_3"]&&\\
				&\barA_N^{(s)}\arrow[rr,pos=0.25,"r_N"]\arrow[rrrr,swap,bend right=70,"(1\times\dots\times N)"]\arrow[dl]&&(\barA_N^{(s)})'\arrow[dl]\arrow[rr,pos=0.25,"(1\times\dots\times N)'"]&&\barA^{(s)}\arrow[dl]\\
                \mathcal{A}_{\Sigma_N^{(s)}}\arrow[rr,"q_N"]&&\mathcal{A}_{(\Sigma_N^{(s)})'} \arrow[rr,"c_N"]&&\mathcal{A}_{\Sigma^{(s)}}&.
	\end{tikzcd}$$
	The left-hand squares are Cartesian, and the maps $\phi_i$ and $\pi_i$ are changes of lattices, hence proper and birational. Thus, for every
	$f\in \spp^{\ast}(\Sigma_N^{(s)})$:
    {\allowdisplaybreaks
	\begin{align*}
		(1\times\dots\times N)_{\ast}(\Phi(f))&=(1\times\dots\times N)'_{\ast}r_{N,\ast}(\Phi(f))\\
					                          &=(1\times\dots\times N)'_{\ast}(\Phi(q_{N,\ast}f))\\
											  &=(1\times\dots\times N)'_{\ast}(\pi_{2,\ast}\pi_2^{\ast}\Phi(q_{N,\ast}f))\\
											  &=\frac{1}{|\Aut((\mathbb{Z}/N\mathbb{Z})^{2g})|}\pi_{3,\ast}(1\times\dots\times N)'_{N,\ast}(\rho_{\ast}\rho^{\ast}\Phi(\phi_2^{\ast}q_{N,\ast}f))\\
											  &=\frac{1}{|\Aut((\mathbb{Z}/N\mathbb{Z})^{2g})|}\pi_{3,\ast}(1\times\dots\times N)_{N,\ast}(\Phi_{\barA^{(s)}(N)}(\phi_2^{\ast}q_{N,\ast}f))\\
											  &=\frac{1}{|\Aut((\mathbb{Z}/N\mathbb{Z})^{2g})|}\pi_{3,\ast}(1\times\dots\times N)_{N,\ast}\left(\Phi_{\barA^{(s)}(N)}\left(\frac{1}{N}\sum_{j=0}^{N-1}T_j^{\ast}\phi_2^{\ast}q_{N,\ast}f)\right)\right)\\
											  &=\frac{1}{|\Aut((\mathbb{Z}/N\mathbb{Z})^{2g})|}\pi_{3,\ast}(1\times\dots\times N)'_{N,\ast}\left(\rho_{\ast}\rho^{\ast}\Phi_{{\barA_N^{(s),N}}'}\left(\frac{1}{N}\sum_{j=0}^{N-1}T_j^{\ast}\phi_2^{\ast}q_{N,\ast}f)\right)\right)\\
											  &=\pi_{3,\ast}(1\times\dots\times N)'_{N,\ast}\left(\Phi_{{\barA_N^{(s),N}}'}\left(\frac{1}{N}\sum_{j=0}^{N-1}T_j^{\ast}\phi_2^{\ast}q_{N,\ast}f)\right)\right).
	\end{align*}}
	Now observe that since $\phi_2^{\ast}q_{N,\ast}f$ is
	$\mathbb{Z}$-periodic with respect
	to the action of $a\in \mathbb{Z}$ by
	$(\mathsf{x}_1,\dots,\mathsf{x}_s,\mathsf{t})\mapsto (\mathsf{x}_1,\dots,\mathsf{x}_{s-1},\mathsf{x}_s+Na\mathsf{t},\mathsf{t})$,
	the piecewise polynomial
	$$\frac{1}{N}\sum_{j=0}^{N-1}T_j^{\ast}\phi_2^{\ast}q_{N,\ast}f$$
	is $\mathbb{Z}$-periodic with respect to the action of
	$a\in \mathbb{Z}$ by $(\mathsf{x}_1,\dots,\mathsf{x}_s,\mathsf{t})\mapsto (\mathsf{x}_1,\dots,\mathsf{x}_{s-1},\mathsf{x}_s+a\mathsf{t},\mathsf{t})$.
	Hence this piecewise polynomial descends along $c_N$ to a piecewise polynomial
	on $\Sigma^{(s),N}$. Concretely, one easily sees that
	$$\frac{1}{N}\sum_{j=0}^{N-1}T_j^{\ast}\phi_2^{\ast}q_{N,\ast}f=c_N^{\ast}\left(\frac{1}{N}\sum_{j=0}^{N-1}T_j^{\ast}c_{N,\ast}\phi_2^{\ast}q_{N,\ast}f\right)$$
	where $c_{N,\ast}$ is defined as in \S\ref{sec:wt_lower_trop}.
	Hence:
    {\allowdisplaybreaks
	\begin{align*}
		(1\times\dots\times N)_{\ast}(\Phi(f))&=\pi_{3,\ast}(1\times\dots\times N)'_{N,\ast}\left(\Phi_{{\barA_N^{(s),N}}'}\left(c_N^{\ast}\left(\frac{1}{N}\sum_{j=0}^{N-1}T_j^{\ast}c_{N,\ast}\phi_2^{\ast}q_{N,\ast}f\right)\right)\right)\\
											  &=\pi_{3,\ast}(1\times\dots\times N)'_{N,\ast}\left({(1\times \dots\times N)'_{N}}^{\ast}\Phi_{{\barA^{(s),N}}}\left(\frac{1}{N}\sum_{j=0}^{N-1}T_j^{\ast}c_{N,\ast}\phi_2^{\ast}q_{N,\ast}f\right)\right)\\
											  &=N^{2g}\pi_{3,\ast}\Phi_{{\barA^{(s),N}}}\left(\frac{1}{N}\sum_{j=0}^{N-1}T_j^{\ast}c_{N,\ast}\phi_2^{\ast}q_{N,\ast}f\right)\\
											  &=N^{2g}\Phi\left(\frac{1}{N}\sum_{j=0}^{N-1}\phi_{3,\ast}T_j^{\ast}c_{N,\ast}\phi_2^{\ast}q_{N,\ast}f\right)\\
											  &=N^{2g}\Phi\left(\frac{1}{N}\sum_{j=0}^{N-1}T_j^{\ast}c_{N,\ast}q_{N,\ast}f\right).
	\end{align*}}
\end{proof}

\subsubsection{Proof of Theorem \ref{thm:wt_lower}}
\label{sec:prf_wt_lower}

\begin{proof}[Proof of Theorem \ref{thm:wt_lower}]
	We prove the result only for $\gamma$ in the subring generated by the $\ell_{i,j}$, $\lambda_i$ and piecewise polynomial classes.
	The case including the $\theta_i$ is similar, additionally using Corollary \ref{cor:pp_corr_theta_precise}.
	By symmetry, it suffices to prove the theorem for weight with
	respect to the $s$-th component, i.e., we prove that
	$$[1\times\dots\times N]_{\ast}(\gamma)$$
	is polynomial in $N^{-1}$ of degree at most $c=\deg(\gamma)$. Since
	the classes $\lambda_i$ are pulled back from $B$ and
	$b$ and $1\times\dots\times N$ are morphisms over $B$, we get
	$$(1\times\dots\times N)_{\ast}b^{\ast}(\lambda_i\cup\gamma)=\lambda_i\cup(1\times\dots\times N)_{\ast}b^{\ast}(\gamma).$$
	Hence $[\lambda_i\cup\gamma]_{(w)}=\lambda_i\cup[\gamma]_{(w)}$. By linearity,
	we may therefore assume that $\gamma$ is a monomial in the classes
	$\ell_{i,j}$ and piecewise polynomials. Similarly, if $s\notin\{i,j\}$,
	then $\ell_{i,j}$ is pulled back along the projection
	$\barA^{(s)}\to \barA^{(2)}$ lifting the projection
	$(\logA)^{s}\to (\logA)^{2}$ onto the $i$-th and $j$-th factors. Since
	both $b$ and $(1\times\dots\times N)$ are morphisms over this
	projection, the same argument as above shows that we can assume
	that $\gamma$ is of the form
	$$\gamma=\prod_{i=1}^{s-1}\ell_{i,s}^{e_i}\cup\Phi(f)$$
	for some $f\in \spp^{\ast}(\Sigma^{(s)})$. Define
	$$\gamma(a_1,\dots,a_{s-1}):=\Phi(f)\cup\prod_{i=1}^{s-1}e^{a_i\ell_{i,s}}$$
	where the $a_i$ are formal indeterminates. The class $\gamma$ is recovered
	from $\gamma(a_1,\dots,a_{s-1})$ by extracting the coefficient of $a_1^{e_1}\cdot\dots\cdot a_{s-1}^{e_{s-1}}$ (up to the constant global factor $\prod_{i=1}^{s-1}e_i!$, which we henceforth ignore).
	For $i\in \{1,\dots,s-1\}$, let $c_{\mathcal P,i,N}$ be the pullback
	of $c_{\mathcal P,N}$ (see Corollary \ref{cor:pp_corr_poincare_precise})
	along the map $\pi_i:\Sigma_N^{(s)}\to \Sigma_N^{(2)}$ induced by
	the projection to the $(i,s)$ factors. By abuse of notation, we also
	write $\pi_i:\Sigma^{(s)}\to \Sigma^{(2)}$ for the projection onto
	the $(i,s)$ factors.
	Since $f$ is $\mathbb{Z}$-periodic, we can write
	$f=\Psi(f')$ for some $f'\in \spp^{\ast}(\tilde{\Sigma}^{(s)})[m]^{\mathsf{gl}}$
	independent of $m$ and such that $\deg_{\mathsf{x}_s}(f)=\deg_{\mathsf{x}_s}(f')=\deg_{m,\mathsf{x}_s}(f')$.
	We compute using Corollary \ref{cor:pp_corr_poincare_precise}
	and Proposition \ref{prop:push_pull_trop_geom_lower}:
	{\allowdisplaybreaks\begin{align*}
			&[1\times\dots\times N]_{\ast}(\gamma(a_1,\dots,a_{s-1}))=N^{-2g}(1\times\dots\times N)_{\ast}b^{\ast}\left(\prod_{i=1}^{s-1}e^{a_i\ell_{i,s}}\cup\Phi(f)\right)\\
			&=N^{-2g}(1\times\dots\times N)_{\ast}\left(\prod_{i=1}^{s-1}b^{\ast}e^{a_i\ell_{i,s}}\cup \Phi(p_N^{\ast}f)\right)\\
			&=N^{-2g}(1\times\dots\times N)_{\ast}\left(\prod_{i=1}^{s-1}\exp\left(\frac{1}{N}(1\times \dots\times N)^{\ast}a_i\ell_{i,s}-\frac{1}{N}\Phi(a_ic_{\mathcal P,i,N})\right)\cup\Phi(p_N^{\ast}f)\right)\\
			&=\prod_{i=1}^{s-1}\exp\left(\frac{1}{N}a_i\ell_{i,s}\right)\cup\Phi\left(\frac{1}{N}\sum_{j=0}^{N-1}T_j ^{\ast}c_{N,\ast}q_{N,\ast}\left(\prod_{i=1}^{s-1}\exp\left(-\frac{1}{N}a_ic_{\mathcal P,i,N}\right)p_N^{\ast}f\right)\right)\\
			&=\prod_{i=1}^{s-1}\exp\left(\frac{1}{N}a_i\ell_{i,s}\right)\cup\Phi\left(\frac{1}{N}\sum_{j=0}^{N-1}T_j ^{\ast}c_{N,\ast}q_{N,\ast}\left(\prod_{i=1}^{s-1}\pi_i^{\ast}\exp\left(-\frac{a_i}{N}(1\times N)^{\ast}\Psi(c_{\mathcal P})+a_ip_N^{\ast}\Psi(c_{\mathcal P})\right)p_N^{\ast}f\right)\right)\\
			&=\prod_{i=1}^{s-1}\exp\left(\frac{1}{N}a_i\ell_{i,s}\right)\cup\Phi\Biggl(\frac{1}{N}\sum_{j=0}^{N-1}T_j ^{\ast}c_{N,\ast}q_{N,\ast}\biggl(\prod_{i=1}^{s-1}\Bigl((1\times\dots\times N)^{\ast}\pi_i^{\ast}\left(\exp\left(-\frac{a_i}{N}\Psi(c_{\mathcal P})\right)\right)\\
			&\cdot\pi_i^{\ast}\left(\exp\left(a_ip_N^{\ast}\Psi(c_{\mathcal P})\right)\right)\Bigr)p_N^{\ast}f\biggr)\Biggr)\\
			&=\prod_{i=1}^{s-1}\exp\left(\frac{1}{N}a_i\ell_{i,s}\right)\\
			&\cup \Phi\left(\frac{1}{N}\sum_{j=0}^{N-1}T_j ^{\ast}\left(\prod_{i=1}^{s-1}\pi_i^{\ast}\left(\exp\left(-\frac{a_i}{N}\Psi(c_{\mathcal P})\right)\right)\cdot c_{N,\ast}q_{N,\ast}\left(p_N^{\ast}\left(\prod_{i=1}^{s-1}\pi_i^{\ast}\left(\exp\left(a_i\Psi(c_{\mathcal P})\right)\right)f\right)\right)\right)\right)
	\end{align*}}
	where on the last line we used the projection formula and $c_{N}\circ q_N=(1\times\dots\times N)$.
	From the definition of $c_{\mathcal P}$ and $T_j$ it is easy to see that
	$$T_j ^{\ast}(\pi_i^{\ast}(\Psi(c_{\mathcal P})))=\pi_i^{\ast}\Psi(c_{\mathcal P})-j\mathsf{x}_i.$$
	Hence we can rewrite the above expression as:
	\begin{equation*}
    \begin{alignedat}{1}
		\prod_{i=1}^{s-1}\exp\left(\frac{1}{N}a_i\ell_{i,s}\right)\cup &\,\Phi\left(\prod_{i=1}^{s-1}\pi_i^{\ast}\left(\exp\left(-\frac{a_i}{N}\Psi(c_{\mathcal P})\right)\right)\cdot \right.\\
        &\qquad \left.\frac{1}{N}\sum_{j=0}^{N-1}T_j ^{\ast}\left(c_{N,\ast}q_{N,\ast}\left(p_N^{\ast}\left(\prod_{i=1}^{s-1}\pi_i^{\ast}\left(\exp\left(a_i \frac{j}{N}\mathsf{x}_i+a_i\Psi(c_{\mathcal P})\right)\right)f\right)\right)\right)\right)\\
		=\prod_{i=1}^{s-1}\exp\left(\frac{1}{N}a_i\ell_{i,s}\right)\cup&\,\Phi\left(\prod_{i=1}^{s-1}\pi_i^{\ast}\left(\exp\left(-\frac{a_i}{N}\Psi(c_{\mathcal P})\right)\right)\cdot\right.\\
		&\qquad \left.\mathfrak{N}_{N,\ast}\left(\Psi\left(\prod_{i=1}^{s-1}\pi_i^{\ast}\left(\exp\left(a_i\left(\frac{j}{N}\mathsf{x}_i+c_{\mathcal P})\right)\right)\right)f'\right)\right)\right).
	\end{alignedat}
    \end{equation*}
	By Theorem \ref{thm:polynomiality_lower} each term in the expansion of the factor $\mathfrak{N}_{N,\ast}(\ast)$ in the indeterminates $a_i$
	is polynomial in $N^{-1}$. The expansions of the other factors are clearly polynomial in $N^{-1}$, so it follows that every term in the expansion
	of the entire expression is polynomial in $N^{-1}$. Since $\deg_{m,\frac{j}{N},\mathsf{x}_s}\left(\frac{j}{N}\mathsf{x}_i+c_{\mathcal{P}}\right)=1$,
	the degree in $N^{-1}$ of the coefficient of the monomial $\prod_{i=1}^{s-1}a_i^{e_i}$ in the expansion of the factor $\mathfrak{N}_{N,\ast}(\ast)$
	is at most $\deg_{m,\mathsf{x}_s}(f')+\sum_{i=1}^{s-1}e_i$ when restricted to any maximal cone $\sigma_{\rho,q,r}$ of $\mathcal U\Sigma^{(s)}$ with $N\ge|q|$. However, notice that the entire expression inside
	$\Phi(\ast)$ is $\mathbb{Z}$-periodic, so if the restriction to $\sigma_{\rho,0,r}$ is polynomial in $N^{-1}$, the same is true for the restriction
	to any $\sigma_{\rho,q,r}$. As in the proof of Corollary
	\ref{cor:glob_poly} one sees that the coefficients of this polynomial
	glue to global piecewise polynomials in $\spp^{\ast}(\Sigma^{(s)})$. Taking into account the remaining factors, the degree in $N^{-1}$ of the
	coefficient of $\prod_{i=1}^{s-1}a_i^{e_i}$ in the entire expression above is at most $\deg_{m,\mathsf{x}_s}(f')+\sum_{i=1}^{s-1}e_i\le \deg_{\mathsf{x}_s}(f)+\sum_{i=1}^{s-1}e_i\le c$.
	Moreover, $[\gamma]_{(w)}$ is obtained from the coefficient of $\prod_{i=1}^{s-1}a_i^{e_i}$ in the above expression by
	first decomposing the piecewise polynomial
	part inside the $\mathfrak{N}_{N,\ast}(\ast)$ into its pure degree parts (with respect to $N^{-1}$)
	and then summing all monomials of degree exactly $w$.
	In particular, every weight $w$ part is again a polynomial in piecewise polynomials and $\ell_{i,j}$, hence tautological.
\end{proof}

\subsection{Explicit calculations in \texorpdfstring{$\R^*(\barA)$}{R(A)}}\label{sec:expl_calc}

When $s=1$, the formulas simplify greatly. In this section, we perform some computations that yield:
\begin{itemize}
    \item An explicit formula for $[\theta^k]^{(w)}$.
    \item An explicit formula for the pushforwards of powers of $\theta$.
    \item A proof of the identity \eqref{eq:evalMn1}.
\end{itemize}
Instead of the coordinates $(x_1,t)$ used in the preceding sections, a more convenient choice of coordinates
for these computations consists of $\mathsf{x}=\mathsf{x}_1$ and $\mathsf{y}=\mathsf{t}-\mathsf{x}_1$, viewed as elements of $\spp^{\ast}(\tilde{\Sigma}^{(1)})[m]$ that are constant in $m$.

For any $a, b, k \geq 0$, let
$$
\mathfrak N(a,b,k, N)=\sum_{i=0}^{N-1}N\mathsf{x}\mathsf{y}\frac{(N^2\theta-1/2((i+1)i\mathsf{x}+(N-i-1)(N-i)\mathsf{y}))^k}{((i+1)\mathsf{x}+(i+1-N)\mathsf{y})^{1-a}(-i\mathsf{x}+(N-i)\mathsf{y})^{1-b}}\,.
$$
From Lemma \ref{lem:combinat_lift} and the computations of the support functions in \S\ref{sec:supp_functions}, it is clear that
\begin{equation}\label{eq:multN-poly}
\mathfrak N_N(\sum c_{a,b,k}\mathsf{x}^a\mathsf{y}^b\theta^k) = \sum c_{a,b,k}\mathfrak N(a,b,k,N)
\end{equation}
lifts the multiplication by $N$ maps to $\mathbb{Q}[\mathsf{x},\mathsf{y},\theta]$, i.e., if $p(x,y,\theta)\in \spp^{\ast}(\Sigma^{(1)})[m]^{\mathsf{gl}}[\theta]$,
then $\mathfrak N_N(p)$ is polynomial in $N$ and the following diagram commutes:
$$
\begin{tikzcd}
	{\spp^{\ast}(\Sigma^{(1)})[m]^{\mathsf{gl}}[\theta]} & {\spp^{\ast}(\Sigma^{(1)})[m]^{\mathsf{gl}}[\theta]} \\
	{\CH^*(\barA)} & {\CH^*(\barA)}
	\arrow["{\mathfrak N_N}", from=1-1, to=1-2]
	\arrow[from=1-1, to=2-1]
	\arrow[from=1-2, to=2-2]
	\arrow["{[N]^*}", from=2-1, to=2-2]
\end{tikzcd}
$$
The advantage of lifting the computation to $\mathbb{Q}[x,y,\theta]$ is that we can formally invert polynomials. This is useful because a straightforward calculation using \eqref{eq:multN-poly} shows that for any $p \in \spp^{\ast}(\Sigma^{(1)})$,
\begin{align*}
\mathfrak N_N(\mathsf{t} \cdot p(\mathsf{x},\mathsf{y},\theta) ) &= \mathsf{t}\cdot\mathfrak N_N(p(\mathsf{x},\mathsf{y},\theta))\,,\\
\mathfrak N_N((2(\mathsf{x}+\mathsf{y})\theta-\mathsf{x}\mathsf{y})\cdot p(\mathsf{x},\mathsf{y},\theta)) &= N^2( 2(\mathsf{x}+\mathsf{y})\theta-\mathsf{x}\mathsf{y}) \cdot \mathfrak N_N(p(\mathsf{x},\mathsf{y},\theta))\,.
\end{align*}

Using these two identities, we have
$$
2\mathsf{t}[\theta^{k+1}]^{(2w+2)} =[2\mathsf{t}\theta^{k+1}]^{(2w+2)} = [\mathsf{x}\mathsf{y}\theta^k]^{(2w+2)}+(2(\mathsf{x}+\mathsf{y})\theta-\mathsf{x}\mathsf{y})\cdot [\theta^{k}]^{(2w)}\,,
$$
so, iterating, we get
\begin{align}\label{eq:theta-recursion}
[\theta^{k}]^{(2w)} =\frac{(2(\mathsf{x}+\mathsf{y})\theta - \mathsf{x}\mathsf{y})^{w}}{2^{w}\mathsf{t}^{w}}[\theta^{k-w}]^{(0)}+\sum_{v=0}^{w-1}\frac{(2(\mathsf{x}+\mathsf{y})\theta - \mathsf{x}\mathsf{y})^v}{2^{v+1}\mathsf{t}^{v+1}}[\mathsf{x}\mathsf{y}\theta^{k-v-1}]^{(2w-2v)}\, ,
\end{align}
for $k\ge w \geq 1$, and $[\theta^k]^{(0)} = \frac{1}{2(\mathsf{x}+\mathsf{y})}[\mathsf{x}\mathsf{y}\theta^{k-1}]^{(0)}$ for $k \geq 1$.

\begin{lemma}\label{lem:weight-pieces-xytheta}
For every $k\geq 0$ and every $0\leq m\leq k$ one has
\begin{equation}\label{eq:weight-piece-q-basis}
 [\mathsf{x}\mathsf{y}\theta^k]^{(2m+2)}
 =
 \frac{\mathsf{x}\mathsf{y}\,\mathsf{t}^{k-2m-1}}{2^k}
 \sum_{\ell=0}^{m}
 \Gamma_{k,m,\ell}\,
 (2(\mathsf{x}+\mathsf{y})\theta-\mathsf{x}\mathsf{y})^{m-\ell}(\mathsf{x}^{2\ell+1}+\mathsf{y}^{2\ell+1}),
\end{equation}
where
$$
 \Gamma_{k,m,\ell}
 =
 \sum_{n=0}^{k-m}
 (-1)^{n+\ell}
 \binom{k}{n+\ell}
 \binom{k-n-\ell}{m-\ell}
 \frac{(2n+2\ell)!}{(2\ell+1)!(2n)!}
 \frac{B_{2n}\!\left(\frac12\right)}{4^{k-m-n}}.
$$
Moreover, $[\mathsf{x}\mathsf{y}\theta^k]^{(w)}=0$ for $w=0$, $w$ odd or $w >2k+2$, for any $k \geq 0$.
\end{lemma}
\begin{proof}
By \eqref{eq:multN-poly},
\[
 \mathfrak N_N(\mathsf{x}\mathsf{y}\theta^k)=N\mathsf{x}\mathsf{y}\sum_{i=0}^{N-1}(N^2\theta-1/2((i+1)i\mathsf{x}+(N-i-1)(N-i)\mathsf{y}))^k
\]
In particular, the coefficient of $N^0$ vanishes. Set
$$
f_i:=\mathsf{t}\left(i+\frac12\right)-N\mathsf{y}\, ,\quad \mathsf{u} = 2\mathsf{t}\theta-\mathsf{x}\mathsf{y}\, .
$$
Then
\[
N^2\theta-1/2((i+1)i\mathsf{x}+(N-i-1)(N-i)\mathsf{y})=\frac{N^2\mathsf{u}+\frac{\mathsf{t}^2}{4}-f_i^2}{2\mathsf{t}}.
\]
Hence
\begin{equation}\label{eq:N-xytheta-expanded}
 \mathfrak N_N(\mathsf{x}\mathsf{y}\theta^k)
 =
 \frac{N\mathsf{x}\mathsf{y}}{2^k\mathsf{t}^k}
 \sum_{p+q+r=k}
 (-1)^r\binom{k}{p,q,r}
 \frac{N^{2p}}{4^q}\mathsf{u}^p\mathsf{t}^{2q}
 \sum_{i=0}^{N-1}f_i^{2r}.
\end{equation}
Using Faulhaber's formula,
$$
\sum_{i=0}^{N-1}(i+a)^m = \frac{B_{m+1}(N+a)-B_{m+1}(a)}{m+1},
$$
where $B_m(x)$ are Bernoulli polynomials (see Appendix \ref{sec:B}), we get
\[
\sum_{i=0}^{N-1}f_i^{2r}=\frac{\mathsf{t}^{2r}}{2r+1}
\left[B_{2r+1}\!\Bigl(\frac12+\frac{N\mathsf{x}}{\mathsf{t}}\Bigr)-B_{2r+1}\!\Bigl(\frac12-\frac{N\mathsf{y}}{\mathsf{t}}\Bigr)\right].
\]
Expanding the Bernoulli polynomials at $1/2$ gives
\begin{equation}\label{eq:Bernoulli-half-expansion}
 B_{2r+1}\!\Bigl(\frac12+z\Bigr)-B_{2r+1}\!\Bigl(\frac12-w\Bigr)
 =
 \sum_{\ell=0}^r
 \binom{2r+1}{2\ell+1}
 B_{2r-2\ell}\!\Bigl(\frac12\Bigr)
 (z^{2\ell+1}+w^{2\ell+1}).
\end{equation}
Substituting $z= N\mathsf{x}/\mathsf{t}$ and $w=N\mathsf{y}/\mathsf{t}$ into \eqref{eq:Bernoulli-half-expansion}, and then into \eqref{eq:N-xytheta-expanded}, we obtain
\begin{align*}
\mathfrak N_N(\mathsf{x}\mathsf{y}\theta^k)
&=
\frac{\mathsf{x}\mathsf{y}}{2^k\mathsf{t}^k}
\sum_{p+q+\ell+n=k}
(-1)^{\ell+n}\binom{k}{p,q,\ell+n}
\binom{2\ell+2n+1}{2\ell+1}
\frac{B_{2n}(1/2)}{4^q(2\ell+2n+1)}\\
&\hspace{3cm}\cdot N^{2p+2\ell+2}\mathsf{u}^p\mathsf{t}^{2q+2n-1}(\mathsf{x}^{2\ell+1}+\mathsf{y}^{2\ell+1}).
\end{align*}
Fix $m\in\{0,\dots,k\}$. The weight $2m+2$ part comes from the terms with $p+\ell=m$. Writing $p=m-\ell$ and $q=k-m-n$, we get
\begin{align*}
[\mathsf{x}\mathsf{y}\theta^k]^{(2m+2)}&=\frac{\mathsf{x}\mathsf{y}\,\mathsf{t}^{k-2m-1}}{2^k}\sum_{\ell=0}^{m}\sum_{n=0}^{k-m}(-1)^{n+\ell}\binom{k}{m-\ell,k-m-n,n+\ell}\\
&\hspace{2cm}\cdot\frac{(2n+2\ell)!}{(2\ell+1)!(2n)!}\frac{B_{2n}(1/2)}{4^{k-m-n}}\mathsf{u}^{m-\ell}(\mathsf{x}^{2\ell+1}+\mathsf{y}^{2\ell+1}).
\end{align*}
Finally, the multinomial coefficient satisfies
\[
\binom{k}{m-\ell,k-m-n,n+\ell}
=
\binom{k}{n+\ell}\binom{k-n-\ell}{m-\ell},
\]
so the inner sum is precisely $\Gamma_{k,m,\ell}$. This proves \eqref{eq:weight-piece-q-basis}.
\end{proof}

\begin{proposition}\label{pro:thetawt}
Let $k\geq 1$. The weight parts of $\theta^k$ are given as follows:
\begin{itemize}
    \item[(a)] For $w=0$, $w$ odd or $w >2k$, $[\theta^k]^{(w)}=0$.
    \item[(b)] For $1 \leq w <k$,
    $$
    [\theta^k]^{(2w)}
    =
    \frac{\mathsf{x}\mathsf{y}\,(\mathsf{x}+\mathsf{y})^{k-2w-1}}{2^k}
    \sum_{v=0}^{w-1}
    \Lambda_{k,w,v}\,
     (2(\mathsf{x}+\mathsf{y})\theta - \mathsf{x}\mathsf{y})^{w-1-v}(\mathsf{x}^{2v+1}+\mathsf{y}^{2v+1}),
    $$
    where
    $$
    \Lambda_{k,w,v}
    :=
    \sum_{r=0}^{w-1-v}\Gamma_{k-r-1,w-r-1,v}\, .
    $$
    \item[(c)] For the top-weight part, we have
    $$
    \frac{[\theta^k]^{(2k)}}{k!} = \St^k_g \in \CH^k(\barA)
    $$
    where $\St_g^k$ is the codimension $k$ tautological class defined in \eqref{eq:St}. In particular, $[N]^*\theta=N^2\theta$ and for any $k\ge 2$, there exists $\vartheta^k(N) \in \R^{k-2}(C)[N]$ such that $[N]^*(\theta^k) = N^{2k} \theta^k + j_*\vartheta^k(N)$.
\end{itemize}
\end{proposition}
\begin{proof}
Parts (a) and (b) are a direct consequence of \eqref{eq:theta-recursion} and Lemma \ref{lem:weight-pieces-xytheta}. For (c), it is enough to show that
\begin{equation}\label{eq:goalll}
    \sum_{k \geq 0} \frac{[\theta^k]^{(2k)}}{k!} = \exp(\theta)\left(1-
    \sum_{k_1,k_2\ge 1} \frac{(-1)^{k_1+k_2}(2k_1)!(2k_2)!}{2^{k_1+k_2}(2k_1+2k_2)!}
    \frac{\mathsf{x}^{k_1}\mathsf{y}^{k_2}}{k_1!k_2!}\right)\, .
\end{equation}
Using \eqref{eq:theta-recursion} and Lemma \ref{lem:weight-pieces-xytheta} we can write
$$
[\theta^k]^{(2k)} = \frac{(2(\mathsf{x}+\mathsf{y})\theta -\mathsf{x}\mathsf{y})^k}{2^k(\mathsf{x}+\mathsf{y})^k} + \frac{\mathsf{x}\mathsf{y}}{2^k(\mathsf{x}+\mathsf{y})^{k+1}}\sum_{v=0}^{k-1}\Lambda_{k,k,v}(2(\mathsf{x}+\mathsf{y})\theta-\mathsf{x}\mathsf{y})^{k-1-v}(\mathsf{x}^{2v+1}+\mathsf{y}^{2v+1})\, ,
$$
where
$$
\Lambda_{k,k,v} = \sum_{r=0}^{k-1-v}\Gamma_{k-r-1, k-r-1, v} = (-1)^v\binom{k}{v+1}\frac{1}{2v+1}
$$
by the hockey stick identity. Consequently, summing over $k$ and using
\[
\frac{1}{k!}\binom{k}{v+1}
=
\frac{1}{(v+1)!(k-v-1)!},
\]
we find
$$
\sum_{k\geq 0}\frac{[\theta^k]^{(2k)}}{k!}
=
\exp\!\left(\frac{2(\mathsf{x}+\mathsf{y})\theta-\mathsf{x}\mathsf{y}}{2(\mathsf{x}+\mathsf{y})}\right)\left(1+\sum_{v\geq 0}\frac{(-1)^v\mathsf{x}\mathsf{y}(\mathsf{x}^{2v+1}+\mathsf{y}^{2v+1})}{2^{v+1}(v+1)!(2v+1)(\mathsf{x}+\mathsf{y})^{v+2}}\right)
$$
Since
$$
\int_0^1((\mathsf{x}+\mathsf{y})u-\mathsf{y})^{2v}\,du=\frac{\mathsf{x}^{2v+1}+\mathsf{y}^{2v+1}}{(2v+1)(\mathsf{x}+\mathsf{y})},
$$
we obtain by direct expansion that
\begin{align}\label{eq:thetamidddle}
\sum_{k\geq 0}\frac{[\theta^k]^{(2k)}}{k!} &=\exp(\theta)\exp\!\left(-\frac{\mathsf{x}\mathsf{y}}{2(\mathsf{x}+\mathsf{y})}\right)\left(\frac{\mathsf{x}\mathsf{y}}{\mathsf{x}+\mathsf{y}}
\int_0^1
\exp\left(-\frac{((\mathsf{x}+\mathsf{y})u-\mathsf{y})^2}{2(\mathsf{x}+\mathsf{y})}\right)\,du\right.\notag\\
&\hspace{5cm}\left.+ \frac{\mathsf{y}}{\mathsf{x}+\mathsf{y}}
\exp\left(-\frac{\mathsf{x}^2}{2(\mathsf{x}+\mathsf{y})}\right)
+\frac{\mathsf{x}}{\mathsf{x}+\mathsf{y}}\exp\left(-\frac{\mathsf{y}^2}{2(\mathsf{x}+\mathsf{y})}\right)\right)\notag\\
&=\exp(\theta)\left(\frac{\mathsf{y}\exp(-\mathsf{x}/2) + \mathsf{x}\exp(-\mathsf{y}/2)}{\mathsf{x}+\mathsf{y}} + \frac{\mathsf{x}\mathsf{y}}{\mathsf{x}+\mathsf{y}}\int_0^1 \exp\left(-\frac{\mathsf{x}u^2 + \mathsf{y}(1-u)^2}{2}\right)\, du\right)
\end{align}
We have
\begin{align*}
\frac{1}{\mathsf{x}+\mathsf{y}}\frac{d^2}{du^2}\exp\left(-\frac{\mathsf{x}u^2+\mathsf{y}(1-u)^2}{2}\right)=\, &-\frac{\mathsf{x}\mathsf{y}}{\mathsf{x}+\mathsf{y}}\exp\left(-\frac{\mathsf{x}u^2+\mathsf{y}(1-u)^2}{2}\right)\\
&-\left(1-\mathsf{x}u^2-\mathsf{y}(1-u)^2\right)\exp\left(-\frac{\mathsf{x}u^2+\mathsf{y}(1-u)^2}{2}\right)
\end{align*}
and integrating both sides from $0$ to $1$ implies that $\sum_{k\geq 0}\frac{[\theta^k]^{(2k)}}{k!}$ equals
\begin{equation}\label{eq:mddlstep}
\exp(\theta)\left(\exp(-\mathsf{x}/2) + \exp(-\mathsf{y}/2) - \int_0^1\left(1-\mathsf{x}u^2-\mathsf{y}(1-u)^2\right)\exp\left(-\frac{\mathsf{x}u^2+\mathsf{y}(1-u)^2}{2}\right)\,du\right)\, .
\end{equation}
After expanding the exponential and using the beta integral $\int_0^1u^{2k_1}(1-u)^{2k_2}\,du=\frac{(2k_1)!(2k_2)!}{(2k_1+2k_2+1)!}$, we see that
$$
\int_0^1\left(1-\mathsf{x}u^2-\mathsf{y}(1-u)^2\right)\exp\left(-\frac{\mathsf{x}u^2+\mathsf{y}(1-u)^2}{2}\right)\,du=\sum_{k_1,k_2\geq0}\frac{(-1)^{k_1+k_2}(2k_1)!(2k_2)!}{2^{k_1+k_2}(2k_1+2k_2)!}\frac{\mathsf{x}^{k_1}\mathsf{y}^{k_2}}{k_1!k_2!}
$$
Then, the terms with $k_1=0$ or $k_2=0$ cancel with the first two terms in the second factor of \eqref{eq:mddlstep}, proving \eqref{eq:goalll}. The last observation follows since all the terms in the second factor of \eqref{eq:mddlstep} except for $1$ are divisible by $\mathsf{x}\mathsf{y} = \frac{1}{2}j_*(1)$.
\end{proof}

\begin{proposition}\label{pro:vanish}
    Let $1 \leq i \leq s$, and consider the projection\footnote{If $s=1$, we take $\barA^{(0)}:=B$.} $p_i : \barA^{(s)} \to \barA^{(s-1)}$ forgetting the $i$-th component. If $\alpha$ has finite weight for $[1 \times \cdots \times N\times \cdots \times 1]^*$ or $[1 \times \cdots \times N\times \cdots \times 1]_*$, then we have
    $$
    p_{i,*} \big([\alpha]^{i,(w)}\big) = \begin{cases}
        p_{i,*} (\alpha)&\text{ if }w = 2g\, ,\\
        0&\text{ otherwise.}
    \end{cases}\, ,\quad p_{i,*} \big([\alpha]_{i,(w)}\big) = \begin{cases}
        p_{i,*} (\alpha)&\text{ if }w = 2g\, ,\\
        0&\text{ otherwise.}
    \end{cases}
    $$
    respectively.
\end{proposition}

\begin{proof}
    To simplify notation, we write
    \begin{equation*}
    N_i := 1 \times \cdots \times N \times \cdots \times 1\,.
    \end{equation*}
    Let $\barA^{(s)}_{N_i}$ be a model of $\barA^{s}$ as in Definition \ref{def:n}, fitting
    into the following commutative diagram:
    \[
    \begin{tikzcd}
        \barA^{(s)}_{N_i} \ar[r,"N_i"] \ar[d,"b"] & \barA^{(s)} \ar[d,"p_i"] \\
        \barA^{(s)} \ar[r,"p_i"] & \barA^{(s-1)} .
    \end{tikzcd}
    \]
    Then
    \[
    p_{i,*} \left(\sum_{w\geq 0} N^w[\alpha]^{i,(w)}\right) = p_{i,*} \circ [N_i]^*(\alpha)
    =
    p_{i,*} b_* N_i^*(\alpha)
    =
    p_{i,*} N_{i,*} N_i^*(\alpha)
    =
    N^{2g} p_{i,*}(\alpha),
    \]
    where the second equality follows from the commutativity of the diagram,
    and the third follows from the projection formula and the fact that $N_i$
    is generically finite of degree $N^{2g}$. Extracting coefficients in $N$ yields the first result.
    For the second formula, we argue similarly, using that
    $$
    p_{i,*} \left(\sum_{w\geq 0} N^{-w}[\alpha]_{i,(w)}\right) = p_{i,*} \circ [N_i]_*(\alpha)
    =
    N^{-2g}p_{i,*} N_{i,*}b^*(\alpha)
    =
    N^{-2g}p_{i,*} b_{*} b^*(\alpha)
    =
    N^{-2g} p_{i,*}(\alpha).
    $$
\end{proof}

This result allows us to compute pushforwards using weights. We showcase this for $s=1$:

\begin{corollary}\label{cor:pushThet}
    $$
    \pi_*(\exp(\theta)\cup\Td^\vee(\cR_\pi)^{-1}) = \exp(D/8)\, , \quad \pi_*(\exp(\theta)) = \sum_{k \geq 0} \frac{(D/8)^{k}}{k!(2k+1)}
    $$
\end{corollary}
\begin{proof}
    By Proposition \ref{pro:vanish}, we have
    \[
    \pi_*\left(\sum_{k\geq0}\frac{[\theta^k]^{(2k)}}{k!}\right)=\pi_*\left(\frac{\theta^g}{g!}\right)=1\, .
    \]
    Let $x=\ell^C$, $y=-\ell^C-2\theta_2^C$, $t=-2\theta_2^C$,
    so that \(x+y=t\), and consider the power series
    \[
    A(x,y)=\sum_{k_1,k_2\geq1}\frac{(-1)^{k_1+k_2}(2k_1)!(2k_2)!}{2^{k_1+k_2}(2k_1+2k_2)!\,k_1!k_2!}x^{k_1-1}y^{k_2-1}\, ,\quad
    R(x,y)=\sum_{m\geq1}\frac{B_{2m}}{(2m)!}\frac{x^{2m-1}+y^{2m-1}}{x+y}.
    \]
    Using Proposition \ref{pro:thetawt}~(c), the expression for the
    Todd class in \eqref{eq:todd}, and Lemma \ref{lem:jstar}~(a)\footnote{While Lemma \ref{lem:jstar} appears later, point (a) is just \cite[Corollary 4.6]{EGH10}}, we obtain
    \begin{align*}
        \pi_*\left(\exp(\theta)\cup\Td^\vee(\cR_\pi)^{-1}\right)
        ={}&
        1+
        \pi_*\left(
        \exp(\theta)
        -
        \sum_{k\geq0}\frac{[\theta^k]^{(2k)}}{k!}
        \right) +
        \pi_*\left(
        \exp(\theta)\cup
        \left(\Td^\vee(\cR_\pi)^{-1}-1\right)
        \right) \notag\\
        ={}&
        1+\frac12\pi_*j_*\left(
        \exp(\theta_1^C+x/2)\cup
        \bigl(A(x,y)+R(x,y)\bigr)
        \right) \notag\\
        ={}&
        1+\frac12\iota_*\pi_{C*}\left(
        \exp(\theta_1^C+x/2)\cup
        \bigl(A(x,y)+R(x,y)\bigr)
        \right).
    \end{align*}
    By \cite[Theorem 9.10]{BMP}, pushforward along the abelian scheme
    \(\pi_C\colon C\to D\) is determined by
    \[
    \pi_{C*}\left(\exp(\theta_1^C+ux)\right)
    =
    \exp(u^2t/2),
    \]
    as a polynomial identity in \(u\). Therefore, for every polynomial
    \(P(x,t)\),
    \[
    \pi_{C*}\left(
    \exp(\theta_1^C+x/2)P(x,t)
    \right)
    =
    \left.
    P(\partial_u,t)\exp(u^2t/2)
    \right|_{u=1/2}.
    \]
    In particular, after replacing \(y\) by \(t-x\), define
    \[
    \mathcal L_t(P)
    :=
    \left.
    P(\partial_u,t)\exp(u^2t/2)
    \right|_{u=1/2}.
    \]
    Using the beta integral
    \[
    \frac{(2k_1)!(2k_2)!}{(2k_1+2k_2)!}
    =
    (2k_1+2k_2+1)
    \int_0^1v^{2k_1}(1-v)^{2k_2}\,dv
    \]
    gives
    $$
    \mathcal L_t\bigl(A(x,t-x)\bigr)=
    \frac1t\int_0^1
    \left(\exp\left(\frac{tv^2}{8}\right)-1\right)\,dv=\frac1t\int_0^1
    \sum_{r\geq1}\frac{t^rv^{2r}}{8^r r!}\,dv=
    \sum_{r\geq1}
    \frac{t^{r-1}}{8^r r!(2r+1)}\,.
    $$
    We next evaluate \(R\). We have
    \[
    \mathcal L_t \left(\frac{x^{2m-1}+y^{2m-1}}{x+y}\right)
    =
    \frac{2}{t}\left(\frac{d}{du}\right)^{2m-1}\exp\left(\frac t2\left(u+\frac12\right)^2\right)\Bigg|_{u=0}\,.
    \]
    By the Euler-Maclaurin identity
    $$
    \int_{-1}^0f(u)\, du = \frac{f(-1)+f(0)}{2} - \sum_{m\geq 1}\frac{B_{2m}}{(2m)!}\left(f^{(2m-1)}(0) - f^{(2m-1)}(-1)\right)\,,
    $$
    we obtain
    $$
    \mathcal L_t\bigl(R(x,t-x)\bigr)=\frac1t\left(\exp(t/8)-\int_{-1}^0\exp\left(\frac t2\left(u+\frac12\right)^2\right)\,du\right)=\sum_{r\geq1}\frac{t^{r-1}}{8^r r!}\frac{2r}{2r+1}\,.
    $$
    We conclude the proof by using $\frac{1}{2}\iota_* t^k = D^{k+1}$.
\end{proof}

\section{The class of the unit section}\label{sec:unit}

This section is devoted to the derivation of two closed formulas for the unit section $e : B \to \barA$ (Theorems \ref{thm:unit} and \ref{thm:altzero}). We also discuss the work of Grushevsky and Zakharov \cite{GZ}.

\subsection{Boundary class}

We follow the notation for the boundary strata of $\barA \to B$ introduced in \S\ref{sec:mum}; see also Figure \ref{fig:thm_bdy}.
Let $\epsilon :E \to \barA^{(2)}$ be the normalization of the exceptional locus of $f: \barA^{(2)}\to \barA^2$. It is isomorphic to a $\PP^1$-bundle over $C \times_D C$.
Recall that $\sR^*(C)\subseteq \CH^*(C)$ is the subring generated by $\theta_1^C,\theta_2^C$ and $\ell^C$ (Example \ref{expl:pp_s_1}).
We define $\sR^{\ast}(C\times_D C)$ to be the subring of $\CH^{\ast}(C\times_D C)$ generated by the pullbacks of classes in $\sR^{\ast}(C)$
from the two factors. Similarly, we define
$\sR^{\ast}(E)$ to be the subring of $\CH^{\ast}(E)$ generated by the pullbacks of classes in $\sR^{\ast}(C\times_D C)$ and the hyperplane class of the $\mathbb{P}^{1}$-bundle $E\to C\times_D C$.

\begin{lemma}\label{lem:Fex}
	For $\beta\in \sR^*(E)$, let $(f\circ \epsilon)_*(\beta) \in \CH^*(\barA^2)$ be the associated correspondence. Then the image of $(f\circ \epsilon)_*(\beta) : \sR^*(\barA)\to \CH^*(\barA)$ lies in the image of $j_*\sR^*(C)$.
\end{lemma}

\begin{proof}
    Consider the following commutative diagram:
    \[
    \begin{tikzcd}
        E \ar[r,"\epsilon"]\ar[d] & \barA^{(2)}\ar[d,"f"]\\
        C\times_{D} C\ar[r,"g"] & \barA^2.
    \end{tikzcd}
    \]
    By the projective bundle formula, there exists $\beta'\in \sR^*(C\times_{D} C)$ such that $(f\circ \epsilon)_*(\beta) = g_*(\beta')$. Therefore, the result follows from the projection formula.
\end{proof}

\begin{definition}\label{def:equiv}
    For $\alpha_1,\alpha_2\in \CH^k(\barA)$, we say $\alpha_1 \equiv \alpha_2$ if there exists $\phi \in \sR^{k-2}(C)$ such that $\alpha_1 = \alpha_2 + j_*(\phi)$.
\end{definition}

Lemma \ref{lem:Fex} implies that the Fourier transform sends any class supported on $E$ to zero under this equivalence relation. Informally,
we may ignore any ``correction term'' supported on $E$.

\begin{proposition}\label{pro:jclosed}
    Let $\ell^C\in \CH^1(C)$ be the first Chern class of the Poincar\'e line bundle.
    \begin{enumerate}[label = (\alph*)]
        \item If $k+m<2g-2$, then $\fm_k (j_*((\ell^C)^m)) \equiv 0$.
        \item If $k+m=2g-2$,
        \[
        \fm_k \left(j_*\frac{(\ell^C)^m}{m!}\right) \equiv i_* \left[ \frac{(-\theta_1^C + u\ell^C - u^2\theta_2^C)^{g-1}}{(g-1)!}\right]_{\text{coeff } u^m}\,.
        \]
    \end{enumerate}
\end{proposition}

\begin{proof}
    Let $\X_{g-1}^n$ be the $n$-fold relative fiber product of $\X_{g-1} \to \A_{g-1}$. The normalization of $\barA|_{D}$ can be identified with $\PP(\P \oplus \CO) \longrightarrow \X_{g-1}^2$. Similarly, the normalization of $\barA^{(2)}|_{C \times_{D} \barA}$ can be identified with
    \[
    E \quad \sqcup \quad \PP(\P_{23} \oplus \CO) \longrightarrow \X_{g-1}^3.
    \]
    Consider the diagram
    \[
    \begin{tikzcd}
    \PP(\P_{23} \oplus \CO) \ar[r] \ar[d]
        & \PP(\P \oplus \CO) \ar[d,"q"] \ar[r,"i"]
        & \barA \\
    \X_{g-1}^3 \ar[r,"p_{2,3}"]
        & \X_{g-1}^2, &
    \end{tikzcd}
    \]
    where the square is Cartesian.

    The restriction of $\P$ to $\PP(\P_{23} \oplus \CO)$ is $\P_{13} \otimes \CO(1)$, see \cite[Theorem 5.1]{vdgK}. Write $h = c_1(\CO(1))$. We then obtain
    \[
    \fm_k \left(j_*\frac{(\ell^C)^m}{m!}\right) \equiv \frac{1}{2} i_* q^*(p_{2,3})_* \left(
    \frac{(\ell_{1,3} +h)^k}{k!} \cup \frac{\ell_{1,2}^m}{m!}
    \right).
    \]

    We apply the closed formula for the proper pushforward of monomials in $\ell_{1,2}$ and $\ell_{1,3}$ proved in \cite[Theorem~9.10]{BMP}. If the degree of the monomial is less than $2g-2$, then $(p_{2,3})_*$ vanishes, which proves part~(a).
    Now consider the case $k + m = 2g - 2$. After expanding $(\ell_{1,3} + h)^k$, any monomial containing a positive power of $h$ is annihilated by $(p_{2,3})_*$. Therefore, the claim follows again from \cite[Theorem~9.10]{BMP}.
\end{proof}

\begin{remark}
    The right-hand side of Proposition \ref{pro:jclosed} differs from \cite[eq. (114)]{BMP} by the global sign $(-1)^{g-1}$. This discrepancy arises because \cite{BMP} uses the opposite convention for the Poincar\'e line bundle from the one adopted in this paper, see \S \ref{subs:extP}.
\end{remark}

\subsection{The key step}

We state a straightforward consequence of the proof of Theorem \ref{thm:wt}:
\begin{corollary}\label{cor:Fwt}
    Let $\epsilon : E \to \barA^{(2)}$ denote the normalization of the exceptional locus.
    For each
    $k \geq 0$, there exists a class $\rho^{k}(N) \in \sR^{k-2}(E)[N]$ polynomial in $N$ of degree at most $k$, such that
    \[
    [1 \times N]^* (\fm_k) = N^k\fm_k + \epsilon_* \rho^k(N)\,.
    \]
\end{corollary}

The following is the key computation in this section.

\begin{proposition}\label{pro:Fleading}
    Under the equivalence relation in Definition \ref{def:equiv}, we have
    \[
    \fm_{2g}(1) \equiv \frac{(-\theta)^g}{g!}\,.
    \]
\end{proposition}

\begin{proof}
    Taking the codimension $g$ component of \eqref{eq:Ftheta}, we obtain
	\begin{equation}\label{eq:Ftheta2}
    \Big[\fm(\exp(-\theta) \cup \Td^\vee(\cR_\pi)^{-1})\Big]_{\codim = g} = \frac{(-\theta +D/8)^g}{g!}\,.
    \end{equation}
    We decompose the left-hand side according to the codimension of $\fm$:
    \[
    \Big[\fm\big(\exp(-\theta)\cup \Td^\vee(\cR_\pi)^{-1}\big)\Big]_{\codim = g} = \fm_0(\cdots) + \cdots + \fm_{2g}(1)\,.
    \]
    By codimension considerations, the terms $\fm_{>2g}$ do not contribute to the codimension $g$ part. We take weight $2g$ on both sides of \eqref{eq:Ftheta2}.

    From the explicit shape of the Fourier transform in Lemma \ref{lem:Fchow}, by Lemma \ref{lem:Npush} and Theorem \ref{thm:wt}, each term $\fm_i(...)$ has weight $\leq i$. For the term $\fm_{2g}(1)$, by Corollary~\ref{cor:Fwt},
    \[
    [\fm_{2g}(1)]^{(2g)} \equiv \fm_{2g}(1)\,.
    \]
    On the right-hand side, since $D$ is a class pulled back from the base $B$, Proposition~\ref{pro:thetawt} shows that
    \[
    \left[\frac{(-\theta + D/8)^g}{g!}\right]^{(2g)} \equiv \left[\frac{(-\theta)^g}{g!}\right]^{(2g)} \equiv \frac{(-\theta)^g}{g!}\,.
    \]
    This completes the proof.
\end{proof}

Consider the correspondence
\[
\fm^\circ : = (-1)^g \cdot \ch(\P^\vee) \cup \Td^\vee(p_1^*\cR_\pi + p_2^*\cR_\pi-\cR_{\pi^{(2)}})\,.
\]
For $\alpha\in\CH^*(\barA)$, Lemma \ref{lem:Fex} implies that
\[
\fm_k (\alpha) \equiv (-1)^{g+k} \cdot \fm_k^\circ(\alpha)
\]
because the Todd class correction is supported on $E$.

For convenience, we introduce the class
\[
\sfS_g^\circ = \left[ \exp(\theta) \cup \left(1 - \frac{1}{2}i_*\sum_{k\geq 0}\frac{|B_{2k+2}|}{(k+1)!}\frac{(-\theta_2^C)^k}{2}\right)\right]_{\operatorname{codim}=g}\,.
\]
\begin{corollary}\label{cor:Fleading2}
    $\big[\fm^\circ(\Td^\vee(\cR_\pi)^{-1})\big]_{\codim =g} \equiv \sfS_g^\circ$.
\end{corollary}
\begin{proof}
    For $0 \le k \le 2g$, consider the individual components
    \[
    \fm^\circ_{2g-k}\Big([\Td^\vee(\cR_\pi)^{-1}]_{\codim = k}\Big).
    \]
    When $k = 0$, this contribution is computed in Proposition~\ref{pro:Fleading}. For $k > 0$, equation~\eqref{eq:todd} implies that the Fourier transform is nonzero only when $k = 2r$ is even. Thus, for $1 \le r \le g$, we are led to compute
    \[
    \frac{(-1)^{r+1}|B_{2r}|}{2(2r)!} \cdot \fm^\circ_{2g-2r}\Bigg(j_*\!\left(
    \frac{(\ell^C)^{2r-1} - (\ell^C + 2\theta_2^C)^{2r-1}}{-2\theta_2^C}\right)\Bigg)\,.
    \]
    Since $\theta_2^C$ is pulled back from the base, by the projection formula, its powers factor out of $\fm^\circ$. Therefore, by Proposition~\ref{pro:jclosed} (a), the only nontrivial contribution is
    \[
    \frac{(-1)^{r+1} |B_{2r}|}{4r}\cdot \fm^\circ_{2g-2r}\left(j_*\frac{(\ell^C)^{2r-2}}{(2r-2)!}\right).
    \]
    Combining this with Proposition \ref{pro:jclosed} (b), we obtain
    \[
    \sfS_g^\circ
    =
    \frac{\theta^g}{g!}
    +
    \frac{1}{2}\, i_*
    \sum_{r=1}^g
    \frac{(-1)^{g+r+1}|B_{2r}|}{2r}
    \Bigg[\frac{(-\theta_1^C + u\ell^C -u^2\theta_2^C)^{g-1}}{(g-1)!}\Bigg]_{\text{coeff } u^{2r-2}}\,.
    \]
    It follows from \cite[Proposition 4.5]{EGH10} that $
    i^*(\theta) = \theta_1^C+U/2$. Since $U=[s_0]+[\sbar_{\infty }]$ and $\ell^{C}=[s_0]-[\sbar_{\infty }]$ are pushed forward from $C$, we obtain that $i_{\ast}(\theta_1^{C}\cdot \alpha)\equiv i_{\ast}(i^{\ast}(\theta)\cdot \alpha)$ and $i_{\ast}(\ell^{C}\cdot \alpha)\equiv 0$ for all $\alpha\in \sR^{\ast}(Y)$.
    Therefore, the formula simplifies to
    \begin{align*}
    \mathsf{S}_g^\circ &\equiv \frac{\theta^g}{g!}
    +
    \frac{1}{2}\, i_*
    \sum_{r=1}^g
    \frac{(-1)^{g+r+1}|B_{2r}|}{2r}
    \Bigg[\frac{(-i^*\theta -u^2\theta_2^C)^{g-1}}{(g-1)!}\Bigg]_{\text{coeff } u^{2r-2}}\\
    &=\left[ \exp(\theta)\cup\left(1 - \frac{1}{2}i_*\sum_{k\geq 0}\frac{|B_{2k+2}|}{2k+2}\frac{(-2\theta_2^C)^k}{2^kk!}\right)\right]_{\operatorname{codim}=g}\,.
    \end{align*}
    This completes the proof.
\end{proof}

We prove the main result of this section.

\begin{proposition}\label{pro:ecorrect}
    There exists $\alpha\in \sR^{g-2}(C)$ such that $[e] = \sfS^g_g + j_*(\alpha)$.
\end{proposition}

\begin{proof}
    Since $\sfS^g_g \equiv \sfS_g^\circ$, we may equivalently show that $[e] \equiv \sfS^\circ_g$. Consider the inverse Fourier transform $\fm^{-1}$. By Lemma \ref{lem:Fchow},
    \[
    \fm^{-1} = \fm^\circ \cup p_1^*\Td^\vee(\cR_\pi)^{-1} \cup p_2^*\Td^\vee(\cR_\pi)^{-1}\,.
    \]
    By the same argument as in \cite[Proposition~5.2]{BMP}, one has $[e] = \fm^{-1}(1)$. Hence
    \[
    [e] = \fm^{-1}(1) = \Td^\vee(\cR_\pi)^{-1} \cup \fm^\circ(\Td^\vee(\cR_\pi)^{-1})\,.
    \]
    Since $\Td^\vee(\cR_\pi)^{-1} \equiv 1$, the claim follows from Corollary \ref{cor:Fleading2}.
\end{proof}

By \cite[Theorem 1.6]{BMP}, Proposition \ref{pro:ecorrect} holds on the semi-abelian part $ G \subset \barA$. By the excision sequence, it follows that $[e]= \sfS^\circ_g + j_*(\alpha)$ for some $\alpha\in \CH^{g-2}(C)$. The nontrivial content of the proposition is therefore to show that it is possible to find $\alpha$ in $\sR^{g-2}(C)$. See Remark \ref{rmk:gz}.

\subsection{Proof of Theorem \ref{thm:unit}}\label{sec:proofA}
To finish our proof of Theorem \ref{thm:unit}, we first study the pullback of the tautological class $\sfS_g^g$ along the morphism $j: C\to \barA$. For this, we require the following result:

\begin{lemma}\label{lem:self-In}
    The self-intersection formulas for $i$ and $j$ are given by
    $$
    i^*\left(\frac{1}{2}i_*(\alpha)\right) = c_1(N_i)\cup\alpha + s_{0,*}(\sbar_{\infty}^*\alpha) + \sbar_{\infty,*}(s_0^*\alpha)\, ,\quad j^*\left(\frac{1}{2}j_*(\beta)\right) = c_2(N_j)\cup\beta\,.
    $$
\end{lemma}

\begin{proof}
	The statement for $j$ follows immediately from the fact that $j$ is an \'etale double cover of its image, which is a smooth closed substack.
	For $i$, we observe that the fiber product $Y_2:=Y\times_{\barA}Y$ is described as follows: $Y_2$ has six components. The first
	two are isomorphic to $Y$ with the projection maps onto the two factors given by the identity. The other four arise from the fact that
	over $i([s_0])=i([\sbar_{\infty }])$, there are two possible preimages corresponding to the zero and infinity sections of the projective
	bundle. For the first two of these four components, the projection onto the first copy of $Y$ is identified with $s_0$ and the projection onto the second
	copy of $Y$ is identified with $\sbar_{\infty }$, while for the third and fourth components, the projection onto the first copy is identified
	with $\sbar_{\infty }$ and the projection onto the second copy is identified with $s_0$. Hence the excess intersection formula gives:
\begin{align*}
	i^{\ast}i_{\ast}(\alpha)=2c_1(N_i)\cup \alpha+2\overline{s}_{\infty,\ast}s_0^{\ast}(\alpha)+2s_{0,\ast}\overline{s}_{\infty}^{\ast}(\alpha)\,,
\end{align*}
	where the first term is the contribution of the two full-dimensional components.
\end{proof}

\begin{lemma}\label{lem:jstar}
    \begin{enumerate}[label=(\alph*)]
        \item $j^*\theta = \theta_1^C + \ell^C/2$.
        \item For $k\geq 1$, $j^*i_*(c_1(N_i)^{k-1}) = 2((-2\theta_2^C-\ell^C)^k + (\ell^C)^k)$.
		\item For $\alpha \in \CH^*(C)$, $j^*j_*(\alpha) = 2 \ell^C \cup (-\ell^C -2\theta_2^C) \cup \alpha$.
    \end{enumerate}
\end{lemma}
\begin{proof}
    Part (a) follows from \cite[Corollary 4.6]{EGH10}.
    Parts (b) and (c) follow from Lemma \ref{lem:self-In} with
    $c_1(N_i)= -2\theta_2^C - [s_0] - [\sbar_\infty]$
    and $c_2(N_j) = \ell^C\cup (-\ell^C - 2\theta^C_2)$.
\end{proof}
By Lemma \ref{lem:jstar}, we get
\begin{align}
    j^*\sfS^g_g
=
\Bigg[
\exp(\theta_1^C+ \ell^C/2) \cup
&\Bigg(
1
+
 \sum_{k\ge 1}
\frac{(-1)^k B_{2k}}{2^k k!}
\,\left((-\ell^C -2\theta_2^C)^k + (\ell^C)^k\right)\label{eq:jS} \\[4pt]
&+
\sum_{k_1,k_2\ge 1}
\frac{(-1)^{k_1+k_2}}{2^{k_1+k_2}}
\sum_{m=0}^{2k_2}
\binom{2k_2}{m}
B_{2k_1+m}
\frac{(\ell^C)^{k_1}}{k_1!}
\cup
\frac{(-\ell^C-2\theta_2^C)^{k_2}}{k_2!}
\Bigg)
\Bigg]_{\codim = g}\,.\nonumber
\end{align}

The intersection theory of the classes $\theta_1^C, \theta_2^C, \ell^C$ is governed by the polynomial
\begin{equation}\label{eq:Q}
    Q(u) = \theta_1^C + u\ell^C + \theta_2^C u^2\, .
\end{equation}
For instance, the shift operator $\sh^* : \sR^*(C) \to \sR^*(C)$ is uniquely determined by
\begin{equation}\label{eq:shiftQ}
\sh^*Q(u) = Q(u+1)\, .
\end{equation}
Let $\vartheta := \operatorname{disc}_u(Q) = (\ell^C)^2 - 4\theta_1^C \cup \theta_2^C$. By \eqref{eq:shiftQ}, $\vartheta$ is shift invariant.

For a polynomial $f(u)$, denote the coefficient of $u^m$ in $f(u)$ by
$[u^m]f(u)$. We fix a preferred additive basis of the ring $\sR^*(C)$ (see
Example \ref{expl:pp_s_1}).

\begin{proposition}\label{pro:basis}
    For any $a,b \geq 0$ and $0 \leq m \leq 2b$, let
    $$
    E_{a,b,m} = \vartheta^a\cup[u^m]Q(u)^b \in \mathbb Q[\theta_1^C, \theta_2^C, \ell^C]\,.
    $$
    The classes $E_{a,b,m}$ form a basis of $\mathbb Q[\theta_1^C, \theta_2^C, \ell^C]$ and the kernel of the natural map
    $$
    \mathbb Q[\theta_1^C, \theta_2^C, \ell^C] \to \sR^*(C)
    $$
    is the ideal $\mathfrak J_g$ with basis given by all $E_{a,b,m}$ with $a+b \geq g$.
\end{proposition}

\begin{proof}
    For the first part, consider the standard $2$-dimensional $\mathfrak{sl}_2$-representation
    $V=\QQ\langle a,b\rangle$. Identify
    \[
    \QQ[\theta_1^C,\theta_2^C,\ell^C]
    \cong
    \operatorname{Sym}^\bullet(\operatorname{Sym}^2 V)
    \]
	by setting $\theta_1^C = a^2, \ell^C = 2ab, \theta_2^C = b^2$. The result then follows from the Clebsch--Gordan decomposition of
    $\operatorname{Sym}^\bullet(\operatorname{Sym}^2 V)$.

    By \cite[Theorem 3.1]{GZ}, the ideal $\mathfrak J_g$ is generated as an ideal by the coefficients of $Q(u)^g$. Therefore, the second part follows from the first.
\end{proof}

The following is a technical lemma describing the kernel of the pullback $j^* : \sR^*(\barA) \to \sR^*(C)$.

\begin{lemma}\label{lem:einvert}
    For $g\ge 1$, let $c_2(N_j)\cup -: \sR^{g-2}(C) \to \sR^{g}(C)$ be the map which multiplies classes in $\sR^{g-2}(C)$ by $c_2(N_j) = \ell^C\cup(-\ell^C-2\theta_2^C)$. The map $c_2(N_j)\cup -$ is an isomorphism of vector spaces.
\end{lemma}

\begin{proof}
    We assign weights multiplicatively on $\mathbb Q[\theta_1^C, \theta_2^C, \ell^C]$ according to
    \begin{equation}\label{eqn:weightsproof}
    \operatorname{wt}(\theta_1^C) = 0\,, \operatorname{wt}(\ell^C)=1\, ,\operatorname{wt}(\theta_2^C)=2\, ,
    \end{equation}
    and work in the basis $E_{a,b,m}$ from Proposition \ref{pro:basis}. Then, $\operatorname{wt}(E_{a,b,m}) = m+2a$ and
    $$
    \sR^{g-2}(C) = \mathbb Q[\theta_1^C, \theta_2^C, \ell^C]^{g-2} =\bigoplus_{\substack{2a+b = g-2\\0 \leq m \leq 2b}}\mathbb Q E_{a,b,m}\, ,\quad \sR^{g}(C) = \bigoplus_{\substack{2a+b = g\\1\leq a,\, 0 \leq m \leq 2b}}\mathbb QE_{a,b,m}
    $$
    In particular, both spaces have the same dimension. The weights on the first part range from $0$ to $2g-4$, and the weights on the second from $2$ to $2g-2$. Multiplication by $(\ell^C)^2$ raises the weight by $2$, and multiplication by $\ell^C \cup \theta_2^C$ raises it by $3$. Therefore, it is enough to show that multiplication by $(\ell^C)^2$ is an isomorphism. Suppose that $\alpha \in \mathbb Q[\theta_1^C, \theta_2^C, \ell^C]\setminus \mathfrak J_g$ has (unweighted) degree $g-2$ and satisfies
    $$
    (\ell^C)^2\cup\alpha \in \mathfrak J_g^g = \bigoplus_{0 \leq m \leq 2g}\mathbb QE_{0,g,m}\, .
    $$
    Decomposing by weights \eqref{eqn:weightsproof}, we find some $\beta \neq 0$ such that $(\ell^C)^2 \cup \beta = [u^m]Q(u)^g$, which is impossible by inspection. So $(\ell^C)^2\cup -$ is injective, and since both vector spaces have the same dimension, $(\ell^C)^2\cup -$ is an isomorphism.
\end{proof}

The proof of the following statement is due to Aaron Pixton.
\begin{proposition}\label{pro:jS0}
    $j^*\sfS^g_g = 0$ in $\CH^g(C)$.
\end{proposition}

\begin{proof}[Proof by A. Pixton]
    To simplify the notation, write $x:= \theta_1^C, y := \ell^C,z := \theta_2^C$ as formal variables of degree $1$. We introduce a formal variable $u$. Let
    \begin{equation}\label{eq:Boperator}
        \fB \colon \QQ[u] \longrightarrow \QQ
    \end{equation}
    be the linear operator defined by $\fB(u^k)=B_k$ for all $k$. Then \eqref{eq:jS} can be written
    as the degree $g$ part of $\fB(F)$, where $F$ is a function defined by
    \begin{align*}
    F(x,y,z;u)
    =
    \exp\!\left(x+\frac{y}{2}\right)
    &\Bigg(
    1
    +
    \exp\!\big((\tfrac{y}{2}+z)u^2\big)-1
    +
    \exp\!\big(-\tfrac{y}{2}u^2\big)-1
    \\
    &\qquad\qquad
    +
    \big(\exp\!\big((\tfrac{y}{2}+z)u^2\big)-1\big)
    \big(\exp\!\big(-\tfrac{y}{2}(1-u)^2\big)-1\big)
    \Bigg).
    \end{align*}
    To show $j^*\sfS^g_g=0$, it suffices to prove that the degree $g$ part of $\fB(F)$ vanishes.

    Canceling terms and rearranging, we obtain
    \[
    F
    =
    \exp(x+uy+u^2z)
    +
    \exp\!\left(x+\frac{y}{2}\right)
    \left(
    \exp\!\left(-\frac{y}{2}u^2\right)
    -
    \exp\!\left(-\frac{y}{2}(1-u)^2\right)
    \right).
    \]
    By Lemma \ref{lem:1} (b), we have
    \[
    \fB\big(u^n-(1-u)^n\big)= -n
    \]
    for all non-negative $n>1$. Hence the formula simplifies to
    \[
    \fB\!\left(
    \exp\!\left(-\frac{y}{2}u^2\right)
    -
    \exp\!\left(-\frac{y}{2}(1-u)^2\right)
    \right)
    =
    y\,\exp\!\left(-\frac{y}{2}\right).
    \]
    Therefore, we get
    \begin{align*}
		[\fB(F)]_{\deg = g}
        &=
        \left[\fB\!\left(\exp(x+uy+u^2z)\right)
        +
        \exp(x)\,y\right]_{\deg = g}\\
        &=\frac{1}{g!}\fB\!\left((x+uy+u^2z)^g\right)\\
        &=0\,,
    \end{align*}
    where the second identity holds since the $u^1$-coefficient of the relation $Q(u)^g=0$ yields $x^{g-1}y=0$, and the third follows from Proposition \ref{pro:basis}. This yields the desired vanishing.
\end{proof}

We finally complete the proof of Theorem \ref{thm:unit}.

\begin{proof}[Proof of Theorem \ref{thm:unit}]
    By Proposition~\ref{pro:ecorrect}, there exists a class $\alpha \in \sR^{g-2}(C)$ such that
    \[
    [e] = \sfS^g_g + j_*(\alpha) \in \CH^g(\barA).
    \]
    We claim that the class $\alpha$ is uniquely determined by $\sfS^g_g$. Indeed, pulling back this equality along the morphism $j \colon C \to \barA$, we note that the image of the unit section lies in the semi-abelian locus of $\barA$ and is therefore disjoint from the image of $j$. Hence $j^*[e] = 0$. By the self-intersection formula, we obtain
    \[
    0 = j^*\sfS^g_g + j^*j_*(\alpha)
  = j^*\sfS^g_g + 2e(N_j) \cup \alpha.
    \]
    The Chern roots of $N_j$ are $\ell^C$ and $-\ell^C - 2\theta_2^C$. By Lemma~\ref{lem:einvert}, the cup product with the Euler class
    \[
    e(N_j) \cup - \; : \sR^{g-2}(C) \longrightarrow \sR^{g}(C)
    \]
    is an isomorphism. Therefore, the correction term $\alpha$ must be
    \[
    \alpha = -\frac{1}{2} e(N_j)^{-1}(j^*\sfS^g_g)\,.
    \]
    The right-hand side vanishes by Proposition~\ref{pro:jS0}, from which we get
	$$[e]=\sfS^{g}_g.$$
    This completes the proof.
\end{proof}

\begin{remark}
    Let $\Mbar_g' \subset \Mbar_g$ be the open locus of integral curves with at most one node, and let $\Jbar_g' \to \Mbar_g'$ be the degree $0$ relative compactified Jacobian. Under the Torelli map $\mathsf{tor} : \Jbar_g' \to \barA$, our formula for $\sfS_g^g$ coincides on $\Jbar_g'$ with the universal double ramification cycle formula $\uniDR_g^g(\emptyset)$ from \cite{BHPSS}.
\end{remark}

\begin{theorem}\label{thm:altzero}
	The unit section also has the formula
    $$
    [e] = \frac{[\theta^g]^{(2g)}}{g!} - \frac12 i_*\left(\sum_{m \geq 2}\frac{B_m}{m}[u^{m-2}]\frac{(\theta_1^C + u\ell^C + u^2\theta_2^C)^{g-1}}{(g-1)!}\right)
    $$
\end{theorem}
Compare this formulation with \cite[Theorem 1.3]{BMP}.
\begin{proof}
    Let $\mathsf{Z}_g$ be the right-hand side of the equation. In the semi-abelian scheme, the boundary is the $\mathbb G_m$-torsor associated to $\mathcal P$, so $\ell^C$ vanishes on the boundary. Therefore,
    $$
    \mathsf{Z}_g \equiv \frac{\theta^g}{g!} - \frac{1}{2}i_*\sum_{m=0}^{g-1}\frac{B_{2m+2}}{(2m+2)}\frac{(\theta_2^C)^m(\theta_1^C)^{g-1-m}}{m!(g-1-m)!} \equiv \mathsf{S}_g^\circ
    $$
    where we have also used Proposition \ref{pro:thetawt} (c) to simplify the top weight part of $\theta^g$. By the argument in the proof of Theorem \ref{thm:unit}, it is enough to show that $j^* \mathsf{Z}_g=0$.

    By the expansion \eqref{eq:thetamidddle} in Proposition \ref{pro:thetawt} and the pullback formulas in Lemma \ref{lem:jstar}, we have
    \begin{align*}
        j^*\left(\sum_{k\geq0}\frac{[\theta^k]^{(2k)}}{k!}\right)
        &=
        \frac{e^{\theta_1^C}}{2\theta_2^C}\cup
        \left(
        \ell^C+2\theta_2^C
        -\ell^C\cup e^{\ell^C+\theta_2^C}
        +\ell^C\cup(\ell^C+2\theta_2^C)
        \int_0^1e^{z\ell^C+z^2\theta_2^C}\,dz
        \right)\\
        &=
        e^{\theta_1^C}
        +\ell^C\int_0^1(1-z)e^{Q(z)}\,dz.
    \end{align*}
    In particular,
    \[
        j^* \frac{[\theta^g]^{(2g)}}{g!}
        =
        \ell^C\int_0^1(1-z)
        \frac{Q(z)^{g-1}}{(g-1)!}\,dz\, ,
    \]
    since $(\theta_1^C)^g=0$ on $\sR^*(C)$. We use the Bernoulli operator $\fB$ defined in \eqref{eq:Boperator}. Then
    \[
        \sum_{m\geq2}\frac{B_m}{m}[u^{m-2}]\,e^{Q(u)}
        =
        \fB\left(\int_0^u z e^{Q(z)}\,dz\right)\, .
    \]
    Putting everything together and using Lemma \ref{lem:self-In}, together with
    $\sbar_{\infty}^*Q(u)=Q(u+1)$, gives
    \begin{align*}
    j^* \mathsf{Z}_g =&\,\ell^C \int_0^1(1-z)\frac{Q(z)^{g-1}}{(g-1)!}\, dz - \fB \left(\int_0^u
        z\frac{(-\ell^C-2\theta_2^C)Q(z)^{g-1}
        +\ell^C Q(z+1)^{g-1}}{(g-1)!}\,dz\right)\\
        =&\, \ell^C\fB \left(\int_0^1(1-z)\frac{Q(z)^{g-1}}{(g-1)!}\, dz + \int_0^u(z-1)\frac{Q(z)^{g-1}}{(g-1)!}\, dz - \int_1^{u+1}(z-1)\frac{Q(z)^{g-1}}{(g-1)!}\, dz\right)\\
        =&\, \ell^C\fB \left(\int_{u}^{u+1}(1-z)\frac{Q(z)^{g-1}}{(g-1)!}\, dz\right).
    \end{align*}
    In the second equality, we have used that $(\ell^C+2z\theta_2^C)Q(z)^{g-1}=0$, which follows from differentiating $Q(z)^g=0$, and changed $z$ to $z-1$ in the third integral. Lemma \ref{lem:1} implies that
    $$
    \fB \left(\int_{u}^{u+1}(1-z)\frac{Q(z)^{g-1}}{(g-1)!}\, dz\right) = \frac{(\theta_1^C)^{g-1}}{(g-1)!},
    $$
    and hence $j^* \mathsf{Z}_g=0$ because $\ell^C \cup (\theta_1^C)^{g-1} =0$, by the weight vanishing for the abelian scheme $C$.
\end{proof}

\subsection{Comparison with the formula of Grushevsky--Zakharov}\label{sec:gz}
We prove a formula for the class of the unit section proposed by Grushevsky--Zakharov \cite{GZ}. Comparing the notation of loc.\ cit.\ with ours, we obtain the identifications $2\Delta = j_*1$ and $2D = i_*1$. Set
\[
\eta_{a,b,c} := \frac{(-1)^{b+c}(2c+2b-1)!!}{2^{3b+3c}a!c!} \sum_{x=0}^b \frac{(2-2^{2c+2x})B_{2c+2x}}{(2c+2b-2x-1)!! (2c+2x-1)!!(b-x)! x!}\,.
\]
The following formula appears in \cite[Remark 1.2]{GZ}:
\[
\mathsf{GZ}_g
= \sum_{a+b+2c=g} \eta_{a,b,c}\theta^a \cup D^b \cup \Delta^c\,.
\]
\begin{corollary}\label{cor:gz=s}
    For any $g$, we have $[e] = \mathsf{GZ}_g$.
\end{corollary}

\begin{proof}
    We first rewrite the formula for $\mathsf{GZ}_g$ to simplify the comparison with $\sfS^g_g$. Using the self-intersection formula, the powers of the boundary divisor class can be written as follows:
    \[
    D^b =
    \begin{cases}
    1, & b=0, \\[4pt]
    \dfrac{1}{2}\, i_*(c_1(N_i)^{\,b-1})
    \;+\;\frac{1}{2}
    j_* \displaystyle\sum_{\substack{k_1+k_2=b \\ k_1,k_2\ge 1}}
    \binom{b}{k_1}
    \,(\ell^C)^{k_1-1}\cup(-\ell^C-2\theta_2^C)^{k_2-1},
    & b\ge 1,
    \end{cases}
    \]
    Also,
    \[
    \Delta^c = \begin{cases}
        1 & c=0\\
        \frac{1}{2} j_*\left((\ell^C\cup(-\ell^C-2\theta_2^C))^{c-1}\right) & c>0.
    \end{cases}
    \]
	Hence we get
    \begin{align*}
    \mathsf{GZ}_g
    &=
    \frac{\theta^g}{g!}
    +
    \sum_{\substack{a+b=g \\ b\ge 1}}
    \eta_{a,b,0}\,
    \theta^a \cup
    \frac{1}{2}\,
    i_*\!\left( c_1(N_i)^{\,b-1} \right)
    \nonumber \\
    &\quad + \sum_{\substack{a+b=g \\ b\ge 1}} \eta_{a,b,0} \theta^a\cup
    \frac{1}{2}
    j_*\sum_{\substack{k_1+k_2=b\\k_1,k_2\ge 1}} \binom{b}{k_1}(\ell^C)^{k_1-1}\cup(-\ell^C-2\theta_2^C)^{k_2-1}\nonumber\\
    &\quad+
    \sum_{\substack{a+b+2c=g \\ c\ge 1}}
    \eta_{a,b,c}\,
    \frac{1}{2}\,\theta^a\cup
    j_*\!\left(
    \bigl(\ell^C\cup(-\ell^C-2\theta_2^C)\bigr)^{c-1}
    \cup
    (-2\theta_2^C)^b
    \right).
    \end{align*}

    We now show that $\sfS^g_g = \mathsf{GZ}_g$ at the level of strata algebras, i.e., without using relations in the Chow ring. Set
    \begin{equation}\label{eq:E}
        E(b,c) := 2^{b+c} a!\cdot \eta_{a,b,c}
    \end{equation}
	(this is independent of the choice of $a$). Comparing the coefficients of each tautological class, we see that for every $c$ with $2c \le g$, it suffices to prove that
    \[
    \sum_{k=0}^c
    \binom{g-2k}{c-k}
    2^{k}\,E(g-2k,k)
    =
    \frac{(-1)^g}{c!(g-c)!}
    \sum_{m=0}^{2c}
    \binom{2c}{m}
    B_{2g-2c+m}\,.
    \]
    This Bernoulli identity is proved in Appendix~\ref{sec:B} (see Theorem~\ref{sp:final}). Consequently, combining this with Theorem \ref{thm:unit}, we obtain
    \[
    [e] = \sfS^g_g = \mathsf{GZ}_g,
    \]
    which proves the claim.
\end{proof}

\begin{remark}\label{rmk:gz}
    We explain oversights in Grushevsky--Zakharov \cite{GZ}.
    In loc. cit., the authors consider the open substack $\X_g \subset \X_g'$ and the normalization $Y$ of its complement, and prove $[e] = \mathsf{GZ}_g$, after restricting to $\X_g$ and $Y$.
    Such an excision argument, however, is {\em not} enough to conclude that $[e] = \mathsf{GZ}_g$.
    The argument can be partially salvaged if one shows that the class $[e]$ lies in $\R^*(\X_g')$. 
    However, we do not know of an argument which shows that $[e]\in \R^*(\X_g')$ while avoiding the
	Fourier transform and the weight decomposition. Even then, the class $D^g$ (which we will show is nonzero in \S\ref{sec:taut})
	appears in the formula $\mathsf{GZ}_g$ with nonzero coefficient $\eta_{0,g,0}$ and vanishes on $\mathcal X_g$ and $Y$, so it is unclear how the argument in \cite{GZ} could determine the value of $\eta_{0,g,0}$.
\end{remark}

\subsection{Pushforwards of tautological classes}

\begin{theorem}\label{thm:small}
    The pushforward map
    $$
    p_{2,*} : \CH^*(\barA^{(2)}) \to \CH^*(\barA)
    $$
    preserves the (small) tautological ring.
\end{theorem}
\begin{proof}
    The restriction of $p_2 : \barA^{(2)}\to \barA$ to the boundary is a blowup followed by a $\mathbb P^1$-bundle and a $(g-1)$-dimensional abelian scheme.
    The normal bundles and the restrictions of the Theta divisors, as well as $\ell$, can be expressed in terms of the hyperplane classes, the exceptional divisors, and tautological classes on the $(g-1)$-dimensional abelian scheme.
    The pushforward from the abelian scheme preserves tautological classes by \cite[Theorem 9.10]{BMP}, so the pushforward sends classes in $\spp^{>0}(\barA^{(2)})[\theta_1, \theta_2, \ell]$ to strata classes \eqref{eq:dec_strata} decorated by elements of the small tautological ring on $\barA$.
	By Corollary \ref{cor:i_* in sR}, these belong to $\sR^{*}(\barA)$.\footnote{While Corollary \ref{cor:i_* in sR} is only proved later in \S\ref{sec:taut}, its proof is formal and independent of the previous results. The placement of its proof was chosen in order to preserve the logical flow.}

    Then, Corollary \ref{cor:F(exp_theta)} shows that
    $$
    p_{2,*}(\exp(\ell -\theta_1)) \in \sR^*(\barA)\,.
    $$
    If we apply $[N]^*$ to both sides, use Lemma \ref{lem:Npush}(a) and Theorem \ref{thm:wt}(a), and move the boundary corrections to the right-hand side of the equation, we obtain
    $$
    p_{2,*}\left(\sum_{a, b\geq 0}(-1)^bN^a\frac{\ell^a\cup\theta_1^b}{a!b!}\right) \in \sR^*(\barA).
    $$
    Separating by power of $N$ and codimension, we get the desired result on $\sR^*(\barA^{(2)})$. Since the $\lambda_i$-classes are pulled back from $B$, we get the result on $\R^*(\barA^{(2)})$.
\end{proof}

\begin{remark}\label{rmk:pushforward}
    In a similar way, one can show that any $p_* : \CH^*(\barA^{(s)}) \to \CH^*(\barA^{(r)})$ preserves (small) tautological rings.
\end{remark}

\begin{proof}[Proof of Theorem \ref{thm:FourierIso}]
    The Fourier transform given by
    $$
    \fm(\alpha) = p_{2,*}(p_1^* \alpha \cup \fm)
    $$
    is an isomorphism by \eqref{eq:fourier_is_iso}.
    Since the Fourier kernel $\fm$ lies in $\sR^*(\barA^{(2)})$, the Fourier transform preserves the tautological ring.
    Moreover, it preserves the small tautological ring by Theorem \ref{thm:small}. The argument for $\mathfrak F^{-1}$ is the same.
\end{proof}

\section{Weight decomposition}\label{sec:weight}

\subsection{Motivic decomposition}

In this section, we prove the first two parts of Theorem \ref{thm:dec}. For a positive integer $N$, consider the rational section
\[
\tau_N : \barA \dashrightarrow \barA^{(2)}, \quad x\mapsto (x,Nx),
\]
and take its Zariski closure. The associated cycle is $[\overline{\tau_N}] \in \CH^g(\barA^{(2)})$.

\begin{lemma}\label{lem:mapvscorr}
    The correspondences $[\overline{\tau_N}]^t$ and $N^{-2g}[\overline{\tau_N}]$ induce the maps
    $$
    [N]^*, [N]_* : \CH^*(\barA) \to \CH^*(\barA)
    $$
    respectively.
\end{lemma}
\begin{proof}
    We have a map $(b,N) : \barA_N \to \barA^{2}$. On the abelian part, it is an isomorphism onto $\tau_N$. Since $\barA_N$ is irreducible, we conclude that \begin{equation}\label{eq:correq}
        (b,N)_*(1) = f_*([\overline{\tau_N}]) \in\corr_B^0(\barA,\barA)
    \end{equation}
    and this implies the result by the projection formula.
\end{proof}

We prove the polynomiality of the class $[\overline{\tau_N}]$ in $N$ using the formula for the unit section $e : B\to \barA$. To do so, consider the following f.s. fiber diagram
\begin{equation}\label{eq:A2s}
    \begin{tikzcd}
    \barA_N^{(2)} \ar[rr, bend left, "\phi_N"] \ar[r] \ar[d,"b"] & \barA_N^{'(2)}\ar[d] \ar[r] & \barA \ar[d,"\beta"] \\
    \barA^{(2)} \ar[r] & (\logA)^2 \ar[r, "\phi_N^\circ"] & \logA,
    \end{tikzcd}
\end{equation}
where the morphism $\phi^\circ_N$ between log abelian schemes is defined by $\phi_N^\circ(x,y) = Nx-y$.

\begin{lemma}\label{lem:flat}
    The morphism $\phi_N :\barA^{(2)}_N \to \barA$ is flat.
\end{lemma}

\begin{figure}
\begin{center}
	\[
	\begin{tikzpicture}[every node/.style={font=\small}]
		\def\u{0.95}
		\def\ybase{1.55}

		\coordinate (A) at (0,\ybase);
		\coordinate (B) at ({4*\u},{\ybase+\u});
		\coordinate (C) at ({-\u},{\ybase+4*\u});
		\coordinate (D) at ({3*\u},{\ybase+5*\u});

		\draw[thick] (A) -- (B) -- (D) -- (C) -- (A);
		\draw[thick] (A) -- (D);

		\coordinate (O) at ({6.5*\u},1.55);
		\draw[->, thick] (O) -- ++(0.9,0.225)
			node[right] {$\mathsf{x}_1$};
		\draw[->, thick] (O) -- ++(-0.225,0.9)
			node[above] {$\mathsf{x}_2$};

		\foreach \k in {0,1,2,3} {
			\draw[thick]
				({\k*\u},{\ybase+\k*\u/4}) --
				({\k*\u},{\ybase+(4.25+\k/4)*\u});
			\draw[densely dashed, gray]
				({\k*\u},0.12) -- ({\k*\u},{\ybase+\k*\u/4});
		}

		\node at ({1.5*\u},{\ybase+5.55*\u})
			{$\mathcal U\Sigma_N^{(2)}$};

		\draw[thick] ({-\u},0) -- ({4*\u},0);
		\node[left, inner sep=2pt] at ({-\u},0) {$\cdots$};
		\node[right, inner sep=2pt] at ({4*\u},0) {$\cdots$};
		\foreach \k in {-1,0,1,2,3,4} {
			\fill ({\k*\u},0) circle (1.4pt);
		}
		\node at ({1.5*\u},-0.72) {$\mathcal U\Sigma^{(1)}$};

		\draw[->, thick]
			({1.5*\u},{\ybase+0.28*\u}) --
			node[right] {$\phi_N$} ({1.5*\u},0.18);
	\end{tikzpicture}
	\]
\end{center}
\caption{\label{fig:phi_n_flat}This figure depicts the tropicalization of the map $\phi_N$ from Lemma \ref{lem:flat} in the case $N=4$. To simplify the
illustration, we only depict the region
$0\leq\mathsf{x}_1,\mathsf{x}_2\leq\mathsf{t}$ and restrict to the links at
$\mathsf{t}=1$.}
\end{figure}
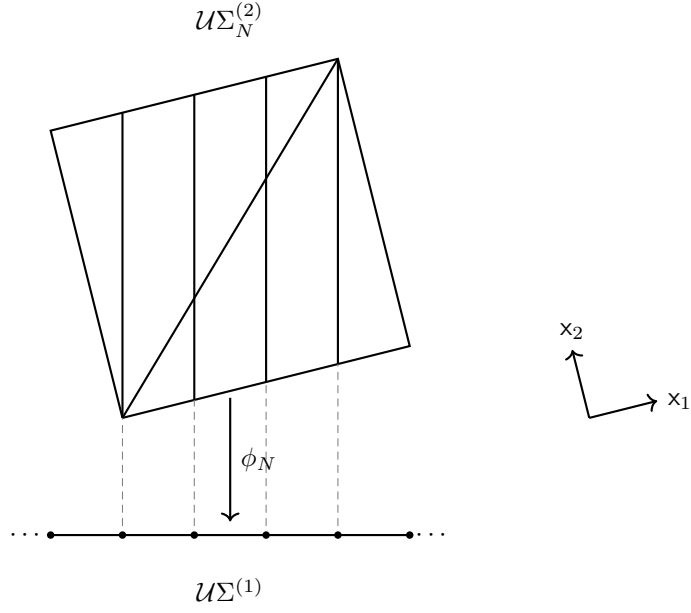

\begin{proof}
	Since $\barA$ is smooth, miracle flatness reduces the claim to showing that the fibers are equidimensional. Since $\phi_N$ is log smooth, this is equivalent to the statement that the induced map on Artin fans maps cones surjectively onto cones.

	We argue tropically; see also Figure \ref{fig:phi_n_flat} for a visualization of the tropical map. The Artin fan of $\barA_N^{(2)}$ is the fiber product of the Artin fans of the other spaces in \eqref{eq:A2s}. Since all subdivisions are $\ZZ$-equivariant, it suffices to work with the cone complexes before taking the quotient. Thus the fiber diagram induces a fiber diagram of cone complexes:
	$$\begin{tikzcd}[column sep=8em]
		\mathcal U\Sigma_N^{(2)}\arrow[d]\arrow[r]&\mathcal U\Sigma^{(1)}\arrow[d]\\
		\mathcal U\Sigma^{(2)}\arrow[r,"{(\mathsf{x}_1,\mathsf{x}_2,\mathsf{t})\mapsto (N\mathsf{x}_1-\mathsf{x}_2,\mathsf{t})}"]&\mathbb{R}\times\mathbb{R}_{\ge 0}.
	\end{tikzcd}$$
	The vertical map is the subdivision along the hyperplanes $\mathsf{x}_1=i\mathsf{t}$ for
    $i\in \ZZ$. Hence the universal cover of the cone complex $\Sigma_N^{(2)}$
    corresponding to the Artin fan of $\barA_N^{(2)}$ is the subdivision of
    $\mathcal U\Sigma^{(2)}$ along the hyperplanes
    \[
    N\mathsf{x}_1-\mathsf{x}_2=i\mathsf{t}.
    \]

    Since $\mathcal U\Sigma_N^{(2)}$ is $\ZZ$-periodic, it suffices to consider the maximal cones in $\tilde{\Sigma}_N^{(2)}$ obtained by restricting to $0 \leq \mathsf{x}_2 \leq \mathsf{t}$. Moreover, since no maximal cone in $\mathcal U\Sigma^{(1)}$ is supported over $\{\mathsf{t}=0\}$, it suffices to restrict to the link over $1\in \RR_{\geq 0}$; equivalently, we may set $\mathsf{t}=1$.
    The link of $\tilde{\Sigma}^{(2)}$ is the subdivision of the square
    $(\mathsf{x}_1,\mathsf{x}_2)\in [0,1]^{2}$ along the diagonal $\{\mathsf{x}_1=\mathsf{x}_2\}$. The link of $\tilde{\Sigma}_N^{(2)}$ is obtained by further subdividing
    along the hyperplanes $N\mathsf{x}_1-\mathsf{x}_2\in \mathbb{Z}$.

    First consider the maximal cones in the region $\mathsf{x}_1 \leq \mathsf{x}_2$. They are the strips
    \[
    \mathsf{x}_2+i-1 \leq N\mathsf{x}_1 \leq \mathsf{x}_2+i,
    \quad
    i \in \{0,\ldots,N-1\}.
    \]
    Restricting to $\mathsf{x}_2=1$, the condition $\mathsf{x}_1\leq \mathsf{x}_2$ is automatic, and the
    strip conditions become
    \[
    i \leq N\mathsf{x}_1 \leq i+1.
    \]
     Under the map
	$(\mathsf{x}_1,1)\mapsto N\mathsf{x}_1-1$ this surjects onto the interval $[i-1,i]$ which is the link of the maximal cone in $\mathcal U\Sigma^{(1)}$
	defined by $0\le \mathsf{x}_1-(i-1)\mathsf{t}\le \mathsf{t}$.

    Similarly, in the region $\mathsf{x}_2 \leq \mathsf{x}_1$, the maximal cones are the regions
    \[
    \mathsf{x}_2+i \leq N\mathsf{x}_1 \leq \mathsf{x}_2+i+1,
    \quad
    i \in \{0,\ldots,N-1\}.
    \]
    Restricting to $\mathsf{x}_2=0$, the first condition is automatic, and the strip conditions become
    \[
    i \leq N\mathsf{x}_1 \leq i+1.
    \]
    Under the map $(\mathsf{x}_1,0)\mapsto N\mathsf{x}_1$, this interval surjects onto $[i,i+1]$,
    which is the link of the maximal cone in $\mathcal U\Sigma^{(1)}$ defined by $0 \leq \mathsf{x}_1-i\mathsf{t} \leq \mathsf{t}$.

    Thus the induced map on Artin fans maps cones surjectively onto cones.
    Therefore $\phi_N$ is flat.
\end{proof}

Using Theorems \ref{thm:unit} and \ref{thm:wt}, we prove the polynomiality of $[\overline{\tau_N}]$.

\begin{proposition}\label{pro:poly}
    The class $[\overline{\tau_N}] \in \CH^g(\barA^{(2)})$ is polynomial in $N$ of degree $2g$.
\end{proposition}

\begin{proof}
    Consider the diagram \eqref{eq:A2s}. We first prove that
    \begin{equation}\label{eq:poly}
        b_*\phi_N^*([e]) = [\overline{\tau_N}] \in \CH^g(\barA^{(2)})\,.
    \end{equation}
    For the unit section $e : B \to \barA$, form the cartesian diagram
    \[
    \begin{tikzcd}
        B_N \ar[r,"\tilde{\phi}_N"] \ar[d,"\tilde{e}"] &
        B \ar[d,"e"]\\
        \barA_N^{(2)} \ar[r,"\phi_N"] &
        \barA .
    \end{tikzcd}
    \]
    By Lemma \ref{lem:flat}, $\phi_N$ is flat, so we have
    \[
    \phi_N^* e_*([B])
    =
    \tilde e_*\tilde\phi_N^*([B])
    =
    \tilde e_*([B_N]).
    \]
    Over $\A_g \subset \A_g'=B$, the morphism $\phi_N$ is smooth, so $B_N$ is
    generically reduced. Moreover, over $\A_g$, the pushforward
    $b_*\tilde e_*([B_N])$ is clearly $[\tau_N]$. Therefore, we get \eqref{eq:poly}.

    On the other hand, we have
    \[
    b_*\circ \phi_N^* = [N\times 1]^* \circ \diff^* : \CH^*(\barA) \to \CH^*(\barA^{(2)}),
    \]
    where $\diff : \barA^{(2)} \to \barA$ is the difference map. Hence
    \begin{equation}\label{eq:taudelta}
    [\overline{\tau_N}] = b_*\phi_N^*([e])
    =
    [N \times 1]^*\circ\diff^*([e])
    =
    [N \times 1]^*\circ\diff^*(\sfS_g^g)\,,
    \end{equation}
    where the first equality follows from \eqref{eq:poly}, the second from the above equality, and the third from Theorem \ref{thm:unit}. Since $\diff$ is a map of log schemes, $\diff^*$ preserves piecewise polynomials. Together with Mumford's formula \eqref{eq:Mumform}, we see that $\diff^* : \sR^*(\barA) \to \sR^*(\barA^{(2)})$. Therefore, the right-hand side is polynomial in $N$ of degree $2g$ by Theorem \ref{thm:wt}. This proves the result.
\end{proof}

In what follows, we view $[\overline{\tau_N}]$ as a class on
$\barA^2$ via the pushforward along $f: \barA^{(2)} \to \barA^2$.

\begin{proposition}\label{pro:nm}
    For $N,M \in \ZZ_{>0}$, $[\overline{\tau_N}] \circ [\overline{\tau_M}] = [\overline{\tau_{NM}}] \in \corr_B^0(\barA,\barA)$.
\end{proposition}

\begin{proof}
    First, we make an observation that is specific to the $s=1$ and torus-rank-one situation.
    For $N \in \ZZ_{>0}$, consider the f.s. fiber diagram
    \begin{equation}\label{eq:nm1}
        \begin{tikzcd}
        \barA_N \ar[r]\ar[d,"b"] \ar[rr, bend left,"N"] & \barA_N' \ar[d] \ar[r] & \barA\ar[d,"\beta"]\\
        \barA \ar[r,"\beta"] & \logA \ar[r,"N"] & \logA.
        \end{tikzcd}
    \end{equation}
	By base-change, the Artin fan of $\barA_N'$ is the cone complex $(\Sigma_N^{(1)})'$
	introduced in \S\ref{sec:wt_lower_trop}.
	From the definition of this cone complex, it readily follows
	that $(\Sigma_N^{(1)})'\cong \Sigma_N^{(1)}$, so
	the morphism $\barA_N\to \barA_N'$ in the above diagram is
	an isomorphism.

    For $N,M \in \ZZ_{>0}$, take $\barA_N$ and $\barA_M$ as in
    \eqref{eq:nm1}. Denote by $b : \barA_N \to \barA, \, b' : \barA_M \to \barA$ the corresponding log modifications. Consider the diagram
    \begin{equation}\label{eq:nm3}
        \begin{tikzcd}
            & X_{NM} \ar[r] \ar[d,"g"] & \barA_N \times \barA_M \ar[d,"h"] \\
            & \barA^3 \ar[r] \ar[dl,"\pi_{1,3}"'] \ar[d,"q"] & \barA^2 \times \barA^2 \ar[d,"q"]\\
            \barA^2 & \barA \ar[r, "\delta"] & \barA \times \barA,
        \end{tikzcd}
    \end{equation}
    where $h$ is given by $(b,N,b',M)$, $q$ is the projection onto the second and third factors, $\delta : \barA \to \barA \times \barA$ is the diagonal embedding, and the squares are cartesian. We also let $\pi_{1,3}$ denote the projection onto the first and third factors.

    We prove that the space $X_{NM}$ is irreducible of dimension $\dim \barA$, and the fundamental classes are related by
    \begin{equation}\label{eq:nm2}
        \delta^! \big( [\barA_N \times \barA_M] \big) = [X_{NM}]
        \in \CH_*(X_{NM}),
    \end{equation}
    where $\delta^!$ is the Gysin pullback along $\delta$.
    Since $\barA'_N = \barA_N$, the following diagram is cartesian:
    \[
    \begin{tikzcd}
        X_{NM} \ar[d] \ar[r] & \barA_M \ar[d,"b'"]\\
        \barA'_N \ar[r,"N"] & \barA .
    \end{tikzcd}
    \]
    Since $\barA_N' \to \barA$ is the pullback of $N : \logA \to \logA$, the above diagram is also an f.s. fiber diagram. Therefore $X_{NM}$ is a log modification of $\barA'_N$, hence irreducible and $\dim X_{NM} = \dim \barA$. For dimension reasons it follows that
    \[
    \delta^! \big([\barA_N \times \barA_M] \big) = c \cdot [X_{NM}] \in \CH_*(X_{NM})
    \]
    for some constant $c \in \QQ$. To check that $c=1$, we restrict to the
    open locus $\A_g \subset \A_g'$, where the equality $c=1$ is clear.

    By \eqref{eq:correq}, we can compute the composition of $[\overline{\tau_N}]$ and $[\overline{\tau_M}]$ using the models $\barA_N$ and $\barA_M$. From the diagram \eqref{eq:nm3}, we obtain
    \begin{align*}
        [\overline{\tau_N}] \circ [\overline{\tau_M}]
        &= \pi_{1,3*} \delta^! h_*\bigl([\barA_N \times \barA_M]\bigr) \\
        &= \pi_{1,3*} g_* \delta^!\bigl([\barA_N \times \barA_M]\bigr) \\
        &= \pi_{1,3*} g_*[X_{NM}],
    \end{align*}
    where the first equality follows from the definition, the second from the compatibility of proper pushforward with Gysin pullback, and the third from \eqref{eq:nm2}. Since $X_{NM}$ is irreducible of dimension $\dim \barA$ and $\pi_{1,3}\circ g$ maps onto the Zariski closure of the rational section $\tau_{NM}$, we have
    \[
    \pi_{1,3*}g_*[X_{NM}] = [\overline{\tau_{NM}}].
    \]
    This completes the proof.
\end{proof}

We prove parts (a) and (b) of Theorem \ref{thm:dec}. Part (c) will be proved in the last part of \S\ref{sec:weight}.
\begin{proof}[Proof of Theorem \ref{thm:dec}(a) and (b)]
    Part (a) was proved in Proposition \ref{pro:poly}, so it remains to prove
    part (b). By part (a), we can write
    \[
    [\overline{\tau_N}]^t = \sum_{k=0}^{2g} \w_k N^k,
    \qquad
    \w_k \in \CH^g(\barA^{(2)}) \to \corr^0_B(\barA,\barA).
    \]
    By Proposition \ref{pro:nm}, for positive integers $N,M$, the following
    identity holds in $\corr_B^0(\barA,\barA)$:
    \begin{equation}\label{eq:orth1}
        \sum_{i=0}^{2g} \w_i (NM)^i
        =
        \left(\sum_{i=0}^{2g} \w_i N^i\right)
        \circ
        \left( \sum_{j=0}^{2g} \w_j M^j \right).
    \end{equation}
    Both sides are polynomials in $N$ and $M$. Comparing the coefficients of
    each monomial in $N$ and $M$,
    \begin{equation}\label{eq:orth}
        \w_i^2 = \w_i, \qquad \w_i \circ \w_j = 0 \quad \text{for } i \neq j.
    \end{equation}
    Let $\Delta_{\barA/B}$ be the relative diagonal for $\pi : \barA \to B$. Since $[\overline{\tau_1}]^t = [\Delta_{\barA/B}]$, we get
    \[
    [\Delta_{\barA/B}] = \sum_{i=0}^{2g} \w_i,
    \]
    where $\{\w_i\}_{i=0}^{2g}$ are orthogonal projectors by \eqref{eq:orth}.
    This proves the relative motivic decomposition.

    Over $\A_g \subset \A_g'$, our $\w_i$ coincide with the weight projectors for $\X_g \to \A_g$ constructed in \cite{DM91}. Therefore, the motivic lift of the decomposition of $R\pi_*\QQ$ follows from the full-support property (Remark \ref{rmk:full}) and the argument in \cite[\S 2.5.2]{PerverseFourierMSY25}. If $\mathrm{char} (\Bbbk) = p>0$, the same argument works with $\overline{\QQ}_\ell$-coefficients with $p \neq \ell$.
\end{proof}

\begin{definition}
    We define the pure weight subspaces of $\CH^*(\barA)$ by
    $$
    \CH^{*,(w)}(\barA) := \im \left(\mathfrak w_w : \CH^*(\barA) \to \CH^*(\barA)\right)\, \quad \CH^*_{(w)}(\barA) = \im \left(\mathfrak w_{2g-w}^t : \CH^*(\barA) \to \CH^*(\barA)\right).
    $$
\end{definition}
By Theorem \ref{thm:dec}, we have canonical decompositions
$$
\CH^*(\barA) = \bigoplus_{w=0}^{2g} \CH^{*,(w)}(\barA) = \bigoplus_{w=0}^{2g} \CH^{*}_{(w)}(\barA),
$$
and by Lemma \ref{lem:mapvscorr}, $[N]^*$ (resp. $[N]_*$) acts by multiplication by $N^w$ (resp. $N^{-w}$) on $\CH^{*,(w)}(\barA)$ (resp. $\CH^{*}_{(w)}(\barA)$). In particular, we have
\begin{corollary}\label{cor:pure}
    Any class $\alpha \in \CH^*(\barA)$ has weights $\leq 2g$ for $[N]^*$ and $[N]_*$ respectively, and each weight part $[\alpha]^{(w)}$, $[\alpha]_{(w)}$ has pure weight $w$ for $[N]^*$ or $[N]_*$, respectively.
\end{corollary}

\begin{remark}
    Theorem \ref{thm:dec} does not produce a multiplicative splitting of $\CH^*(\barA)$; see, for example, Proposition \ref{pro:thetawt}. Following \cite{BMSY1}, one can show using Corollary \ref{cor:pushThet} that, when $g\ge 2$, there is no multiplicative splitting of the perverse filtration on $\mathsf{H}^*(\barA,\QQ)$.
\end{remark}

\subsection{Fourier vanishing}

Recall that $\fm_i,\fm^{-1}_i\in\CH^i(\barA^{(2)})$ denote the codimension $i$ components of $\fm$ and $\fm^{-1}$, respectively.

\begin{theorem}\label{thm:fv}
    \begin{enumerate}[label = (\alph*)]
        \item If $i+j<2g$, $\fm_i^{-1} \circ \fm_j = 0$.
        \item If $i+j <2g$, $\fm_i \circ \fm_j^{-1} = 0$.
    \end{enumerate}
\end{theorem}

\begin{proof}
    By Lemma \ref{lem:Fchow}, the classes $\fm_i^{-1}$ and $\fm_j$ on $\barA^{(2)}$ (before pushforward to $\barA^2$) are polynomials in $\ell$ and piecewise polynomial classes.
    Let $p_{i,j} : \barA^{(3)} \to \barA^{(2)}$ be the projection onto the $i$-th and $j$-th factors. Consider the action of $[1 \times N \times 1]^*$ on $\CH^*(\barA^{(3)})$. By Theorem \ref{thm:wt},
    \[
    [1\times N \times 1]^*
    \bigl(p_{2,3}^*\fm_i^{-1} \cup p_{1,2}^*\fm_j\bigr)
    \]
    is polynomial in $N$ of degree at most $i+j$. By assumption, $i+j<2g$. Therefore, by Proposition \ref{pro:vanish},
    \[
    \fm_i^{-1} \circ \fm_j =
    p_{1,3*}\bigl(p_{2,3}^*\fm_i^{-1} \cup p_{1,2}^*\fm_j\bigr) = 0\,.
    \]
    This proves part (a), and part (b) follows similarly.
\end{proof}

\begin{remark}
    In \cite{PerverseFourierMSY25}, Fourier vanishing is proved using Adams
    operations in algebraic $K$-theory. The argument above gives an independent
    proof using the theory of weights.
\end{remark}

\subsection{Comparison of filtrations}
\begin{definition}
    The {\em (Chow-theoretic) perverse filtration} is defined by
    \[
    P_k \CH^*(\barA) := \sum_{w \le k} \im \Big(\fm_w : \CH^*(\barA) \to \CH^*(\barA) \Big)\,.
    \]
\end{definition}

In \S \ref{sec:mbyNmaps} and \S\ref{sec:Nlow}, we introduced the {\em weight filtrations} $W_k$ and $\widehat{W}_k$. We have
\[
W_k \CH^*(\barA) = \bigoplus_{w\leq k}\CH^{*,(w)}(\barA)\, ,\quad \widehat{W}_k \CH^{*}(\barA) = \bigoplus_{w\le k} \CH^*_{(w)}(\barA)\,.
\]

The following is the main result of this section. The key input is the weight calculation for $[N]_*$.
\begin{proposition}\label{pro:P=W}
    For all $k$, $P_k \CH^*(\barA) = W_k \CH^*(\barA) = \widehat{W}_k\CH^*(\barA)$.
\end{proposition}

\begin{proof}
	By Lemma \ref{lem:Npush}\footnote{The analogous statement for $[N]_{\ast}$ can be proven similarly.}, we have
    $$
    [N]^* \mathfrak F_k(\alpha) = p_{2*} (p_1^*(\alpha) \cup [1\times N]^* \mathfrak F_k)\, ,\quad [N]_* \mathfrak F_k(\alpha) = p_{2*} (p_1^*(\alpha) \cup [1\times N]_* \mathfrak F_k).
    $$
    Therefore, the inclusions
    $$
    P_k\subseteq W_k \quad\text{and}\quad P_k\subseteq\widehat{W}_k
    $$
    are a consequence of the weight estimates for tautological classes in Theorems \ref{thm:wt} and \ref{thm:wt_lower}.

    We first prove that $W_k \subseteq P_k$.
    By Theorem \ref{thm:fv}~(a) and Lemma \ref{lem:mapvscorr}, it suffices to prove the vanishing
    \begin{equation}\label{eq:PW1}
        \mathfrak F_k^{-1}\circ \mathfrak w_i =0 \quad \text{for } i+k<2g \quad \text{in } \CH^g(\barA^{(2)}).
    \end{equation}
	Consider the lift of the relative diagonal $\Delta : \barA \to \barA^{(2)}$ (Lemma \ref{lem:maps_regular}). By abuse of notation, we write
    \[
    \Delta := \Delta_*(1) \in \CH^g(\barA^{(2)}),
    \]
    and since $\diff$ is flat, $\diff^*([e]) = \Delta$. By \eqref{eq:taudelta}, we have
    \[
    \sum_{i=0}^{2g} N^i \mathfrak w_i
    =
    [\overline{\tau_N}]^t
    =
    [1\times N]^*(\Delta)
    \]
    in $\CH^g(\barA^{(2)})$.
    Let $p_{1,3} : \barA^{(3)} \to \barA^{(2)}$ be the projection onto the first and third factors, and let $\Delta_{1,2} \in \CH^g(\barA^{(3)})$ be the pullback of $\Delta$ along the projection onto the first two factors. By Lemma \ref{lem:Npush} and the projection formula,
    \begin{align*}
        \mathfrak F^{-1}\circ [1\times N]^*(\Delta)
        &=
        p_{1,3*}
        \bigl(
            b_*(1\times N \times 1)^*(\Delta_{1,2})
            \cup \mathfrak F_{23}^{-1}
        \bigr) \\
        &=
        p_{1,3*}
        \bigl(
            \Delta_{1,2}
            \cup
            N^{2g}[1\times N \times 1]_*(\mathfrak F_{23}^{-1})
        \bigr).
    \end{align*}
    Taking the codimension $k$ part of $\mathfrak F^{-1}$ and applying Theorem \ref{thm:wt_lower} (c), we obtain
    \[
    \sum_{i=0}^{2g} N^i\, \mathfrak F_k^{-1}\circ\mathfrak w_i
    =
    \sum_{j \leq k} N^{2g-j} \,
    p_{1,3*}(\Delta_{1,2}\cup \alpha_{k,j})
    \]
    for some classes $\alpha_{k,j}\in \CH^k(\barA^{(3)})$. Comparing powers of $N$ gives \eqref{eq:PW1}.

    Similarly, to prove $\widehat{W}_k \subseteq P_k$ it is enough to show that
    \begin{equation}\label{eq:PW2}
        \fm^{-1}_k \circ \mathfrak w_{2g-i}^t = 0 \quad \text{for } i+k < 2g \quad \text{in } \CH^g(\barA^{(2)}).
    \end{equation}
    Since $[\overline{\tau_N}] = [N \times 1]^* \Delta$, we have, by Lemma \ref{lem:Npush},
    $$
    \fm^{-1}_k \circ [\overline{\tau_N}] = \fm^{-1}_k \circ [N \times 1]^* \Delta = [N \times 1]^* \fm_k^{-1}\circ \Delta = [N \times 1]^* \fm_k^{-1}\,,
    $$
    which is a polynomial in $N$ of degree at most $k$ by Theorem \ref{thm:wt}. Therefore, $\fm_k^{-1} \circ \mathfrak w_j^t =0$ whenever $k<j$, which proves \eqref{eq:PW2}.
\end{proof}

We finally complete the proof of Theorem \ref{thm:dec}.

\begin{proof}[Proof of Theorem \ref{thm:dec}~(c)]
    By Proposition \ref{pro:P=W}, it suffices to prove the multiplicativity of
    the perverse filtration $P_k \CH^*(\barA)$. By Theorems \ref{thm:deq} and \ref{thm:fv}, the triple $(\barA,B,G)$ is a dualizable abelian fibration satisfying Fourier vanishing. Hence the multiplicativity of $P_k$ follows from \cite[Theorem 2.6]{PerverseFourierMSY25}. This completes the proof.
\end{proof}

\section{Tautological ring of the universal family}\label{sec:taut}

\subsection{Relations between \texorpdfstring{$\lambda$}{lambda} classes}

We study relations between the Chern classes of the Hodge bundle. For the duration of this section, we choose a
smooth toroidal compactification of $\mathcal A_g$, which we denote by $\overline{\mathcal A}_g$.

Recall from \cite{FC} that we may choose $\overline{\mathcal{A}}_g$ such that there exists a universal semi-abelian scheme $\mathcal G_g$ over $\overline{\mathcal A}_g$, with zero section $e: \overline{\mathcal A}_g\to \mathcal G_g$ and the Hodge bundle extends to $\overline{\mathcal{A}}_g$
as the conormal bundle to the zero section of $\mathcal G_g$:
$$\mathbb E := N_{e|\mathcal G_g}^{\vee}.$$
We denote its Chern classes by $\lambda_i = c_i(\mathbb E)$. 

There is a canonical map from $\overline{\mathcal A}_g$ to the Satake compactification:
$$
\beta : \overline{\mathcal A}_g \to \overline{\mathcal A}_g^{\mathrm{Sat}} = \mathcal A_g \sqcup \mathcal A_{g-1} \sqcup \ldots \sqcup \mathcal A_0.
$$
From the stratification of the Satake compactification, we obtain open loci in $\overline{\mathcal{A}}_g$ parametrizing degenerations
of torus rank at most $r$, for $r=0, \ldots , g$:
$$
\overline{\mathcal A}_g^{\leq r} := \beta^{-1}(\mathcal A_g \sqcup \ldots \sqcup \mathcal A_{g-r})\,.
$$
The goal of this section is to determine a presentation of the \emph{Lambda ring}:
$$
\Lambda^*(\overline{\mathcal A}_g^{\leq r}) :=\im \left(\QQ[\lambda_1, \ldots , \lambda_g] \to \CH^*(\overline{\mathcal A}_g^{\leq r})\right)\,.
$$
\begin{itemize}
	\item When $r=g$, \cite[Section 5]{vdGcoh} shows that
		\begin{equation}\label{eq:Lambdabarring}
			\Lambda^*(\overline{\mathcal A}_g) \cong\frac{\QQ[\lambda_1, \ldots , \lambda_g]}{\langle c(\EE\oplus \EE^\vee)-1\rangle }\,.
		\end{equation}
	\item When $r=0$, \cite[Theorem 1.5]{vdG99} proves that there is a single extra relation $\lambda_g=0$ in $\Lambda^{\ast}(\mathcal{A}_g)$.
\end{itemize}

The complement $\mathcal Z^{> r} := \overline{\mathcal A}_g\setminus \overline{\mathcal A}_g^{\leq r}$ is part of the toroidal boundary. By \cite[Section 7]{Pinkthesis}, there exists a proper surjective morphism
$$
\xi:\bigcup_{i=1}^K\overline{\mathcal T}_i \to \mathcal Z^{> r},
$$
where $\overline{\mathcal T}_i$ is a compactification of a toric variety bundle $\mathcal T_i \to \mathcal X_{g-s_i}^{s_i}$, for certain $s_i> r$\footnote{In the notation of \cite{Pinkthesis}, $\mathcal T_i$ is a mixed Shimura variety.}. Moreover, by \cite[Proposition 6.25, Proposition 6.26]{Pinkthesis} we can take the compactifications $\overline{\mathcal T}_i$ to be smooth and such that there is a regular map $\pi_i:\overline{\mathcal T}_i \to \overline{\mathcal A}_{g-s_i}$. This yields the following commutative diagram:
\begin{equation}\label{eq:Satake}
\begin{tikzcd}
	\overline{\mathcal T}_i && \Abar_g \\
	\Abar_{g-s_i} & {\Abar^\mathrm{Sat}_{g-s_i}} & {\Abar^\mathrm{Sat}_g}.
	\arrow["\xi", from=1-1, to=1-3]
	\arrow["\pi_i",from=1-1, to=2-1]
	\arrow["\beta", from=1-3, to=2-3]
	\arrow["\beta"', from=2-1, to=2-2]
	\arrow[from=2-2, to=2-3]
\end{tikzcd}
\end{equation}
\begin{lemma}\label{lem:lambdacomp}
    In the situation described above, let $\xi_i :\overline{\mathcal T}_i \to \mathcal Z^{> r}$ be the restriction of $\xi$. Then, there is a short exact sequence on each $\overline{\mathcal T}_i$
    $$
    0 \to \pi_i^* \mathbb E \to \xi_i^*\mathbb E \to \mathcal V \to 0,
    $$
    where $\mathcal V$ is a vector bundle of rank $s_i$ whose Chern classes vanish in $\CH^*(\overline{\mathcal T}_i)$. In particular, $\xi_i^*\lambda_j = \pi_i^*\lambda_j$, which vanishes whenever $j > g-r$.
\end{lemma}

\begin{proof}
    Let $\mathcal G_{g-s_i}$ be the pullback of the universal semi-abelian scheme from $\Abar_{g-s_i}$ to $\overline{\mathcal T}_i$. Since $\mathcal G_g \big|_{\mathcal T_i}$ has torus rank equal to $s_i$, it has a closed subtorus\footnote{See \cite[Definition I.2.1.]{FC} for the definition of a torus.} $\mathcal H^\circ $ of rank $s_i$ by \cite[Corollary I.2.11]{FC}. By \cite[Proposition I.2.9.]{FC}, $\mathcal H^\circ $ extends to a closed torus $\mathcal H \subseteq \mathcal G_g|_{\overline{\mathcal T}_i}$, and we obtain a short exact sequence of semi-abelian schemes
    $$
    1 \longrightarrow \mathcal H \longrightarrow\mathcal G_g|_{\overline{\mathcal T}_i} \longrightarrow \mathcal B \longrightarrow 0.
    $$
    By the construction in \cite[page 106]{FC}, $\mathcal B\big|_{\mathcal T_i}$ is the universal family $\mathcal X_{g-s_i}$. Hence, by \cite[Proposition I.2.7]{FC}, $\mathcal B = \mathcal G_{g-s_i}$. From this we obtain a short exact sequence of normal bundles at the zero section
    \begin{equation}\label{eq:sess}
    0 \longrightarrow N_{1|\mathcal H} \longrightarrow \xi^* N_{e|\mathcal G_g} \longrightarrow N_{e|\mathcal G_{g-s_i}} \longrightarrow 0.
    \end{equation}
    By \cite[Exp.~X Theorem~5.16]{SGA3II}, there exists a finite \'etale cover of $\overline{\mathcal T}_i$ over which $\mathcal H$ is isomorphic to $\mathbb G_m^{s_i}$. Therefore, the Chern classes of $N_{1|\mathcal H}$ are torsion. Dualizing \eqref{eq:sess}, the lemma follows.
\end{proof}

The last part of Lemma \ref{lem:lambdacomp} was obtained in \cite{CMOP} by different methods. Following their work, we obtain well-defined maps
$$
\epsilon_k: \CH^{*}(\overline{\mathcal{A}}_g^{\leq r}) \to \CH^{*+g-k}(\overline{\mathcal{A}}_g)\,, \quad \alpha \mapsto \lambda_{g-k} \cup \overline{\alpha}\, ,
$$
for any $0\le k \leq r$, where $\overline{\alpha}$ is any extension of $\alpha$.

\begin{theorem}\label{thm:lambdarels}
    For each $r<g$, we have
    \begin{equation}\label{eq:Lambda_isom}
    \Lambda^*(\overline{\mathcal A}_g^{\leq r}) \cong \frac{\QQ[\lambda_1, \ldots , \lambda_g]}{\langle c(\mathbb E \oplus \mathbb E^\vee)-1, \lambda_g\cdots\lambda_{g-r}\rangle}
    \end{equation}
\end{theorem}
\begin{proof}
    Mumford's relation $c(\mathbb E \oplus \mathbb E^\vee)=1$ was shown in \cite{EV04} to hold on any smooth toroidal compactification of $\mathcal A_g$.
    
	We prove the second relation by induction in $r$. The $r=0$ case is \cite[Proposition 1.2]{vdG99}. Hence, by the induction hypothesis, we may assume that $\lambda_g\cup\cdots \cup \lambda_{g-r}$ vanishes on $\Abar_{g}^{\le r}$ and want to show the same statement for $r+1$. By the excision sequence, we find
	$\alpha_i \in \CH^*(\overline{\mathcal T}_i)$ for $i=1, \ldots ,K$ such that $\lambda_g\cup\cdots \cup \lambda_{g-r} = \xi_{1,\ast}(\alpha_1)+\ldots +\xi_{K,\ast}(\alpha_K)$. Then, by Lemma \ref{lem:lambdacomp} and \eqref{eq:Satake}, we have
    \begin{align*}
    \lambda_g\cup \cdots \cup \lambda_{g-r-1}\big|_{\Abar_g^{\leq r}} &= \sum_{i=1}^K \xi_{i,*}(\alpha_i\cup \xi_i^*(\lambda_{g-r-1}))\big|_{\Abar_g^{\leq r}} \\
    &= \sum_{i=1}^K \xi_{i,*}\left(\alpha_i \cup \pi_i^*\left(\lambda_{g-r-1}\big|_{\mathcal A_{g-s_i}}\right)\right)\,,
    \end{align*}
    which vanishes by the case $r=0$.

    It remains to show that these are the only relations. Using Mumford's relation and $\lambda_g \cup \cdots\cup  \lambda_{g-r}=0$, every class $\alpha$ in the Lambda ring can be written in the form
    \begin{equation}\label{eq:Lambdamonoms}
        \alpha =\sum_{j=0}^r\sum_{\underline{u}^j \in \{0,1\}^{g-j-1}}c_{j, \underline{u}^j}\lambda_{1}^{u^j_1}\cup  \cdots\cup  \lambda_{g-j-1}^{u^j_{g-j-1}}\cup \lambda_{g-j+1}\cup\cdots\cup \lambda_g.
    \end{equation}
    Thus, \eqref{eq:Lambda_isom} is equivalent to proving that the monomials in \eqref{eq:Lambdamonoms} are linearly independent.
    
    Suppose that $\alpha =0$, where $\alpha$ is as in \eqref{eq:Lambdamonoms}. Then
    $$
    0 = \epsilon_0 \left(\alpha\right) = \sum_{\underline{u}^0 \in \{0,1\}^{g-1}}c_{0, \underline{u}^0}\lambda_{1}^{u^0_1}\cup  \cdots\cup  \lambda_{g-1}^{u^0_{g-1}}\cup \lambda_{g} \in \Lambda^*(\Abar_g),
    $$
    since $\lambda_g^2=0$ by Mumford's relation. By \eqref{eq:Lambdabarring}, the monomials in this sum are elements of a basis of $\Lambda^*(\Abar_g)$, so all coefficients $c_{0,\underline{u}^0}$ must vanish. Then,
    $$
    0 = \epsilon_1 \left(\alpha\right) =\sum_{\underline{u}^1 \in \{0,1\}^{g-2}}c_{1, \underline{u}^1}\lambda_{1}^{u^1_1}\cup  \cdots\cup  \lambda_{g-2}^{u^1_{g-2}}\cup \lambda_{g-1} \cup \lambda_{g} \in \Lambda^*(\Abar_g)
    $$
    using the relation $\lambda_{g-1}^2 \cup \lambda_g=0$, which follows from Mumford's relation. Therefore, the coefficients $c_{1, \underline{u}^1}$ vanish. Continuing this way, applying the evaluations $\epsilon_2(\alpha) , \ldots , \epsilon_r(\alpha)$, we conclude that all coefficients in \eqref{eq:Lambdamonoms} must vanish, using the relations $\lambda_k^2 \cup \ldots \cup \lambda_{g-1}\cup\lambda_g =0$.
\end{proof}

\subsection{Tautological relations}

We prove the tautological relations in Theorem \ref{thm:rel}.

\begin{proof}[Proof of Theorem \ref{thm:rel}]
    By Proposition \ref{pro:thetawt}, for each $c \geq 0$, we have
    \[
    \frac{[\theta^c]^{(2c)}}{c!} = \St_g^c \in \R^c(\barA),
    \]
    where $\St^c_g$ is a tautological class defined in \eqref{eq:St}. By Corollary \ref{cor:pure}, the class $[\theta^c]^{(2c)}$ has pure
    weight $2c$. By Theorem \ref{thm:dec}, $\CH^{*,{(w)}}(\barA)$ vanishes for
    $w > 2g$. Therefore, if $c > g$, then $\St_g^c = 0$.
\end{proof}

We also prove the relations in Theorem \ref{thm:taut} that are pulled back from $B$:
\begin{proposition}\label{pro:relB}
    The relations $\lambda_g\cup D = 0$, $\lambda_{g-1}\cup D=0$ (for $g>1$) and $\lambda_g = \frac{B_{2g}}{2^gg!}D^g$ hold.
\end{proposition}

\begin{proof}
    The first two follow from Lemma \ref{lem:lambdacomp} and the vanishing $\lambda_{g-1} =0 \in \CH^*(\mathcal A_{g-1})$; see \cite{vdG99}.

    Since $\mathbb E$ is the conormal bundle to the unit section, we have $(-1)^g\lambda_g = e^*([e])$. By pulling back \cite[Theorem~1.6]{BMP} along the unit section $e$ and using the fact that $e^*(\theta) = 0$, we obtain
    \begin{equation}\label{eq:lambdag}
        (-1)^g\lambda_g
        = \frac{1}{2}\iota_*\left(\frac{(-1)^gB_{2g}}{2^gg!}c_1(N_\iota)^{g-1}\right) = (-1)^g \frac{B_{2g}}{2^gg!}D^g\,,
    \end{equation}
    where the last equality follows from the fact that $\iota$ is generically $2\!:\!1$ onto its image. See also \cite{EvdG05}. This proves the last relation.
\end{proof}

\begin{corollary}\label{cor:non-zero}
    The class $\theta^g\cup D^g\cup \lambda_{g-2}\cup \cdots\cup \lambda_{1}$ is nonzero on $\barA$.
\end{corollary}
\begin{proof}
    Since $\pi_*(\theta^g\cup D^g\cup \lambda_{g-2}\cup \cdots\cup \lambda_{1}) = \frac{2^g(g!)^2}{B_{2g}}\lambda_g\cup\lambda_{g-2}\cup\cdots\cup \lambda_1$ by Proposition \ref{pro:relB}, the result follows from Theorem \ref{thm:lambdarels}.
\end{proof}

\subsection{The kernel of \texorpdfstring{$i_*$}{i}}

Let $i : Y \to \partial \barA \subset \barA$ be the finite morphism to the boundary of $\barA$ as in Theorem \ref{thm:bdy}. The pushforward map $i_* : \CH^*(Y) \to \CH^*(\barA)$ is not injective, as we now explain. Let $\sh : C \to C$ be the shift operator \eqref{eq:sh}, with inverse $\shb$. For $\beta \in \sR^*(C)$, define $\zeta(\beta)$ by
\begin{equation}\label{eq:zeta}
    (1+\sh^*)\zeta(\beta) = (\sh^*-1)\beta .
\end{equation}
Since $\sh^*$ is unipotent, $\zeta$ is well defined. Let $\mathfrak K$ be the $\QQ$-vector subspace of $\sR^*(Y)$ generated by
\[
\ell^C\cup\beta + U\cup\zeta(\beta),
\quad
\beta \in \sR^*(C),
\]
where $U = 2h-\ell^C =[s_0] + [\overline{s}_\infty]$.

\begin{lemma}\label{lem:K-in-ker}
    $\mathfrak{K}\subseteq \ker (i_* : \sR^*(Y) \to \CH^*(\barA) )$.
\end{lemma}

\begin{proof}
    The equality \eqref{eq: j = is0=isinf} implies that
$$
i_*\circ s_{0,*} = i_*\circ \overline{s}_{\infty,*} : \CH^*(C) \to \CH^*(\barA).
$$
Let $\alpha \in \sR^*(C)$. If $\pi_{\mathbb{P}} : Y \to C$ denotes the $\mathbb P^1$-bundle (see Figure \ref{fig:thm_bdy}), then
$$
s_0^* \pi_{\mathbb{P}}^*(\alpha) = \overline{s}_{\infty}^*\pi_{\mathbb{P}}^*\shb^*(\alpha).
$$
Omitting $\pi_{\mathbb{P}}^*$ from the notation, this gives
$$
i_*(\alpha\cup [s_0]) = i_*(\shb^*\alpha \cup [\sbar_{\infty}])
$$
for any class $\alpha \in \sR^*(C)$. Writing $[s_0]$ and $[\sbar_{\infty}]$ in terms of $U$ and $\ell^C$, we get
$$
i_*\left(\alpha\cup \frac{U+\ell^C}{2}\right) = i_*\left(\shb^*\alpha \cup \frac{U-\ell^C}{2}\right).
$$
Applying this to $\alpha=\sh^*\gamma$, we get
$$
i_*\left(\sh^*\gamma\cup (U+\ell^C)\right) = i_*\left(\gamma \cup (U-\ell^C)\right)\, .
$$
After rewriting the above, we get
$$
(\sh^*\gamma - \gamma) \cup U +(\sh^*\gamma + \gamma) \cup \ell^C \in \ker (i_* : \sR^*(Y) \to \CH^*(\barA) ).
$$
By the unipotence of $\sh^*$, every $\beta\in \sR^{\ast}(C)$ can be written as $\beta=\sh^*\gamma+\gamma$ for some $\gamma\in \sR^{\ast}(C)$. This proves the Lemma.
\end{proof}

In particular, we obtain the following result:

\begin{corollary}\label{cor:i_* in sR}
    Pushforward along the maps $i$ and $j$ preserves (small) tautological rings.
\end{corollary}

\begin{proof}
    By the projection formula and the compatibility of $\lambda$-classes from Lemma \ref{lem:lambdacomp}, it is enough to work with small tautological rings. For $j_*$, note that any class in $\sR^*(C)$ can be written as a polynomial in
    $$
    \theta_1^C + \ell^C/2 = j^*\theta\, , \ell^C = \alpha_1(N_j)\, ,-\ell^C - 2\theta_2^C = \alpha_2(N_j)\, ,
    $$
	and it is clear that
    $$
    \frac{1}{2}j_*(j^*(\theta^k)\cup\alpha_1(N_j)^b\cup\alpha_2(N_j)^c) = \theta^k\cup\frac{1}{2}j_*(\alpha_1(N_j)^b\cup\alpha_2(N_j)^c) \in \sR^*(\barA)\, .
    $$
    For $i_*$, it suffices to work modulo $\mathfrak K$. Then, from the shape of the generators of $\mathfrak K$, every class in $\sR^*(Y)/\mathfrak K$ can be written as $U^{0/1}$ times a polynomial in
    $$
    i^*(\theta) = \theta_1^C + \frac{1}{2}U\, ,i^*(D) = -2\theta_{2}^C.
    $$
    It is clear that for every $\alpha$, we have
    $$
    \frac{1}{2}i_*(U\cup\alpha) = \frac{1}{2}j_*(\alpha + \sh^*\alpha) \in \sR^*(\barA)\, ,
    $$
    and the other cases are dealt with the projection formula. Therefore, both $i_*$ and $j_*$ preserve the (small) tautological ring.
\end{proof}

Let
\begin{equation}\label{eq:Ypairing}
    \langle \, ,\,\rangle : \sR^d(Y) \times \sR^{2g-1-d}(Y)  \to \sR^{2g-1}(Y) = \QQ \,U \vartheta^{g-1} 
\end{equation}
denote the intersection pairing. Since $\sR^*(C)$ is Gorenstein, the pairing \eqref{eq:Ypairing} is perfect. Let $\mathfrak N = \mathfrak K ^\perp$
be the orthogonal complement with respect to the above pairing, and let $\mathfrak M = \sR^*(Y)/\mathfrak K$, viewed as a graded vector space.

\begin{lemma}\label{lem:normalization-pairing}
    The subspace $\mathfrak N$ agrees with the image of the pullback map
    $$
    i^* : \sR^*(\barA) \to \sR^*(Y).
    $$
\end{lemma}

\begin{proof}
    Every class in $\sR^*(Y)$ can be written uniquely as
    $\alpha+U\cup \beta$, where $\alpha,\beta\in\sR^*(C)$. Since $\sh^*$ fixes
    $\vartheta^{g-1}$, its adjoint for the perfect intersection pairing on
    $\sR^*(C)$ is $\shb^*$. It follows that $\zeta$, defined in \eqref{eq:zeta}, is skew-adjoint. Hence, for every $\gamma\in\sR^*(C)$,
    the coefficient of $U\vartheta^{g-1}$ in
    \[
    (\alpha+U\cup \beta)
    \cup\bigl(\ell^C\cup\gamma+U\cup\zeta(\gamma)\bigr)
    \]
    is the coefficient of $\vartheta^{g-1}$ in
    \[
    \bigl(\ell^C\cup\beta-\zeta(\alpha)\bigr)\cup\gamma.
    \]
    Since the pairing on $\sR^*(C)$ is perfect, we have
    \begin{equation}\label{eq:N-condition}
    \alpha+U\cup \beta\in\mathfrak N
    \quad\Longleftrightarrow\quad
    \ell^C\cup\beta=\zeta(\alpha).
    \end{equation}
    By the definition of $\zeta$, condition \eqref{eq:N-condition} is equivalent to
    \begin{equation}\label{eq:sft-inv}
        s_0^*(\alpha +U\cup \beta) = \overline{s}_\infty^*(\alpha+U\cup \beta)\,.
    \end{equation}
    Following \cite{GZ}, we call these classes \emph{shift invariant}. It is clear that $\im (i^*) \subset \mathfrak N$.

    The pullback formulas are determined by $i^*\theta = \theta_1^C + \frac{U}{2}$, the identities \eqref{eq:normalbunex}, and the self intersection formula in Lemma \ref{lem:self-In}. From this, we get a commutative diagram
    \[\begin{tikzcd}
	{\sPP^*(\Sigma^{(1)})[\theta]} & {\frac{\QQ[\theta_1^C, \theta_2^C, \ell^C, U]}{(U^2-(\ell^C)^2)}} \\
	{\sR^*(\barA)} & {\sR^*(Y)}
	\arrow["I", from=1-1, to=1-2]
	\arrow[from=1-1, to=2-1]
	\arrow[from=1-2, to=2-2]
	\arrow["{i^*}", from=2-1, to=2-2]
    \end{tikzcd}\]
    where
    $$
    I(p(\mathsf{x}, \mathsf{y}, \theta)) =  \frac{p_0 + p_{\infty}}{2} + U \frac{p_0-p_{\infty}}{2\ell^C}\, ,
    $$
    and $p \in \sPP^*(\Sigma^{(1)})[\theta]$ is a polynomial in $\mathsf{x}$, $\mathsf{y}$ and $\theta$ that satisfies the gluing condition $p(\mathsf{t}, 0, \theta) = p(0,\mathsf{t}, \theta)$. The expressions $p_0$, $p_\infty$ are given by
    $$
    p_0 = p \left(\ell^C, -2\theta_2^C -\ell^C, \theta_1^C + \frac{\ell^C}{2}\right)\, ,\quad p_\infty = p \left(\ell^C-2\theta_2^C, -\ell^C, \theta_1^C - \frac{\ell^C}{2}\right)= \shb^*p_0\,.
    $$
    Indeed, since the relation $U^2-(\ell^C)^2=0$ makes the map $I$ a homomorphism, it suffices to check that $I(\alpha)=i^*(\alpha)$ for $\alpha \in \{\theta, \mathsf{x}+\mathsf{y}, \mathsf{x}\mathsf{y}, \mathsf{x}\mathsf{y}(\mathsf{x}-\mathsf{y})\}$.

The action of the shift operator lifts to an action on $\QQ[\theta_1^C,\theta_2^C,\ell^C]$, so there is a well-defined notion of shift invariant polynomials in $\frac{\QQ[\theta_1^C, \theta_2^C, \ell^C, U]}{(U^2-(\ell^C)^2)}$, defined by the analogue of \eqref{eq:N-condition}.
    
    \textbf{Claim:} Any shift invariant class is the image of a shift invariant polynomial.
    
    The space of relations
    $$
    \mathfrak J = \ker \big(\QQ[\theta_1^C, \theta_2^C, \ell^C] \to \sR^*(C)\big)
    $$
    was described in Proposition \ref{pro:basis}. Consider the map
    $$
    \Upsilon : \QQ[\theta_1^C, \theta_2^C, \ell^C]^{\oplus 2}\to \QQ[\theta_1^C, \theta_2^C, \ell^C]\,, \quad (q_1, q_2) \mapsto \ell^Cq_2 - \zeta (q_1)\,,
    $$
    where $\zeta$ is the natural lift of \eqref{eq:zeta} to $\QQ[\theta_1^C, \theta_2^C, \ell^C]$. By \eqref{eq:N-condition}, the \emph{class} $q_1(\theta_1^C, \theta_2^C, \ell^C) + U \cup q_2(\theta_1^C, \theta_2^C, \ell^C)$ is shift invariant \eqref{eq:sft-inv} if and only if $\Upsilon(q_1, q_2) = 0$ modulo the relations in $\sR^*(C)$. Therefore, it is enough to show that
    \begin{equation}\label{eq:goallll}
    \mathfrak J \cap \im(\Upsilon) = \Upsilon(\mathfrak J \times \mathfrak J).
    \end{equation}
    Since $\sh^*+1$ is an automorphism and commutes with $\sh^{\ast}-1$, we have $\im(\zeta) = \im(\sh^* -1)$, and it is immediate to see that $\im(\sh^*-1)$ is the ideal generated by $\theta_2$, modulo $\ell^C$. Thus $\im (\Upsilon)$ is the ideal generated by $\ell^C$ and $\theta_2^C$, and looking at the definition of the $E_{a,b,m}$ it is then clear by Proposition \ref{pro:basis} that the left side of \eqref{eq:goallll} has a basis given by $E_{a,b,m}$ with $a+b \geq g$, and $m \geq 1$ if $a=0$.
    Consider the derivation
    $$
    \delta = \ell^C \frac{\partial}{\partial \theta_1^C} +2\theta_2^C \frac{\partial}{\partial \ell^C}\,.
    $$
    It is the infinitesimal generator of $\sh^*$, so $\sh^* = \exp(\delta)$, $\delta(Q(t)) = \partial_t Q(t)$ and $\delta(\vartheta) =0$, so
    $$
    \delta E_{a,b,m} =(m+1)E_{a,b,m+1}\,.
    $$
    Inserting this into the expression \eqref{eq:zeta} for $\zeta$, we get
    $$
    \zeta (E_{a,b,m}) = \tanh \left(\frac{\delta}{2}\right) (E_{a,b,m}) = \frac{m+1}{2}E_{a,b,m+1} + \text{ a linear combination of }E_{a,b,m+3}, \, E_{a,b,m+5},\ldots \,.
    $$
    It follows that each $E_{a,b,m}$ with $m>0$ lies in $\Upsilon(\mathfrak{J}\times\mathfrak{J})$. On the other hand, expanding $\vartheta = -2(b+1)(\ell^C)^2 -4\theta_1^C\cup \theta_2^C + (2b+3)(\ell^C)^2$ yields
    $$
    E_{a+1,b,0} = -\frac{4}{b+2} E_{a, b+2, 2} +  \frac{2b+3}{b+1}\ell^CE_{a, b+1,1} \in \Upsilon(\mathfrak J \times \mathfrak J)\,.
    $$
    This proves that both sides of \eqref{eq:goallll} agree.
    
    By the claim, it is enough to show that $I$ is surjective onto the space of shift invariant \emph{polynomials}. Suppose that $q_1(\theta_1^C, \theta_2^C, \ell^C) + U q_2(\theta_1^C, \theta_2^C, \ell^C)$ is a shift invariant expression, and let
    $$
    F(\theta_1^C, \theta_2^C,\ell^C) = q_1 +\ell^C q_2\, , \quad G(\theta_1^C, \theta_2^C,\ell^C) = q_1 - \ell^Cq_2 \in \QQ[\theta_1^C, \theta_2^C, \ell^C]\,.
    $$
    In these terms, shift invariance is equivalent to $\sh^*F=G$. Define
    \begin{equation}\label{eq:i^*inverse}
        r(\mathsf{x}, \mathsf{y},\theta) :=F\left(\theta-\frac{\mathsf{x}}{2}, -\frac{\mathsf{x}+\mathsf{y}}{2},\mathsf{x}\right)=G\left(\theta-\frac{\mathsf{y}}{2},-\frac{\mathsf{x}+\mathsf{y}}{2}, -\mathsf{y}\right)\,.
    \end{equation}
    Then, by shift invariance, $r(\mathsf{t},0,\theta) = r(0,\mathsf{t},\theta)$. Substituting \eqref{eq:i^*inverse} into the definitions of $p_0$, $p_\infty$, we obtain
    $$
    p_0=F(\theta_1^C,\theta_2^C,\ell^C)=q_1+\ell^Cq_2\, \quad p_\infty = G(\theta_1^C,\theta_2^C,\ell^C)=q_1 -\ell^Cq_2,
    $$
    so $I(r)=q_1 +Uq_2$, as required.
   \end{proof}

Let $\mathfrak A^*$ denote the quotient of
$\spp^*(\Sigma^{(1)})[\theta]$ by the relations in Theorem
\ref{thm:taut}, restricted to the small tautological ring; explicitly, we quotient by the ideal generated by $\widetilde{\sfS}_{g}^{g+1}$ and the pushforward of relations in $\sR^*(Y)$ by $i_*$.

\begin{lemma}\label{lem:two-exact-sequences}
For every $0\leq d\leq g$, pushforward from $Y$ and restriction to $Y$
give exact sequences
\begin{align}
 0&\longrightarrow \mathfrak M^{d-1}
 \stackrel{\frac12i_*}{\longrightarrow} \mathfrak A^d
 \longrightarrow \QQ\,\theta^d
 \longrightarrow0,\label{eq:low-exact}\\
 0&\longrightarrow
 \QQ\,\theta^{g-d}D^g
 \longrightarrow \mathfrak A^{2g-d}
 \xrightarrow{\ i^*\ }\mathfrak N^{2g-d}
 \longrightarrow0.\label{eq:high-exact}
\end{align}
Furthermore, $\mathfrak A^{>2g}=0$ and
$\mathfrak A^{2g}=\QQ\theta^gD^g$.
\end{lemma}

\begin{proof}
There are no relations of degree at most $g$: the relations from $Y$ acquire
one additional degree after pushforward from the boundary, while the first
weight relation is $\widetilde{\sfS}_g^{g+1}$. Therefore, a simple calculation in $\QQ[\mathsf{x}, \mathsf{y}, \theta]$ gives
$$
\dim (\mathfrak A^d) = \binom{d+1}{2} +1 \text{ for } d =0,\ldots , g.
$$
Using the basis $E_{a,b,m}$ from Proposition \ref{pro:basis}, we have
$$
\dim \sR^r(C) = \begin{cases}
    \binom{r+2}{2}&\text{ if }0\leq r<g,\\[2mm]
  \binom{2g-r}{2}&\text{ if }g\leq r\leq2g-2,\\[2mm]
  0&\text{ otherwise}.
\end{cases}
$$
Therefore, since $\mathfrak M^r = \operatorname{coker} \left((\ell^C\cup -, \zeta) : \sR^{r-1}(C) \to  \sR^r(C) \oplus U \sR^{r-1}(C)\right) $, and
$$
\ker (\ell^C\cup -, \zeta) = \bigoplus_{a+b=g-1} \QQ\,\vartheta^a(\theta_2^C)^b,
$$
we obtain
\begin{equation}\label{eq:Mdim}
\dim \mathfrak M^r  = \dim \sR^r(C) + \begin{cases}
    1&\text{ if }g\leq r\leq2g-1\,,\\
 0&\text{ otherwise}.
\end{cases}
\end{equation}
Indeed, $\zeta(\beta)=0$ means that $\beta$ is fixed by the shift
operator $\sh^*$.  The invariant line in
$\vartheta^a\operatorname{Sym}^{2b}V$ ($V$ as in Proposition \ref{pro:basis}) is
$\QQ\vartheta^a(\theta_2^C)^b$, and multiplication by $\ell^C$ kills it
exactly when $a+b=g-1$.  The weight decomposition used in
\eqref{eqn:weightsproof} shows that invariant lines with different $a+b$
cannot cancel.
This proves \eqref{eq:low-exact}.

Set $e=2g-d$. Using the weight relation $\widetilde{\sfS}_g^{g+1}=0$ and \eqref{eq:Mdim} we have
\[
 \dim \mathfrak A^e\leq \dim \mathfrak M^{e-1} + \begin{cases}
     1&\text{ if }e=g\\
     0&\text{ if }e>g
 \end{cases}\, \leq1+\binom{d+1}{2}\,.
\]
On the other hand,
\[
 \dim \mathfrak N^e=\dim \mathfrak M^{d-1}=\binom{d+1}{2}\,,
\]
since they are orthogonal complements. Restriction $\mathfrak A^e\to \mathfrak N^e$ is surjective by Lemma~\ref{lem:normalization-pairing}.

The class $\theta^{g-d}D^g$ lies in the kernel because
\[
 i^*(\theta^{g-d}D^g)
 =(\theta_1^C+U/2)^{g-d}(-2\theta_2^C)^g=0\,,
\]
since $(\theta_2^C)^g = [u^{2g}]Q(u)^g$. However, $\theta^{g-d}D^g$
is nonzero by Corollary \ref{cor:non-zero}. The preceding upper bound now forces
\eqref{eq:high-exact}. The vanishing $\mathfrak A^{>2g}=0$ follows from the vanishing $\sR^{>2g-1}(Y)=0$.
\end{proof}

\begin{corollary}\label{thm:Gor}
We have $\sR^*(\barA) = \mathfrak A^*$. In particular, $\sR^*(\barA)$ is Gorenstein with socle in degree $2g$ and the ideal of relations in $\sR^*(\barA)$ is generated by $\widetilde{\sfS}_g^{g+1}$ and the pushforward of relations on the boundary.
\end{corollary}

\begin{proof}
Fix $0\leq d\leq g$ and choose a vector space splitting of
 \eqref{eq:high-exact}. A canonical splitting of \eqref{eq:low-exact} is given by $\theta^d$. Then the intersection pairing on $\mathfrak A^*$ in complementary degrees has the form
\[
 \begin{array}{c|cc}
  &\theta^{g-d}D^g&\mathfrak N^{2g-d}\\ \hline
  \theta^d&\theta^gD^g&*\\
  \mathfrak M^{d-1}&0&P_d
 \end{array}.
\]
The lower-left block vanishes because the restriction of $\theta^{g-d}D^g$ to $Y$ is zero, while the projection formula identifies $P_d$ with the matrix of the bilinear map \eqref{eq:Ypairing}. Therefore, $\mathfrak A^*$ is Gorenstein. Since by Corollary \ref{cor:non-zero} the generator of the socle is nonzero in $\sR^*(\barA)$, the map $\mathfrak A^* \to \sR^*(\barA)$ is an isomorphism.
\end{proof}

\begin{proof}[Proof of Theorem \ref{thm:taut}]
	The relations in $\mathfrak I_g$ allow us to write every class in $\R^*(\barA)$ in the form
    $$
    \beta = \sum_{\substack{0\leq k\leq g,\\\underline{u}\in \{0,1\}^{g-2}}}c_{k, \underline{u}}\,\theta^k\cup \lambda_1^{u_1}\cup \cdots \cup \lambda ^{u_{g-2}}_{g-2} \cup \lambda_{g-1} + \sum_{\substack{\alpha \in \mathcal B_g,\\
    \underline{v}\in \{0,1\}^{g-2}}}d_{\alpha, \underline{v}}\,\lambda_1^{v_1}\cup \cdots \cup \lambda_{g-2}^{v_{g-2}} \cup \alpha\, ,
    $$
	where $\mathcal B_g$ is a basis of $\sR^*(\barA)$. We show that the monomials in these sums are linearly independent. Assume that $\beta=0$. By restricting to the interior, we have that the $c_{k, \underline{u}}$ are all zero. For $\alpha \in \mathcal B_g$, let $\gamma $ be its dual under the intersection pairing on $\mathfrak A^*$ which exists by Corollary \ref{thm:Gor}. Similarly, for $\underline{v}$ fixed,
	let $\gamma'$ be (some lift of) the dual of $\lambda_1^{v_1}\cup\dots\cup \lambda_{g-2}^{v_{g-2}}$ in $\Lambda^{\ast}(\overline{\mathcal{A}}_g)/(\lambda_g,\lambda_{g-1})\cong \Lambda^{\ast}(\overline{\mathcal{A}}_{g-2})$.
	Then,
    $$
    \beta \cup \gamma \cup \gamma' = d_{\alpha, \underline{v}}\theta^g\cup D^g \cup \lambda_1\cup \cdots \cup \lambda_{g-2},
    $$
    and therefore $d_{\alpha, \underline{v}}=0$ by Corollary \ref{cor:non-zero}.
\end{proof}

\appendix

\section{Bernoulli identities}\label{sec:B}
\subsection{Basic identities}
For the reader's convenience, we collect standard identities among Bernoulli numbers.
The Bernoulli polynomial $B_n(u)$ is defined by
\begin{equation*}
    \frac{xe^{ux}}{e^x-1} = \sum_{n=0}^\infty B_n(u) \frac{x^n}{n!}\,.
\end{equation*}
When $u=0$, we recover the usual definition of the Bernoulli numbers
\begin{equation}\label{eq:B}
    F(x) = \frac{x}{e^x-1} = \sum_{k=0}^\infty B_k \frac{x^k}{k!}\,.
\end{equation}
By definition, $B_n(u) = \sum_{k=0}^n \binom{n}{k}B_ku^{n-k}$.
\begin{lemma}\label{lem:1}
    \begin{enumerate}[label=(\alph*)]
        \item $\sum_{k=0}^n \binom{n}{k} B_k = B_n$ for $n\neq 1$.
        \item $\sum_{k=0}^n (-1)^k \binom{n}{k} B_k = n+(-1)^nB_n$.
        \item $\sum_{k=0}^n \binom{n}{k} (2-2^k) B_k = 2^n B_n$ for $n\neq 1$.
    \end{enumerate}
\end{lemma}
\begin{proof}
    Part~(a) follows from expanding $x= (e^x-1) F(x)$. Part~(b) follows from the standard identity for Bernoulli polynomials,
    \[
    B_n(x+1) - B_n(x) = nx^{n-1}\,,
    \]
    by substituting $x=-1$. Part~(c) follows from the identity
    \[
    F(2u) e^u
    =
    \frac{2u e^u}{(e^u-1)(e^u+1)}
    =
    \frac{2u}{e^u-1} - \frac{2u}{e^{2u}-1} = 2F(u) -F(2u)\,.
    \]
\end{proof}

\begin{lemma}\label{lem:3}
    Consider the function $S(u) = \frac{u}{\sinh u}$. Then
    \[
    S(u) = \sum_{k=0}^\infty \frac{(2-2^{2k}) B_{2k}}{(2k)!}u^{2k}\,.
    \]
\end{lemma}
\begin{proof}
    By a simple calculation,
    \[
    e^u F(2u) = \frac{2u e^u}{e^{2u}-1} = \frac{u}{\sinh u} = S(u)\,.
    \]
    Combining this identity with Lemma \ref{lem:1}, we get the result.
\end{proof}
\begin{lemma}\label{lem:4}
    Let $f(n,m) = \sum_{k=0}^n \binom{n}{k} B_{k+m}$. Then $f(n,m) = (-1)^{n+m} f(m,n)$.
\end{lemma}
\begin{proof}
    Consider the generating series
    \[
    F(u,v) = \sum_{n,m \geq 0} f(n,m) \frac{u^n}{n!} \frac{v^m}{m!}.
    \]
    By a simple calculation,
    \[
    F(u,v) = \frac{e^u(u+v)}{e^{u+v}-1}.
    \]
    Clearly $F(u,v) = F(-v,-u)$, so the result follows.
\end{proof}
In particular, Lemma \ref{lem:4} shows that the coefficients of $(\ell^C)^{k_1-1}\cup(-\ell^C-2\theta_2^C)^{k_2-1}$ in \eqref{eq:S} are symmetric under exchanging $k_1$ and $k_2$.

\subsection{Identity for the comparison}
We now prove the key Bernoulli identity used in the proof of Corollary~\ref{cor:gz=s}.
\begin{theorem}\label{sp:final}
    Let $E(b,c)$ be defined as in \eqref{eq:E}. For each $c$ with $2c\leq g$, we have
    \[
    \sum_{k=0}^c \binom{g-2k}{c-k} 2^{k} E(g-2k,k) = \frac{(-1)^g}{c!(g-c)!} \sum_{m=0}^{2c} \binom{2c}{m} B_{2g-2c+m}\,.
    \]
\end{theorem}

The following lemma will be used to simplify the binomial coefficients in Theorem~\ref{sp:final}. For a polynomial $p(x)$, we denote by $[x^m]\,p(x)$ the coefficient of the monomial $x^m$.

\begin{lemma}\label{lem:Hypergeometric}
    Let $g,c,n$ be non-negative integers such that $n \leq g$, and $c \leq g$. Consider the polynomial $\Phi_{g,c}(x) = (1-x)^{2c}(1+x)^{2g-2c}$. Then,
    \begin{equation}\label{eq:hypgeo}
        \sum_{k=0}^{\min\{c,n,g-n\}}\frac{n!(-1)^k2^{4k}}{(2n)!(2g-2n)!}\binom{g-2k}{c-k}\frac{(2(g-k))!}{(g-k)!}\frac{\binom{g-n}{k}}{(n-k)!} = \frac{1}{c!(g-c)!}[x^{2n}]\Phi_{g,c}(x)\, .
    \end{equation}
\end{lemma}
\begin{proof}
Let $\Gamma(z)$ denote the Gamma function. We use the Pochhammer symbol
$$
(a)_w = a(a+1)\ldots (a+w-1) = \frac{\Gamma(a+w)}{\Gamma(a)}\,,
$$
and the hypergeometric series
    $$
    \prescript{}{p}{F}_q\left(\begin{array}{c}
        a_1, \ldots , a_p \\
        b_1, \ldots , b_q
    \end{array};x\right) = \sum_{k \geq 0}\frac{(a_1)_k\ldots (a_p)_k}{(b_1)_k\ldots (b_q)_k}\frac{x^k}{k!}\, .
    $$
    For background on hypergeometric series, see \cite{hyperbook}. We write
    \begin{align}
    \begin{split}\label{eq:binom}
    \frac{n!\binom{g-n}{k}}{(n-k)!} &= \frac{(-n)_k(n-g)_k}{k!}\, ,\\
    \binom{g-2k}{c-k} &= \binom{g}{c}\frac{(-c)_k(c-g)_k}{(-g)_{2k}}\,, \\
    (-1)^k2^{4k}\frac{(2g-2k)!}{(g-k)!} & =2^{4k}\frac{(2g)!}{g!}\frac{(-g)_k}{(-2g)_{2k}}\,.
    \end{split}
    \end{align}

    We simplify the left-hand side of \eqref{eq:hypgeo} using the hypergeometric series. The duplication formula for the Gamma function,
    \[
    \Gamma(2z) =
    \frac{2^{2z-1}}{\sqrt{\pi}}\Gamma(z) \Gamma(z + \frac{1}{2})\,,
    \]
    implies that the Pochhammer symbol satisfies
    \[
    (a)_{2k} = 2^{2k}(a/2)_k(\frac{1}{2}+\frac{a}{2})_k\,.
    \]
    Therefore, the left-hand side of \eqref{eq:hypgeo} simplifies to
    $$
    \frac{\binom{g}{c}\frac{(2g)!}{g!}}{(2n)!(2g-2n)!}\sum_{k}\frac{(-n)_k(n-g)_k(-c)_k(c-g)_k(-g)_k}{(-\frac{g}{2})_k(\frac{1-g}{2})_k(-g)_k(-g+\frac{1}{2})_k\cdot k!} = \frac{\binom{2g}{2n}}{c!(g-c)!}\cdot{}_{4}{F}_3\left(\begin{array}{c}
        -n, n-g, -c, c-g \\
        -\frac{g}{2}, \frac{1-g}{2}, -g+\frac{1}{2}
    \end{array};1\right)\, .
    $$
    By \cite[(2.31)]{Hypergeometric}, we have
    $$
    {}_{4}{F}_3\left(\begin{array}{c}
        -n, n-g, -c, c-g \\
        -\frac{g}{2}, \frac{1-g}{2}, -g+\frac{1}{2}
    \end{array};1\right) = {}_{3}{F}_2\left(\begin{array}{c}
        -2c, 2c-2g, -n \\
        -g+\frac{1}{2}, -g
    \end{array};1\right)\, .
    $$

    Using again the duplication formula and rewriting the Pochhammer symbols in terms of factorials according to \eqref{eq:binom}, we obtain
    {\allowdisplaybreaks
    \begin{align*}
    \binom{2g}{2n}\cdot{}_{3}{F}_2\left(\begin{array}{c}
        -2c, 2c-2g, -n \\
        -g+\frac{1}{2}, -g
    \end{array};1\right) &= \binom{2g}{2n}\sum_{k}\frac{(-2c)_k(2c-2g)_k(-n)_k}{(-g+\frac{1}{2})_k (-g)_k k!}\\
    &=\binom{2g}{2n}\sum_{k}2^{2k}\frac{(-2c)_k(2c-2g)_k(-n)_k}{(-2g)_{2k} k!}\\
    &=\binom{2g}{2n}\sum_{k}(-4)^{k}\frac{(2c)!(2g-2c)!n!(2g-2k)!}{(2c-k)!(2g-2c-k)!(n-k)!(2g)! k!}\\
    &=\binom{2g}{2n}\binom{2g}{2c}^{-1}\sum_{k}(-4)^{k}\binom{2g-2k}{2c-k}\binom{n}{k}\\
    &=\binom{2g}{2n}\binom{2g}{2c}^{-1}\sum_{k}(-4)^{k}\binom{n}{k}[x^{2c-k}](1+x)^{2g-2k}\\
    &=\binom{2g}{2n}\binom{2g}{2c}^{-1}[x^{2c}]\sum_{k}\binom{n}{k}(-4)^{k}x^k(1+x)^{2g-2k}\\
    &=\binom{2g}{2n}\binom{2g}{2c}^{-1}[x^{2c}](1+x)^{2g}\left(1-4\frac{x}{(1+x)^2}\right)^n\\
    &=\binom{2g}{2n}\binom{2g}{2c}^{-1}[x^{2c}](1+x)^{2g}\frac{(1-x)^{2n}}{(1+x)^{2n}}\,.
    \end{align*}
    }
    Finally, for any $a,b,N$, one checks that
    $$
    \binom{N}{b}[x^a](1-x)^{b}(1+x)^{N-b} = \binom{N}{a}[x^b](1-x)^{a}(1+x)^{N-a}\, .
    $$
    This completes the proof.
\end{proof}

\begin{proof}[Proof of Theorem \ref{sp:final}]
    First, we rewrite $E(b,c)$ using the identity $(2x-1)!! = \frac{(2x)!}{2^xx!}$. We obtain
    \begin{align*}
        E(b,c)
        &=\frac{(-1)^{b+c}(2c+2b)!c!}{2^{2b+c}(b+c)!}\sum_{x=0}^b\frac{(2-2^{2c+2x})B_{2c+2x}\binom{c+b-x}{c}\binom{c+x}{c}}{(2c+2b-2x)!(2c+2x)!}\,.
    \end{align*}
    In particular,
    \begin{align*}
        E(g-2k,k) &=\frac{(-1)^{g+k}(2g-2k)!k!}{2^{2g-3k}(g-k)!}\sum_{x=0}^{g-2k}(2-2^{2k+2x})B_{2k+2x}\frac{\binom{g-k-x}{k}\binom{k+x}{k}}{(2(g-k-x))!(2(k+x))!}\,.
    \end{align*}

    Set
    $$
    F(g,c) = \sum_{k=0}^c \binom{g-2k}{c-k} 2^{k} E(g-2k,k)\,.
    $$
    Let $n=x+k$. Then $0 \le n \le g$ and $k \le n \le g-k$. Factoring out the Bernoulli numbers and applying Lemma~\ref{lem:Hypergeometric}, we obtain
    \begin{align*}
        F(g,c) &=\sum_{k=0}^{c}\sum_{x=0}^{g-2k}\binom{g-2k}{c-k}\frac{(-1)^{g+k}(2g-2k)!k!}{2^{2g-4k}(g-k)!}(2-2^{2k+2x})B_{2k+2x}\frac{\binom{g-k-x}{k}\binom{k+x}{k}}{(2(g-k-x))!(2(k+x))!}\\
        &=\frac{(-1)^g}{2^{2g}}\sum_{n=0}^g\frac{(2-2^{2n})B_{2n}}{(2n)!}\frac{n!}{(2g-2n)!}\sum_{k=0}^{\min\{c,n, g-n\}}(-1)^k2^{4k}\binom{g-2k}{c-k}\frac{\binom{g-n}{k}(2g-2k)!}{(n-k)!(g-k)!}\\
        &=\frac{(-1)^g}{2^{2g}c!(g-c)!}\sum_{n=0}^g(2-2^{2n})B_{2n}\sum_{m=0}^{2c}(-1)^m\binom{2c}{m}\binom{2g-2c}{2n-m}\\
        &=\frac{(-1)^g}{2^{2g}c!(g-c)!}\sum_{n=0}^g(2-2^{2n})B_{2n}[x^{2n}]\left((1-x)^{2c}(1+x)^{2g-2c}\right)\,.
    \end{align*}
    By Lemma \ref{lem:3}, and writing $\partial=\frac{d}{dx}$, we can rewrite the last expression as
    \begin{align*}
        F(g,c) &= \frac{(-1)^g}{c!(g-c)!}\frac{1}{2^{2g}}\left(\frac{\partial}{\sinh (\partial)}(1-x)^{2c}(1+x)^{2g-2c}\right)(0)\, .
    \end{align*}
    Using the identity $(e^{\partial }f)(0) = f(1)$, we compute
    \begin{align*}
        \left(\frac{\partial}{\sinh (\partial)}(1-x)^{2c}(1+x)^{2g-2c}\right)(0)&=\left(e^{\partial}\frac{2\partial}{e^{2\partial}-1}(1-x)^{2c}(1+x)^{2g-2c}\right)(0)\\
        &=\left(\frac{2\partial}{e^{2\partial}-1}(1-x)^{2c}(1+x)^{2g-2c}\right)(1)\\
        &=\sum_{m=0}^{\infty}2^{m}B_{m}[(x-1)^m]\left((1-x)^{2c}(1+x)^{2g-2c})\right)\\
        &=\sum_{m=0}^{\infty}2^mB_{m}\binom{2g-2c}{2g-m}2^{2g-m}\\
        &=2^{2g}\sum_{r=0}^{2g-2c} B_{r+2c}\binom{2g-2c}{r}\,.
    \end{align*}
    The result now follows from Lemma \ref{lem:4}.
\end{proof}

\newcommand{\etalchar}[1]{$^{#1}$}

\noindent Department of Mathematics, University of Michigan \\
\noindent younghan@umich.edu

\vspace{8pt}

\noindent Department of Mathematics, ETH Z\"urich\\
\noindent jeremy.feusi@math.ethz.ch

\vspace{8pt}

\noindent Department of Mathematics, ETH Z\"urich\\
\noindent aitor.iribarlopez@math.ethz.ch

\vspace{8pt}

\noindent Department of Mathematics, University of Rome ``La Sapienza'' \\
\noindent samouil.molcho@uniroma1.it

\vspace{8pt}


\begin{thebibliography}{BBGHdJ26}

\bibitem[ACM{\etalchar{+}}16]{SkeletonsAndFAbramo2015}
D.~Abramovich, Q.~Chen, S.~Marcus, M.~Ulirsch, and J.~Wise.
\newblock Skeletons and fans of logarithmic structures.
\newblock In {\em Nonarchimedean and tropical geometry}, Simons Symp., pages 287--336. Springer, 2016.

\bibitem[AW18]{Birational_inva_Abramo_2013}
D.~Abramovich and J.~Wise.
\newblock Birational invariance in logarithmic {G}romov-{W}itten theory.
\newblock {\em Compos. Math.}, 154(3):595--620, 2018.

\bibitem[Ale02]{Alexeev}
V.~Alexeev.
\newblock Complete moduli in the presence of semiabelian group action.
\newblock {\em Ann. of Math. (2)}, 155(3):611--708, 2002.

\bibitem[AN99]{AleNak}
V.~Alexeev and I.~Nakamura.
\newblock On {M}umford's construction of degenerating abelian varieties.
\newblock {\em Tohoku Math. J. (2)}, 51(3):399--420, 1999.

\bibitem[AF24]{AD}
G.~Ancona and D.~Fratila.
\newblock Ng\^o's support theorem and polarizability of quasi-projective commutative group schemes.
\newblock {\em \'Epijournal G\'eom. Alg\'ebrique}, 8:Art. 11, 10, 2024.

\bibitem[Ari11]{Arinkin1}
D.~Arinkin.
\newblock Cohomology of line bundles on compactified {J}acobians.
\newblock {\em Math. Res. Lett.}, 18(6):1215--1226, 2011.

\bibitem[Ari13]{Arinkin2}
D.~Arinkin.
\newblock Autoduality of compactified {J}acobians for curves with plane singularities.
\newblock {\em J. Algebraic Geom.}, 22(2):363--388, 2013.

\bibitem[AF16]{AF}
D.~Arinkin and R.~Fedorov.
\newblock Partial {F}ourier-{M}ukai transform for integrable systems with applications to {H}itchin fibration.
\newblock {\em Duke Math. J.}, 165(15):2991--3042, 2016.

\bibitem[BHP{\etalchar{+}}23]{BHPSS}
Y.~Bae, D.~Holmes, R.~Pandharipande, J.~Schmitt, and R.~Schwarz.
\newblock Pixton's formula and {A}bel-{J}acobi theory on the {P}icard stack.
\newblock {\em Acta Math.}, 230(2):205--319, 2023.

\bibitem[BMSY26]{BMSY1}
Y.~Bae, D.~Maulik, J.~Shen, and Q.~Yin.
\newblock On generalized {B}eauville decompositions.
\newblock {\em Selecta Math. (N.S.)}, 32(3):Paper No. 45, 2026.

\bibitem[BM26]{BM}
Y.~Bae and S.~Molcho.
\newblock (in preparation), 2026.

\bibitem[BMP25]{BMP}
Y.~Bae, S.~Molcho, and A.~Pixton.
\newblock Fourier transforms and {A}bel-{J}acobi theory.
\newblock {\em arXiv:2506.06116}, 2025.

\bibitem[BP26]{BP}
Y.~Bae and A.~Pixton.
\newblock (in preparation), 2026.

\bibitem[BS22]{BS1}
Y.~Bae and J.~Schmitt.
\newblock Chow rings of stacks of prestable curves {I}.
\newblock {\em Forum Math. Sigma}, 10:Paper No. e28, 47, 2022.
\newblock With an appendix by Bae, Schmitt and Jonathan Skowera.

\bibitem[Bea86]{Beauville86}
A.~Beauville.
\newblock Sur l'anneau de {C}how d'une vari\'{e}t\'{e} ab\'{e}lienne.
\newblock {\em Math. Ann.}, 273(4):647--651, 1986.

\bibitem[Bea10]{beauvilleSL2}
A.~Beauville.
\newblock The action of {$\rm SL_2$} on abelian varieties.
\newblock {\em J. Ramanujan Math. Soc.}, 25(3):253--263, 2010.

\bibitem[BBDG82]{BBDG}
A.~A. Beilinson, J.~Bernstein, P.~Deligne, and O.~Gabber.
\newblock Faisceaux pervers.
\newblock In {\em Analysis and topology on singular spaces, {I} ({L}uminy, 1981)}, volume 100 of {\em Ast\'erisque}, pages 5--171. Soc. Math. France, Paris, 1982.

\bibitem[BBGHdJ26]{BGHdJ}
A.~M. Botero, J.~I. Burgos~Gil, D.~Holmes, and R.~de~Jong.
\newblock Pure extension of the theta divisor over the moduli space of abelian varieties.
\newblock {\em arXiv:2602.22162}, 2026.

\bibitem[Bri96]{Brion}
M.~Brion.
\newblock Piecewise polynomial functions, convex polytopes and enumerative geometry.
\newblock In {\em Parameter spaces ({W}arsaw, 1994)}, volume~36 of {\em Banach Center Publ.}, pages 25--44. Polish Acad. Sci. Inst. Math., Warsaw, 1996.

\bibitem[CMOP25]{CMOP}
S.~Canning, S.~Molcho, D.~Oprea, and R.~Pandharipande.
\newblock Tautological projection for cycles on the moduli space of abelian varieties.
\newblock {\em Algebr. Geom.}, 12(6):736--768, 2025.

\bibitem[Cap94]{Caporaso94}
L.~Caporaso.
\newblock A compactification of the universal {P}icard variety over the moduli space of stable curves.
\newblock {\em J. Amer. Math. Soc.}, 7(3):589--660, 1994.

\bibitem[CCUW20]{AModuliStackCavali2017}
R.~Cavalieri, M.~Chan, M.~Ulirsch, and J.~Wise.
\newblock A moduli stack of tropical curves.
\newblock {\em Forum Math. Sigma}, 8:Paper No. e23, 93, 2020.

\bibitem[CH00]{CortiHanamura}
A.~Corti and M.~Hanamura.
\newblock Motivic decomposition and intersection {C}how groups. {I}.
\newblock {\em Duke Math. J.}, 103(3):459--522, 2000.

\bibitem[Del74]{Hodge3}
P.~Deligne.
\newblock Th\'eorie de {H}odge. {III}.
\newblock {\em Inst. Hautes \'Etudes Sci. Publ. Math.}, 44:5--77, 1974.

\bibitem[DG70]{SGA3II}
M.~Demazure and A.~Grothendieck, editors.
\newblock {\em Sch{\'e}mas en groupes. {II}: Groupes de type multiplicatif, et structure des sch{\'e}mas en groupes g{\'e}n{\'e}raux}, volume 152 of {\em Lecture Notes in Mathematics}.
\newblock Springer-Verlag, Berlin, 1970.
\newblock S{\'e}minaire de G{\'e}om{\'e}trie Alg{\'e}brique du Bois Marie 1962--64 (SGA 3).

\bibitem[DM91]{DM91}
C.~Deninger and J.~Murre.
\newblock Motivic decomposition of abelian schemes and the {F}ourier transform.
\newblock {\em J. Reine Angew. Math.}, 422:201--219, 1991.

\bibitem[EvdG05]{EvdG05}
T.~Ekedahl and G.~van~der Geer.
\newblock Cycles representing the top {C}hern class of the {H}odge bundle on the moduli space of abelian varieties.
\newblock {\em Duke Math. J.}, 129(1):187--199, 2005.

\bibitem[EdGFS26]{EdGS}
P.~Engel, O.~de~Gaay~Fortman, and S.~Schreieder.
\newblock Combinatorics and {H}odge theory of degenerations of abelian varieties: {A} survey of the {M}umford construction.
\newblock {\em arXiv:2507.15695}, 2026.

\bibitem[EGH10]{EGH10}
C.~Erdenberger, S.~Grushevsky, and K.~Hulek.
\newblock Some intersection numbers of divisors on toroidal compactifications of {$\mathcal{A}_g$}.
\newblock {\em J. Algebraic Geom.}, 19(1):99--132, 2010.

\bibitem[EV02]{EV04}
H.~Esnault and E.~Viehweg.
\newblock Chern classes of {G}auss-{M}anin bundles of weight 1 vanish.
\newblock {\em $K$-Theory}, 26(3):287--305, 2002.

\bibitem[FC90]{FC}
G.~Faltings and C.-L. Chai.
\newblock {\em Degeneration of abelian varieties}, volume~22 of {\em Ergebnisse der Mathematik und ihrer Grenzgebiete (3)}.
\newblock Springer-Verlag, Berlin, 1990.

\bibitem[Feu24]{J}
J.~Feusi.
\newblock The logarithmic picard group of a logarithmically smooth scheme, {E}{T}{H} {Z}\"urich, {M}aster thesis, 2024.

\bibitem[FIL]{troAV}
J.~Feusi and A.~Iribar~L\'{o}pez.
\newblock The moduli space of tropical abelian varieties.
\newblock in preparation.

\bibitem[Ful98]{Fulton}
W.~Fulton.
\newblock {\em Intersection theory}, volume~2 of {\em Ergebnisse der Mathematik und ihrer Grenzgebiete. 3. Folge. A}.
\newblock Springer-Verlag, Berlin, second edition, 1998.

\bibitem[GR04]{hyperbook}
G.~Gasper and M.~Rahman.
\newblock {\em Basic hypergeometric series}, volume~96 of {\em Encyclopedia of Mathematics and its Applications}.
\newblock Cambridge University Press, Cambridge, second edition, 2004.
\newblock With a foreword by Richard Askey.

\bibitem[Gro72]{grothendieck1973groupes}
A.~Grothendieck.
\newblock {\em Groupes de monodromie en g\'eom\'etrie alg\'ebrique. {I}}, volume Vol. 288 of {\em Lecture Notes in Mathematics}.
\newblock Springer-Verlag, Berlin-New York, 1972.
\newblock S\'eminaire de G\'eom\'etrie Alg\'ebrique du Bois-Marie 1967--1969 (SGA 7 I).

\bibitem[GZ14]{GZ}
S.~Grushevsky and D.~Zakharov.
\newblock The zero section of the universal semiabelian variety and the double ramification cycle.
\newblock {\em Duke Math. J.}, 163(5):953--982, 2014.

\bibitem[KKN08a]{KKN2}
T.~Kajiwara, K.~Kato, and C.~Nakayama.
\newblock Logarithmic abelian varieties.
\newblock {\em Nagoya Math. J.}, 189:63--138, 2008.

\bibitem[KKN08b]{KKN1}
T.~Kajiwara, K.~Kato, and C.~Nakayama.
\newblock Logarithmic abelian varieties. {I}. {C}omplex analytic theory.
\newblock {\em J. Math. Sci. Univ. Tokyo}, 15(1):69--193, 2008.

\bibitem[KKN15]{KKN4}
T.~Kajiwara, K.~Kato, and C.~Nakayama.
\newblock Logarithmic abelian varieties, {P}art {IV}: {P}roper models.
\newblock {\em Nagoya Math. J.}, 219:9--63, 2015.

\bibitem[KKN18]{KKN5}
T.~Kajiwara, K.~Kato, and C.~Nakayama.
\newblock Logarithmic abelian varieties, {P}art {V}: {P}rojective models.
\newblock {\em Yokohama Math. J.}, 64:21--82, 2018.

\bibitem[KKN21]{KKN7}
T.~Kajiwara, K.~Kato, and C.~Nakayama.
\newblock Logarithmic abelian varieties, part {VII}: moduli.
\newblock {\em Yokohama Math. J.}, 67:9--48, 2021.

\bibitem[KP19]{KP19}
J.~L. Kass and N.~Pagani.
\newblock The stability space of compactified universal {J}acobians.
\newblock {\em Trans. Amer. Math. Soc.}, 372(7):4851--4887, 2019.

\bibitem[Kat89]{Kato}
K.~Kato.
\newblock Logarithmic structures of {F}ontaine-{I}llusie.
\newblock In {\em Algebraic analysis, geometry, and number theory ({B}altimore, {MD}, 1988)}, pages 191--224. Johns Hopkins Univ. Press, Baltimore, MD, 1989.

\bibitem[Koo18]{Hypergeometric}
T.~H. Koornwinder.
\newblock Quadratic transformations for orthogonal polynomials in one and two variables.
\newblock In {\em Representation theory, special functions and {P}ainlev\'e{} equations---{RIMS} 2015}, volume~76 of {\em Adv. Stud. Pure Math.}, pages 419--447. Math. Soc. Japan, Tokyo, 2018.

\bibitem[KvdG10]{vdgK}
A.~Kouvidakis and G~van~der Geer.
\newblock The rank-one limit of the {F}ourier-{M}ukai transform.
\newblock {\em Doc. Math.}, 15:747--763, 2010.

\bibitem[Laz04]{Positive1}
R.~Lazarsfeld.
\newblock {\em Positivity in algebraic geometry. {I}}, volume~48 of {\em Ergebnisse der Mathematik und ihrer Grenzgebiete. 3. Folge. A}.
\newblock Springer-Verlag, Berlin, 2004.

\bibitem[MSY25]{PerverseFourierMSY25}
D.~Maulik, J.~Shen, and Q.~Yin.
\newblock Perverse filtrations and {F}ourier transforms.
\newblock {\em Acta Math.}, 234(1):1--69, 2025.

\bibitem[MSY26]{Daf}
D.~Maulik, J.~Shen, and Q.~Yin.
\newblock Dualizable abelian fibrations.
\newblock {\em arXiv:2602.19318}, 2026.

\bibitem[MRV19]{MRV1}
M.~Melo, A.~Rapagnetta, and F.~Viviani.
\newblock Fourier-{M}ukai and autoduality for compactified {J}acobians. {I}.
\newblock {\em J. Reine Angew. Math.}, 755:1--65, 2019.

\bibitem[MPS23]{MPS}
S.~Molcho, R.~Pandharipande, and J.~Schmitt.
\newblock The {H}odge bundle, the universal 0-section, and the log {C}how ring of the moduli space of curves.
\newblock {\em Compos. Math.}, 159(2):306--354, 2023.

\bibitem[MR24]{A_case_study_of_Molcho_2021}
S.~Molcho and D.~Ranganathan.
\newblock A case study of intersections on blowups of the moduli of curves.
\newblock {\em Algebra Number Theory}, 18(10):1767--1816, 2024.

\bibitem[MW22]{Molcho_Wise_2022}
Samouil Molcho and Jonathan Wise.
\newblock The logarithmic picard group and its tropicalization.
\newblock {\em Compositio Mathematica}, 158(7):1477–--1562, 2022.

\bibitem[Mum69]{Mum_ext}
D.~Mumford.
\newblock Bi-extensions of formal groups.
\newblock In {\em Algebraic {G}eometry ({I}nternat. {C}olloq., {T}ata {I}nst. {F}und. {R}es., {B}ombay, 1968)}, volume~4 of {\em Tata Inst. Fundam. Res. Stud. Math.}, pages 307--322. Tata Inst. Fund. Res., Bombay, 1969.

\bibitem[Mum72]{Mumfordanalytic}
D.~Mumford.
\newblock An analytic construction of degenerating abelian varieties over complete rings.
\newblock {\em Compositio Math.}, 24:239--272, 1972.

\bibitem[Mum83]{mumford_kodaira}
D.~Mumford.
\newblock On the {K}odaira dimension of the {S}iegel modular variety.
\newblock In {\em Algebraic geometry---open problems ({R}avello, 1982)}, volume 997 of {\em Lecture Notes in Math.}, pages 348--375. Springer, Berlin, 1983.

\bibitem[Nam76]{namikawa76}
Y.~Namikawa.
\newblock A new compactification of the {S}iegel space and degeneration of {A}belian varieties. {II}.
\newblock {\em Math. Ann.}, 221(3):201--241, 1976.

\bibitem[Nam80]{Namikawa}
Y.~Namikawa.
\newblock {\em Toroidal compactification of {S}iegel spaces}, volume 812 of {\em Lecture Notes in Mathematics}.
\newblock Springer, Berlin, 1980.

\bibitem[Ng{\^o}10]{Ngo}
B.~C. Ng{\^o}.
\newblock Le lemme fondamental pour les alg\`ebres de {L}ie.
\newblock {\em Publ. Math. Inst. Hautes \'Etudes Sci.}, 111:1--169, 2010.

\bibitem[Ogu18]{LecturesOnLogOgus2018}
A.~Ogus.
\newblock {\em Lectures on logarithmic algebraic geometry}, volume 178 of {\em Cambridge Studies in Advanced Mathematics}.
\newblock Cambridge University Press, Cambridge, 2018.

\bibitem[PRSS25]{LogarithmicTauPandha2024}
R.~Pandharipande, D.~Ranganathan, J.~Schmitt, and P.~Spelier.
\newblock Logarithmic tautological rings of the moduli spaces of curves.
\newblock {\em Adv. Math.}, 474:Paper No. 110291, 100, 2025.

\bibitem[Pin90]{Pinkthesis}
R.~Pink.
\newblock {\em Arithmetical compactification of mixed {S}himura varieties}, volume 209 of {\em Bonner Mathematische Schriften [Bonn Mathematical Publications]}.
\newblock Universit\"at Bonn, Mathematisches Institut, Bonn, 1990.
\newblock Dissertation, Rheinische Friedrich-Wilhelms-Universit\"at Bonn, Bonn, 1989.

\bibitem[{Sta}26]{stacks-project}
The {Stacks project authors}.
\newblock The stacks project.
\newblock \url{https://stacks.math.columbia.edu}, 2026.

\bibitem[vdG99]{vdG99}
G.~van~der Geer.
\newblock Cycles on the moduli space of abelian varieties.
\newblock In {\em Moduli of curves and abelian varieties}, volume E33 of {\em Aspects Math.}, pages 65--89. Friedr. Vieweg, Braunschweig, 1999.

\bibitem[vdG13]{vdGcoh}
G.~van~der Geer.
\newblock The cohomology of the moduli space of abelian varieties.
\newblock In {\em Handbook of moduli. {V}ol. {I}}, volume~24 of {\em Adv. Lect. Math. (ALM)}, pages 415--457. Int. Press, Somerville, MA, 2013.

\end{thebibliography}
\end{document}